\documentclass[11pt]{article}
\usepackage{latexsym}
\usepackage{amssymb}
\usepackage{graphicx}
\usepackage{enumerate}
\usepackage{amssymb}
\usepackage{amsmath}
\usepackage{wasysym}
\usepackage{color}
\usepackage{tikz}
\usepackage{hyperref}
\hypersetup{hypertex=true,
	colorlinks=true,
	linkcolor=blue,
	anchorcolor=blue,
	citecolor=blue}
\usepackage{amsthm}
\usepackage{cleveref}
\usepackage{xcolor}
\crefformat{section}{\S#2#1#3} 
\crefformat{subsection}{\S#2#1#3}
\crefformat{subsubsection}{\S#2#1#3}
\usepackage{imakeidx}
\makeindex[columns=2, title={Index of Parts I and III}, intoc]
\makeindex[name=p2, columns=2, title={Index of Parts II}, intoc]

\numberwithin{equation}{section}

\usepackage{tikz}
\theoremstyle{plain}
\newtheorem{thm}{{T}heorem}[section]
\newtheorem{assum}[thm]{{A}ssumption}
\newtheorem{cond}[thm]{{C}ondition}
\newtheorem{obs}[thm]{{O}bservation}
\newtheorem{lem}[thm]{{L}emma}
\newtheorem{cor}[thm]{{C}orollary}
\newtheorem{prop}[thm]{{P}roposition}
\newtheorem{rem}[thm]{{R}emark}
\theoremstyle{definition}

\newtheorem{defn}[thm]{{D}efinition}
\def\rW{\mathrm{W}}

\def\hS{{\hat{S}}}

\def\hE{{\hat{E}}}

\def\hV{{\hat{V}}}
\def\hE{{\hat{E}}}
\def\chH{{\check{H}}}
\def\Cozt{{\mathcal{C}_{t,T}}}	
\def\Dozt{{\mathcal{D}_{t,T}}}	
\def\hV{{\hat{V}}}
\def\hE{{\hat{E}}}
\def\hA{{\hat{A'}}}
\def\chH{{\check{H}}}
\def\fh{{\mathfrak{h}}}

\def\eb{{\mathbf{e}}}
\def\pb{{\mathbf{p}}}
\def\tP{\tilde{\P}}

\def\ps{{\frac{\p}{\p s}}}
\def\bZ{\mathcal{Z}}
\def\tiT{{\widetilde{\T}}}
\def\btiT{{\overline{\widetilde{\T}}}}
\def\bfD{{\mathbf{D}}}

\def\bd{\mathrm{bd}}

\def\rs{\mathrm{s}}
\def\ru{\mathrm{u}}

\def\cC{{\mathcal{C}}}

\def\hotrf{h^{\cF}_{1,T,R}}

\def\tD{{\tilde{\Delta}}}
\def\chH{{\check{H}}}
\def\cF{{\mathcal{F}}}
\def\A{\mathcal{A}}
\def\tl{{\tilde{\l}}}

\def\V{{\mathcal{V}}}
\def\H{\mathcal{H}}
\def\bfE{{\mathbf{E}}}
\def\dims{{\dim(S)}}

\def\trW{\widetilde{\rW}}
\def\lan{\langle}

\def\bfT{\mathbf{T}}
\def\M{\mathbf{M}}

\def\bbf{\mathfrak{f}}

\def\bfH{{\mathbf{H}}}
\def\bfR{{\mathbf{R}}}
\def\ran{\rangle}

\def\Re{\mathrm{Re}}
\def\la{\mathrm{la}}
\def\mL{\mathrm{L}}
\def\cE{{\mathcal{E}}}
\def\ppi{{{\pi^{\mathrm{sp}}}}}

\def\sm{\mathrm{sm}}
\def\wch{{\widetilde{\ch}}}

\def\End{\mathrm{End}}

\def\K{\mathbb{K}}

\def\M{\mathcal{M}}

\def\rank{\mathrm{rank}}
\def\Cc{\mathcal{C}}
\def\tf{\tilde{f}}

\def\rel{\mathrm{rel}}
\def\bd{\mathrm{bd}}

\def\lan{\langle}
\def\ran{\rangle}

\def\abs{\mathrm{abs}}
\def\Tr{\operatorname{Tr}}
\def\B{{\mathcal{B}}}

\def\Df{{\mathbf{D}}}

\def\bz{{ \overline{Z}}}
\def\bm{ \overline{M}}

\def\Hess{\mathrm{Hess}}

\def\Z{\mathbb{Z}}

\def\ch{h}
\def\wch{\widetilde{\ch}}
\def\supp{\mathrm{supp}}

\def\P{{\mathcal{P}}}
\def\dvol{\mathrm{dvol}}

\def\l{\lambda}
\def\a{\alpha}

\def\I{\mathcal{I}}
\def\b{\beta}

\def\R{\mathbb{R}}
\def\F{ \overline{F}}
\def\hc{\hat{c}}

\def\C{\mathbb{C}}
\def\sqt{\sqrt{\frac{\kappa}{A}}}
\def\Im{\mathrm{Im}}

\def\Dom{\mathrm{Dom}}

\def\cD{\mathcal{D}}

\def\half{\frac{1}{2}}
\def\Id{\operatorname{Id}}

\def\Cb{{\underline{\C}}}

\def\cn{c}
\def\J{\mathcal{J}}
\def\K{\mathcal{K}}
\def\T{\mathcal{T}}
\def\Cn{C}

\def\bh{{\mathbf{h}}}
\def\T{\mathcal{T}}

\begin{document}
	
	\def\o{\omega}
	\def\O{\Omega}
	\def\p{\partial}
	\def\s{{\theta}}
	\def\a{\alpha}
	\def\half{\frac{1}{2}}
	\def\l{\lambda}
	\def\dist{\mathrm{dist}}
	\def\Zs{{Z_\s}}
	\def\Zso{{Z_{\s,1}}}
	\def\Zsw{{Z_{\s,2}}}
	\def\bzs{{\bz_\s}}
	\def\bzso{{\bz_{\s,1}}}
	\def\bzsw{{\bz_{\s,2}}}
	\def\chS{{\check{S}}}
	\def\chE{{\check{E}}}
	\def\chA{{\check{A'}}}
	\def\hchS{{\hat{\check{S}}}}
	\def\hchE{{\hat{\check{E}}}}
	\def\hchA{{\hat{\check{A'}}}}
	
	\newcommand{\be}{\begin{equation}}
		\newcommand{\ee}{\end{equation}}	
	\newcommand{\ba}{\begin{aligned}}
		\newcommand{\ea}{\end{aligned}}
	
\newcommand{\bs}{\begin{split}}
	\title{Generalized Morse Functions, Excision and Higher Torsions}
	\date{}                                           
	
	\author{Martin Puchol
		\footnote{
			Universit\'e Paris-Saclay, CNRS, Laboratoire de Math\'ematiques d’Orsay, 91405, Orsay, France, martin.puchol@math.cnrs.fr. 
		}
		\and Junrong Yan
		\footnote{Department of Mathematics, Northeastern University, 02215, Boston, USA, j.yan@northeastern.edu.
		} 
	}
	
	\maketitle

		\begin{abstract}
Comparing invariants from both topological and geometric perspectives is a central theme in index theory. In this paper, we compare higher analytic and topological torsions and establish a version of the higher Cheeger--M\"uller/Bismut--Zhang theorem. Bismut--Goette \cite{bismut2001families} previously achieved this comparison under the assumption of the existence of fiberwise Morse functions satisfying the fiberwise Thom--Smale transversality condition (TS condition). To fully generalize the theorem, it is natural to remove this assumption. Notably, unlike fiberwise Morse functions, fiberwise generalized Morse functions (GMFs) always exist\footnote{See Remark~\ref{existence}}. In this work, we extend the framework of \cite{bismut2001families} by considering a fibration $M \to S$ equipped with a unitarily flat complex vector bundle $F \to M$ and a fiberwise GMF $f$, while retaining the TS condition. A key difficulty in this generalization lies in analyzing the birth--death singularities intrinsic to GMFs, which do not appear in the classical Morse case. 
\end{abstract}


	\tableofcontents
	
	\section{Introduction}

	\subsection{Overview}
	
	In \cite{ray1971r}, Ray and Singer introduced the analytic torsion (RS-torsion) associated with a unitarily flat vector bundle over a closed Riemannian manifold $M$, and conjectured that this analytic torsion coincides with the classical Reidemeister-Franz torsion (RF-torsion). The latter was constructed by Franz \cite{franz1935torsion}, Reidemeister \cite{reidemeister1935homotopieringe} and de Rham \cite{rham1950complexes} when $F$ is acyclic, and subsequently extended to the non-acyclic case  \cite{rham1950complexes, milnor1966whitehead, MR35437}. It was the first algebraic-topological invariant which distinguished between certain homotopy equivalent but non-homeomorphic spaces \cite{franz1935torsion,reidemeister1935homotopieringe}. The Ray-Singer conjecture was later proved independently by Cheeger \cite{cheeger1979analytic} and M\"uller \cite{muller1978analytic}, and is known as the Cheeger-M\"uller theorem. Moreover, the theorem was extended to broader classes of flat vector bundles: M\"uller extended it to unimodular flat bundles \cite{muller1993analytic}, while simultaneously Bismut and Zhang extended it to arbitrary flat vector bundles \cite{bismutzhang1992cm}. The (extended) theorem is now known as the Cheeger-M\"uller/Bismut-Zhang theorem. Further generalizations include equivariant settings \cite{bz2,lr,lu} and manifolds with boundary \cite{bruning2013gluing}.

	Wagoner \cite{MR0520491} conjectured that RF-torsion and RS-torsion can be extended to define invariants of a smooth fibration $\pi: M \rightarrow S$ with a flat complex vector bundle $F\to M$ on the total space $M$.  Bismut and Lott \cite{bismut1995flat} confirmed the analytic part of Wagoner’s conjecture by
	constructing analytic torsion forms (BL-torsion), which are even forms on $ S $. 
	These forms appear in the Grothendieck-Riemann-Roch (GRR) theorem for flat bundles (see \eqref{GRR-Bismut-Lott}) proved in \cite{bismut1995flat}. Lott also used BL-torsion to define a secondary K-theory \cite{lott2000secondary}. An interesting relation between the Bismut-Freed connection and analytic torsion forms was observed by Dai and Zhang \cite{DaiZhang}.
	
	On the topological side, Igusa developed the Igusa-Klein torsion (IK-torsion), a higher topological torsion \cite{igusa2002higher}. As an application of IK-torsion, Goette, Igusa, and Williams \cite{goette2014exotic1,goette2014exotic} uncovered exotic smooth structures in fiber bundles. Also, Dwyer, Weiss, and Williams \cite{dwyer2003parametrized} constructed another version of higher topological torsion (DWW-torsion).

A natural and fundamental problem is to understand the relationship between higher topological torsion and analytic torsion forms, namely whether a higher Cheeger--M\"uller/Bismut--Zhang theorem holds.

	There are two approaches to to establishing such a higher Cheeger-M\"uller/Bismut-Zhang theorem. The first approach was developed by Bismut and Goette \cite{bismut2001families}, who extended the argument of Bismut and Zhang \cite{bismutzhang1992cm}. Under the assumption that there exists a fiberwise Morse function $f\colon M\to \R$  satisfying the fiberwise Thom-Smale conditions, they compared the BL-torsion with a version of the higher topological torsion constructed from the Morse complex of $f$.  Goette \cite{goette2001morse,goette2003morse} later extended these results to the case of fiberwise Morse functions whose gradient vector fields do not necessarily satisfy the Thom–Smale conditions. In addition, Bismut and Goette \cite{bismut2001families} studied BL-torsion in the equivariant case. Further related works on torsion forms in the equivariant settings  are discussed in \cite{bismut2004equivariant, bunke2000equivariant}. For a comprehensive overview of higher torsions, we refer the reader to the survey by Goette \cite{goette2009torsion}.
	
	The second approach is the axiomatization method. Higher torsion invariants were axiomatized by Igusa \cite{igusa2008axioms}, who moreover showed that the IK-torsion satisfies his axioms. Igusa’s work is based on two main axioms: the additivity axiom and the transfer axiom. Igusa proved that any higher torsion invariant satisfying his axioms is necessarily a linear combination of IK-torsion and the higher Miller-Morita-Mumford class \cite{morita1987characteristic,mumford1983towards, miller1986homology}. Badzioch, Dorabiala, Klein and Williams \cite{badzioch2011equivalence} established that the DWW-torsion satisfies Igusa’s axioms. For the BL-torsion, the transfer axiom follows from the work of Ma \cite{ma1999functoriality}. 
	The remaining additivity axiom was proved more recently by the author M.P. together with Zhang and Zhu \cite{puchol2020adiabatic}. The author J.Y. also provides an alternative proof \cite{Yanforms}.  Combining  \cite{ma1999functoriality,puchol2020adiabatic} and Igusa’s axiomatization theorem, M. P., Zhang, and Zhu were able to prove a version of the higher Cheeger-M\"uller/Bismut-Zhang theorem for \textbf{trivial} bundles in \cite{puchol2021comparison}.

	Our work, in a way, integrates the two methods mentioned above, as we explain below.

On the one hand, we we seek to make full use of the results and methods of Bismut--Goette \cite{bismut2001families}, but without assuming the existence of fiberwise Morse functions. We note the following.
\begin{rem}\label{existence}
According to the work of Eliashberg and Mishachev \cite{eliashberg2000wrinkling}, a fiberwise generalized Morse function always exists. 
\end{rem}

In \cref{combitor}, assuming that we have a generalized Morse function $f: M \to \R$ satisfying certain fiberwise Thom-Smale transversality conditions (see Definition \ref{defn38} and Definition \ref{defn310}), we will define a higher combinatorial torsion using the Thom-Smale complex of a double suspension $(\M,\bbf)$ of $(M,f)$. This space consists in an enlargement of the fiber, together with a function on this enlarged space; see Definition \ref{doublesuspension}. The notion of multiple suspension was introduced in \cite[Definition 4.5.3]{igusa2002higher} to define IK-torsion.

Using new techniques developed by J.Y. \cite{Yantorsions,Yanforms}, together with Agmon-type estimates established by J.Y. and Dai \cite{DY2020cohomology}, we reduce the problem to two cases. Either the problem falls in the framework of Bismut--Goette \cite{bismut2001families}, where the function is fiberwise Morse and satisfies the TS condition, or to a similar setting in which the fibers are manifolds with boundary, see \Cref{idea} for more details.

On the other hand, complementing the previous approach, we incorporate the axiomatization method. Indeed, it was used by M.P., Zhang and Zhu  to establish the higher Cheeger-M\"uller/Bismut-Zhang for trivial bundles  in \cite{puchol2021comparison}. To extend the theorem to a general unitarily flat bundle $F$, we can consider its relative version, obtained by subtracting the torsion associated with the trivial flat bundle of rank $ \mathrm{rk}(F) $.

Lastly, we describe some potential applications of the higher Cheeger--M\"uller/Bismut--Zhang theorem. 
Note that if the higher Cheeger--M\"uller/Bismut--Zhang theorem is established in a general setting, it can be used to transfer known results from one type of torsion invariant to the other. 
For instance, to the best of our knowledge, the rigidity theorem for the Bismut--Lott torsion proved by Bismut and Goette, announced in \cite{bismut2000rigidite} and contained in \cite{bismut2001families}, does not yet admit a topological analogue. 
Also, the comparison of higher torsion invariants is closely related to the transfer index conjecture formulated by Bunke and Gepner \cite{buge}.

	\subsection{Basic settings}
    \label{sect-basic-setting}
	
	Let $\pi: M \to S$ be a smooth fibration with closed fiber $Z$ of dimension $n+1$. Let $(F \to M,\nabla^F)$ be a unitarily flat vector bundle of rank $m$ equipped with a compatible Hermitian metric $h^F$ (i.e. $\nabla^F h^F = 0$). Let $\Cb^m \to M$ be the trivial complex bundle of rank $m$, equipped with the canonical Hermitian metric $h^{\Cb^m}$ and the canonical flat connection $\nabla^{\Cb^m}$. Let $TZ \to M$ be the vertical subbundle of $TM \to M$. With a metric $g^{TZ}$ on $TZ \to M$, and a splitting 
	\[TM = T^H M \oplus TZ,\]
	one can define the Bismut-Lott analytic torsion form $\mathcal{T}(T^H M, g^{TZ}, h^F) \in  \Omega^{\bullet}(S)$ (see \cite{bismut1995flat} or \cref{defnbl}), whose degree 0 component is the fiberwise RS-torsion.
	
	Let $d^Z :  \Omega^{\bullet}(Z, F|_Z) \to \Omega^{\bullet+1}(Z, F|_Z)$ denote the exterior differentiation along fibers induced by $\nabla^F$, and let $H := H^\bullet(Z, F|_Z)$ be the associated de Rham cohomology. Then $H \to S$ carries a natural Gauss-Manin flat connection $\nabla^H$ (cf. \cite{bismut1995flat}). We fix a Hermitian metric $h^H$  on $H \to S$.
	
	With these data, we will define a modified the Bismut-Lott torsion form $\tilde{\mathcal{T}}(T^H M, g^{TZ}, h^F, h^H)$  as in \cite[Definition 2.8]{goette2009torsion} (see Definition \ref{assotor}).
	Similarly, for the trivial Hermitian flat bundle $(\Cb^m \to M, h^{\Cb^m}, \nabla^{\Cb^m})$ and a Hermitian metric $h^{H_{\Cb^m}}$ on $H_{\Cb^m} := H(Z, \Cb^m|_Z)$, we define the corresponding modified torsion form $\widetilde{\mathcal{T}}(T^H M, g^{TZ}, h^{\Cb^m}, h^{H_{\Cb^m}})$.
    
      In this paper, we will study a relative version of modified torsion forms  $$\overline{\widetilde{\mathcal{T}}}(T^H M, g^{TZ}, h^F, h^H, h^{H_{\Cb^m}}),$$ defined as the positive degree component of $$\widetilde{\mathcal{T}}(T^H M, g^{TZ}, h^F, h^H) - \widetilde{\mathcal{T}}(T^H M, g^{TZ}, h^{\Cb^m}, h^{H_{\Cb^m}}).$$ 
	Throughout this paper,  a bar `` $  \bar{ } $ " added on a torsion  or characteristic form denotes its relative version, obtained by subtracting the corresponding form associated with the trivial bundle, and then removing the degree 0 component.

	Now, let $f:M\to\R$ be fiberwise generalized Morse function $f$ (See Definition \ref{generalizedMorse}),i.e., a smooth function such that, for each fiber, the critical points of $f_\s := f|_{Z_\s}$ for all $\s \in S$ are either nondegenerate or birth-death points.
	
	Let $\Sigma^i\left(f_\s\right)$ be the set of all critical points of $f_\s$ of index $i$, and set
	$$
	\Sigma^{i}(f):=\cup_{\s \in S} \Sigma^{ i}\left(f_\s\right) \times \{\s\} \quad \text{and}\quad\Sigma(f):=\cup_{i}\Sigma^i(f).
	$$
	
	Let $\Sigma^{(1, i)}\left(f_\s\right) \subset \Sigma\left(f_\s\right)$ be the set of all birth-death points of $f_\s$ of index $i$, and set
	$$
	\Sigma^{(1, i)}(f)=\cup_{\s \in S} \Sigma^{(1, i)}\left(f_\s\right) \times \{\s\} \subset \Sigma(f) .
	$$

    Define $\Sigma^{(1)}(f_\theta)$ and $\Sigma^{(1)}(f)$ as the unions over all $i$ of $\Sigma^{(1,i)}(f_\theta)$ and $\Sigma^{(1,i)}(f)$, respectively. Similarly, let $\Sigma^{(0, i)}(f)$ and $\Sigma^{(0)}(f)$ denote the sets of nondegenerate points of index $i$ and of all indices, respectively. Then $$\Sigma(f)=\Sigma^{(0)}(f)\cup\Sigma^{(1)}(f).$$ We emphasize out that $\Sigma^1$ refers to the set of all critical points of Morse index 1, while $\Sigma^{(1)}$  refers to the set of all birth-death points. Let \be L := \pi(\Sigma^{(1)}(f)) \quad\text{and}\quad S' := S - L.\ee Up to performing a small perturbation of $f$, we can assume that $L$ is a $(\dim(S)-1)$-immersed submanifold of $S$. 
	
	Next, to better control birth-death points, we enlarge the fiber as follows.
	\begin{defn}[Double Suspension]\label{doublesuspension}
		Let $f:M\to \R$ be a smooth function. A double suspension of the pair $(M,f)$ is the pair $(\M,\bbf)$, where
		\[\M=M\times\R^N\times\R^N \quad\text{ and }\quad  \bbf(p,x_1,x_2)=f(p)-|x_1|^2+|x_2|^2\text{ for }(p,x_1,x_2)\in M\times\R^N\times\R^N,\]
		where $N$ is a large enough {even} number.
		
		If $\pi:M \to S$ is a smooth fibration with fiber $Z$, then $\pi^{\mathrm{sp}}: \M \to S$ is a smooth fibration with fiber $\bZ := Z \times \mathbb{R}^N \times \mathbb{R}^N$, and the unitarily flat vector bundle $(F,\nabla^F,h^F)$ on $M$ induces a unitarily flat vector bundle $(\cF,\nabla^\cF,h^\cF)$ on $\M$.
	\end{defn}
	
	Let $(\M,\bbf)$ be a double suspension of $(M,f)$ for a sufficiently large $N$. Then $\Sigma(\bbf) = \Sigma(f) \times \{0\} \times \{0\}$. In \cref{combitor}, we will prove the existence of a metric $g^{T\bZ}$ on $T\bZ \to \M$ such that:
	\begin{itemize}
		\item Outside a compact subset $K$, $g^{T\bZ}$ coincides with $g^{TZ} \oplus g^{T\R^N} \oplus g^{T\R^N}$, where $g^{T\R^N}$ is the standard metric on $\R^N$ .
		\item Birth-death points are independent of all other critical points (see Definition \ref{defn39}).
	\end{itemize}
 More precisely, we will choose $g^{T\bZ}$ so that Condition \ref{5a'} is satisfied.
    
	In \cref{defct}, we associate a $\Z$-graded complex vector bundle $V \to S'$ with the pair $(\M,\bbf)$, together with a flat superconnection $\mathcal{A}'$ on $V \to S'$ induced by $\nabla^F$ and the metric $g^{T\bZ}$ that satisfies the conditions mentioned above. Here the rank of $V$ may vary across different connected component of $S'$. Using this setup, we define the higher combinatorial torsion form $\mathcal{T}(\mathcal{A}', h^{V})$ on $S'$. Let $\bfH$ be the $ L^2 $-cohomology of $( \Omega^{\bullet}(\bZ, \cF|_{\bZ}), d^{\bZ})$ with respect to the $ L^2 $ metric induced by $ g^{T\bZ} $ and $e^{-2\bbf}h^{\cF} $. Then, roughly, $\bfH$ is a shifted version of $H$ (see \eqref{isom-H_bfH} and \eqref{eq157}). Next, we define the modified combinatorial torsion form $\widetilde{\mathcal{T}}(\mathcal{A}' , h^{V }, h^{\bfH })$   and a relative version of modified combinatorial torsion form $\overline{\widetilde{\mathcal{T}}}(\mathcal{A}' , h^{V }, h^{\bfH },h^{\bfH_{\Cb^m}})$ on $S'$. Moreover, we show that all of these forms extend smoothly to $S$. Based on the discussion in \cite[$\S$10]{goette2003morse} (particularly \cite[Remark 10.7]{goette2003morse}) and on the framing principle \cite[$\S$6]{igusa2002higher}, $\mathcal{T}(\mathcal{A}' , h^{V})$ should be equivalent to higher Igusa-Klein torsion. In this paper, however, we focus on addressing the analytical difficulties near birth-death points, and the precise relation between $\mathcal{T}(\mathcal{A}' , h^{V })$ and the higher Igusa-Klein torsion will be studied in a subsequent work.

	\subsection{Main result}
	
	Our main result compares the relative modified torsions introduced above. To define these torsions, we used metrics $h^H$, $h^{H_{\Cb^m}}$ on the one hand and $ h^{\bfH }$, $h^{\bfH_{\Cb^m} }$ on the other. To compare the corresponding torsions, these metrics should be linked as follows. We require that $h^{\bfH}$ is the image of $h^H$ under the isomorphism $c_1\colon H\simeq \bfH$ constructed in \eqref{isom-H_bfH}. In the same way, we require that $h^{\bfH_{\Cb^m}}$ is the image of $h^{H_{\Cb^m}}$ under the analogous isomorphism $H_{\Cb^m}\simeq \bfH_{\Cb^m}$, obtained by replacing $F$ with the trivial bundle ${\Cb^m}$.

    Let $(\M,\bbf)$ be a large enough double suspension of $(M,f)$. We consider a metric $g^{T\bZ}$ on the fibers of $\M$ as explained in \cref{sect-basic-setting}.
    
	\begin{thm}\label{main1}Let $f\in C^\infty(M)$ be a fiberwise generalized Morse function.
		Suppose $(\bbf, g^{T\bZ})$ satisfies the fiberwise Thom-Smale transversality condition (see Definition \ref{defn38}) and $F\to M$ is unitarily flat. Then
		in $ \Omega^{\bullet}(S)/d^S \Omega^{\bullet}(S)$,
		$$\btiT(T^HM,g^{TZ},h^F,h^H,h^{H_{\Cb^m}})=\btiT(\mathcal{A}' , h^{V }, h^{\bfH },h^{\bfH_{\Cb^m} }).$$
	\end{thm}
	\begin{rem}
		This paper focuses on addressing the analytical difficulties near birth-death points. Removing fiberwise Thom-Smale transversality conditions is, in some sense, more of an algebraic problem, which we leave for futur work. A first step in this direction is to define the combinatorial torsion form. This can be achieved using a similar idea as in \cref{defct} (particularly \Cref{obs5}) and giving a smooth version of Igusa-Klein's higher algebra structure (see \cite[$\S$3 and $\S$4]{igusa2002higher} or \cite[$\S$2]{igusa2005higher}), as Goette does in \cite[$\S$1]{goette2001morse}. Once the combinatorial torsion form is defined, the comparison with Bismut–Lott torsion can be approached by generalizing the analytic techniques developed in this paper to the setting without the transversality condition. This would ultimately yield a fully generalized higher Cheeger-M\"uller/Bismut-Zhang theorem.


	\end{rem}
	
	Let $w\in \Omega^{\bullet}(S)$ be such that $\int_{K}w=0$ for any closed oriented submanifold $ K \subseteq S $. By the Stokes theorem, we have in particular $\int_{B}dw=\int_{\partial B}w=0$ for any embedded closed ball $ B \subseteq S $, hence $dw=0$. Let $[w]$ be the de Rham cohomology class of $w$. Our hypothesis on $w$, Thom’s realization theorem and the de Rham theorem then imply that $[w]=0$. As a conclusion, $w=0$ in $ \Omega^{\bullet}(S)/d^S \Omega^{\bullet}(S)$. From this discussion, we see that to prove Theorem \ref{main1}, it is sufficient to consider the case where $ S $ is a closed manifold.
	



\subsection{Main ideas and difficulties}\label{idea}

To illustrate our approach, we temporarily assume that each fiber contains at most one birth-death point. Then $\Sigma^{(1)}(f_\s)$ consists of 0 or 1 point for every $\s\in S$. We further temporarily assume that $\Sigma^{(1)}(f) \subset M$ and $L = \pi(\Sigma^{(1)}(f)) \subset S$ are embedded submanifolds, with $L$ of codimension 1.

Let $\Omega_1\subset S$  be a small tubular neighborhood of $L$,  and let $\V\subset M$ be a small tubular neighborhood   of $\Sigma^{(1)}(f)$ such that $\pi(\V)=\Omega_1$. Assume that $\pi|_{\V}:\V\to \Omega_1$ is a bundle such that $\Df_\theta:=(\pi|_\V)^{-1}(\s)$ is a small $n+1$-ball for $\s \in\Omega_1$. For $\s\in L$, $\Df_\s$ contains the unique birth-death points in $Z_\s$. 

Let $\Omega_0\subset\subset S-L$ be an open subset such that $\{\Omega_1,\Omega_0\}$ covers $S$. On  ${\pi^{-1}(\Omega_0)}$, the function $f$ is  fiberwise Morse. One may then expect Bismut-Goette's result \cite{bismut2001families} to apply in $\Omega_0$, so that Theorem \ref{main1} holds in $\Omega_0$. 


To establish Theorem \ref{main1} on $ \Omega_1 $, we use techniques introduced in \cite{Yantorsions,Yanforms} by J.Y. to study the gluing formula for analytic torsion and analytic torsion forms. The idea is to perform a Witten-type deformation , but for functions that are  not Morse functions\footnote{ In \cite{puchol2020adiabatic}, a similar Witten deformation for non-Morse functions is introduced by M.P, Zhang and Zhu. However, the approach of J.Y. does not involve elongating the cylindrical region. This permit us to effectively control the newly created critical points  appearing during the deformation with respect to non-Morse functions (see \Cref{sixcrits}).
  }.

\def\bm{ \overline{M}}
\def\F{F}
Let $Y\subset Z$ be a hypersurface cutting $Z$ into two pieces $Z_1$ and $Z_2$ (see Figure \ref{fig1}). 

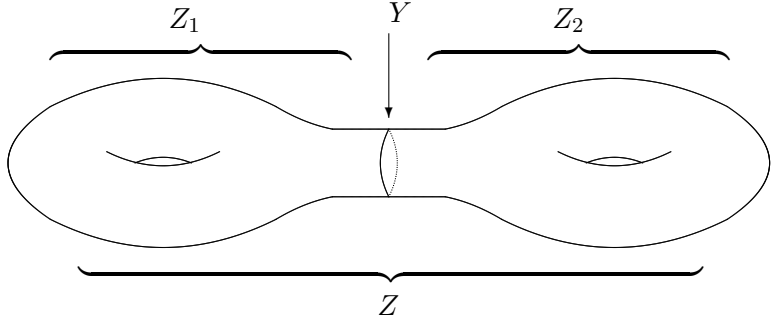
\begin{figure}[h]
	\setlength{\unitlength}{0.75cm}
	\centering
	\begin{picture}(18,6.5)
		\qbezier(3,2)(5,1)(7,2) 
		\qbezier(3,4)(5,5)(7,4) 
		\qbezier(3,2)(1.5,3)(3,4) 
		\qbezier(4,3.2)(5,2.7)(6,3.2) 
		\qbezier(4.5,3)(5,3.2)(5.5,3) 
		\qbezier(7,2)(7.5,2.3)(8,2.4) 
		\qbezier(7,4)(7.5,3.7)(8,3.6) 
		\qbezier(11,2)(13,1)(15,2) 
		\qbezier(11,4)(13,5)(15,4) 
		\qbezier(15,2)(16.5,3)(15,4) 
		\qbezier(12,3.2)(13,2.7)(14,3.2) 
		\qbezier(12.5,3)(13,3.2)(13.5,3) 
		\qbezier(11,2)(10.5,2.3)(10,2.4) 
		\qbezier(11,4)(10.5,3.7)(10,3.6) 
		\qbezier(8,2.4)(9,2.4)(10,2.4) 
		\qbezier(8,3.6)(9,3.6)(10,3.6) 
		
		\qbezier(9,2.4)(8.7,3)(9,3.6) 
		\qbezier[30](9,2.4)(9.3,3)(9,3.6) 
		
		\put(9,5.3){\vector(0,-1){1.5}} 
		
		\put(3,4.8){$\overbrace{\hspace{40mm}}^{}$} \put(5.1,5.4){$Z_{1}$}  
		\put(9.7,4.8){$\overbrace{\hspace{40mm}}^{}$} \put(11.9,5.4){$Z_{2}$}  
		\put(3.5,1.2){$\underbrace{\hspace{83mm}}_{}$} \put(8.8,0.3){$Z$}  
		\put(9,5.5){$Y$}
	\end{picture}
	\caption{Cutting $ Z $ along $ Y $.}
	\label{fig1}
\end{figure}

We then construct a family of non-Morse functions $p_A$ such that $p_A\equiv0$ if $A=0$ and as $A\to\infty$, informally, the ``critical sets" of $p_A$ consist of $Z_1$ and $Z_2$ with respective Morse index $0$ and $1$.

Let $d_A:=d+dp_A\wedge$ be the Witten deformation with respect to $p_A$, and $\Delta_A$ be the corresponding Hodge Laplacian.
\def\l{\lambda}
\def\abs{\mathrm{abs}}
\def\rel{\mathrm{rel}}
For $A=0$, $\Delta_A$ corresponds to the original Laplacian on $Z$. As $A\to+\infty$, the eigenvalues of $\Delta_A$ converge to the eigenvalues of the Laplacians on $Z_1$ and $Z_2$ with relative and absolute boundary conditions, respectively. This observation underlies the proof of the gluing formulas in \cite{Yantorsions,Yanforms}. 

{
	Now, on $\Omega_1$, consider $\Omega_1'\subset\subset\Omega_1$ such that $\Omega_1'\supset L$. Let $\phi: S \rightarrow \mathbb{R}_+ \cup \{+\infty\}$ be smooth and compactly supported in $\Omega_1$, such that  $\phi^{-1}(\{+\infty\})=\Omega_1'$. We deform fiberwisely the function $f$ by the non-Morse function $p_A$ described above, with $Z_1$ and $Z_2$  replaced by $\Df_\s$ and $Z_\s - \Df_\s$ and with parameter $A = \phi(\s)$, defining $\mathbf{f}:= f + p_A$\footnote{In the main text of the paper, we will use different notation for $\mathbf{f}$, because some additional small perturbations are needed, which we do not discuss here.}.

    For $\s \in \Omega_1'$, since $A = \phi(\s) = \infty$, the torsion forms separate into two components corresponding to $Z_\s - \Df_\s$ and  $\Df_\s.$ The contribution from $\Df_\s$ vanishes because on a small ball  $(F,\nabla^F,h^F)$ can be trivialized  and we have removed the contribution of the trivial bundle when defining the relative version of the torsion (see Theorems \ref{int20} and \ref{diffsm1}).  Furthermore, on $\Omega_1'$,  $\mathbf{f}|_{Z_\s-\Df_\s}$ is Morse, so extending Bismut-Goette's result to the manifold with boundariy $Z_\s-\Df_\s$ (see Theorem \ref{int2}) we obtain the contribution of this piece.  Thus, we get Theorem \ref{main1}  in $\Omega_1'$. Since for $\s\notin\Omega_1$, $A=\phi(\s)=0$, we leave everything unchanged outside $\Omega_1.$
}

			

\begin{rem}\label{extra-effort}
	Note that Bismut-Goette's theorem and the gluing formula hold only modulo $ d^S\Omega^\bullet(S) $. However, a locally exact differential form may not be globally exact. Thus,  using these two results to prove that Theorem \ref{main1} holds modulo exact forms in $\Omega_0$ and $\Omega_1$, is not enough. This is why our work requires extra effort to establish formulas at the level of differential forms, both for the Bismut-Goette theorem and for the gluing formula (see Theorems \ref{int1}, \ref{int20}  and \ref{int2}, and Remark \ref{remthminr20}). In particular,  we perform the deformation using $\phi$ and $p_A$ to ensure that the ``gluing formula" holds at the level of differential forms on $\Omega_1$. Such 'refined' versions are made possible by two facts: first we are considering a relative comparison formula, and second the set we cut out in the fibers above $\Omega_1'$ is very simple.
\end{rem}

We encounter two main challenges in our approach. First, on fibers over $ \Omega_1 \setminus \Omega_1' $ where the function $ \phi $ is nonzero but not large , the function $ \mathbf{f} $ has highly singular critical points. Second, on fibers where $ \phi $ is large—including at fibers where $ \phi = +\infty $—the function $ \mathbf{f} $ behaves like a (generalized) Morse function, but new critical points appear (see \Cref{model}).

To handle the birth–death pairs as well as new added critical points that emerge at fibers near $ \phi = +\infty $, we use the notion of \emph{double suspension} of the pair $ (M, f) $ (see \Cref{combitor}). This technique appears in \cite[Definition~4.5.3]{igusa2002higher}. In regions where $ \phi $ is large, these newly introduced critical points can be made irrelevant  by choosing an appropriate metric on the double suspended space (see \Cref{sec53}).

To address the singular critical points in the region where $ \phi $ is not large but nonzero, we combine techniques from Bismut–Lebeau \cite[§12 and §13]{bismut1991complex} (see also, e.g., Bismut–Goette \cite{bismut2001families}, Bismut–Zhang \cite{bismutzhang1992cm}) with Agmon-type estimates for noncompact manifolds developed by Dai and J.Y.~\cite{DY2020cohomology} (see \Cref{PartII}).


\subsection{Organization of the paper}
This paper is separated into three parts.
Part \ref{PartI} consists of \cref{sec2}--\ref{ineresult}. We begin by reviewing basic concepts in \cref{sec2}. In \cref{secfff}, we  revisit the notion of generalized Morse functions and establish Lemma \ref{indepbd}, which is similar to \cite[Lemma 4.5.1]{igusa2002higher}, to ensure the independence of certain critical points  from other critical points. We also associate a set of non-Morse functions with a fiberwise generalized Morse function. In \cref{model}, we study the model--Witten deformation for associated non-Morse functions near birth-death points. In \cref{combitor}, we  define the higher combinatorial torsion, which, according to the discussion in \cite[$\S$10]{goette2003morse}, should be equivalent to higher Igusa-Klein torsion.  In \cref{ineresult}, we  state several intermediate results and then use them to establish our main result under the assumption that each fiber has at most one birth-death point.  Next, Part \ref{PartII} consists of   \cref{setting-partII}--\cref{sec9} and aims to provide the necessary estimates to prove the intermediate results. Finally,  Part \ref{PartIII} covers \cref{sec7}--\ref{sec11}, where we  prove the intermediate results,  and \cref{ineresult1}, where we  explain how to treat the case where the fibers may have more than one birth-death point.

 Part \ref{PartII} is mainly devoted to introducing analytic techniques for handling singular critical points. In order to avoid introducing many new notations, the notations used in this part are re-defined and carry similar—but not necessarily identical—meanings as in the rest of the paper. To help the reader keep track, we provide two separate notation index lists: one for Part \ref{PartII} and one for the rest of the paper.


\subsection{Notations and conventions}\label{sec21}
Let $X$ be a smooth manifold and $E\to X$ be a complex vector bundle. We denote by 
$C^\infty(X)$ the space of smooth functions on $X$, by $\Gamma(X,E)$, or simply $\Gamma(E)$ when $X$ is understood, the space of smooth sections of $E$, and by
$\Omega^\bullet(X,E)$ the space of smooth $E$-valued differential forms on $X$. If $E=\underline{\C}$ is the trivial bundle of rank one, we will abbreviate this last notation as $\Omega^\bullet(X)$. Moreover, when adding a subscript '$c$' to one of these spaces (e.g., $C^\infty_c(X)$), we mean that the corresponding objects have compact support. Given $h^E$ a Hermitian bundle on $E$, we frequently identify $E$ with its anti-dual bundle $ \overline{E}^*$ via $h^E$. If moreover we have a Riemannian metric $g^{TX}$ on $X$, we denote by $L^2\Omega^\bullet(X,E)$ the space of  $E$-valued differential forms on $X$ which are $L^2$ with respect to $g^{TX}$ and $h^E$.

For  a complex vector bundle $F\to X$ which carries a flat connection, we denote by $d^X$ the induced de Rham-type differential on $\Omega^\bullet(X,F)$. In particular, for $F=\underline{\C}$ with its canonical connection, $d^X$ is the usual de Rham differential.  
For  a Riemannian metric $g^{TX}$ on $X$ and a Hermitian metric $h^F$ on $F$, we denote by $d^{X,*}$ the formal adjoint of $d^X$ with respect to the $L^2$ metric on $\Omega^\bullet(X,F)$ induced by $g^{TX}$ and $h^F$. 

Assume now that $X$ has a boundary $Y$. We  denote by $e_{\mathfrak{n}}$ the inward pointing unit normal vector field on $Y$, and by $e^{\mathfrak{n}}$ its dual. Given a flat vector bundle $(F,\nabla^F)$ on $X$, a form $\omega\in \Omega^{\bullet}(X,F)$ is said to satisfy the absolute (resp. relative) boundary condition if $i_{e_{\mathfrak{n}}}\omega=i_{e_{\mathfrak{n}}}d^X\omega=0$  (resp.  $e^{\mathfrak{n}}\wedge \omega=e^{\mathfrak{n}}\wedge d^{X,*}\omega=0$) on $Y$. Given a Hermitian vector bundle with Hermitian connection $ (E, h^E,\nabla^E) $ on $X$, let $ H$ be an operator of the form $H = \Delta^E + C $, where $ \Delta^E =(\nabla^E)^*\nabla^E$ is the Bochner Laplacian and $ C \in \Gamma(X,\End(E)) $ is a self-adjoint. When we consider the operator $H$ with Neumann boundary condition, we mean $H$ acting on $\{ w \in \Gamma(E)\::\: \nabla^E_{e_{\mathfrak{n}}} w = 0 \quad \text{on } Y\}
$.

For any two open set $U,V$, $U\subset\subset V$ means that $ \bar{U}\subset V.$ 

We adopt the following ``\textbf{bar convention}"\index{bar convention} in this paper: when we add a bar `` $  \bar{ } $ " to a torsion or characteristic form (e.g., $ \overline{{\mathcal{T}}}(T^H M, g^{TZ}, h^F)$) it indicates that the corresponding term for the trivial Hermitian bundle with trivial connection $(\Cb^m \to M,\nabla^{\Cb^m},h^{\Cb^m})$ is subtracted, and the degree 0 component is removed. This  should not be confused with complex conjugation.

For any $\Z$-graded object $G$, e.g., a differential forms or a vector bundle, we denote by $G^l$ its component of degree $l$, and we set $G^{>k}=\sum_{l>k}G^l$.

\subsection*{Acknowledgments} 
The authors sincerely thank Yeping Zhang for his contribution to the completion of this paper.
 In the early stage of this work, he suggested the idea of removing small neighborhoods of birth-death points using cutting techniques and initiated a discussion that facilitated the authors' collaboration. 
 His input is a crucial point for our project.

J.Y. extends heartfelt thanks to Prof. Xianzhe Dai and Prof. Weiping Zhang for their invaluable support and encouragement. Also, J.Y. is thankful to BICMR for fostering an excellent academic atmosphere, allowing dedicated focus on this project. J.Y. is partially supported by Boya Postdoctoral Fellowship at Peking University and China Postdoctoral Science Foundation (No. 2023M730092).

M.P. is grateful to Prof. Xiaonan Ma for many helpful discussions about the topic covered here.

\newpage

\part{Topological torsion and relative comparison formula}\label{PartI}

This part consists of \cref{sec2}--\ref{ineresult}. In \cref{sec2},  we start by  reviewing basic concepts. In \cref{secfff}, we  review the notion of generalized Morse functions and study the independence of certain critical points  from other critical points. We also construct a set of non-Morse functions associated with a fiberwise generalized Morse function. In \cref{model}, we  study the model-- Witten deformation for associated non-Morse functions near birth-death points. In \cref{combitor}, we  define the higher combinatorial torsion.  In \cref{ineresult}, we  present several intermediate results and then use them to establish our main result under the assumption that each fiber has at most one birth-death point.

\section{Preliminaries}\label{sec2}
Let $S$ be a closed manifold, and $\pi: M \to S$ be a smooth fibration with closed $n+1$-dimensional fiber $Z$, and let $(F,\nabla^F,h^F)$ be a unitarily flat Hermitian vector bundle on $M$.

\subsection{Review on generalized metric and Chern-Weil theory for flat vector bundles}

Let $E=E^+\oplus E^-\to S$ be a $\mathbb{Z}_2$-graded complex vector bundle, and $\overline{E}^*$ be the antidual bundle of $E$. Let $\bar{\cdot}^*$ be the even antilinear map from $\Lambda^{\bullet}\left(T^* S\right) \widehat{\otimes} \operatorname{End}(E)$ into $\Lambda^{\bullet}\left(T^* S\right) \widehat{\otimes} \operatorname{End}(\overline{E}^*)$ defined by the following relations:
\begin{itemize}
	\item if $\alpha,\alpha'\in \Lambda^{\bullet}\left(T^* S\right) \widehat{\otimes} \operatorname{End}(E)$, then $\overline{\alpha\alpha'}^*=\overline{\alpha'}^*\overline{\alpha}^*$;
	\item if $\omega \in T^*S\otimes\C$, then $\overline{\omega}^*=-\overline{\omega}$;
	\item if $B\in \operatorname{End}(E)$, then $\overline{B}^*\in \operatorname{End}(\overline{E}^*)$ is the usual transpose of $B$.
\end{itemize}
If $A'=\nabla^E+S$ is a superconnection on $E$, with $\nabla^E$ a connection and $S$ a section of $$\big(\Lambda^{\bullet}\left(T^* S\right) \widehat{\otimes} \operatorname{End}(E)\big)^{\mathrm{odd}},$$ then we define
$$\overline{A'}^*=\nabla^{\overline{E}^*} + \overline{S}^*,$$
where $\nabla^{\overline{E}^*}$ is the connection on $\overline{E}^*$ induced by $\nabla^E$.

{\begin{defn}[Definition 2.43 in \cite{bismut2001families}]
		A smooth section $\mathbf{h}^E$ of $\left(\Lambda^{\bullet}\left(T^* S\right) \widehat{\otimes} \operatorname{Hom}\left(E,  \overline{E}^*\right)\right)^{\text {even }}$ is said to be a generalized\index{generalized metric} Hermitian metric if
		$$
		\overline{\mathbf{h}^E}^*=\mathbf{h}^E,
		$$
		and if the degree $0$ component $h^E:=(\mathbf{h}^{E})^{[0]} \in \operatorname{Hom}\left(E,  \overline{E}^*\right)^{\text {even }}$ of $\mathbf{h}^E$ defines a standard Hermitian metric on $E$. Then $E^{+}$ and $E^{-}$ are orthogonal with respect to $h^E$.
	\end{defn}

    For any generalized Hermitian metric $\mathbf{h}^E$, since $(\mathbf{h}^{E})^{[0]} $ is invertible, so is $\mathbf{h}^E$.
	
	\begin{defn}[Definition 2.44 in \cite{bismut2001families}]	\label{def-*,g}
		Let $A'$ be a flat superconnection on $E$. The adjoint superconnection $A''$ of $A'$ with respect to $\mathbf{h}^E$ is defined by the formula,
		$$
		A''=\left(\mathbf{h}^E\right)^{-1} \circ{\overline{A'}^*} \circ\mathbf{h}^E.
		$$
	\end{defn}
	
	\begin{rem}
		
		When $\bh^E$ is a standard Hermitian metric $h^E$, this notion of adjoint superconnection coincide with the one defined in \cite[\S 1(d)]{bismut1995flat}.
	\end{rem}
}

Let $h:\C\to\C$ be the holomorphic function 
\[h(a)=a\exp(a^2),\forall a\in\C.\]

Let $\psi: \Omega^{\bullet}(S) \rightarrow \Omega^{\bullet}(S)$ as follows,
$$
\psi\index{psi@$\psi$} \omega=(2 \pi i)^{-k/2} \omega, \quad \text { for } \omega \in \Omega^{ k}(S).
$$

For any  flat superconnection $A'$ and any generalized Hermitian metric $\mathbf{h}^E$ on $E$, we define
\begin{equation}
    \begin{aligned}
        &X^{\bh^E}:=\frac{1}{2}{\left(A''-A'\right)}=\frac{1}{2}(\bh^E)^{-1}(A'\bh^E),\\
        &h(A',\bh^E)\index{hAhE@$h(A',\bh^E)$}:={(2 \pi i)^{\half}}\psi\Tr_s\left(h\big(X^{\bh^E}\big)\right){\in \Omega^\bullet(S)}.
    \end{aligned}
\end{equation}
Then, as in standard Chern-Weil theory, we have \be\label{chernweil} d^Sh(A',\bh^E)=0.\ee

Let $\bh^{E}_0$ and $\bh^{E}_1$ be two generalized metrics on $E$. Let $\chS:=S\times[0,1]$\index{Scheck@$\chS$} and $l$ be the coordinate for $[0,1]$. Let $\chE\to \chS$\index{Echeck@$\chE$} be the pullback of $E\to S$ under the canonical projection $\chS\to S$ and $\chA:=A'+dl\wedge\frac{\p}{\p l}$ be the associated connection on $\chE$.  For $l\in[0,1]$,  $\check{\I}_l$ denotes the embedding $s\in S\mapsto (s,l)\in \chS$. Let $\bh^{\chE}$ be a generalized metric on $\chE$, such that
\be\label{hhat} \begin{aligned}
    &\check{\I}_i^*\bh^{\chE}=\bh^E_i \text{ for }i=0,1\\
    &\bh^{\chE} \text{ has no exterior variable $dl$, i.e., } i_{\frac{\p}{\p l}}\bh^{\chE}=0.
\end{aligned} 
 \ee

For any differential form $w\in \Omega^{\bullet}(\chS)$, 
Let $w^{dl}$ denotes the $dl$-components for $w$, i.e., $w^{dl}=dl\wedge (i_{\frac{\p}{\p l}}w).$ 
Let $\bh^E_l:=\check{\I}_l^*\bh^{\chE}$, which is a generalized Hermitian metric. Then we can see that the $dl$-component of $h(\chA,\bh^{\chE})$ is given by
\[h(\chA,\bh^{\chE})^{dl}=dl\wedge\psi \operatorname{Tr}_{\mathrm{s}}\left[\frac{1}{2}\left(\bh_l^E\right)^{-1} \frac{\partial \bh_l^E}{\partial l} h^{\prime}\left(X^{\bh^E_l}\right)\right].\]
Here we have identified $E$ and $\overline{E}^*$ via the metric $h_l^E:=(\bh_l^E)^{[0]}.$

We define
\be\label{htilde}\tilde{h}(A',\bh^{\chE})\index{htildeAhE@$\tilde{h}(A,\bh^{\chE})$}:=\int_0^1h(\chA,\bh^{\chE})^{dl}=\int_0^1dl\wedge\psi \operatorname{Tr}_{\mathrm{s}}\left[\frac{1}{2}\left(\bh_l^E\right)^{-1} \frac{\partial \bh_l^E}{\partial l} h^{\prime}\left(X^{\bh^E_l}\right)\right].\ee

As in standard Chern-Weil theory,  \eqref{chernweil} gives $d^{\chS}h(\chA,\bh^{\chE})=0$, which implies
\[h(A',\bh^E_1)-h(A',\bh^E_0)=d^S\tilde{h}(A',\bh^{\chE}).\]

Moreover, again as in standard Chern-Weil theory, we can see that the class of $\tilde{h}(A',\bh^{\chE})$ in $ \Omega^{\bullet}(S)/d^S \Omega^{\bullet}(S)$ is independent of the choice of $\bh^{\chE}$  that satisfying \eqref{hhat}, we then denote this class by $\tilde{h}(A',\bh_0^E,\bh_1^E).$ \\


We now assume that $E=\bigoplus_{k=0}^nE^k$ is $\Z$-graded. Let  $N^E=\bigoplus_{k=0}^n k\Id_{E^k}$  be the associated number operators. Then $\hE$ is also $\Z$-graded, and we denote its number operator by $N^\hE$.

\begin{defn}
	We say that a superconnection $A'$ on $E$ is of total degree 1, if we have decomposition
	\[A'=\sum_{j \geq 0} A'_j,\]
	such that $A'_1$ is a connection on $E$ which preserves the $\mathbb{Z}$-grading and for $j \in \mathbb{N}-\{1\}, A'_j$ is an element of $\Omega^j\big(S, \operatorname{Hom}(E^{\bullet}, E^{\bullet+1-j})\big).$
\end{defn}

Let $A'$ be a  superconnection of total degree 1. We assume that $A'$ is flat, i.e., $A'^2=0$. Set $v:=A'_0$, then the flatness of $A'$ implies $v^2=0$ and $[A'_1,v]=0$.

Let $H:=H^*(E,v)$ be the bundle of cohomology with respect to $v$, equipped with   be the Gauss-Manin connection $\nabla^H$ induced by $A'$. Let
\begin{equation}
    \begin{aligned}
        &\chi'(E):=\sum_{i=0}^n(-1)^ii\cdot \rank(E^i), \\
          &\chi(H):=\sum_{i=0}^n(-1)^i\rank\big(H^i\big),
          &\chi'(H):=\sum_{i=0}^n(-1)^ii\cdot \rank(H^i).
    \end{aligned}
\end{equation}

Let $\bh^E$ be a generalized Hermitian metric on $E$. Let $h_{L^2}^E$ be the metric on $H$ induced by $h^E:=\big(\bh^E\big)^{[0]}$ and Hodge theory.

Let $\hS:=S\times(0,\infty)$\index{Shat@$\hS$}. We denote by $t$  the coordinate for $(0,\infty)$. Let $\hE\to\hS$\index{Ehat@$\hE$} be the pullback of $E\to S$ with respect to the canonical projection $\hS\to S$, and let $\hA:=A'+dt\wedge\frac{\p}{\p t}$ be the associated connection on $\hE\to\hS.$ Let $\hat{\I}_{t}$ be the embedding $s\in S\mapsto(s,t)\in \hS$. Suppose that we have a generalized metric $\bh^{\hE}$ satisfying
\begin{cond}\label{assum21}
	\begin{enumerate}[(1)]
		\item For any $s_1,s_2\in E$ \be\label{congen1}\hat{\I}_1^*\bh^\hE(s_1,s_2)=\bh^E(s_1,s_2);\ee
		\item there exists $g' \in  \Omega^{\bullet}(S ;\mathrm{End}(E))$ such that $g'$ is even, and as $t \rightarrow \infty$, for any $s\in E$,
		\be\label{congen}
		\begin{aligned}
			\hat{\I}_t^*\left(t^{\frac{n-N^\hE}{2}} \bh^\hE t^{-\frac{N^\hE}{2}}\right)(s) & =s+\frac{g^{\prime}(s)}{\sqrt{t}}+O\left(\frac{1}{t}\right)(s), \\
			\text { and } \quad \hat{\I}_{t}^*\left(t^{\frac{N^\hE}{2}}\left(\bh^\hE\right)^{-1} \frac{\p\bh^\hE}{\partial t} t^{-\frac{N^\hE}{2}}\right)(s) & =\left(N^E-\frac{n}{2}\right) \frac{s}{t}+O\left(t^{-\frac{3}{2}}\right)(s),
		\end{aligned}
		\ee
		where both $O(\cdot)$ are even elements of $ \Omega^{\bullet}(S ;\mathrm{End}(E))$. Here we identify $\hE$ and $\overline{\hE}^*$ by the metric $h^{\hE}=\big(\bh^{\hE}\big)^{[0]}$.
		\item $\bh^{\hE}$ contains no exterior variable $dt$, i.e., it can be considered as a smooth family, parameterized by $t\in(0,\infty)$, of generalized metrics on $E$.
	\end{enumerate}
\end{cond}

We define a differential form as follows:
\begin{defn}\label{defn26} For any $\tau>0$,
	\[\T^{\mL}_\tau(A',\bh^{\hE})\index{TAhELTau@$\T^{\mL}_\tau(A',\bh^{\hE})$}:=-\int_\tau^\infty h(\hA,\bh^\hE)^{dt}-\left(\chi'(H)+\Big(\chi^{\prime}(E)-\frac{n}{2} \chi(H)\Big)h^{\prime}\Big(\frac{\sqrt{-t}}{2}\Big)\right) \frac{dt}{2 t}.\]
	By \cite[Proposition 2.48]{bismut2001families}, $\T_\tau^{\mL}(A',\bh^\hE)$ is well-defined.
	We will also set
	\[\T(A',\bh^{\hE})\index{TAhE@$\T(A,\bh^{\hE})$}:=-\int_0^\infty h(\hA,\bh^\hE)^{dt}-\left(\chi'(H)+\Big(\chi^{\prime}(E)-\frac{n}{2} \chi(H)\Big)h^{\prime}\Big(\frac{\sqrt{-t}}{2}\Big)\right) \frac{dt}{2 t}\]
	if $\T(A',\bh^{\hE})$ is well-defined.
	
\end{defn}	
\def\can{\mathrm{can}}
\begin{rem}\label{remeq}
	Let $h^E$ be a Hermitian metric on $E$. We then have a canonical generalized metric on  $\hE$, defined by
	\[h_{\can}^\hE(\cdot,\cdot)|_{S\times\{t\}}:=t^{-\frac{n}{2}} h^E\left(t^{\frac{N^\hE}{2}} \cdot, t^{\frac{N^\hE}{2}} \cdot\right).\]
	If, for $l\geq2$, every $l$-component of $A'$ vanishes,  
	$\T(A',h^E)\index{TAhE@$\T(A',h^E)$}:=\T(A',h^{\hE}_\can)$ is well-defined. Here $\T(A',h^E)$ is actually the famous analytic torsion form defined in \cite[Definition 2.20 ]{bismut1995flat}.
\end{rem}

Now let $\bh^{E}_0$ and $\bh^{E}_1$ be two generalized metrics on $E$, and, for $i=0,1$, let $\bh^{\hE}_i$  be a generalized metrics on $\hE$ that satisfy Condition \ref{assum21} with respect to $\bh^{E}_i$.

Let $\hchS:=S\times[0,1]\times(0,\infty)$\index{Shatcheck@$\hchS$} with coordinates $(l,t)\in[0,1]\times(0,\infty)$. Let  $\hchE\to \hchS$\index{Ehatcheck@$\hchE$} be the pullback of $E\to S$ with respect to canonical projection $\hchS\to S$, and $\hchA=A'+dl\wedge\frac{\p}{\p l}+dt\wedge\frac{\p}{\p t}$. Let $\hat{\check{\I}}_t$ be the canonical embedding $\chS\to \chS\times\{t\}\subset \hchS.$ Let $\bh^{\hchE}$ be a generalized metric on $\hchE$ which satisfies
\begin{itemize}
    \item \eqref{hhat} with $(S,E, \bh^{E}_0, \bh^{E}_1)$ replaced by $(\hS,\hE,\bh^{\hE}_0,\bh^{\hE}_1)$,
    \item Condition \ref{assum21} with $(E,S, \hE,\bh^{E},\bh^{\hE})$ replaced by $(\chE,\chS,\hchE,\bh^{\chE},\bh^{\hchE})$ where $\bh^{\chE}:=\hat{\check{\I}}_{t=1}^*\bh^{\hchE}$.
\end{itemize}
We then set $$\B^L(A',\bh^{\hchE}):=-\int_\tau^\infty \int_0^1h(\hchA,\bh^{\hchE})^{dtdl},$$ where as before, for $w\in \Omega^{\bullet}(\hchS),$ $w^{dtdl}:=dt\wedge dl\wedge i_{\frac{\p}{\p l}}i_{\frac{\p}{\p t}}w$ is the $dt\wedge dl$-component of $w$. By \cite[Proposition 2.48]{bismut2001families}, $h(\hchA,\bh^{\hchE})^{dt}=\frac{(\chi'(H)+O(t^{-\half}))dt}{2t}$ as $t\to\infty$, so $\B^L(A',\bh^{\hchE})$ is well defined.


Let $\chH\to \chS$ be the pullback bundle of $H\to S$ with respect to the canonical projection $\chS\to S$, equipped with the Gauss-Manin connection  $\nabla^{\chH}$  induced by $\chA$. Let $h_{L^2}^{\chE}$ be the metric on $\chH$ induced by $h^{\chE}:=\big(\bh^{\chE}\big)^{[0]}$ and  Hodge theory.
We have

\begin{thm}[{\cite[Theorems 2.50  {and 2.52}]{bismut2001families}}]\label{thm109}
	Let $\bh^{E}$  be a generalized metrics on $E$, and let $\bh^{\hE}$  be a generalized metrics on $\hE$ that satisfy Condition \ref{assum21}. Then the form $\T_\tau^{\mL}(A',\bh^\hE)$ is even and satisfies
	$$
	d^S \T^{\mL}_\tau(A',\bh^\hE)=h\bigg(A', \hat{\I}_\tau^*\Big(\bh^\hE\Big)\bigg)-h\left(\nabla^H, h^E_{L^2}\right) .
	$$
	
	Let $\bh^{E}_0$ and $\bh^{E}_1$ be two generalized metrics on $E$, and, for $i=0,1$, let $\bh^{\hE}_i$  be a generalized metrics on $\hE$ that satisfy Condition \ref{assum21} with respect to $\bh^{E}_i$. Then
	$$
	\T_\tau^{\mL}(A',\bh_1^\hE)-\T_\tau^{\mL}(A',\bh_0^\hE)=\wch\left(A',\hat{\check{\I}}_{\tau}^*\Big(\bh^{\hchE}\Big) \right)-\wch\left(\nabla^H, h_{L^2}^\chE\right)+d^S\B^L\big(A',\bh^{\hchE}\big),
	$$
	where $\hat{\check{\I}}$, $\bh^{\hchE}$ and $\B^L\big(A',\bh^{\hchE}\big)$ have been defined above.
\end{thm}
\def\gen{0}
\if\gen1
\subsection{Generalized metric on infinite dimensional vector spaces}

Now we assume that $E=\sum_{i=1}^nE^i\to S$ is an infinite dimensional $\mathbb{Z}$-graded complex Hilbert bundle, and its fiberwise metric (inner product) is denoted by $h^E_0.$ 
\begin{defn}\label{defgen2}
	A generalized Hermitian metric $\bh^E$ on  $\mathbb{Z}$-graded complex Hilbert bundle $E:=\oplus_{i=0}^n E^i \rightarrow S$ is an $\mathbb{R}$-bilinear map $\bh^E:\left(\Lambda T^* S\widehat{\otimes} V\right)^2 \rightarrow \Lambda T^* S$, such that
	\begin{enumerate}[(a)]
		\item the component $h^E:=\left(\bh^E\right)^{[0]}:E\otimes E \rightarrow \mathbb{C}$ is an inner product. Moreover, there exists $C>1$, s.t. $C^{-1}h^E(s,s)\leq h^E_0(s,s)\leq Ch^E(s,s),\forall s\in \Gamma(E).$ 
		\item\label{defgenb} Let $g^{TS}$ be a metric on $TS\to S$. Then it induces a metric on $\Lambda T^*S$, which is still denoted by $g^{TS}$. For any $\a\in\Gamma(\Lambda T^*S)$, $|\a|:=\sqrt{g^{TS}(\a,\a)}.$ Then we assume that, there exists $C>0$, s.t. $|(\bh^E)^{[l]}(s,s)|\leq Ch^E(s,s),\forall s\in\Gamma(E),l\geq 1.$
		\item $\bh^E$ is even with respect to the total grading of $\Lambda T^* S \widehat{\otimes} E$;
		\item for $s_1, s_2 \in E$, we have
		$$
		\bh^E(s_2, s_1)=(-1)^{\left[\frac{\operatorname{deg}\left(\bh^E(s_2, s_1)\right)}{2}\right]  }\overline{\bh^E(s_1, s_2)} ;
		$$
		\item for $\alpha, \beta \in \Lambda T^* S$ and $s_1, s_2 \in E$, we have
		$$
		\bh^E(\alpha \widehat{\otimes} s_1, \beta \widehat{\otimes} s_2)=(-1)^{\left[\frac{\operatorname{deg}(\beta)}{2}\right]} \alpha \wedge \bh^E(s_1, s_2) \wedge  \overline{\beta}.
		$$
		As a result, $\bh^E|_{E\times E}$ determines $\bh^E.$
	\end{enumerate}
\end{defn}
Let $\mathrm{End}(E)$ denotes the spaces of \textbf{bounded} linear endomorphisms, and $\mathrm{Hom}(E^i,E^j)$ denotes the space of \textbf{bounded} morphisms from $E^i$ to $E^j$.

\begin{defn}
	For an endomorphism $B \in  \Omega^{\bullet}(S ;\mathrm{End}(E))$, we define the adjoint endomorphism $B^{*,g}$ by
	$$
	\bh^E(Bs_1,s_2)=(-1)^{\mathrm{deg}(B)(\mathrm{deg}(B)+\mathrm{deg}(s_1)+\mathrm{deg}(s_2))}\bh^E(s_1, B^{*,g}s_2), \forall s_1,s_2\in \Lambda T^*S\hat{\otimes}E,
	$$
	and for a superconnection $A$ on $E$, we define its adjoint connection by
	$$
	d^S\bh^E(s_1,s_2)=\bh^E(As_2,s_2)+(-1)^{\mathrm{deg}(A)(\mathrm{deg}(A)+\mathrm{deg}(s_1)+\mathrm{deg}(s_2))}\bh^E(s_1,A^{*,g}s_2)
	, \forall s_1,s_2\in \Lambda T^*S\hat{\otimes}E.
	$$
\end{defn}

Let $v^i:E^i\to E^{i+1}$ be a unbounded close operator with dense domain. Let $v:=\oplus v^i$.   

Let $A$ be a  superconnection of total degree $1$. Moreover, its $0$-component is $v$. We further assume that
\begin{enumerate}[({\textbf{A}}1)]
	\item  $A$ is flat (So $v^2=0$). 
	\item\label{assumv2} Let $v^{*,g}$ be the adjoint of $v$ with respect to $\bh^E$, and $v^*$ be the adjoint of $v$ with respect to $h^E$. By \eqref{defgenb} in Definition \ref{defgen2} and the definition of adjoint operators, we can see that $\Dom(v^{*,g})=\Dom(v^*)$. We then assume that $\Dom(v^{*,g})\supset \Dom(v).$  Here $\Dom(T)$ denotes the domain of an unbounded operator $T.$
	\item\label{assumv3} We assume that $v+v^*$ is essential self-adjoint. Moreover, fiberwisely, $(v+v^*)^2$ has discrete spectral $\{\l_k\}$, such that each one has finite multiplicities. Furthermore, there exist  constants $c,\tau>0$, s.t. $\l_k\geq ck^{\tau}$ if $k$ is large enough. By \cite[Proposition A.3]{DY2020cohomology}, we have Hodge theory for $v.$ That is, let $H(E,v):=\ker(v)/\Im(v)$, then as vector bundles, $H(E,v)\cong \ker(v)\cap\ker(v^*)=\ker(v+v^*).$ 
	\item Let $F:=A-A^{*,g}-(v-v^{*})$, then $F$ is nilpotent. Next we will define $F_i$ inductively: set $F_0:=F$, and $F_i:=[F_{i-1},v-v^*].$
	We assume that there exists constant $C_i>0,$ s.t. $|F_i|(\s)\leq C_i,\forall \s\in S.$
	
	Here for $T\in  \Omega^{\bullet}(S,\mathrm{End}(E))$, $|T|(\s):=\sup_{s\in E_\s,h^E(s,s)=1}|T(s)|$, for a differential form $\a\in \Omega^{\bullet}(S), |\a|:=\sqrt{g^{TS}(\a,\a)}.$
\end{enumerate}

If (\text{A}1)-(\text{A}4) are satisfied,  we could set $X^{\bh^E}:=\frac{1}{2}(A-A^{*,g}),$ and
$$h(A,\bh^E):=\psi\Tr_s\left(h\left(X^{\bh^E}\right)\right).$$ 
The well-definedness of $h(A,\bh^E)$ follows the following:
\begin{prop}
	If (\text{A}1)-(\text{A}4) are satisfied, $h\left(X^{\bh^E}\right)$ is a trace class operator.
\end{prop}
\begin{proof}
	Set $V:=
	\frac{1}{2}(v-v^*).$ 
	
	By (\textbf{A}3), $h(V)$ is of trace class.
	
	By (\textbf{A}4), proceeding as in the proof of \cite[Lemma 4.9]{Yanforms}, we can see that, if $u$ is a unit eigenvector of $V$ with respect to an eigenvalue $\mu$, then for any $l>0$, we have
	\[\left|h^E\left(\left(h(X^{\bh^E})-h(V)\right)u,u\right)\right|\leq \frac{C_l}{|\mu|^l+1}.\]
	
	As a result, by (\textbf{A}3) again, $h(X^{\bh^E})-h(V)$ is of trace class. As a result, $h(X^{\bh^E})$ is of trace class.

\end{proof}

Let $H:=H^*(E,v)$ be the bundle of cohomology. Let
$\chi'(V):=\sum_{i=0}^n(-1)^ii\cdot \rank(V^i),\chi'(H):=\sum_{i=0}^n(-1)^ii\cdot \rank(H^i) $ and $\chi(H):=\sum_{i=0}^n(-1)^i\rank(H^i(E,v)).$

Let $h_{L^2}^E$ be the metric on $H$ induced from $h^E:=(\bh^E)^{[0]}$ by Hodge theory, and $\nabla^H$ be the Gauss-Manin connection on $H$ induced by $A$.


Suppose we have a generalized metric $\bh^{\hE}$, s.t.
\begin{enumerate}[(a)]
	\item For any $s_1,s_2\in E$, \be\label{congen11}\hat{\I}_1^*\bh^\hE(s_1,s_2)=\bh^E(s_1,s_2);\ee
	\item there exists $g' \in  \Omega^{\bullet}(S ;\mathrm{End}(E))$ being even, and as $t \rightarrow \infty$, for any $s\in E$
	\be\label{congen01}
	\begin{aligned}
		\hat{\I}_t^*\left(t^{\frac{n-N^\hE}{2}} \bh^\hE t^{-\frac{N^\hE}{2}}\right)(s)& =s+\frac{g^{\prime}(s)}{\sqrt{t}}+O\left(\frac{1}{t}\right)(s), \\
		\text { and } \quad \hat{\I}_t^*\left(t^{\frac{N^\hE}{2}}\left(\bh^\hE\right)^{-1} \frac{\p\bh^\hE}{\partial t} t^{-\frac{N^\hE}{2}}\right) (s)& =\left(N^\hE-\frac{n}{2}\right) \frac{s}{t}+O\left(t^{-\frac{3}{2}}\right)(s),
	\end{aligned}
	\ee
	where both $O(\cdot)$ are even elements of $ \Omega^{\bullet}(S ;\mathrm{End}(E))$. Here we identify $\hE$ and $\overline{\hE}^*$ by the metric $h^{\hE}.$ For any $s\in (\hE)^k$, $N^{\hE}s:=ks.$
\end{enumerate}

Let $\hat{A}:=A+dt\wedge\frac{\p}{\p t}$. We assume that \def\hv{{\hat{v}}}
\begin{enumerate}[({\textbf{B}}1)]
	\item  $\hA$ is flat.
	\item\label{assumv21} Let $\hv^{*,g}$ be the adjoint of $v$ with respect to $\bh^{\hE}$, and $\hv^*$ be the adjoint of $v$ with respect to $h^{\hE}$. By \eqref{defgenb} in Definition \ref{defgen2} and the definition of adjoint operators, we can see that $\Dom(\hv^{*,g})=\Dom(\hv^*)$. We assume that $\Dom(\hv^{*,g})\supset \Dom(v).$ 
	\item\label{assumv31} $v+\hv^*$ is essential self-adjoint. Moreover, at each fiber at the point $(\s,t)\in \hS$, $(v+\hv^*)^2$ has discrete spectrals $\{\l_k(\s,t)\}$, such that each one has finite multiplicities. Furthermore, there exist a positive smooth function $c:(0,\infty)\to (0,\infty)$ and a constant $\tau>0$, s.t. $\l_k(\s,t)\geq c(t)k^{\tau}$ if $k$ is large enough. Moreover, $c(t)\sim t$ if $t$ is large enough.
	\item Let $\hat{F}:=\hA-\hA^{*,g}-(v-\hv^{*})$, then $\hat{F}$ is nilpotent. Next we will define $\hat{F}_i$ inductively: set $\hat{F}_0:=\hat{F}$, and $\hat{F}_i:=[\hat{F}_{i-1},v-\hv^*].$
	We assume that there exists constant $C_i>0,$ s.t. $|\hat{F}_i|(\s,t)\leq C_i(1+t^{-\frac{\dim(S)}{2}})$.
\end{enumerate}

\begin{defn}
	Then for any $\tau>0$, we can define
	\[\T^{\mL}_\tau(A,\bh^{\hE}):=-\int_\tau^\infty h(\hA,\bh^\hE)^{dt}-\chi'(H) \frac{dt}{2 t}.\]
	If there exists a constant $\chi$, s.t. $h(\hA,\bh^{\hE})^{dt}=\frac{\chi dt}{t}+O(t^{-c})$ as $t\to0^+$ for some constant $c>-1$, then
	we will put
	\[\T(A,\bh^{\hE}):=-\int_0^\infty h(\hA,\bh^\hE)^{dt}-\left(\chi'(H)+\frac{\chi-2\chi'(H)}{2}h'\left(\frac{\sqrt{-1}\sqrt{t}}{2}\right)\right)\frac{dt}{2 t}. \]
\end{defn}	

\begin{prop}\label{welldefntor}
	$\T_\tau^{\mL}(A;\bh^\hE)$ is well-defined.	
\end{prop} 	
\begin{proof}
	Set $\hat{V}:=
	\frac{1}{2}(v-\hv^*)$ and $\hat{X}^{\bh^{\hE}}:=\half(\hA-\hA^{*,g}).$ 
	
	By (\textbf{B}3), $\Tr_s(N^Eh'(\hat{V}))-\chi'(H)=O(e^{-ct})$ if $t$ is large enough.
	
	By (\textbf{B}4), proceeding as in the proof of \cite[Lemma 4.9]{Yanforms}, we can see that, if $u$ is an eigenvector of $\hat{V}$ with respect to a non-zero eigenvalue $\mu$, then for any $l>0$, we have
	\[\left|h^E\left(N^{\hE}\left(h'(\hat{X}^{\bh^\hE})-h'(V)\right)u,u\right)\right|\leq \frac{C_l}{\sqrt{t}(|\mu|^l+1)}\] if $t$ is large enough.
	
	As a result, by (\textbf{B}3) again, $\Tr_s\left(N^{\hE}\left(h'(\hat{X}^{\bh^\hE})-h'(V)\right)\right)=O(t^{-\half})$ . 
	
	Since $h(\hA,\bh^{\hE})^{dt}=\frac{dt}{t}\wedge\psi\Tr_s(N^{\hE}h'(\hat{X}^{\hE})),$ the result follows.

\end{proof}

Let $\bh^{E}_0$ and $\bh^{E}_1$ be two generalized metrics on $E$, and $\bh^{\hE}_0$ and $\bh^\hE_1$ be two generalized metric on $\hE$ that satisfy \eqref{congen11}, \eqref{congen01}, and assumption (\textbf{A}1)-(\textbf{A}4).

Let $\chS:=S\times[0,1]$ for $l\in[0,1]$, $\chA=A+\frac{\p}{\p l}dl$ and $\chE\to \chS$ be the pullback of $E\to S$ with respect to canonical projection $\chS\to S.$ Let $\bh^{\chE}$ be a generalized metric on $\chE$, s.t. $\bh^{\chE}|_{S\times\{0\}}=\bh_0^{E}$ and $\bh^{\chE}|_{S\times\{1\}}=\bh_1^{E}.$

Let $\hchS:=\chS\times(0,\infty)=S\times[0,1]\times(0,\infty)$ with coordinates $(l,t)\in[0,1]\times(0,\infty)$, $\hchA=\chA+\frac{\p}{\p t}dt$, $\hchE\to \hchS$ be the pullback of $(\chE\to \chS,\chA)$ with respect to canonical projection $\hchS\to \chS.$ Let $\bh^{\hchE}$ be a generalized metric on $\hchE$ that satisfies \eqref{congen}, \eqref{congen1}, and assumption (\textbf{B}1)- (\textbf{B}4) (with $(E,S, \hE,\bh^{E},\bh^{\hE})$ replaced by $(\chE,\chS,\hchE,\bh^{\chE},\bh^{\hchE})$).

We set $\B^L(A,\bh^{\hchE}):=-\int_1^\infty \int_0^1h(\hchA,\bh^{\hchE})^{dtdl}$.

In the proof of Proposition \ref{welldefntor}, we have seen that $h(\hchA,h^{\hchE})^{dt}=\frac{\chi'(H)+O(t^{-\half})dt}{2t}$ as $t\to\infty$, so $\B^L(A,\bh^{\hchE})$ is well defined.


Let $\chH\to \chS$ be the pullback bundle of $H\to S$ with respect to the canonical projection $\chS\to S$.

Let $h_{L^2}^{\chE}$ be the metric on $\chH$ induced from $h^{\chE}:=(\bh^{\chE})^{[0]}$ by Hodge theory, and $\nabla^{\chH}$ be the Gauss-Manin connection on $\chH$ induced by $\chA$. 

It follows from the standard Chern-Weil type theory that, 
\begin{thm}\label{thm1091}
	The form $\T_\tau^{\mL}(A,\bh^\hE)$ is even and satisfies
	$$
	d^S \T^{\mL}_\tau(A,\bh^\hE)=h\left(A, \hat{\I}_\tau^*(\bh^\hE)\right)-h\left(\nabla^H, h^E_{L^2}\right) .
	$$
	
	Let $\bh^{E}_0$ and $\bh^{E}_1$ be two generalized metrics on $E$, and $\bh^{\hE}_0$ and $\bh^\hE_1$ be two generalized metrics on $\hE$ that satisfy \eqref{congen1} , \eqref{congen}, and assumption (\textbf{B}1)-(\textbf{B}4), then
	$$
	\T^{\mL}_\tau(A,\bh_1^\hE)-\T_\tau^{\mL}(A,\bh_0^\hE)=\wch\left(A,\hat{\check{\I}}_\tau^*(\bh^{\hchE}) \right)-\wch\left(\nabla^H, h_{L^2}^\chE\right)+d^S\B^L(A,\bh^{\hchE}).
	$$
	
\end{thm}
\fi

\subsection{Definition of the Bismut-Lott analytic torsion forms}\label{defnbl}

{Let $(F,\nabla^F)\to M$ be a flat vector bundle, and let $h^F$ be a metric on $F$.}

Let $T^H M$ be a sub-bundle of $T M$ such that
\begin{equation}
\label{defTHM}
  T M=T^H M \oplus T Z .  
\end{equation}
Let $P^{T Z}$ denote the projection from $T M$ to $T Z$ with respect to the above decomposition. If $U \in T S$, let $U^H$ be the lift of $U$ in $T^H M$, s.t. $\pi_* U^H=U$.

Let $\bfE^i=\Omega^i(Z,F|_Z)$, which is the smooth infinite-dimensional bundle over $S$ whose fiber at $\s \in S$ is $\bfE^i_{\s}=\Omega^i(\Zs,F|_{Z_\s})$. Let $\bfE=\oplus_{i=0}^{\operatorname{dim} Z} \bfE^i$. Using \eqref{defTHM}, we have an identification
$$
\Gamma\left(S ,\bfE^i\right)\simeq\Gamma\left(M, \Lambda^i\left(T^* Z\right) \otimes F\right).
$$

For  $U\in \Gamma(S,TS)$, the Lie differential $L_{U^H}$ acts on $\Gamma(S, \bfE)$. For $s \in \Gamma(S , \bfE)$, set
\begin{equation}
    \label{def-nablabfE}
    \nabla_U^\bfE s=L_{U^H} s .
\end{equation}
Then $\nabla^\bfE$ is a connection on $\bfE$ preserving the $\mathbb{Z}$-grading.

If $U_1, U_2$ are vector fields on $S$, put
$$
\bfT\left(U_1, U_2\right)=-P^{T Z}\left[U_1^H, U_2^H\right] \in \Gamma(M, T Z) .
$$
We denote by $i_\bfT\in \Omega^2\left(S, \operatorname{Hom}\left(\bfE^{\bullet}, \bfE^{\bullet-1}\right)\right)$ the 2-form on $S$ which, to vector fields $U_1, U_2$ on $S$, assigns the operation of interior multiplication by $\bfT\left(U_1, U_2\right)$ on $\bfE$. Let $d^M: \Omega^{\bullet}(M,F)\to \Omega^{\bullet+1}(M,F)$ be the exterior differentiation induced by $\nabla^F.$
Let also $d^{Z}: \Omega^{\bullet}(Z,F|_Z)\to\Omega^{\bullet+1}(Z,F|_Z)$ be the exterior differentiation along fibers induced by $\nabla^F$. We view $d^Z$ as an element of $\Gamma\left(S ; \operatorname{Hom}\left(\bfE^{\bullet}, \bfE^{\bullet+1}\right)\right)$. By \cite[Proposition 3.4]{bismut1995flat}, we have
\be\label{decdm}
d^M=d^Z+\nabla^\bfE+i_\bfT .
\ee
Thus $d^M$ can be viewed as a flat superconnection of total degree 1 on $\bfE$. As $(d^M)^2=0$, we have
$$
\left(d^Z\right)^2=0 \text{ and }[\nabla^\bfE, d^Z]=0.
$$

Let $g^{T Z}$ be a metric on $T Z$. Let $h^\bfE$ be the $L^2$ metric on $\bfE$ induced by $h^F$ and $g^{TZ}$. Let $\nabla^{\bfE,*}$ (resp. $d^{Z ,*}$) be the adjoint connection of $\nabla^\bfE$ (resp. the fiberwise formal adjoint of $d^Z$) with respect to ${h^\bfE}$. Set
$$
D^Z=d^Z+d^{Z *}, \quad \nabla^{\bfE, u}=\frac{1}{2}\left(\nabla^\bfE+\nabla^{\bfE,*}\right) .
$$
Let $d^{M,*}$ be the adjoint superconnection of $d^M$ with respect to $h^\bfE$, as defined in Definition~\ref{def-*,g}. For $X \in T Z$, we denote by $X^\flat \in T^* Z$ its metric dual with respect to $g^{T Z}$, then by \cite[Proposition 3.7]{bismut1995flat},
\be \label{decdmstar}d^{M,*}= d^{Z,*}+\nabla^{\bfE,*}-\bfT^
\flat\wedge.\ee
Let $N^Z$ be the number operator of $\bfE$, i.e., $N^Z$ acts on   $\Gamma\left(S, \Lambda^k\left(T^* Z\right) \otimes F\right)$ by multiplication by $k$. For $t>0$, set
\[
\begin{array}{l}
	C_t^{\prime}=t^{N^Z / 2} d^M t^{-N^Z / 2}, \quad C_t^{\prime \prime}=t^{-N^Z / 2}d^{M,*} t^{N^Z / 2}, \\
	C_t=\frac{1}{2}\left(C_t^{\prime}+C_t^{\prime \prime}\right), \quad D_t=\frac{1}{2}\left(C_t^{\prime \prime}-C_t^{\prime}\right) ,
\end{array}
\]
then $C_t^{\prime \prime}$ is the adjoint of $C_t^{\prime}$ with respect to $h^\bfE$, $C_t$ is a superconnection, $D_t$ is an odd element of $\Omega^\bullet(S, \End (\bfE))$ and
$$
C_t^2=-D_t^2.
$$
For $X \in T Z$, set $c(X)=X^\flat \wedge-i_X$. Then
$$
C_t=\frac{\sqrt{t}}{2} D^Z+\nabla^{\bfE, u}-\frac{1}{2 \sqrt{t}} c(\bfT).
$$

For any $t>0$, the operator $D_t$ is a fiberwise-elliptic differential operator. Then $h\left(D_t\right)$ is a fiberwise trace class operator. For $t>0$, put
$$
h^{\wedge}\left(C_t^{\prime}, h^\bfE\right)\index{hWedgeCthE@$h^{\wedge}\left(C_t^{\prime}, h^\bfE\right)$}:=\psi \operatorname{Tr}_s\left(\frac{N^Z}{2} h^{\prime}\left(D_t\right)\right).
$$

Put
$$
\begin{array}{l}
	\chi(Z,F):=\sum_{j=0}^{\operatorname{dim} Z}(-1)^j \operatorname{rank} H^j(Z, F), \\
	\chi^{\prime}(Z, F):=\sum_{j=0}^{\operatorname{dim} Z}(-1)^j j \operatorname{rank} H^j\left(Z, F\right).
\end{array}
$$

\begin{defn}[{\cite[Definition 3.24]{bismut1995flat}}]\label{torc}
	The analytic torsion form $\mathcal{T}\left(T^H M, g^{T Z}, h^F\right)$ is a form on $S$ which is defined by
	\begin{align*}
		&\ \ \ \ \mathcal{T}\left(T^H M, g^{T Z},h^F\right)\index{TTHMGTZhF@$\mathcal{T}\left(T^H M, g^{T Z},h^F\right)$}\\
		&=-\int_0^{+\infty}\left(h^{\wedge}\left(C_t^{\prime}, h^E\right)-\frac{\chi^{\prime}(Z, F)}{2} 
		-\frac{  \chi(Z,F)\dim Z-2 \chi^{\prime}(Z, F)}{4} h^{\prime}\Big(\frac{\sqrt{-1} \sqrt{t}}{2}\Big)\right) \frac{d t}{t} .
	\end{align*}
\end{defn}

It follows from \cite[Theorem 3.21]{bismut1995flat} that $\mathcal{T}\left(T^H M, g^{T Z}, h^F\right)$ is well defined. The degree $0$ part of $\mathcal{T}\left(T^H M, g^{T Z}, h^F\right)$ is nothing but the fiberwise Ray-Singer analytic torsion.

Let $H\index{H@$H$}:=H^*(Z,F|_Z)\to S$ be the vector bundle whose fiber at $\s\in S$ is $H^*(Z_\s,F_\s),\s\in S$.  By Hodge theory, we have \be\label{Hodgeiso}\ker(D^Z)\cong H,\ee so the restriction of the $L^2$ metric $h^\bfE$ on $\ker(D^Z)\subset \Gamma(S,\bfE)$ induces a metric $h^H_{L^2}$  on $H\to S$.

According to \cite[Definition 2.4 and Proposition 2.5]{bismut1995flat}, the vector bundle $H\to S$ possesses a flat connection $\nabla^H$. Moreover, under the isomorphism \eqref{Hodgeiso}, \begin{equation}
\label{connection-on-H}
    \nabla^H\index{nablaH@$\nabla^H$}=\P^{\ker(D^Z)}\nabla^E\P^{\ker(D^Z)}
\end{equation}
where $\P^{\ker(D^Z)}$ is the orthogonal projection with respect to $\ker(D^Z)$. { Let $\nabla^{TZ}$ be the Bismut connection associated with $T^HM$ and $g^{TZ}$ as defined in \cite[Definition 1.6]{MR0813584}.
	Let $e(TZ,\nabla^{TZ})\in\Omega^{\dim Z}\big(M,o(TZ)\big)$ be the associated Euler form (cf. \cite[(3.17)]{bismutzhang1992cm}). Bismut and Lott \cite[Theorem 3.23]{bismut1995flat} established the following refined ``Riemann-Roch-Grothendieck" theorem for flat vector bundles:
	\begin{equation}
		\label{GRR-Bismut-Lott}
		d^S\mathcal{T}\left(T^H M, g^{T Z},h^F\right) =
		\int_Z e(TZ,\nabla^{TZ}) h(\nabla^F,h^F) -
		h\big(\nabla^{H},h^H_{L^2}\big) \;.
	\end{equation}
}

\begin{defn}\label{assotor}
	Given a Hermitian metric $h^H$ on $H\to S$, the modified analytic torsion form is defined as
	\[\tiT(T^HM,g^{TZ},h^F,h^H)\index{TTHMgTZhFhHtilde@$\tiT(T^HM,g^{TZ},h^F,h^H)$}:=\T(T^HM,g^{TZ},h^F)+\widetilde\ch(\nabla^{H},h^{\chH}),\]
	where $\chH\to \chS$ is the pullback of the bundle $H\to S$ with respect to the canonical projection $\chS\to S$  and $h^{\chH}$ is the metric on $\chH\to\chS$ that is given by $
	h^{\chH}|_{S\times\{l\}}=(1-l)h_{L^2}^H+lh^H.$   
	
	Moreover, if we fix a metric $h^{H_{\Cb^m}}$ on the bundle $H_{\Cb^m}:=H^*(Z,\Cb^m|_Z)\to S$\index{H0@$H_{\Cb^m}$}, we can define the form $\btiT(T^HM,g^{TZ},\nabla^F,h^F,h^H,h^{H_{\Cb^m}})\index{TTHMgTZnablaFhFhHhH0tildebar@$\btiT(T^HM,g^{TZ},\nabla^F,h^F,h^H,h^{H_{\Cb^m}})$}$ according to our bar convention of Section \ref{sec21}.
    \end{defn}

    Note that in this definition,    as the metric $h^{\chH}$ is fixed, we can view $\widetilde\ch(\nabla^{H},h^{\chH})$ as a differential form rather than just a class in $ \Omega^{\bullet}(S)/d^S \Omega^{\bullet}(S)$. Moreover, by Chern-Weil theory, if we replace $h^{\chH}$ with any metric $h^{\chH'}$ containing no exterior variable $dl$, and such that $\check{\I}_0^*h^{\chH'}=h^H_{L^2}$ and $\check{\I}_1^*h^{\chH'}=h^H$, the class of $\tiT(T^HM,g^{TZ},\nabla^F,h^F,h^H)$ in $ \Omega^{\bullet}(S)/d^S \Omega^{\bullet}(S)$ is unchanged.

{Observe that if $h^F$ is a flat metric,} then for $f\in C^\infty(M)$, we have $(e^{-f}h^F)^{-1}\nabla^F(e^{-f}h^F)=d^Mf$. Thus, by \cite[Theorem 3.24]{bismut1995flat}, we get the following.
\begin{prop}\label{indepf}
	If $h^F$ is a flat metric, then for any $f\in C^\infty(M)$ and any two metric $g_0^{TZ}$ and $g_1^{TZ}$ on $TZ$,
	\[\btiT(T^HM,g_0^{TZ},h^F,h^H,h^{H_{\Cb^m}})=\btiT(T^HM,g_1^{TZ},e^{-f}h^F,h^H,h^{H_{\Cb^m}})\]
	in $ \Omega^{\bullet}(S)/d^S \Omega^{\bullet}(S)$.
\end{prop}

\section{Fiberwise Generalized Morse Functions}\label{secfff}

\begin{defn}[Generalized Morse Functions]\label{generalizedMorse}
	A generalized Morse function\index{generalized Morse function} $ f $ on a smooth manifold is a smooth function with only nondegenerate and birth-death critical points. Near a birth-death critical point, in suitable coordinates, $ f $ is given by
	\[
	x_0^3 - \sum_{j=1}^i x_j^2 + \sum_{k=i+1}^n x_k^2 + C,
	\]
	where $ C $ is a constant.
\end{defn}

By Remark \ref{existence}, fiberwise generalized Morse function always exists.

In this section, we associate a set of non-Morse functions with each given generalized Morse function. We will also establish a lemma (Lemma \ref{indepbd}), similar to \cite[Lemma 4.5.1]{igusa2002higher}, to ensure the independence of certain critical points (Definition \ref{defn39}) from other critical points.

For this Section, we fix $S$ a closed manifold,  $\pi: M \to S$  a smooth fibration with closed $n+1$-dimensional  fiber $Z$, and  $(F,\nabla^F,h^F)$  a unitarily flat Hermitian vector bundle on $M$.

\subsection{Fiberwise generalized Morse functions}
\label{sec-fGMF}

Let $f:M\to \R$ be a fiberwise {generalized Morse function}\index{generalized Morse function!fiberwise}. That is, fiberwisely,  the critical points of $f_\s:=f|_{Z_\theta}$ are either nondegenerate or birth-death points. 

Let $\Sigma^i\left(f_\s\right)\index{Sigma@$\Sigma^i(f)$}$ be the set of all critical points of $f_\s$ of index $i$ and let
$$
\Sigma^{i}(f):=\cup_{\s \in S} \Sigma^{ i}\left(f_\s\right) \times \{\s\},\Sigma(f):=\cup_{i}\Sigma^i(f).
$$

Let $\Sigma^{(1, i)}\left(f_\s\right) \subset \Sigma\left(f_\s\right)\index{Sigma1@$\Sigma^{(1, i)}(f)$}$ be the set of all birth-death points of $f_\s$ of index $i$ and let
$$
\Sigma^{(1, i)}(f)=\cup_{\s \in S} \Sigma^{(1, i)}\left(f_\s\right) \times \{\s\} \subset \Sigma(f) .
$$

Let $\Sigma^{(1)}\left(f_\s\right), \Sigma^{(1)}(f)$ be the union of all $\Sigma^{(1, i)}\left(f_\s\right), \Sigma^{(1, i)}(f)$, respectively. Similarly, let $\Sigma^{(0, i)}(f)\index{Sigma0@$\Sigma^{(0, i)}(f)$}$ and $\Sigma^{(0)}(f)$ denote sets of nondegenerate points of index $i$ and of all indices, respectively. Let $\Sigma(f)=\Sigma^{(0)}(f)\cup\Sigma^{(1)}(f).$ The reader should be careful that $\Sigma^1$ refers to the set of all critical points of Morse index 1, while $\Sigma^{(1)}$ refers to the set of all birth-death points.

{The following lemma, taken from \cite{igusa2002higher} (see also \cite{igusa1987space}), gives the standard local form of $f$ near its critical points.}
\def\mfV{{\mathfrak{V}}}
\def\mfU{{\mathfrak{U}}}
\begin{lem}[Lemma 4.4.1 in \cite{igusa2002higher}]\label{lem441}
	There are two neighborhoods $V(i) \subseteq U(i)$ of $\Sigma^{(1, i)}(f)$ in $M$ for each $i$ and smooth embeddings
	$$
	\left(\pi, \phi_\alpha\right)=\left(\pi, u_{0,\a}, u_{1,\a}, \cdots, u_{n,\a}\right): V_\alpha \rightarrow S \times \mathbb{R}^{n+1}
	$$
	defined separately on open sets $V_\alpha$ covering the entire critical set $\Sigma(f)$, as well as $\mathfrak{U}=\cup U(i)$ and $\mathfrak{V}=\cup V(i)$\index{V@$\V$}, so that (We will write $u_{k,\a}$ as $u_{k}$ if it will not cause any confusion):
	\begin{enumerate}[(1)]
		\item\label{3a1}  On $V_\alpha \cap V_\beta, \phi_\alpha$ and $\phi_\beta$ differ by a rotation $g_{\alpha \beta}: \pi\left(V_\alpha \cap V_\beta\right) \rightarrow O(n+1)$.
		\item $f |_{\left(V_\alpha \cap \mfV\right)}=c\circ \pi+u_0^3- t_1 u_0-\left\|\left(u_1, \cdots, u_i\right)\right\|^2+\left\|\left(u_{i+1}, \cdots, u_n\right)\right\|^2$ where $t_1$ is a nonsingular smooth function on {$V$} which locally factors through $\pi$ (i.e., its restriction to each $V_\alpha$ factor through $\pi$).
		\item \label{3a3}$f |_{\left(V_\alpha \cap \mfU\right)}=g\circ\left(\pi, u_0\right)-\left\|\left(u_1, \cdots, u_i\right)\right\|^2+\left\|\left(u_{i+1}, \cdots, u_n\right)\right\|^2$ for some $g\in C^\infty\left((\pi, u_0)(\mfU)\right)$.
		\item $f|_{\left(V_\alpha-\mfU\right)}=c\circ\pi+\sum_{j=0}^n \pm u_j^2$.
	\end{enumerate}
\end{lem}

\begin{rem}\label{strgp}
	{Comparing the structure of $f$ on $V_\a$ and $V_\b$, we see that on $V_\a\cap V_\b\cap \mfV$, we have a more refined structure of $g_{\alpha \beta}: \pi\left(V_\alpha \cap V_\beta\cap \mfV\right) \rightarrow \{1\}\times O(i)\times O(n-i).$}
\end{rem}
\def\withlem{0}
\if\withlem1
With Lemma \ref{lem441}, we can defined fiberwise framed functions:
\begin{defn}[Framed Function]
	A framed function on a smooth manifold $Z$ is a pair $(f, \xi)$ where
	\begin{enumerate}[(1)]
		\item  $f: Z \rightarrow \mathbb{R}$ is a smooth function with only nondegenerate and birth-death critical points. Near a birth-death, by definition, $f$ is given by
		$$
		x_0^3-\sum_{j=1}^i x_j^2+\sum_{k=i+1}^n x_k^2+C
		$$
		\item $\xi=\left(\xi_1, \xi_2, \cdots, \xi_i\right)$ is a framing (vector space basis) for the negative eigenspace of $\mathrm{Hess} f$ at each critical point.
	\end{enumerate}
\end{defn}
\begin{defn}[Fiberwise Framed Functions]\label{fff}
	Given a smooth fiberation $M \rightarrow S$ with fiber $Z$, a fiberwise framed function on $M$ is a fiberwise generalized Morse function $f: M \rightarrow \mathbb{R}$ together with a smooth family of framings. At each birth-death point of index $i$ we define the $i+1^{\mathrm{st}}$ framing vector $v_{i+1}$ to be $\frac{\p}{\p u_0}$ in standard coordinates and we require that each vector $v_i$ be a smooth function on its domain
	$$
	\bigcup_{j+k \geq i} \Sigma^{(j, k)}(f) \subset M,(j=0,1).
	$$
\end{defn}
\fi


\def\fk{{\mathfrak{k}}}
Using Lemma \ref{lem441}, up to performing a small perturbation of $ f $ {inside $\mfV$}, we can assume that the fiberwise generalized Morse function  $ f $ gives the following stratification structure of $ S $ for some $ \fk \in \mathbb{N} $:
\begin{itemize}
	\item Let $S_k\index{S@$S_k$}:=\{\s\in S: \bbf_\s \mbox{ has exactly $k$ birth-death points}\}$ for $k=0,1,\cdots,\fk,$ 
	\item $\cup_{k=0}^\fk S_k=S,$
	\item $S_\fk\neq \emptyset$ is a closed submanifold of $S$,
	\item $ \overline{S}_{k}=\cup_{j=k}^\fk S_j$ for $k=0,1,\cdots,\fk$.
\end{itemize}

\textbf{Throughout the paper, we assume the metric $g^{TZ}$ on $TZ \to M$ to be standard\index{standard form of a metric near critical points} near $\Sigma(f)$}, i.e., in the coordinates of Lemma  \ref{lem441} on $V_\a$, \be\label{standardmetric} g^{TZ}=du_{0,\a}\otimes du_{0,\a}+\cdots+du_{n,\a}\otimes du_{n,\a}.\ee By item (\ref{3a1}) in Lemma \ref{lem441}, such a $g^{TZ}$ exists.


Let $L:=\pi(\Sigma^{(1)}(f))\subset S\index{L@$L$}$ and $\Omega:=\pi(\mfV)\index{Omega@$\Omega$}$. With the metric $g^{TZ}$ above, we may as well assume that $\mfV$ is a tubular neighborhood of $\Sigma^{(1)}(f)$ and $\mfV_\s:=\mfV\cap Z_\s,\s\in\Omega$ consists of several disjoint balls of radius $100$ (after doing some rescaling on $g^{TZ}$ and $f$), with center being the birth-death points if $\s\in L$. 



{
	\begin{defn}
    \label{def-covering-Omegak}
    Using the stratification of $ S $ described above, we construct inductively three open covers of $ S $, $ \{\Omega_k\}_{k=0}^\fk $, $ \{\Omega_k'\}_{k=0}^\fk $, and $ \{\Omega_k''\}_{k=0}^\fk $\index{Omegak@$\Omega_k$, $\Omega_k'$ and $\Omega_k''$}, as follows 
	\begin{enumerate}[(1)]
		\item For $k=1,\cdots,\fk$, we impose $\Omega_k''\subset\subset\Omega_k'\subset\subset \Omega_k$.
		\item We choose $\Omega_\fk''\subset\subset\Omega_\fk'\subset\subset\Omega_\fk$ to be three tubular neighborhoods of $S_\fk$ such that $ \overline\Omega_\fk\subset \Omega$ and $\forall\s\in\Omega_{\fk}$, $\mfV_\s$ consists of exactly $\fk$ balls.
		\item Let $k\in\{1,\dots,\fk-1\}$, and assume we have constructed $\Omega_j,\Omega_j'$ and $\Omega_j''$ for $j\in\{k+1,\cdots,\fk\}$. Then we choose $\Omega_{k}''\subset\subset\Omega_{k}'\subset\subset\Omega_{k}$ to be tubular neighborhoods of $S_{k}-\cup_{j=k+1}^\fk\Omega_j''$, such that:
		\begin{enumerate}
			\item $ \overline{\Omega}_{k}\subset\Omega$ and $ \overline{\Omega}_{k}\cap  \overline{S}_{k+1}=\emptyset$.
			\item $\forall\s\in\Omega_{k}$, $\mfV_\s$ consists of 
            $k$ balls.
		\end{enumerate}
		\item Finally we choose $\Omega_0''$ such that $\{\Omega_k''\}_{k=0}^\fk$ forms an open cover of $S$ {and $\overline{\Omega}_{0}\cap  \overline{S}_{1}=\emptyset$}. We then set $\Omega_0 = \Omega_0' = \Omega_0''$.
	\end{enumerate}	    
	\end{defn}
		Figure \ref{fig-Omega} illustrates our construction near $S_2$. The two straight lines represent the immersed manifold $L$, which self-intersects at $S_2$ (represented by the dotted point). The blue disk represents $\Omega_2$, while the gray parts represent $\Omega_1$.

	Through this construction, we can see that \begin{obs}\label{obs7}
		$L\subset\cup_{k=1}^\fk \Omega_k''.$ 
		Moreover, since $\Omega_k$ is a tubular neighborhood of $S_k - \cup_{j=k+1}^\fk \Omega_j$, for any $\s \in \Omega_k$, there is a canonical way to select $k$ balls among $\mfV_\s$. Thus we have a $k$-balls bundle over $\Omega_k$, denoted as $\mfV_k \to \Omega_k$.   \end{obs}
	
	\definecolor{colorone}{RGB}{9,139,189}
	\definecolor{colortwo}{RGB}{152,151,185}

	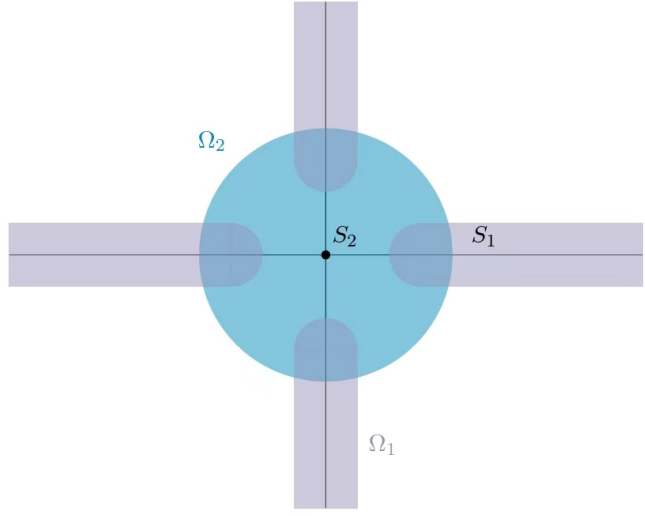
\begin{figure}[h]
		\begin{center}
			\begin{tikzpicture}
				\draw (-5,0) -- (5,0);
				\draw (0,-4) -- (0,4);
				\fill[fill=colorone, opacity=0.5] (0,0) circle(2);
				\fill[fill=colortwo,opacity=0.5](-0.5,1.5) rectangle (0.5,4);
				\fill[colortwo,opacity=0.5] (-0.5, 1.5) arc[start angle=180, end angle=360, radius=0.5];
				\fill[fill=colortwo,opacity=0.5](-0.5,-1.5) rectangle (0.5,-4);
				\fill[colortwo,opacity=0.5] (0.5, -1.5) arc[start angle=0, end angle=180, radius=0.5];
				\fill[fill=colortwo,opacity=0.5](-1.5,-0.5) rectangle (-5,0.5);
				\fill[colortwo,opacity=0.5] (-1.5,0.5) arc[start angle=90, end angle=-90, radius=0.5];
				\fill[fill=colortwo,opacity=0.5](1.5,-0.5) rectangle (5,0.5);
				\fill[colortwo,opacity=0.5] (1.5,0.5) arc[start angle=90, end angle=270, radius=0.5];
				\fill [black] (0,0) circle (2pt);
				\node at (0.3,0.3){$S_2$};
				\node at (-1.8,1.8){\textcolor{colorone}{$\Omega_2$}};
				\node at (2.5,0.3){$S_1$};
				\node at (0.9,-3){\textcolor{colortwo}{$\Omega_1$}};
			\end{tikzpicture}
		\end{center}
		\caption{Local picture near $S_2$}
		\label{fig-Omega}
	\end{figure}

	\subsection{Associated non-Morse functions}\label{assotwo}

		
	
	For $k=1,\cdots,\fk$, let $\phi_k:S\to [0,+\infty]$ be a smooth function  such that
	\begin{enumerate}[(1)]
		\item $\phi_k$ has a compact support inside $\Omega_k$,
		\item $\phi_k^{-1}(+\infty) = \Omega''_k$.
	\end{enumerate}
	
	For $k=1,\cdots,\fk$, let $\phi_{k,R}$ be a smooth family of functions in $C^\infty(S)$, parameterized by $R \in [0, +\infty)$. Here, $\phi_{k,R}$ should be thought of as a smoothed version of the function $\min\{\phi_{k}, R\}$, such that:
	\begin{enumerate}[(1)]
		\item $\phi_{k,R=0}\equiv0.$
		\item $\forall R\geq 0$, $\phi_{k,R}$ has a compact support inside $\Omega_k$.
		\item When $R$ is large enough $\phi_{k,R}|_{\Omega_k''}=R$, $\phi_{k,R}|_{\Omega_k'}\leq \phi_k|_{\Omega_k'}$, $\phi_{k,R}|_{\Omega_k-\Omega_k'}=\phi_k$ and $\phi_{k,R}|_{\Omega_k-\overline{\Omega_k''}}<R$.  
		\item $\exists C>0$, $\forall R\geq 0$, $|d^S\phi_{k,R}|\leq CR$ and $\left|\frac{\partial\phi_{k,R}}{\partial R}\right|\leq C$.
	\end{enumerate}

	Fix $r_1,r_2\in(0,1),r_1<r_2$ to be determined later (See Definition \ref{fixr1}).
	Let $q_A$\index{qA@$q_A$} be a smooth family of smooth function on $[0,\infty)$ parametrized by $A\in[0,\infty)$, s.t. (see also the picture below)
	\begin{enumerate}[(1)]
		\item $q_{A=0}\equiv0.$
		\item $q_A|_{[0,r_1]}\equiv -A(r_2-r_1)^2/2$, and $q_A|_{(r_2,\infty)}\equiv 0.$
		\item $q_A(\frac{r_1+r_2}{2})=-\frac{(r_1-r_2)^2}{4}.$
		\item {Near but not at $r_1$ and $r_2$, $q_A$ is quadratic:}
		$$\begin{aligned}
			& q_A|_{[r_1,0.51r_1+0.49r_2]}(s)=A\varphi\big(2e^{A^2}(s-r_1)\big)\times (s-r_1)^2/2-A(r_2-r_1)^2/2  \qquad\text{and}\\
			& q_A|_{[0.49r_1+0.51r_2,r_2]}(s)=-A\varphi\big(2e^{A^2}(r_2-s)\big)\times (s-r_2)^2/2,
		\end{aligned}$$
		where $\varphi\in C_c^\infty([0,\infty))$ is a cut-off function such that $0\leq\varphi\leq1$, $\varphi_{[0,(r_2-r_1)/8]}\equiv0,$ $\varphi_{[(r_2-r_1)/4,\infty]}\equiv1,$ $|\varphi'|\leq \cn_1$ and $|\varphi''|\leq \cn_2$ for some constants $\cn_1$ and $\cn_2$ depending only on $r_2-r_1.$
		\item $\Cn_1(r_2-r_1)A\leq{q_A'(s)}\leq 2\Cn_1(r_2-r_1) A$, $|q_A''|\leq \Cn_2A$ for some universal positive constants $\Cn_1$ and $
		\Cn_2$ whenever $s\in[0.51r_1+0.49r_2,0.49r_1+0.51r_2].$
		
	\end{enumerate}
	The existence of $q_A$ follows from \cite[Lemma 2.3]{Yantorsions}.
	
	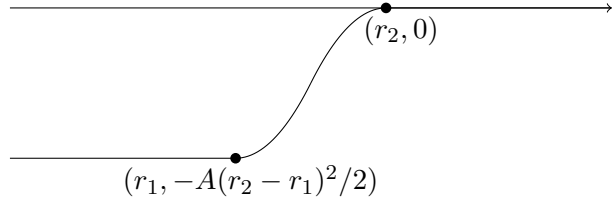
\begin{figure}[h]
		\begin{center}
			\begin{tikzpicture}
				\draw[->] (-4,1) -- (4,1);
				\draw[domain=-4:-1] plot (\x,{-1});
				\draw[domain=-1:0] plot (\x,{(\x+1)*(\x+1)-1});
				\draw[domain=0:1] plot (\x,{-(\x-1)*(\x-1)+1});
				\draw[domain=1:4] plot (\x,{1});
				
				\fill [black] (1,1) circle (2pt);
				
				\fill [black] (-1,-1) circle (2pt);
				\node at (-0.8,-1.3){$(r_1,-A(r_2-r_1)^2/2)$};
				\node at (1.2,0.7){$(r_2,0)$};
			\end{tikzpicture}
		\end{center}
		\label{fig3}
		\caption{The graph of $q_A$}
	\end{figure}

	{By Observation \ref{obs7}, we have $\fk$ functions $q_{k,A}\in C^\infty(\pi^{-1}(\Omega_k))$, such that for  $\s\in\Omega_k$
		\[
		q_{k,A}|_{Z_\s-\V_{k,\s}}=0 \text{ and } q_{k,A}=q_A\left(\sqrt{u_0^2+\cdots+u_n^2}\right) \text{ on each component of } \mfV_{k,\s}.
		\]
	}	
	
	{For $k=1,\cdots,\fk$, let $\psi_{k,R}\index{psikR@$\psi_{k,R}$}\in C^\infty(\pi^{-1}(\Omega_k))$ be such that for each $\s\in\Omega_k$, $\psi_{k,R}|_{Z_\s}=q_{k,A}$ for $A=\phi_{k,R}(\s).$ Then $\psi_{k,R}$ has a compact support inside $\pi^{-1}(\Omega_k)\cap \mfV$, and we can extend  $\psi_{k,R}$ to a smooth function on $M$.}

	
	\begin{assum}\label{3B} 
		To avoid heavy notation and overwhelming the reader with technical details, \textbf{we will temporarily assume that $S_k$ is empty, $k\geq3.$ That means on each fiber, $f$ has at most two birth-death points.} 
	\end{assum}
	Throughout this paper, {except in \cref{ineresult1}},  we shall assume that $f$ satisfies the Assumption \ref{3B}.
	
	\begin{defn}[Associated non-Morse functions]\label{assononmor}
		Let $f_{1,T,R}:=Tf+T\psi_{1,R}$ and $f_{2,T,R,R'}=Tf+T\psi_{1,R}+T\psi_{2,R'}$.
		We will call $\{f_{1,T,R},f_{2,T,R,R'}\}$ the associated  non-Morse functions\index{non-Morse function} for the fiberwise generalized Morse function $f$.
	\end{defn}

	\begin{rem}\label{rmkkpts}
		{In the same way, if $ f_{\s} $ admits at most $ \fk $ birth-death points for $ \s \in S $ with $ \fk > 2 $, then there will be $ \fk $ associated non-Morse functions as described in Definition \ref{assononmor}.}
	\end{rem}

	As $(F,\nabla^F,h^F)$ is a unitarily flat vector bundle, we can also assume that $F$, $h^F$ and $\nabla^F$ are standard\index{standard form of a unitarily flat vector bundle near critical points} inside $\V$, in the following sense.
	\begin{cond}\label{assum32}
		\begin{enumerate}[(1)]
			\item\label{assum4} Let $i:\Omega_1\to \mfV_1$ be the section  giving the center of the ball that forms of the bundle $\mfV_1\to\Omega_1$ of Observation \ref{obs7}. Then $F|_{\mfV_1}=\pi^*\circ i^*F|_{\mfV_1}$,  $h^F|_{\mfV_1}$ is the pullback of some Hermitian metric $h^{i^*F|_{\mfV_1}}$ and $\nabla^F|_{\mfV_1}$ is the pullback of some compatible unitarily connection $\nabla^{i^*F|_{\mfV_1}}$.
			\item \label{assum5}Let $i_1,i_2:\Omega_2\to \mfV_2$ be the two sections giving the centers of the balls that form the bundle $ \pi|_{\mfV_2}:\mfV_2\to\Omega_2$. Let $\mfV_{2,a}$, $a=1,2$ be the two connected components of $\mfV_2$, so that $\pi_a:\mfV_{2,a}\to\Omega_2$ is a one-ball bundle. Then $F|_{\mfV_2}=\pi_1^*\circ i_1^*F|_{\mfV_2}\sqcup\pi_2^*\circ i_2^*F|_{\mfV_2}$,  $h^F|_{\mfV_2}$ is the pullback of some Hermitian metric $h^{i_1^*F|_{\mfV_1}}\oplus h^{i_2^*F|_{\mfV_2}}$ and  $\nabla^F|_{\mfV_2}$ is the pullback of some compatible unitarily connection $\nabla^{i_1^*F|_{\mfV_2}}\oplus \nabla^{i_2^*F|_{\mfV_2}}$.
			\item The Hermitian metrics $h^{i^*F|_{\mfV_1}}$ and $h^{i_1^*F|_{\mfV_1}} \oplus h^{i_2^*F|_{\mfV_2}}$ agree on $\Omega_1 \cap \Omega_2$, and the unitarily connections $\nabla^{i^*F|_{\mfV_1}}$ and $\nabla^{i_1^*F|_{\mfV_2}} \oplus \nabla^{i_2^*F|_{\mfV_2}}$ also agree on $\Omega_1 \cap \Omega_2$.
		\end{enumerate} 
		
\end{cond}}

\subsection{Independence of critical points}
Let $g^{TZ}$ be a metric on $TZ$.

\begin{defn}
	Let $f:Z\to\R$ be a (generalized) Morse function and let $p$ be a critical point of $f$.
	
	A vector field $\mathcal{X}\subset\Gamma(TZ)$ is called a gradient-like vector field, if 
	\begin{enumerate}[(1)]
		\item $\mathcal{X}=-\nabla f$ in a small neighborhood of critical points of $f$.
		\item Away from critical points of $f$, $g^{TZ}(\mathcal{X},\nabla f)<0$.
	\end{enumerate}
\end{defn}

\begin{defn}\label{defn38}
	Let $f:Z\to\R$ be a (generalized) Morse function and let $p$ be a critical point of $f$.
	
	Let $\varphi^t$ be the flow generated by a gradient-like flow $\mathcal{X}$, then the unstable manifold of $p$ with respect to $\mathcal{X}$ is defined as
	$\rW^\ru_{\mathcal{X}}(p):=\{q\in Z: \lim_{t\to-\infty}\phi^t(q)=p\}$; 
	the stable manifold of $p$ is defined as
	$\rW^\rs_\mathcal{X}(p):=\{q\in Z: \lim_{t\to\infty}\phi^t(q)=p\}$.

	For any $c>0$, the unstable level set of $p$ is defined as $\rW^\ru_{\mathcal{X},c}(p):=\{q\in Z: \lim_{t\to-\infty}\phi^t(q)=p\}\cap\{q\in Z: f(q)= f(p)-c\}$; 
	for any $c>0$, the stable level set of $p$ is defined as $\rW^\rs_{\mathcal{X},c}(p):=\{q\in Z: \lim_{t\to\infty}\phi^t(q)=p\}\cap\{q\in Z: f(q)=f(p)+c\}$.
	
	We say $(f,\mathcal{X})$ (or $(f,g^{TZ})$) satisfied \textbf{Thom-Smale transversality
		condition} (see also \cite[p. 219]{bismutzhang1992cm}), if any pair $(p,q)$ of critical point of $f$, $\rW^{\rs}_{\mathcal{X}}(p)$	 and $\rW_{\mathcal{X}}^{\ru}(q)$ (or $\rW^{\rs}_{-\nabla f}(p)$	 and $\rW_{-\nabla f}^{\ru}(q)$) intersect transversely.
	We will also abbreviate $\rW^{\rs/\ru}_{\mathcal{X}}(p)/\rW^{\rs/\ru}_{\mathcal{X},c}(p)$ as $\rW^{\rs/\ru}(p)/\rW_c^{\rs/\ru}(p)$ if there is no need to emphasis the vector field.
\end{defn}

\begin{defn}[Independence of critical points]\label{defn39}
	Let $f:Z\to \R$ be a generalized Morse function. We say two critical points $p$ and $q$ of $f$ are independent (for $\mathcal{X}$), if 
	$(\rW_\mathcal{X}^\ru(p)\cup \rW_\mathcal{X}^\rs(p))\cap (\rW_\mathcal{X}^\ru(q)\cup \rW_\mathcal{X}^\rs(q))=\emptyset$. 
	
	We say a critical point $p$ of $f$ is independent (for $\mathcal{X}$) if $p$ is independent of any other critical points of $f.$

\end{defn}
{This definition is taken from \cite[p. 62]{Hatcher}. An other way to state it is that $p$ and $q$ are independent iff there is no point flowing from one to the other for $t$ going from $-\infty$ to $+\infty$.}

Suppose we have a splitting \be\label{spliting111} TM=T^HM\oplus TZ\ee and a metric $g^{TS}$ on $TS\to S$. Then {we take $g^{TM}$\index{gTM@$g^{TM}$} to be $g^{TZ}\oplus g^{TS}$ } with respect to the splitting above.
\begin{defn}[Fiberwise gradient vector fields]\label{defn310}
	Let $f\in C^\infty(M)$ be a fiberwise generalized Morse functions for the fibration $\pi:M\to S,$ and $g^{TM}$ be the Riemannian metric introduced above.

	Then the fiberwise gradient vector field is $\nabla^v f:=P^{TZ}\nabla f$. Here $P^{TZ}$ is the projection with respect to the splitting \eqref{spliting111} above, and $\nabla f$ is the gradient of $f$ with respect to $g^{TM}.$ One can see that $\nabla^v f$ is independent of the choice of $g^{TS}.$
	
	Likewise, we could introduce fiberwise gradient-like vector fields, fiberwise stable/unstable manifolds, fiberwise Thom-Smale transversality conditions, fiberwise independence, etc.
	
\end{defn}


The following lemma would be needed, and is {an adaptation} of \cite[Lemma 4.5.1]{igusa2002higher}.   

\begin{lem}\label{indepbd}
	Let $f:M\to\R$ be a fiberwise (generalized) Morse function with respect to the fibration $\pi:M\to S$. 
	
	Let $\Sigma'$ be equal to $\Sigma^{(1,k)}(f)$ or an connected components of $\Sigma^{(0,k)}(f)$ that is disjoint from $\Sigma^{(1)}(f).$ {We assume that, for any $p\in \Sigma'$, there are $c=c(p)>0$ and a smooth closed ball $D_{p}^{\rs/\ru}\subset \{x\in Z_{\pi(p)}:f_{\pi(p)}(x)=f_{\pi(p)}(p)\pm c(p)\}$ depending smoothly on $p\in\Sigma'$, such that $\rW^{\rs/\ru}_{-\nabla^v f|_{Z_{\pi(p)}},c}(p)\subset D_{p}^{\rs/\ru}$. Here by `closed ball', we mean a smooth manifold with boundary that is homeomorphic to a bounded closed ball $\{x\in \R^{k'}:|x|\leq1\}$ for some $k'.$}
	
	Let $U,U'$ be any two small enough neighborhood of $\Sigma'$, s.t. $ \bar{U}\subset U'$.


	{For any $q\in \Sigma^{(j,l)}(f)$ satisfying
		\be\begin{aligned}
			&\pi(p)=\pi(q), \qquad |f_{\pi(p)}(q)-f_{\pi(p)}(p)|>c \qquad\text{and}\\
			&1+\dims\leq l,k \leq n+1-\dims-j-1=\operatorname{dim} Z-\operatorname{dim} S-j-1 \footnotemark,
		\end{aligned}\ee
		there is a fiberwise gradient-like vector field $\mathcal{X}$  for $f$, which fiberwisely
		equal to the vertical gradient $\nabla^v f$ outside $U'- \bar{U}$, and for which {any} $p\in \Sigma'$ is independent of $q$.}
	
	{An other way of saying this is that for any such $q$, there is a deformation of metric $g^{TZ}$ inside $U'- \bar{U}$ such that $p$ is independent of $q$ for $\nabla^v f$.}

	\footnotetext{i.e., Morse points have index $\leq n-\dims$ and birth-death points have index $\leq n-\dims-1$.}
\end{lem}
{
	\begin{rem}\label{rem-indepbd}
		In this lemma, in the case where  $\Sigma'=\Sigma^{(1,k)}(f)$, by the proof of \cite[Lemma 4.5.1]{igusa2002higher} we can see that the existence of a ball containing  $\rW^{\rs/\ru}_{-\nabla^v f|_{Z_{\pi(p)}},c}(p)$ is in fact automatic.
\end{rem}}
\begin{proof} The proof is essentially the same as the proof of \cite[Lemma 4.5.1]{igusa2002higher}.
	
	By symmetry, we may assume that  $f(p)>f(q)$, so we must have $f(p)-c>f(q)$ by assumption.
	
	Fix $c'\in(0,c),$ such that $(f(p)-c,f(p)-c')$ has no critical value of $f$. Let $V_{c,c'}$ be some open subset, s.t., $$ V_{c,c'}\subset  U_{c,c'}:=\left\{x\in Z_{\pi(p)}: f_{\pi(p)}(x)\in\left( f_{\pi(p)}(p)-c,f_{\pi(p)}(p)-c'\right)\right\},$$ $${V}_{c,c'}\supset  U_{c,c'}\cap \rW^\ru(p).$$
	Let $L_{c}:=\{x\in Z_{\pi(p)}: f_{\pi(p)}(x)= f_{\pi(p)}(p)-c\}$.

	Recall that an isotopy of a smooth manifold $X$ is a family of diffeomorphism $\Phi_\tau:X\to X$ parametrized by $\tau\in[0,1]$, s.t. $\Phi_0$ is the identity map.
	Deforming the gradient field to gradient-like vector field inside $V_{c,c'}$ (or deforming the metric inside $V_{c,c'}$) will result in an isotopy of the level surface $L_c$. More explicitly, let $\mathcal{X}$ be any gradient-like vector field of $f$, s.t. outside $V_{c,c'}$, $\mathcal{X}=-\nabla^vf.$ Let $\mathcal{X}(\tau):=\tau \mathcal{X}-(1-\tau)\nabla^vf$, and define a vector field inside $ \bar{U}_{c,c'}$ as follows: $\tilde{\mathcal{X}}(\tau):=\frac{-\mathcal{X}(\tau)}{g^{TZ}\left(\mathcal{X}(\tau),\nabla^vf\right)}$. Then we can see that $\tilde{\mathcal{X}}(\tau)(f)=-1.$  Let $\phi^t_\tau$ be the gradient flow generated by $\tilde{\mathcal{X}}(\tau).$ 
	One can also see easily that $\Phi_{\tau}:=\phi^{c-c'}_\tau\circ\phi^{c'-c}_0$ is an isotopy  of $L_c$ and $\Phi_\tau(\rW^{\ru/\rs}_{-\nabla^vf,c}(p))=\rW^{\ru/\rs}_{\mathcal{X}(\tau) ,c}(p)$.

	{To make $p$ independent of $q$, we construct an isotopy of the level surface $L_{c}$ as follows. Assume that we have a point $w(p)$ near or inside the disc $D^\ru_p$ so that $\rW^{\rs}(q)\cap L_{c}$ doesn't hit that point. Then by an isotopy of $L_c$, we can shrink the ball $D^\ru_p $ into a small neighborhood of $w(p)$ so that $\rW^\ru_{c}(p) $ and $\rW^\rs(q) \cap L_c$ don't meet (also note that $L_c$ is a manifold without boundary, and $D^\ru_p$ is a manifolds with boundary, so $L_c$ contains {the closure} of $D^\ru_p$). We can realize this isotopy by either deforming the metric inside $V_{c,c'}$ or deforming the gradient inside $V_{c,c'}$ to a gradient-like vector field.}
	
	Hence, we need a smooth function $w: \Sigma' \to M$ so that
	\begin{enumerate}[(1)]
		\item $w(p) \in L_{c}$ for all $p \in \Sigma'$,
		\item $w(p)$ lies arbitrarily close to $D_p^{\ru}$,
		\item $w(p)\notin \rW^\rs(q)$.
	\end{enumerate}
	
	In fact, any generic choice of a function $w$ satisfying (a) and (b) will also satisfy (c) by transversality. Indeed, the set of points in $L_{c}\cap W^\rs(q)$ has codimension $l$ in $L_c$, and the corank is at least $\dim(S)+1$ so in an $\dims$-parameters family $\left(\Sigma'\right.$ being a $\dim(S)$ or $\dims-1$-manifold) this set will be disjoint from a generic $w(p)$.
\end{proof}

\section{Model Near Birth-death Points}\label{model}

In this section, we will study the model near birth-death points, specifically focusing on Witten deformation for associated non-Morse functions on $\R^{n+1}$. We will introduce a family of generalized Morse functions $f_y: \R^{n+1} \to \R$ parameterized by $y \in (-\delta, \delta)$ for some small $\delta > 0$, which will be our model. We will then deform $f_y$ using the non-Morse function $q_A$ introduced in \cref{assotwo}. When $A$ is large, we will see that this deformation results in the addition of six new Morse points near the birth-death point. In \cref{combitor}, we will use Lemma \ref{indepbd} to ensure these points are independent of all other critical points. \\

In this section, we denote the standard norm on $\R^k$ by $|\cdot|$, for any $k$.

For any $r>0,$ let $D(r):=\left\{u \in \mathbb{R}^{n+1}:|u| \leq r\right\}$\index{Dr@$D(r)$} and $S(r):=\left\{u \in \mathbb{R}^{n+1}:|u|= r\right\}$\index{Sr@$S(r)$}. Also, for any $r''>r'>0$,
$$
D\left(r', r''\right)=\left\{u \in \mathbb{R}^n: r' \leq |u| \leq r''\right\}
\index{Drr@$D(r',r'')$}.$$

Fix $i\in\{2,\cdots,n-1\}$. For $y\in\R$ small enough, let $f_y\in C^\infty(\R^{n+1})$ be given by
\[f_y\index{fy@$f_y$}(u)=u_0^3-yu_0-|(u_1,\cdots,u_i)|^2+|(u_{i+1},\cdots,u_{n})|^2, u=(u_0,\cdots,u_n)\in\R^{n+1}.\]

Fix $r_1,r_2\in\R$, s.t. $0<r_1<r_2<\frac{1}{14}$, and let $\delta\in(0,\frac{r_1}{36})$ be small enough. 

Then there exists $\eta\in C^\infty([0,\infty))$ such that
\begin{cond}\label{e}
	\begin{enumerate}[(1)]
		\item\label{e1} $\eta\geq 0$ ,$\eta|_{[0,\frac{r_1}{6}]}\equiv0$ and $\eta|_{[\frac{5r_2}{2},\infty)}\equiv0.$
		\item\label{e2} $\eta|_{(r_1,r_2)}(s)=\delta s.$
		\item\label{e3} $|\eta'|<2\delta.$
		\item\label{e4}
		$\eta(s)>0$ if $s\in(\frac{r_1}{6},\frac{5r_2}{2})$
	\end{enumerate}
\end{cond}
Let $\tilde{\eta}\in C^\infty([0,\infty))$ satisfying
\begin{cond}\label{f}
	\begin{enumerate}[(1)]
		\item $0\leq\tilde\eta\leq 1$. 
		\item $\tilde\eta|_{[\frac{r_1}{2},2r_2]}\equiv1$, $\tilde\eta|_{[0,\frac{r_1}{6}]}\equiv0$ and $\tilde\eta|_{[\frac{5r_2}{2},\infty)}\equiv0.$
		\item\label{f3} $|\tilde{\eta}'|\leq \frac{2}{r_1} {(\leq \frac{1}{18\delta})}.$
	\end{enumerate}
\end{cond}

\begin{lem}\label{nocritical}
	For $|y|<\delta^2$, the function (generalized) Morse function $\tilde{f}_y$ given by
	$\tilde{f}_y\index{ftilde@$\tilde{f}_y$}(u)=
	f_y(u)-\eta(|u|)u_1+y\tilde{\eta}(|u|)u_0$ 
	has the same critical points as $f_y$, with same Morse indices.
\end{lem}
\begin{rem}
  We introduce the perturbation $ y\tilde\eta(|u|)u_0 $ to ensure that $ \tilde{f}_y|_{D(r_1,r_2)} $ remains independent of $ y $, which simplifies the computations in \Cref{threecrits} and \Cref{sixcrits}. The term $ -\eta(|u|)u_1 $ is introduced to guarantee that $ f_{A,y}$ defined below is a generalized Morse function if $A$ is large.
\end{rem}
\begin{proof}
Since $ f_y $ and $ \tilde{f}_y $ are identical outside the region $ D\left(\frac{r_1}{6}, \frac{5r_2}{2}\right) $, it suffices to prove that $ \tilde{f}_y $ has no critical points within this region.

Let $ \gamma = |(u_1, \cdots, u_n)| $, and note that it follows from the expression for the gradient that
\be\label{eq41}
|\nabla f_y| = \sqrt{(3u_0^2 - y)^2 + {4}|(u_1, \cdots, u_n)|^2} \geq \frac{|3u_0^2 - y| + {2}\gamma}{\sqrt{2}}.
\ee
Now, observe that by the properties of $ \eta $ and $ \tilde{\eta} $, on the region $ D\left(\frac{r_1}{6}, \frac{5r_2}{2}\right) $, we have
\be
|\nabla \eta(|u|) u_1| = \left| \eta'\left(|u|\right) u_1 \nabla |u| + \eta(|u|) \nabla u_1 \right| \leq  4\delta |u|,
\ee
and 
\be\label{eq43}
|\nabla y \tilde{\eta}(|u|) u_0|=\big|y \tilde{\eta}(|u|) \nabla u_0+y\tilde{\eta}'(|u|)u_0\nabla |u|\big|  \leq \delta^2 + \frac{\delta |u|}{18}.
\ee
Now consider two cases:
\begin{enumerate}[{Case} 1]
	\item If $ \gamma^2 \geq |u|^2 / 2 $.
	In this case, using \eqref{eq41}-\eqref{eq43}, on $ D\left(\frac{r_1}{6}, \frac{5r_2}{2}\right) $, we obtain
	\[
	|\nabla \tilde{f}_y| \geq \left( {1} - 4\delta - \frac{\delta}{18} \right) |u| - \delta^2 \geq\frac{|u|}{ {2}}-\delta^2\geq \frac{r_1}{{12}} - \delta^2 > 0.
	\]
	
	\item  If $ \gamma^2 \leq |u|^2 / 2 $.
	Then we have $ u_0^2 \geq |u|^2 / 2 $. Again, applying \eqref{eq41}-\eqref{eq43}, on $ D\left(\frac{r_1}{6}, \frac{5r_2}{2}\right) $, 	
\end{enumerate} 
\[
	|\nabla \tilde{f}_y| \geq \frac{3 |u|^2}{2\sqrt{2}}-\frac{|y|}{\sqrt{2}} - 4\delta |u| - \delta^2 - \frac{\delta |u|}{18} \geq \left( \frac{r_1}{4\sqrt{2}} - 5\delta \right) |u| - 2\delta^2 \geq \frac{r_1 |u|}{{72}} -2 \delta^2 \geq \frac{r_1^2}{432} -2 \delta^2 > 0.
	\]
In both cases, we ensure that $ \tilde{f} $ has no critical points within $ D\left(\frac{r_1}{6}, \frac{5r_2}{2}\right) $.

	\def\thatis{1}
	\if\thatis0
	 that is, the equation $\frac{\p}{\p u_k}\tf=0,k=0,1,\cdots n$ has no solution inside $D(\frac{r_1}{6},\frac{5r_2}{2})$.
	
	One computes
	\begin{align}
		&\label{du0} \frac{\partial}{\partial u_0} \tilde{f}=3 u_0^2-y-\eta^{\prime}(|u|) \frac{u_1 u_0}{|u|}+y\tilde{\eta}(|u|)+y\tilde{\eta}'(|u|)\frac{u_0^2}{|u|}, \\
		&\label{du1} \frac{\partial}{\partial u_1} \tilde{f}=-2 u_1-\eta(|u|)-\eta^{\prime}(|u|) \frac{u_1^2}{|u|}+y\tilde{\eta}'(|u|)\frac{u_0u_1}{|u|}, \\
		&\label{duj} \frac{\partial}{\partial u_j} \tilde{f}=-2 u_j-\eta^{\prime}(|u|) \frac{u_1 u_j}{|u|}+y\tilde{\eta}'(|u|)\frac{u_0u_j}{|u|},\qquad j\in\{2,\cdots,i\}, \\
		&\label{duk} \frac{\partial}{\partial u_k} \tilde{f}=2 u_k-\eta^{\prime}(|u|) \frac{u_k u_1}{|u|}+y\tilde{\eta}'(|u|)\frac{u_0u_k}{|u|}, \qquad k\in\{i+1,\cdots,n\}.
	\end{align}

	{Assume that $u$ is a critical point of $\tf$ with $u\in D(\frac{r_1}{6},\frac{5r_2}{2})$.}
	\begin{itemize}
		\item If $u_1=0$, then by \eqref{du1}, $\frac{\p}{\p u_1}\tf=-\eta(|u|)\neq0.$ {Hence $u_1\neq 0$.}
		\item {By item (\ref{e3}) in Condition \ref{e}, item (\ref{f3}) in Condition \ref{f}, \eqref{duj}, \eqref{duk},  and the fact that $\frac{|u_1|}{|u|}\leq 1$, $\frac{\p}{\p u_\ell}\tf=0$ implies that $u_\ell=0$ for $\ell\in\{2,\cdots,n\}$ if $\delta$ is small enough.}

		On the other hand, {we claim} that $\frac{\p}{\p u_0}\tf=0$ implies that \be\label{u0est}|u_0|\leq 2\delta.\ee Indeed, if $|u_0|>2\delta$, as $|y|<\delta^2$, by item (\ref{e3}) in Condition \ref{e} and item (\ref{f3}) in Condition \ref{f} we get
		\[
		\begin{aligned}
			&\ \ \ \ 3 u_0^2-y-\eta^{\prime}(|u|) \frac{u_1 u_0}{|u|}+y\tilde{\eta}(|u|)+y\tilde{\eta}'(|u|)\frac{u_0^2}{|u|}\geq 3u_0^2-\delta^2-4\delta|u_0| \\
			&{> 6\delta|u_0|-\delta^2-4\delta|u_0| >3\delta^2 >0.}
		\end{aligned}\]
		{Recall that $\delta<\frac{r_1}{24}$, so \eqref{u0est} gives $|u_0|\leq \frac{r_1}{12}$. Since moreover $|u|^2=u_0^2+u_1^2\geq \frac{r_1^2}{36}$,  we have $u_1^2\geq \frac{r_1^2}{48}\geq 3u_0^2$. Hence,} 
		\be\label{u1est}|u_1|\geq {\frac{\sqrt{3}}{2}|u|}> \frac{|u|}{2}.\ee
		
		Now, as $|y|<\delta^2$, by item (\ref{e1}) and (\ref{e3}) in Condition \ref{e}, item (\ref{f3}) in Condition \ref{f} and \eqref{u1est}, we have $|\eta(|u|)|\leq 4\delta |u_1|$ , $|\eta^{\prime}(|u|) \frac{u_1^2}{|u|}|\leq 2\delta|u_1|$ and $|y\tilde{\eta}'(|u|)\frac{u_0u_1}{|u|}|\leq \delta|u_1|/12\leq 2\delta|u_1|$. Thus by \eqref{du1}, 
		$$\left|\frac{\partial}{\partial u_1} \tilde{f}\right|\geq (2-8\delta)|u_1|\geq (1-4\delta)|u|.$$
		Since $\delta<r_1/24<1/4$, this is a contradiction.
	\end{itemize}

	As a conclusion, we have proved that $\tf$ has no critical point inside $D(\frac{r_1}{6},\frac{5r_2}{2}).$\fi
\end{proof}
\begin{rem}\label{rem41}
	Let $ \varrho\in[0,1] $. Then one can see from the proof of Lemma \ref{nocritical} that the function $ f_y(u) - \varrho\big(\eta(|u|)u_1 + y\tilde{\eta}(|u|)u_0\big) $ also has the same critical points as $ f_y $.
\end{rem}

Let $q_A\in C^\infty[0,\infty)$ be the function defined in \cref{assotwo}. Let \be\label{defn fay} f_{A,y}(u)\index{fAy@$f_{A,y}$}:=\tf_y(u)+q_A(|u|),\ee we would like to find critical point of $ f_{A,y}(u)$ when $A$ is large, note that $q_A$ depends only on the radial direction, and that $q_A'>0$ on $(r_1,r_2)$. We first prove the following.
\begin{lem}\label{threecrits}
	Let $\rho=|u|$. For $\rho\in(r_1,r_2)$,  $\tilde{f}_y|_{S(\rho)}$, as a function on $S(\rho)$, has exactly three critical points inside the set $\left\{u\in S(\rho):\left(\frac{\p}{\p\rho}\tilde{f}_y\right)\Big|_{S(\rho)}(u)<0\right\}$. Their Morse index is $0$, $i$ and $i-1$ respectively. 
\end{lem}
\begin{proof}
The proof of this lemma follows directly from basic multivariable calculus.

	By our constructions, $ \tf_{y}|_{D(r_1,r_2)} $ is independent of $y$, so we will prove the our lemma for $y=0$.
	
	On $D(r_1,r_2)$,
	\begin{equation}
		\label{formuletf}
		\tf_{0}=u_0^3-(u_1^2+\dots+u_i^2) + (u_{i+1}^2+\dots+u_n^2)-\delta |u|u_1,
	\end{equation}
	so we compute
	\begin{align}
		&\label{duf1} \frac{\partial}{\partial u_1} {\tf}_{0}=-2 u_1-\delta|u|- \frac{\delta u_1^2}{|u|}, \\
		&\label{dufj} \frac{\partial}{\partial u_j} {\tf}_{0}=-2 u_j- \frac{\delta u_1 u_j}{|u|}, \qquad j\in\{2,\cdots,i\}, \\
		&\label{dufk} \frac{\partial}{\partial u_k} {\tf}_{0}=2 u_k- \frac{\delta u_k u_1}{|u|} ,\qquad k\in\{i+1,\cdots,n\}.
	\end{align}
	Assume that $u=(u_0,\cdots,u_n)\in S(\rho)$ is a critical point of $\tf_0|_{S(\rho)}$ satisfying $\left(\frac{\p}{\p\rho}\tilde{f}_y\right)\Big|_{S(\rho)}(u)<0$. Then  the vector $\left(\frac{\partial}{\partial u_0} {\tf}_{0}(u),\cdots,\frac{\partial}{\partial u_n} {\tf}_{0}(u)\right)$ must have the same direction as $-u$, i.e., 
    \begin{equation}
    \label{direction-gradient}
     \exists\lambda>0 \quad \text{ such that } \quad \left(\frac{\partial}{\partial u_0} {\tf}_{0}(u),\cdots,\frac{\partial}{\partial u_n} {\tf}_{0}(u)\right)=-\lambda u .   
    \end{equation}
    Using \eqref{duf1} and  \eqref{direction-gradient}, we see that $u_1\neq0$, because if not, we would have $-\delta |u| = -\lambda u_1 =0$ and $|u|\neq0$. Also, by \eqref{dufk} and \eqref{direction-gradient}, and because $2 - \delta u_1/|u|>0$, we must have $u_k=0$ for $k\in\{i+1,\cdots,n\}$. Finally, comparing \eqref{duf1} and \eqref{dufj}, and using \eqref{direction-gradient}, we see that $u_j=0$ for $j\in\{2,\cdots,i\}.$

	\def\u{{\bar{u}}}\def\diag{\mathrm{diag}}
	Now	assume that $\u=(\u_0,\u_1,0,\cdots,0)\in S(\rho)$ with $\u_1\neq0$, and we would like to solve $\nabla^{S(\rho)}\left(\tf_0|_{S(\rho)}\right)(\u)=0$ under the restriction that $\frac{\p}{\p \rho} \tf_0(\u)<0$, where $\nabla^{S(\rho)} \left(\tf_0|_{S(\rho)}\right)$ is the gradient of $\tf_0|_{S(\rho)}$ on $S(\rho)$.
	\def\sign{{\mathrm{sign}}}
	Since $\u_1\neq0$, we can see that $\sigma_l:=\u_l/\rho$, $l\neq1$ can serve as a coordinates on $S(\rho)$ near $\u$. Under these coordinates,
	\be\label{coordtf}\tf_0|_{S(\rho)}=\rho^3\sigma_0^3-\rho^2(1-\sigma_0^2)+2\rho^2(\sigma_{i+1}^2+\cdots +\sigma_n^2)-\sign(\u_1)\delta\rho^2\sqrt{1-\sigma_0^2-\sigma_2^2-\cdots -\sigma_n^2}.\ee
	\begin{enumerate}[(1)]
		\item If $\u_0=0$, by \eqref{coordtf}, we  see that $\nabla^{S(\rho)} \Big(\tf_0|_{S(\rho)}\Big)(\bar{u})=0$. So $v_{1,\rho}=(0, \rho,0,\cdots,0)$ and $v_{2,\rho}=(0,-\rho,0,\cdots,0)$ are critical points of $\tf_0|_{S(\rho)}$. Moreover, by \eqref{coordtf}, for $b=1$ or $2$,
		\be\label{vbrho}
		\frac{\p}{\p \rho} \tf_0(v_{b,\rho})=-2\rho+(-1)^b2\delta\rho<0 
		\ee
		and
        \begin{multline}
        \label{tangential-hessian-v_b}
            \left(\frac{\p^2}{\p \sigma_l\p\sigma_{l'}}\left(\tf_0|_{S(\rho)}\right)\right)_{l,l'\neq1}(v_{b,\rho})\\
		=\rho^2\diag(2+(-1)^{b+1}\delta,(-1)^{b+1}\delta,\cdots,(-1)^{b+1}\delta,4+(-1)^{b+1}\delta,\cdots, 4+(-1)^{b+1}\delta).
        \end{multline}
		So $v_{1,\rho}$ has Morse index $0$ and $v_{2,\rho}$ has Morse index $i-1$.
		\item If $\u_0\neq0$, we  have $\frac{\p}{\p \sigma_l}\left(\tf_0|_{S(\rho)}\right)(\u)=0$ for $l\geq2$ and
		\[\frac{\p}{\p \sigma_0}\left(\tf_0|_{S(\rho)}\right)(\u)=3\rho^3\sigma_0^2+2\rho^2\sigma_0+\sign(\u_1)\delta\rho^2\frac{\sigma_0}{\sqrt{1-\sigma_0^2}}.\]
		
		Since $\sigma_0\neq0$, Solving $\frac{\p}{\p \sigma_0}\left(\tf_0|_{S(\rho)}\right)(\u)=0$ is equivalent to solve
		\be\label{solvedsigma0}
		3\rho\sigma_0+2+\frac{\sign(\u_1)\delta}{\sqrt{1-\sigma_0^2}}=0.
		\ee
		If $\sign(\u_1)=1$, since $\rho<r_2<1/14,\sigma_0\in(-1,1)$, one can see \eqref{solvedsigma0} has no solution.
		
		Assume now that $\sign(\u_1)=-1$. Let $h(s)=	3\rho s+2-\frac{\delta}{\sqrt{1-s^2}}$. Then $h'(s)=3\rho-\frac{\delta s}{(1-s^2)^{3/2}}$ and $h''(s)=-\delta \left( \frac{3 s^2}{(1-s^2)^{5/2}}+\frac{1}{(1-s^2)^{3/2}}\right)<0$. Hence, $h'(s)=0$ has exactly one solution $s_0$ on $(-1,1)$, and $s_0\in(0,1)$. As $\lim_{s\to\pm1}h'(s)=\mp\infty$, $\lim_{s\to\pm1}h(s)=-\infty$ and $h(0)=2-\delta>0$, we get that $h(s)=0$ for has exactly two solutions $s_{1,\rho}$ and $s_{2,\rho}$ on $(-1,1)$, with $s_{1,\rho}\in(-1,0)$ and $s_{2,\rho}\in (0,1)$.
		
		By \eqref{coordtf} and assuming $h(\sigma_0)=0$
		\be\label{computation-der-rho-for-h(s)=0}\ba	& \frac{\p}{\p \rho} \tf_0(\u)=3\rho^2\sigma_0^3-2\rho(1-\sigma_0^2)+{2\delta\rho}{\sqrt{1-\sigma_0^2}}\\
		&=3\rho^2\sigma_0^3-2\rho(1-\sigma_0^2)+{2\delta\rho}{\sqrt{1-\sigma_0^2}}+ h(\sigma_0)\cdot \rho\cdot(1-\sigma_0^2)\\
		&=3\rho^2\sigma_0+\delta\rho\sqrt{1-\sigma_0^2}.\ea\ee
	\end{enumerate}
    So if $\sigma_0=s_{2,\rho}$, then $\frac{\p}{\p \rho} \tf_0(\u)>0$, which is in contradiction with our initial hypothesis. Thus we have now to look at the case $\sigma_0=s_{1,\rho}$.

    First, we claim that $s_{1,\rho}\in(-1,-\frac{23}{24})$. Note that $0<\rho<r_2<1/14$ and $\delta<r_2/36<1/(14\times 36)$, so from  $\lim_{s\to-1}h(s)=-\infty$ and $$h\left(-\frac{23}{24}\right)=-\frac{23\rho}{8}+2-\frac{\delta}{\sqrt{1-\frac{23^2}{24^2}}}>2-\frac{23}{8\times 14}-\frac{1}{14\times 36\sqrt{1-\frac{23^2}{24^2}}}>0,$$
    we get our claim. As a result, if $\sigma_0=s_{1,\rho}$, using \eqref{computation-der-rho-for-h(s)=0}, $s_{1,\rho}\in(-1,-\frac{23}{24})$ and $\delta<r_1/36<\rho/36,$ we find
    \be\label{drhofw}\frac{\p}{\p \rho}\tf_0(\u)=3\rho^2 s_{1,\rho}+\delta\rho\sqrt{1-s_{1,\rho}^2}<\rho^2\big(-3(23/24)+\sqrt{1-23^2/24^2}/36\big)<0.\ee

    As a conclusion, $w_\rho=(\rho s_{1,\rho},-\rho\sqrt{1-s_{1,\rho}^2},0,\cdots,0)$ is another critical point of $\tf_0|_{S(\rho)}$ satisfying $\frac{\p}{\p \rho}\tf_0(\u)<0$. Moreover, by \eqref{coordtf},
\begin{equation}
\begin{aligned}
    &\left(\frac{\p^2}{\p \sigma_l\p\sigma_{l'}}\left(\tf_0|_{S(\rho)}\right)\right)_{l,l'\neq1}(w_\rho)\\
    &=\rho^2 \diag\left(s_{1,\rho}h'(s_{1,\rho}),-\frac{\delta}{\sqrt{1-s_{1,\rho}^2}},\cdots,-\frac{\delta}{\sqrt{1-s_{1,\rho}^2}},4-\frac{\delta}{\sqrt{1-s_{1,\rho}^2}},\cdots, 4-\frac{\delta}{\sqrt{1-s_{1,\rho}^2}}\right).
\end{aligned}
\end{equation}
As $h(s_{1,\rho})=0$, $s_{1,\rho}\in(-1,-\frac{23}{24})$ and $\rho<1/14$, we have $4-\frac{\delta}{\sqrt{1-s_{1,\rho}^2}} = 2-3\rho s_{1,\rho}>0$. Moreover $h'(s_{1,\rho})>0$, so we find that the index of $w_\rho$ is $i$.
\end{proof}
\begin{rem}\label{threecrit-rmk}
 From the proof above, it follows that as $\rho$ varies from $r_1$ to $r_2$, the three critical points on $S(\rho)$ define embeddings of three intervals.
\end{rem}
\begin{lem}\label{sixcrits}
	Assume that $|y|<\delta^2.$
	Let $f_{A,y}\in C^\infty(\R^{n+1})$ be given by $f_{A,y}(u)=\tf_y(u)+q_A(|u|)=f_y(u)-\eta(|u|)u_1{+y\tilde{\eta}(|u|)u_0}+q_A(|u|)$. When $A$ is large enough, $f_{A,0}$ has exactly $7$ critical points. One of them is the original {birth-death point $0$ with index $i$}, three of them {, denoted by $v_{1,-},v_{2,-},w_-$,} are {Morse points} near $S_{r_1}$, with index {$0,i-1,i$} respectively, and three of them {, which are denoted by $v_{1,+},v_{2,+},w_+$,} are {Morse points} near $S_{r_2}$, with Morse index  {$1,i,i+1$} respectively. Similarly, $f_{A,y}$ for $y>0$ has exactly $8$ critical points and $f_{A,y}$ for $y<0$ has exactly $6$ critical points, and the localization and indexes of the 6 new created critical points is the same as for $f_{A,0}$.
\end{lem}
\begin{proof}

The proof of this lemma follows directly from basic multivariable calculus.

	{Note that on $D(r_1)$ and outside of $D(r_2)$, $f_{A,y}(u)=\tf_y(u)+ \mathrm{Cst}$, so by Lemma \ref{nocritical}, it has only 0 has a critical point on these regions. Moreover, by our constructions, $ f_{A,y}|_{D(r_1,r_2)} $ is independent of $y$, so to prove Lemma \ref{sixcrits}, we only have to show that $ f_{A,0}|_{D(r_1,r_2)} $ has six critical points with the specified Morse indices.}
	
	By point (5) in the construction of $q_A$ in  \cref{assotwo}, if $A$ is large enough, $|\nabla f_{A,0}|^2\neq 0$ for $|u|\in[0.51r_1+0.49r_2,0.49r_1+0.51r_2]$.
	
	By Lemma \ref{nocritical}, there exists $c>0$, s.t. $|\nabla\tf_0|^2|_{D(r_1,r_2)}\geq c$. One can check by a direct computation that when $|u|\in [r_2-e^{-A^2}(r_2-r_1)/2,r_2]$, {$|\nabla (q_A(|u|)|\to 0$} uniformly as $A\to \infty.$ Hence $|\nabla f_{A,0}|^2\neq 0$ if $|u|\in [r_2-e^{-A^2}(r_2-r_1)/2,r_2]$ and $A$ is large.

Now we would like to find critical point on $D(0.49r_1+0.51r_2,r_2-e^{-A^2}{(r_2-r_1)}/2)$. We will use notation in the proof of \Cref{threecrits}.
By \Cref{threecrits}, as $q_A'>0$ on $(0.49r_1+0.51r_2,r_2-e^{-A^2}{(r_2-r_1)}/2)$, it suffices to solve 
\[
\left(\frac{\p}{\p \rho}(\tilde{f}_0+q_A)\right)(u_\rho)=0
\]
for $u_\rho=v_{1,\rho},v_{2,\rho}$ and $w_\rho$.

On $D(0.49r_1+0.51r_2,r_2-e^{-A^2}{(r_2-r_1)}/2)$, \be\label{drhoqa}\frac{\p}{\p \rho}q_A=A(r_2-\rho).\ee
Thus, by \eqref{vbrho} and \eqref{drhoqa},  on $D(0.49r_1+0.51r_2,r_2-e^{-A^2}{(r_2-r_1)}/2)$, one can see that if $\rho^+_b=\frac{Ar_2}{A+2+2(-1)^{b+1}\delta}$, for $b=1,2$, \[
\left(\frac{\p}{\p \rho}(\tilde{f}_0+q_A)\right)(v_{b,\rho^+_b})=0.
\]

Set $v_{b,+}=v_{b,\rho^+_b},b=1,2$. Then one can check from \eqref{coordtf} and \eqref{drhoqa} that 
\[
\frac{\p^2}{\p \rho^2}f_{A,0}(v_{b,+})=-A-2+2(-1)^b\delta\quad\text{and }\quad \frac{\p^2}{\p \rho\p\sigma_j}f_{A,0}(v_{b,+})=0.
\]
 Together with Hessian in the tangential direction we computed in \eqref{tangential-hessian-v_b}, we can see that $v_{1,+}$ and $v_{2,+}$ have Morse indices $1$ and $i$ respectively. 

Let $\varrho_1 = 0.49r_1 + 0.51r_2$ and $\varrho_2 = r_2 - \frac{e^{-A^2}}{2} (r_2 - r_1)$. Let $K= 3(23/24)-\sqrt{1-23^2/24^2}/36$, then, using \eqref{drhofw} along with the discussion above it, we have $-3\rho^2\leq \frac{\p}{\p \rho}\tf_0(w_{\varrho_i}) \leq -K \rho^2$. With equation \eqref{drhoqa}, this gives that for sufficiently large $A$, we have  
\[
\left(\frac{\partial}{\partial \rho}(\tilde{f}_0 + q_A)\right)(w_{\varrho_1}) > -3\varrho_1^2 + 0.49A (r_2 - r_1) > 0
\]  
and 
\[
\left(\frac{\partial}{\partial \rho}(\tilde{f}_0 + q_A)\right)(w_{\varrho_2}) \leq -K \varrho_2^2 + \frac{A e^{-A^2}}{2} (r_2 - r_1) < 0.
\] 

Thus, the equation  
\be\label{eqwrho}
\left(\frac{\partial}{\partial \rho}(\tilde{f}_0 + q_A)\right)(w_{\rho}) = 0
\ee 
has at least one solution on $[\varrho_1,\varrho_2]$. 

Next, we will prove the uniqueness of the solution to \eqref{eqwrho}.  Set  
\[
H(\rho, \phi) = f_{A,0}(\rho \cos(\phi), \rho \sin(\phi), 0, \dots, 0).
\]  
Let $\phi_1 \in (-\pi,0)$ be such that $\cos(\phi_1) = -\frac{23}{24}$. By the discussion of the case $\bar{u}_0\neq 0$ in the proof of Lemma \ref{threecrits}  (above \eqref{drhofw}), proving the uniqueness of the solution to \eqref{eqwrho} reduces to showing that $H$ has a unique critical point in the domain $[\varrho_1, \varrho_2] \times [-\pi, \phi_1]$.
  
	Note that on $D(0.49r_1+0.51r_2,r_2-e^{-A^2}{(r_2-r_1)}/2)$, we have,
	\begin{equation}\label{fA0-on-middle-to-almost-r2}
		f_{A,0}=u_0^3-(u_1^2+\dots+u_i^2) + (u_{i+1}^2+\dots+u_n^2)-\delta |u|u_1-\frac{A}{2}(|u|-r_2)^2,
	\end{equation} so
			$$H(\rho,\phi) = \rho^3\cos^3(\phi)-\rho^2\sin^2(\phi)-\delta \rho^2\sin(\phi) -\frac{A}{2}(\rho-r_2)^2$$ 
			and
		\be\begin{aligned}
			\label{duarho} &F_1(\rho,\phi):=\frac{\partial}{\partial \rho} H=3\rho^2 \cos^3(\phi)-2\rho\sin^2(\phi)-{2\rho}\delta\sin(\phi)+{A(r_2-\rho)}, \\
	&F_2(\rho,\phi):=\frac{\partial}{\partial \phi} H=-3\rho^3 \cos^2(\phi)\sin(\phi)-2\rho^2\sin(\phi)\cos(\phi)-\delta\rho^{{2}}\cos(\phi).
		\end{aligned}\ee

		We now compute the Jacobian matrix of $F=(F_1,F_2)=\nabla H$ on $\Big[0.49r_1+0.51r_2,r_2-\frac{e^{-A^2}(r_2-r_1)}{2}\Big]\times[-\pi,\phi_1]$. Using \eqref{duarho} and the fact that $\rho \in (0,r_2)\subset (0, 1/14),\delta\in(0,{\rho}/{24})$ we find that if $A$ is large,
		\[\begin{aligned}
			\frac{\partial}{\partial \rho}F_1&=6\rho\cos^3(\phi)-2\sin^2(\phi)-2\delta\sin(\phi)-A<-\frac{A}{2}, \\
			\frac{\partial}{\partial \phi} F_2 &= -3\rho^3\cos(\phi)+9\rho^3\cos(\phi)\sin^2(\phi)-2\rho^2\cos(2\phi)+\delta\rho^2\sin(\phi)\\
			& < -2\rho^2\cos(2\phi) + 13\rho^3 < \rho^2(-1+13\rho) < -\frac{1}{14^3}.
		\end{aligned}
		\]
		Moreover, we have:
		\[
		\left|\frac{\partial}{\partial \phi}F_1\right|=\left|\frac{\partial}{\partial \rho}F_2\right| = \Big| -9\rho^2\cos^2(\phi)\sin(\phi)-4\rho\sin(\phi)\cos(\phi)-2\delta\rho\cos(\phi)\Big|\leq 9 \rho^2 + 4 \rho + 2\delta\rho.
		\]
		Hence, the Jacobian matrix of $F$ is negative definite at each point if $A$ is large. As $[\varrho_1,\varrho_2]\times[-\pi,\phi_1]$ is convex, this implies that $F$ is injective.

         As we already proved that \eqref{eqwrho} as at least a solution, we know that $F=0$ has a unique solution, say $(\rho_{A},\phi_{A})$, on $[\varrho_1,\varrho_2]\times[\pi,\phi_1]$. Moreover, using the first line of \eqref{duarho}, there exists $c_1,c_2>0$, s.t.
		\begin{equation}
			\rho_{A} \in \left( \frac{Ar_2}{A + c_1}, \frac{Ar_2}{A + c_2}\right) \text{ and } \phi_{A} \in (\pi,\phi_1).  
		\end{equation}
		
		Set $w_+:=w_{\rho_A}=(\rho_{A}\cos(\phi_{A}),\rho_{A}\sin(\phi_{A}),0,\cdots,0)$. One computes form \eqref{fA0-on-middle-to-almost-r2} that $\frac{\partial^2}{\partial u_k \partial u_l} f_{A,0}(w_+) = 0$ if $2 \leq k < l \leq n$. For $2 \leq j \leq i$, using the fact that $F_1(\rho_A,\phi_A)=0$, we find:
		\[
		\frac{\partial^2}{\partial u_j^2} f_{A,0}(w_+) = -2 - 3 \rho_{A} \cos^3(\phi_{A}) + 2 \sin^2(\phi_{A}) +\delta\sin(\phi_{A}) < -1,
		\]
		and similarly, for $i + 1 \leq k \leq n$:
		\[
		\frac{\partial^2}{\partial u_k^2} f_{A,0}(w_+) = 2 - 3 \rho_{A} \cos^3(\phi_{A}) + 2 \sin^2(\phi_{A}) +\delta\sin(\phi_{A}) > 1.
		\]
		
		Therefore, together with the negativity of the Jacobian matrix of $F$, we can see that $w_+$ has Morse index $i + 1$.

	Similarly, on $D(r_1,0.51r_1+0.49r_2)$ we can find another three critical points $v_{1,-},v_{2,-}$ and $w_-$, which are near $S_{r_1}$ and have Morse index $0,i-1,i$ respectively. This complete the proof.
\end{proof}

\begin{rem}\label{indepr}
	Assuming that the conditions of Lemma \ref{sixcrits} holds, we can see from its proof that there exist $(A,y)$-independent positive constants $c=c(r_1,r_2,\delta),c'=c'(r_1,r_2,\delta)$ and $C=C(r_1,r_2,\delta)$, such that if $A$ is large, any two distinct critical points $p,q\in\{v_{1,+},v_{2,+},w_+\}$ satisfy $|p-q|\geq c$, $|f_{A,y}(p)-f_{A,y}(q)|\geq c'$ and $|f_{A,y}(p)|\leq C.$ Moreover, limits $\lim_{A\to\infty}v_{1,+}(A)$, $\lim_{A\to\infty}v_{2,+}(A)$ and $\lim_{A\to\infty}w_{+}(A)$ exist.
\end{rem}

\def\next{1}
\if\next0
Next we would like to study the unstable/stable level sets of $w_{\pm}$ and $v_{2,\pm}$.

Let $\rW^{\rs/\ru}_{A}(p)$ and $\rW^{\rs/\ru}_{A,c}(p)$ denotes the stable/unstable manifolds and stable/unstable level sets of a critical point $p$ of $f_{A,y}$ with respect to the vector field $-\nabla f_{A,y}.$
We will abbreviate  $\rW^{\rs/\ru}_{A}(p)$ and $\rW^{\rs/\ru}_{A,c}(p)$ as  $\rW^{\rs/\ru}(p)$ and $\rW^{\rs/\ru}_{c}(p)$ sometimes.
By the construction of $f_{A,y}$, it's easy to see that 

\begin{lem}\label{sep}
	Assume that $|y|\leq \delta^2.$
	{ Let $\rho:=|u|,$ and $\frac{\p}{\p \rho}:=\sum_{l=0}^n\frac{u_l}{\rho}\frac{\p}{\p u_l}.$ {For any $\varepsilon>0$}, there is $C_0>3$, independent from $(A,\delta,y)$, s.t
		\be\label{drhof}\frac{\partial}{\partial \rho} f_{A, 0}>{\varepsilon} \text{ if } u\in D\left(\frac{Ar_1}{A-C_0},\frac{Ar_2}{A+C_0}\right).\ee    
		
		Moreover, there exists $(A,y,\delta)$-independent constant $c_0>0$, s.t. for any $c>c_0$, \be\label{splitpn}\rW^{\rs/\ru}_{A,c}(w_+)\cap {D\left(\frac{Ar_2}{A+C_0},r_2\right)=\emptyset\mbox{ 
				and } \rW^{\rs/\ru}_{A,c}(v_{a,+})\cap D\left(\frac{Ar_2}{A+C_0},r_2\right)}=\emptyset,a=1,2,\ee
		if $A$ is large enough. }
\end{lem}
\begin{proof}
	We prove the case for $y=0$, the general cases are similar.
	If $\rho>0$, we compute $$\begin{aligned}
		& \frac{\partial}{\partial \rho} f_{A, 0}=\frac{1}{\rho}\left(3 u_0^3-2 \sum_{j=1}^i u_j^2+2 \sum_{k={i+1}}^n u_k^2\right)-\eta^{\prime}(\rho) u_1-\eta(\rho) \frac{u_1}{\rho}+q_A^{\prime}(\rho) \\
		& \geq q_A^{\prime}(\rho)-C'_0 \rho
	\end{aligned}$$   
	for some $(A,{\delta,y})$-independent constant $C'_0>3.$        
	
	Hence, by the construction of $q_A$ in  \cref{assotwo}, we can see that when $A$ is large, {if $C_0>\frac{\varepsilon+C'_0}{r_1}$}, then \eqref{drhof} holds.      
	
	A straightforward computation shows that there exists a $A$-independent constant $C_2>0$, s.t. 
	\begin{equation}
		\label{sup-difference-outer-annulus}
		\sup _{u, u^{\prime} \in D(\frac{Ar_2}{A+C_0},r_2)}\left| f_{A, 0}(u)-f_{A, 0}(u')\right| \leq C_2 r_2^2.
	\end{equation}
	Choose $c_0=2C_2r_2^2$, note that $|w_+|,|v_{2,+}|\in \Big(\frac{Ar_2}{A+3},r_2\Big)$ when $A$ is large enough, so \eqref{sup-difference-outer-annulus} implies \eqref{splitpn}.
\end{proof}
\def\hrW{{\widehat{\rW}}}
\begin{rem}
	\label{Ws/u_c_separated}
	For $p=w_{+}$ or $v_{a,+},a=1,2$, we define $\trW^{\rs/\ru}_{A,c}(p):=\rW^{\rs/\ru}_{A,c}(p)\cap\{u:|u|> r_2\}$ and $\hrW^{\rs/\ru}_{A,c}(p):=\rW^{\rs/\ru}_{A,c}(p)\cap\{u:|u|< r_2\}$. Then by \eqref{splitpn}, $\trW^{\rs/\ru}_{A,c}(p)$ and $\hrW^{\rs/\ru}_{A,c}(p)$ are separated by $D(\frac{Ar_2}{A+C_0},r_2)$. In \cref{combitor} {(see the paragraph below \eqref{decvr})}, we will show that ensuring the independence of the new critical points described in Lemma \ref{sixcrits} only involves handling $\trW^{\rs/\ru}_{A,c}(p)$. 
\end{rem}

We will also abbreviate $\trW^{\rs/\ru}_{A,c}(p)$ as $\trW^{\rs/\ru}_{c}(p)$ sometimes. Because of the above remark, our next lemma will focus on studying $\trW^{\rs/\ru}_{A,c}(p)$.

\begin{lem}\label{lemdisk}
	{Assume that $|y|\leq \delta^2.$}
	There exists $(A,y,\delta)$-independent constant $c_0>0$, s.t. for any $c>c_0$, there exist smooth closed balls $D^{\rs/\ru}_c(p)\subset \{u:f_{A,y}(u)=f(p)\pm c\}\cap\{u:|u|> r_2\}$ s.t. $\trW^{\rs/\ru}_{A,c}(p)\subset D^{\rs/\ru}_c(p)$ for $p=w_+$ or $v_{a,+}$, $a=1,2$.
	
	Moreover, if we choose $c=c_{A,y,\delta}>c_0$ such that $f_{A,y}(p)\pm c$ is independent of $A$, $y$ and $\delta$, then the set $\{u:f_{A,y}(u)=f(p)\pm c\}\cap\{u:|u|> r_2\}$ is independent of $A$, $y$ and $\delta$, and $D^{\rs/\ru}_c(p)$ can be taken independent of $A$, $y$ and $\delta$.
\end{lem}
\begin{proof}
	We prove the case for $y=0$, the general cases are similar. In this proof, for $u=(u_0,\dots, u_n)\in\R^{n+1}$, we will set $u^-=(u_1,\dots,u_i)$ and $u^+=(u_{i+1},\dots,u_n)$. 
	
	Let $\phi^t$ be the flow generated by $-\nabla f_{A,0}.$

	A straightforward computations shows that

	\be\label{eq20}
	\begin{cases}
		|\nabla f_{A,0}|(u)\leq C|u|^2 \mbox{ for some $A$-independent $C$ if $|u|\geq \frac{Ar_2}{A+3}$;}\\
		|\nabla f_{A,0}|^2(u)\geq C^{-1}|u|^2 \mbox{ if $|u|\in(\frac{5r_2}{2},+\infty)$};\\
		\frac{\p}{\p u_l}f_{A,0}(u)=-2u_l \mbox{ if $1\leq l\leq i$, $|u|\geq \frac{5r_2}{2}$};\\
		\frac{\p}{\p u_l}f_{A,0}(u)=2u_l \mbox{ if $1\leq l\leq i$, $|u|\geq \frac{5r_2}{2}$}.
	\end{cases}
	\ee
	
	These imply that \be\label{unsta}\trW^{\ru}(v_{2,+})\cap S_{3r_2}\subset\Big\{u\in S_{3r_2}:\|u^+\|\leq \frac{5r_2}{2}\Big\}.\ee Indeed, if $u\in S_{3r_2}$ satisfies $\|u^+\|>5r_2/2$, we claim that $ \left(\cup_{t<0}\phi^t(u)\right)\cap S_{5r_2/2}=\emptyset$, so it is impossible that $\lim_{t\to-\infty}\phi^t(u)=v_{2,+}$. If this claim is not true, i.e., $ \left(\cup_{t<0}\phi^t(u)\right)\cap S_{5r_2/2}\neq\emptyset$, then there exists $t_0{<0}$, s.t. $|\phi^{t}(u)|>5r_2/2$ for $t\in{(t_0,0]}$, but $|\phi^{t_0}(u)|=5r_2/2$. Let $v=(v_0,\cdots,v_n)=\phi^{t_0}(u)$. By the fourth statement of \eqref{eq20}, $v_{i+1}^2+\cdots+v_n^2>u_{i+1}^2+\cdots u_n^2>(5r_2/2)^2,$ which contradicts the fact that $
	\phi^{t_0}(u)
	\in S_{5r_2/2}.$ 
	
	Similarly,
	$$\trW^{\rs}(v_{2,+})\cap S_{3r_2}\subset\Big\{u\in S_{3r_2}:\|u^-\|\leq \frac{5r_2}{2}\Big\}.$$

	Notice that \be\label{lemdisk1}|f_{A,0}(v_{2,+})|=\left|-|v_{2,+}|^2+\delta|v_{2,+}|-\frac{A(2-2\delta)^2(r_2^2)}{(A+2-2\delta)^2}\right|\leq 3r_2^2\ee if $A$ is large. So there exists {$c''_0$} large enough, s.t. $f_{A,0}(v_{2,+})-c''_0<0$
	and  $f_{A,0}(v_{2,+})+c''_0>0.$
	
	Let $L_c=\{u\in\R^n:f_{A,0}(u)=f_{A,0}(v_{2,+})-c\}$.
	By the first statement of \eqref{eq20} (also note that $|v_{2,+}|>\frac{Ar_2}{A+3}$), we have $\sup_{u\in D(|v_{2,+}|,3r_2)}|f(u)-f(v_{2,+})|\leq C (3r_2)^3$. So there exists $(A,y,\delta)$-independent {$c'_0>c''_0>0$} large enough, s.t. if $c>c'_0$,  \be\label{lcdempet} L_c\cap D(|v_{2,+}|,3r_2)=\emptyset.\ee

	{Assume from now on that $c>c'_0$.} Let $U_c\subset L_c$ be the set of points $u=(u_0,\cdots,u_n)\in L_c\cap D(3r_2,\infty)$, s.t.
	\begin{enumerate}[(1)]
		\item $u_0>3r_2$ \textbf{or}
		\item $\|u^+\|> \frac{5r_2}{2}$.
	\end{enumerate}
	
	On $D(3r_2,\infty)$,
	\be\label{lemdisk2} f_{A,0}=u_0^3-\|u^-\|^2+\|u^+\|^2,\ee
	so
	\be\label{lemdisk3}\nabla f_{A,0}=(3u_0^2,-2u_1,\cdots,-2u_i,2u_{i+1},\cdots, 2u_n).\ee
	
	We must have \be
	\label{lemdisk4}
	\forall u\in U_c ,\;\left(\cup_{t<0}\phi^t(u)\right)\cap \left(S_{3r_2}\cap\trW^\ru(v_{2,+})\right)=\emptyset.\ee Otherwise, there exists $u\in U_c$, such that $\phi^{t_0}(u)\in S_{3r_2}\cap\trW^\ru(v_{2,+})$ for some $t_0<0$.  By \eqref{lcdempet}, there exists
	$t_1\in[t_0,0]$, s.t. $|\phi^{t}(u)|>3r_2$ for $t\in(t_1,0]$ but $|\phi^{t_1}(u)|=3r_2.$ If $u_0>3r_2$, then by \eqref{lemdisk3}, the $0$-th component of $\phi^{t_1}(u)$ is greater than $u_0>3r_2$, which contradicts with $|\phi^{t_1}(u)|=3r_2$. If $\|u^+\|>5r_2/2$, we also get a contradiction using  \eqref{unsta}. Thus, \eqref{lemdisk4} holds.

	By {\eqref{lcdempet} and }\eqref{lemdisk4}, we have
	$$\trW^{\ru}_c(v_{2,+})\subset L_c\cap\{u\in\R^{n+1}:u_0\leq3r_2,\|u^+\|\leq \frac{5r_2}{2}\}\cap D({3}r_2,\infty).$$
	But  by \eqref{lemdisk2} and our choice of $c_0$ and $c_0'$, $L_c\cap\{u\in\R^n:u_0\leq3r_2,\|u^+\|\leq 5r_2/2\}\cap D({3}r_2,\infty)$ is contained in a smooth closed topological ball that is independent of $A.$ Indeed, set $\kappa= c-f_{A,0}(v_{2,+})>0$, then 
	$$\begin{aligned}
		&L_c\cap\Big\{u\in\R^{n+1}:u_0\leq3r_2,\|u^+\|\leq \frac{5r_2}{2}\Big\}\cap D({3}r_2,\infty)\\
		&= \Big\{u\in\R^{n+1}\: :\: u_0^3-\|u^-\|^2+\|u^+\|^2=-\kappa,\, u_0\leq3r_2,\, \|u^+\|\leq \frac{5r_2}{2}\Big\}.
	\end{aligned}$$
	As a consequence, by \eqref{lemdisk1}, if we set $K=27r_2^3+(5r_2/2)^2+\kappa$, we get that $L_c\cap\{u\in\R^{n+1}:u_0\leq3r_2,\|u^+\|\leq 5r_2/2\}\cap D({3}r_2,\infty)$ is included in the image $D^{\ru}_c(v_{2,+})$ of $B^{\R^i}(0,K)\times B^{\R^{n-i}}(0, 5r_2/2)$ under the map $\psi\colon (s,t)\mapsto \big((\|s\|^2-\|t\|^2-\kappa)^{1/3},s,t\big)$, which is a homeomorphism on its image. Note also that, although $\psi$ is not smooth on $\{(s,t): \|s\|^2-\|t\|^2=\kappa\}$, the set $D^{\ru}_c(v_{2,+})$ is still smooth. 
	
	Moreover, by \eqref{lemdisk1}, we can take $c_0>c'_0$ large enough so that $c>c_0$ implies $\kappa > (3r_2)^{3/2}+(3r_2)^2$. In that case, $u=\psi(s,t)$, we have either $|s|>3r_2$, in which case $\|u\|>3r_2$, or $|s|\leq 3r_2$, in which case   $\kappa - \|s\|^2+\|t\|^2 > (3r_2)^{3/2}$ so that again, $\|u\|>3r_2$. Hence, if $c>c_0$, $D^{\ru}_c(v_{2,+})\subset D(3r_2,\infty)$.
	
	Now, if we choose $c=c_{A,y,\delta}$ so that $\kappa$ is independent of $A$, $y$ and $\delta$, then by \eqref{lcdempet} and \eqref{lemdisk2}, $\{u:f_{A,y}(u)=f(v_{2,+})- c\}\cap\{u:|u|> r_2\}=\{u:f_{A,y}(u)=-\kappa \}\cap\{u:|u|> 3r_2\}$ is independent of $A$, $y$ and $\delta$, as well as the map $\psi$, so $D^{\ru}_c(v_{2,+})$ can be taken independent of $A$, $y$ and $\delta$.
	
	Let $\tilde{L}_c=\{u\in\R^{n+1}:f_{A,0}(u)=f_{A,0}(v_{2,-})+c\}.$
	Similarly, we can show that for $c$ large enough,
	$$\trW^{\rs}_c(v_{2,+})\subset \tilde{L}_c\cap\Big\{u\in\R^{n+1}:u_0\geq-3r_2,\|u^-\|\leq \frac{5r_2}{2}\Big\}\cap D({3}r_2,\infty).$$
	
	While for the same reason, $\tilde{L}_c\cap\{u\in\R^{n+1}:u_0\geq-3r_2,\|u^-\|\leq 5r_2/2\}\cap D({3}r_2,\infty)$ is contained in a closed ball, that is independent of $A$, $y$ and $\delta$ if $c$ is large and satisfies that $f_{A,y}(p)+c$ does not depend on $A$, $y$ and $\delta$. 
	
	Proceeding in the same way, we also have the same statement for {$v_{1,+}$ and } $w_+.$
\end{proof}
\fi

\section{Analytic and Combinatorial Torsion Forms}\label{combitor}
\subsection{Double suspension and analytic torsion forms}\label{suspen}
In this section, we will introduce the concept of double suspension. The purpose of performing double suspension is to ensure that the Morse index of all critical points satisfies the lower and upper bounds specified in Lemma \ref{indepbd}. Then, by using Lemma \ref{indepbd}
we will show that the product metric on the doubly suspended space can be deformed to ensure that the birth-death point and the six new Morse points described in \cref{model} are independent of other critical points (except for the birth-death point or the two points that converge to the birth-death point as $t_1 \to 0$ in Lemma \ref{lem441}). Based on this, we will introduce higher combinatorial torsion.

Another reason for considering double suspension is the following: for any gradient flow $\gamma$ on a closed manifold, as $t \to \pm\infty$, $\gamma(t)$ must converge to a critical point. However, on a noncompact manifold, the flow may diverge to infinity. This observation, combined with the upper and lower bounds of the Morse index of critical points, enable us to make certain critical points independent.

Let $\pi: M \to S$ be a smooth fibration with closed fiber $Z$ with $\dim(Z)=n+1$. Let $(F \to M,\nabla^F)$ be a unitarily flat vector bundle of rank $m$ equipped with a compatible Hermitian metric $h^F$. Let $g^{TZ}$ be a metric on $TZ \to M$, and fix a splitting $TM = T^H M \oplus TZ$. We also fix a metric $g^{TS}$ on $TS$ as set $g^{TM}=g^{TZ}\oplus g^{TS}$ with respect to the preceding splitting.

\begin{defn}\label{doublesuspension1}
	Let $f:M\to \R$ be a smooth function. A double suspension\index{double suspension} of the pair $(M,f)$ is the pair $(\M,\bbf)$\index{M@$\M$}\index{f@$\bbf$}, where
	$\M:=M\times\R^N\times\R^N$ and for $(p,x_1,x_2)\in M\times\R^N\times\R^N$.
	\[\bbf(p,x_1,x_2)=f(p)-|x_1|^2+|x_2|^2,\]
	Here $N$ is a large enough \textbf{even} integer, to be fixed below.
\end{defn}

Let $f:M\to\R$ be fiberwise generalized Morse function $f$ (See Definition \ref{generalizedMorse}).

One can see ${\ppi \colon}\M\to S$ is a fibration with fiber $\bZ\index{Z@$\bZ$}:=Z\times\R^N\times\R^N$, and \begin{equation}
    \label{Sigma(bbf)}
    \Sigma(\bbf)=\Sigma(f)\times\{0\}\times\{0\}.
\end{equation}
{Moreover, if $\mathfrak{p}=(p,0,0)\in \Sigma(\bbf)$, then the Morse indices of $\mathfrak{p}$ and $p$ are linked by $\mathrm{ind}_\bbf(\mathfrak{p})=\mathrm{ind}_f(p)+N$. 
	\begin{cond}\label{N-large}
		From now on, we fix $N $ sufficiently large so that $N\geq n+1$ and the Morse indices of all critical points of $\bbf_\s$, $\s \in S$ satisfy the lower and upper bounds specified in Lemma \ref{indepbd}.
\end{cond}}

Let $T^H\M:=T^HM\oplus\{0\}\oplus\{0\}$\index{THM@$T^H\M$} and 
\be \label{defgTZprime} g^{T\bZ'}:=g^{TZ}\oplus g^{T\R^N}\oplus g^{T\R^N}.\index{gTZp@$g^{T\bZ'}$} \ee
Let $(\cF,h^{\cF})$\index{F@$\cF$}\index{hF@$h^{\cF}$} be the pull back of $(F,h^F)$ with respect to the canonical map $\M\to M$, and $\nabla^{\cF}\index{nablaF@$\nabla^{\cF}$}=\nabla^F+\sum_{i=1}^{2N}du_{i+n}\wedge\frac{\p}{\p u_{n+i} }$. Here $g^{T\R^N}$ is the standard metric on $\R^N$ and we denoted $(x_1,x_2)=(u_{n+1},\cdots,u_{n+2N})$.

By the theory of harmonic oscillator and proceeding as in the proof of \cite[Proposition 3.28]{bismut1995flat}, we can see that we can define, for $T>0$, a differential form $$\T(T^H\M,g^{T\bZ'},e^{-2T\bbf}h^{\cF})$$ associated with the above data,  as in Definition \ref{torc}, and that moreover, in positive degree we have $$\T(T^H\M,g^{T\bZ'},e^{-2T\bbf}h^{\cF})^{>0}=\T(T^HM,g^{TZ},e^{-2Tf}h^{F})^{>0}.$$

Let $ g^{T\bZ}\index{gTZ@$g^{T\bZ}$} $ be a perturbation of $ g^{T\bZ'} $ above, such that outside a compact neighborhood of $ M \times \{0\} \times \{0\} \subset \M $, $ g^{T\bZ} = g^{T\bZ'} $.
Using the standard finite-propagation-speed argument, we can see that $$\T(T^H\M,g^{T\bZ},e^{-2T\bbf}h^{\cF})\index{TTHMgTZexphF@$\T(T^H\M,g^{T\bZ},e^{-2T\bbf}h^{\cF})$}$$ is also well-defined. In the sequel, we always assume that $(\bbf, g^{T\bZ})$ satisfies the fiberwise Thom-Smale transversality condition (see Definition \ref{defn38}).

Let $ \bfH_{T} $\index{HT@$\bfH_T$ and $\bfH$} be the $ L^2 $-cohomology of $ ( \Omega^{\bullet}(\bZ, \cF|_{\bZ}), d^{\bZ}) $ with respect to the $ L^2 $ metric induced by $ g^{T\bZ} $ and $ e^{-2{T}\bbf}h^{\cF} $ (for the definition of $ L^2 $-cohomology, see \cite{DY2020cohomology}). 	Let $ \bfH := \bfH_{T=1} $, then $\bfH\to S$ carries a natural Gauss-Manin-type connection $\nabla^{\bfH}$.

Let $\bfH_{T}'$\index{HTp@$\bfH_{T}'$} be the $ L^2 $-cohomology of $ ( \Omega^{\bullet}(\bZ, \cF|_{\bZ}), d^{\bZ}) $ with respect to the $ L^2 $ metric induced by $ g^{T\bZ'} $ and $ e^{-2{T}\bbf}h^{\cF} $. 	Let $ \bfH' := \bfH_{T=1}' $.		
Recall that $ H\index{H@$H$} \to S $ is the cohomology bundle associated with $ M \to S $ and the flat bundle $ F \to M $. Then we have a canonical isomorphism between $H$ and $\bfH_{T}'$ given by the following map of closed form \be \label{hbfh}
w\mapsto \frac{e^{-2T|x_1|^2}}{{(2\pi T)}^{\frac{N}{2}}}w \wedge dx_1,\ee
where $w$ is a closed form on $Z$. Here $dx_1$ is the $N$-form given as follows: write $x_1=(u_{n+1},\cdots,u_{n+N})$, then $dx_1:=du_{n+1}\wedge\cdots\wedge du_{n+N}.$

Let
\be\label{defjt} \J_T\index{JT@$\J_T$}: \bfH_{T}'\to \bfH' \ee be the isomorphism induced by the following map of closed forms $\frac{e^{-2T|x_1|^2}}{{(2\pi T)}^{\frac{N}{2}}}w \wedge dx_1\mapsto \frac{e^{-2|x_1|^2}}{{(2\pi )}^{\frac{N}{2}}}w \wedge dx_1,$ where $w$ is a closed form on $Z$.

There is a canonical isomorphism \be\label{defnkt} \K_T\index{KT@$\K_T$}:\bfH_{T}' \to \bfH_{T} \ee given as follows: given an element $[w]$ in $ \bfH_{T}' $ represented by a smooth closed $ L^2 $-form $ w $, then $w$ is clearly $L^2$-integrable with respect to the measure induced by $e^{-2T\bbf}h^{\cF}$ and $g^{T\bZ}$, so we simply map $[w]$ to the element in $ \bfH_{T} $ represented by $ w $. It is also tautological to check that this map is well defined.

{Notice that, even if $\M$ is non-compact, \cite[Definition 2.4 and Proposition 2.5]{bismut1995flat} are still valid in our context. We denote by $\nabla^\bfH$\index{nablaH@$\nabla^\bfH$} the corresponding connection on $\bfH$, as in \eqref{connection-on-H}.}

Let $ h^{\bfH}_{T,L^2} $\index{hHTL2@$h^{\bfH}_{T,L^2}$} be the Hodge metric on $ \bfH $ induced by $ g^{T\bZ} $, $ e^{-2T\bbf}h^{\cF} $, and $\K_1\J_T\K_T^{-1}$. Let $ h^{\bfH,\prime}_{T,L^2} $ be the Hodge metric on $ \bfH $ induced by $ g^{T\bZ'} $, $ e^{-2T\bbf}h^{\cF} $, and $\K_1\J_T$. 	

Let $ h^{H}_{T,L^2} $\index{hHTL2@$h^{H}_{T,L^2}$} be the Hodge metric on $ H $ induced by $ g^{TZ} $, $ e^{-2Tf}h^F $.  Composing the isomorphism \eqref{hbfh} with $\J_T$ and $\K_1$ we get an isomorhpism $c_T\colon H \to \bfH$. Moreover, by \eqref{hbfh} and \eqref{defjt} we have
\begin{equation}
    \label{isom-H_bfH}
 c_T=c_1\colon H \overset{\sim}{\longrightarrow} \bfH.
\end{equation}
Under this isomorphism,  $h^{H}_{T,L^2}$ corresponds to $ h^{\bfH,\prime}_{T,L^2}$.
Given a metric $ h^H $\index{hH@$h^H$} on $ H \to S $, let $ h^{\bfH} $\index{hHT@$h^{\bfH}$ and $h^{\bfH}$} denotes the metric on $ \bfH \to S $ induced by isomorphism \eqref{isom-H_bfH}.

Recall that $H_{\Cb^m}\to S$ is the cohomology bundle associated with  $M\to S$ and the trivial bundle $\Cb^m\to M$.  Similarly, let $ \bfH_{\Cb^m} $\index{HCm@$\bfH_{\Cb^m}$} be  $ L^2 $-cohomology bundle associated with the trivial flat bundle $ \Cb^m \to \M $, defined in the same way as $ \bfH $. 
As above, given a metric $ h^{H_{\Cb^m}} $\index{hHCm@$h^{H_{\Cb^m}}$} on $ H_{\Cb^m} \to S $, we have the corresponding metrics $ h^{\bfH_{\Cb^m}} $ and $ h^{\bfH_{\Cb^m}}$ \index{hHTCm@$h^{\bfH_{\Cb^m}}$ and $ h^{\bfH_{\Cb^m}}$} on $ \bfH_{\Cb^m} $.

Similarly to Definition \ref{assotor}, we can define
$\btiT(T^H\M,g^{T\bZ'},e^{-2T\bbf}h^{\cF},h^{\bfH},h^{\bfH_{\Cb^m}})$. More explicitly, the metric $h^{H}$ and $h^{H}_{L^2}$ in Definition \ref{assotor} are replaced by $h^{\bfH}$ and $h^{\bfH,\prime}_{T,L^2}$ respectively, etc. The discussion above implies that  $$\btiT(T^H\M,g^{T\bZ'},e^{-2T\bbf}h^{\cF},h^{\bfH},h^{\bfH_{\Cb^m}})=\btiT(T^HM,g^{TZ},\nabla^F,e^{-2Tf}h^F,h^H,h^{H_{\Cb^m}}).$$ 

Using the standard finite-propagation-speed argument, we can see that $$\btiT(T^H\M,g^{T\bZ},e^{-2T\bbf}h^{\cF},h^{\bfH},h^{\bfH_{\Cb^m}})\index{TTHMgTZexphFhHThH0Ttildebar@$\btiT(T^H\M,g^{T\bZ},e^{-2T\bbf}h^{\cF},h^{\bfH},h^{\bfH_{\Cb^m}})$}$$ is also well-defined as in Definition \ref{assotor} (the metric $h^{H}$ and $h^{H}_{L^2}$ in Definition \ref{assotor} are replaced by $h^{\bfH}$ and $h^{\bfH}_{T,L^2}$ respectively, etc.). Moreover,  proceeding as in \cite[Theorem 3.24]{bismut1995flat}, in $ \Omega^{\bullet}(S)/d^S \Omega^{\bullet}(S)$,
\[\btiT(T^H\M,g^{T\bZ'},e^{-2T\bbf}h^{\cF},h^{\bfH},h^{\bfH_{\Cb^m}})=\btiT(T^H\M,g^{T\bZ},e^{-2T\bbf}h^{\cF},h^{\bfH},h^{\bfH_{\Cb^m}}).\]

Lastly, if $\tilde{\bbf}$ is another smooth function on $\M$, s.t. $\tilde{\bbf}=\bbf$ outside some compact subset of $\M$, then proceeding as in \cite[Theorem 3.24]{bismut1995flat}, in $ \Omega^{\bullet}(S)/d^S \Omega^{\bullet}(S)$,
\[\btiT(T^H\M,g^{T\bZ},e^{-2T\tilde{\bbf}}h^{\cF},h^{\bfH},h^{\bfH_{\Cb^m}})=\btiT(T^H\M,g^{T\bZ},e^{-2T\bbf}h^{\cF},h^{\bfH},h^{\bfH_{\Cb^m}}).\]


In summary, together with Proposition \ref{indepf}, we have shown that


\begin{thm}\label{thm52}
	
	If $g^{T\bZ}=g^{T\bZ'}$ and $\tilde{\bbf}=\bbf$ outside some compact subset of $\M$, then for any $T>0,$ the following identity holds
	\[\btiT(T^H\M,g^{T\bZ},e^{-2T\tilde{\bbf}}h^{\cF},h^{\bfH},h^{\bfH_{\Cb^m}})=\btiT(T^HM,g^{TZ},h^F,h^H,h^{H_{\Cb^m}})\]
	in $ \Omega^{\bullet}(S)/d^S \Omega^{\bullet}(S)$,
	
\end{thm}

\subsection{Definition of higher combinatorial torsions}\label{defct}

Before going any further, it should be noted that in this section, we will perform a number of deformation of our data, but ultimately, by Remark \ref{torsion-indep-perturbation}, the combinatorial torsion we define will not depend on the exact choice of these perturbations.

\subsubsection{{Combinatorial complex and its torsion}}
\label{combinatorial-complex}

Let $f:M\to\R$ be a fiberwise generalized Morse function satisfying the conditions described in \cref{secfff}, and $(\M,\bbf)$ is a double suspension of $(M,f)$ for an even number $N$ satisfying Condition \ref{N-large}. Recall $g^{T\bZ'}$ of \eqref{defgTZprime}. 

\begin{lem}
\label{flow-lines-avoid-neighborhoods}
Let $ g^{T\bZ}$ be a metric which coincides with $ g^{T\bZ'} $  outside of a compact neighborhood of $ M \times \{0\} \times \{0\} \subset \M $. Let $\s_0\in S$, and let $p_0$ be a critical point in $\bZ_{\s_0}$. We assume that $p_0$ is independent of the other critical points, denoted by $p_l$, $l\geq 1$. 

Fix a collection of disjoint balls $U_l\subset \bZ_{\s_0}$ centered at $p_l$, for $l\geq 0$, such that $\big(\rW^{\ru}(p_0)\cup \rW^{\rs}(p_0)\big)\cap  \left(\cup_{l \geq 1}  \bar{U}_l \right)=\emptyset$. Then the following holds:

    \begin{enumerate}
        \item There exists $\rho_0 \in (0,1)$ such that for any $p\in\bZ_{\s_0}$ with $d_{\bZ_{\s_0}}(p,p_0) < \rho_0$ (where $d_\bZ$ is the distance induced by $g^{T\bZ}$) and for any flow line $\gamma$ of $
    \pm\nabla \bbf_{\s_0}$ that passes through $p$, $\gamma$ does not intersect $\cup_{l \geq 1}  \bar{U}_l$.
    \item If  $U'_l\subset U_l$ is a smaller ball centered at $p_l$, for $l\geq 1$, then there exists $\epsilon_0 \in (0,1)$ satisfying the following.
    On $B_{\epsilon_0}:=\{\s\: :\: d_S(\s_0,\s) < \epsilon_0\}$ (where $d_S$ is the distance induced by $g^{TS}$), $\M$ can be trivialized as $\bZ_{\s_0}\times B_{\epsilon_0}$, and if we fix such a trivialization, then for any $(p,\s)\in \bZ_{\s_0}\times B_{\epsilon_0}$ such that $d_{\bZ_ \theta}(p,p_0) < \rho_0/2$ and for any flow line $\gamma$ of $-\nabla \bbf_{\s}$ starting at $(p,\s)$, $\gamma$ does not intersect $\cup_{l \geq 1}  \bar{U'_l}$.
    \item The above $\rho_0$ can be chosen uniformly on $\s_0\in S$ and on critical points $p_0\in\bZ_{\s_0}$.
    \end{enumerate}
\end{lem}

\begin{proof}
    Let us prove the first part of the Lemma. Let us assume on the contrary that there is a sequence $\{q_k\}$ converging to $p_0$, with gradient flows $\gamma_k$ with respect to $\pm \nabla \bbf_{\s_0}$ starting from $q_k$ and proceeding to $\cup_{l \geq 1}  \bar{U}_l$. 
    
    Let $U'_0\subset U_0$ be a smaller ball centered at $p_0$. Let $q_{1,k} \in \partial U_0'$ be the point where $\gamma_k$ intersects $\partial U_0'$ for the last time, and let $q_{2,k} \in \cup_{l \geq 1} \partial U_l$ be the point where $\gamma_k$ intersects $\cup_{l \geq 1} \partial U_l$ for the first time. Let $\tilde{\gamma}_k$ be the segment of $\gamma_k$ that goes from $q_{1,k}$ to $q_{2,k}$. 
			
			Let $ |\nabla \bbf_{\s_0}|^2g^{TZ} $ denote the Agmon metric on $ Z - U_0'-\bigcup_{l\geq 1} U_l  $, and let $ \dist(\cdot,\cdot) $ be the distance induced by this Agmon metric on $ Z -U_0'- \bigcup_{l\geq 1} U_l  $. By \cite[Lemma 4.2]{DY2020cohomology}, (which is a general version of Proposition \ref{prop84}), it follows that $ \tilde{\gamma}_k $ are contained in the compact set
			\[ K := \bigg\{ q \in Z-U_0'-\cup_{l\geq1}U_l : \dist(q, \partial U_0') \in \Big[ 0, \sup_{q' \in \partial U_0', q'' \in \bigcup_{l \geq 1} \partial U_l} \big|\bbf_{\s_0}(q') - \bbf_{\s_0}(q'')\big| \Big] \bigg\}. \]
			
			On $ K $, both $ |\nabla \bbf_{\s_0}|^2 $ and $ \Big|\nabla_{\nabla \bbf_{\s_0}}\nabla \bbf_{\s_0}\Big|^2 $ are bounded, where $ \nabla \bbf_{\s_0} $, $ |\nabla \bbf_{\s_0}|^2 $, etc., are defined with respect to the original metric $ g^{TZ} $. Thus, by the Arze\'a–Ascoli Theorem, up to taking a subsequence we may assume that $ \tilde{\gamma}_k $ converges in the $ C^1 $ topology to $ \tilde{\gamma}_\infty $. As a result, $ \tilde{\gamma}_\infty $ is also a flow line with respect to $ \pm \nabla \bbf_{\s_0} $.
			
			Since $ q_k \to p_0 $, we must have, again up to taking a subsequence, that $ q_{1,k} $ converge to a point $ q_{1,\infty} \in \partial U_0' \cap \rW^{{\ru/\rs}}(p_0) $.			As a result, there exists a flow line with respect to $ -\nabla \bbf_{\s_0} $ connecting $ p_0 $ and $ \bigcup_{l \geq 1}  \bar{U}_l $, which leads to a contradiction.

            We now turn to the proof of the second part of Lemma \ref{flow-lines-avoid-neighborhoods}. For $\s$ close enough to $\s_0$, we fix a trivialization of $\M\to S$ as in the statement, so that we can  see $\bbf_{\s}$ as a deformation of $\bbf_{\s_0}$ on $\bZ_{\s_0}$. By standard ODE theory, we then conclude that the gradient flow of $\bbf_\theta$ on $\bZ_{\s_0}-\left(\cup_{l \geq 1}  U'_l \right)- \{p\::\: d_\bZ(p,p_0) < \rho_0/2\}$ is close to the one of $\bbf_{\s_0}$. As a result, if we have a flow line $\gamma$ with respect to $\pm \nabla \bbf_\s$ starting from $p$ with $d_\bZ(p,p_0) < \rho_0/2$, and flowing to $\cup_{l \geq 1}  \bar{U'_l}$, we can deduce that there is a flow line $\gamma_0$ with respect to $\pm \nabla \bbf_{\s_0}$ flowing from $\{p\::\: d_\bZ(p,p_0) < \rho_0\}$ to $\cup_{l \geq 1}  U_l $, which is not possible.    

            Finally, the third part of Lemma \ref{flow-lines-avoid-neighborhoods} is a consequence of the compacity of $S$ and of the fact that there are only finitely many critical points.
\end{proof}

Let $\V$\index{V@$\V$} be a tubular neighborhood of $\Sigma^{(1)}(\bbf)$ in $\M$ that satisfies the same conditions as $\mfV$ in Lemma \ref{lem441} for $f$. { By \eqref{Sigma(bbf)}, $\bbf$ gives a  stratification of $S$ as in \cref{sec-fGMF}, which is just the same  as $f$.} We also assume that Condition \ref{assum32} holds. We can see that $\Omega=\pi(\V)$\index{Omega@$\Omega$} is a small ``tubular" neighborhood of $L:=\pi(\Sigma^{(1)}(f))={\ppi}(\Sigma^{(1)}(\bbf))$. Let also $\Omega_0''$, $\Omega_1$, $\Omega_1''$, etc., be open sets that are constructed as described in \cref{assotwo}.

Moreover, fiberwisely, {we assume that} $\V_\s:=\pi^{-1}(\s)\cap \V$, $\s\in \pi(\V)$, consists of one or two balls of {large} radius, say equal to $100$, with respect to the metric $g^{T\bZ}$, and with centers being birth-death points if $\s\in L.$ For $r\in(0,100)$, let $\V(r)\subset \V$\index{Vr@$\V(r)$} denotes a neighborhood of $\Sigma^{(1)}(f)$, such that $\pi(\V(r))=\pi(\V)$ and $\V(r)_\s:=\pi^{-1}(\s)\cap \V(r), \s\in \pi(\V)$, consists of {the closed balls $\V(r)_\s$ with the same center as $\V_\s$ but with} radius $r$. For $r'<r$, let $\V(r,r'):=\V(r')-\mathring{\V}(r).$\index{Vrrp@$\V(r,r')$}

Let $ r_1, r_2 $ be the constants appearing in the definition of $ q_A $ in \cref{assotwo}.  Then we take $r_2<\rho_0/3$, with $\rho_0$ in Lemma \ref{flow-lines-avoid-neighborhoods}.

Observe that in the proof of \cite[Lemma 4.5.1]{igusa2002higher} (or Lemma \ref{indepbd}), the metric is deformed in an annulus around the set $\Sigma'$. Thus, using \cite[Lemma 4.5.1]{igusa2002higher} (or Lemma \ref{indepbd} and Remark \ref{rem-indepbd}) and Lemma \ref{flow-lines-avoid-neighborhoods}, and choosing $\V$ and $r_2$ small enough, there exist a metric $g^{T\bZ}$ satisfying the following condition:
\begin{cond}\label{5a'}
	\begin{enumerate}[(1)]
		\item Fiberwisely, for $\s\in\Omega$, birth-death points, or critical points of $\bbf_\s$ which converge to birth-death points as $t_1 \to 0$ (where $t_1$ is the function in Lemma \ref{lem441}), are independent of other critical points under $g^{T\bZ}$.
		\item we have $g^{T\bZ}=g^{T\bZ'}$ on $\M-\V$. 
		\item the metric $g^{T\bZ}=g^{T\bZ'}$ is standard inside $\V(r_1/2)$ in the sense of \eqref{standardmetric}.  
	\end{enumerate}
\end{cond}
\textbf{We furthermore assume that $(\bbf, g^{T\bZ})$ satisfies the fiberwise Thom-Smale transversality condition}.

Let $(\cF^*,\nabla^{\cF^*})$ be the dual bundle and dual connection of $(\cF,\nabla^{\cF}).$

Now let $S':=S-\pi(\Sigma^{(1)}(\bbf))$\index{Sprime@$S'$}. Then we can see that $\pi:\Sigma(\bbf)|_{S'}\to S'$ is a finite covering map, which may have different number of sheets in different connected components of $S'$.

For each $ p \in \Sigma^{(0)}(\bbf) $, let $ \s = \ppi(p) $. Let $ \rW^{\ru/\rs}(p) \subset \bZ_\s $ be the fiberwise unstable/stable set with respect to the fiberwise vertical gradient vector field $-\nabla^v\bbf$ of $ \bbf $ for the metric $ g^{T\bZ} $.
Let $T^{\ru}_p\bZ:=T_p\rW^{\ru}(p)$ and $T^\rs_p\bZ:=T_p\rW^{\rs}(p)$, then $T_p\bZ=T^{\ru}_p\bZ\oplus T^{\rs}_p\bZ$. Let $o^{\ru/\rs}_p$ be determinant lines $\det(T^{\ru/\rs}_p\bZ\otimes\C)$, which is well-defined on $\Sigma(\bbf)-\Sigma^{(1)}(\bbf)$.

For each $p\in\Sigma^{(1,i)}( {\bbf})$, let $\s=\ppi(p).$ Under the coordinate given by Lemma \ref{lem441}, let $o_p^{u,1}$ be generated by $(\frac{\p}{\p u_1}\wedge\cdots\wedge\frac{\p}{\p u_i})\otimes \C$, and $o_p^{u,2}$ by be generated by $(\frac{\p}{\p u_0}\wedge\frac{\p}{\p u_1}\wedge\cdots\wedge\frac{\p}{\p u_i})\otimes \C.$ One can see that $ o_p^{u,a}$ ($a=1,2$) is well-defined on $\Sigma^{(1)}( {\bbf})$, and could be extended to  {the small} tubular neighbor $\V$ of $\Sigma^{(1)}( {\bbf})$.

We define a vector bundle ${V}_k\to S'$ as follows: 
\begin{equation}
\label{Def-V_k}
	({V}_k)_\s:=\oplus_{p\in\Sigma^{(0,k)}(\bbf), {\ppi}(p)=\s} o^\ru_p\otimes \cF^*_{p} \quad \text{for} \quad \s\in S'.
\end{equation}
{Note that in particular, the rank of $V_k$ is fixed on each connected component of $S'$.}

For $p\in\Sigma^{(0,k)}(\bbf),q\in\Sigma^{(0,k-1)}(\bbf)$ such that $ {\ppi}(p)= {\ppi}(q)\in S'$, let $\M(p,q)$ be the set of flow line of $-\nabla^v\bbf$ {going from $p$ to $q$}. By \cite[Remark 4.3 and Lemma 7.6]{DY2020cohomology}, this set is finite.

{
	For each $\gamma\in\M(p,q)$, let $\tau_{\gamma}$ be the parallel transport along $\gamma$ associated with $\nabla^{\cF^*}.$ For each $\gamma\in\M(p,q)$, we can define $n_\gamma\in o_p^{\ru,*}\otimes o_q^{\ru}$ according to the orientation, see \cite[\S 5.1]{bismut2001families}. Then we have a differential $\tilde\p: {V}_k\to {V}_{k-1}$, such that for $x\in o_p^\ru$ and $s\in \cF^*_{p}$:
	$$\tilde\p (x\otimes s)=\sum_{q\in \Sigma^{(0,k-1)}(\bbf), {\ppi}(q)= {\ppi}(p)} \sum_{\gamma\in\M(p,q)}n_\gamma x\otimes\tau_\gamma s,$$
	where $o_p^{\ru,*}$ is the dual bundle of $o_p^{\ru}.$
	Note that, roughly, if one fixes an orientation of $T^{\ru}_p\bZ$ at each $p$, then $n_\gamma$ is simply $\pm1$ (see \cite[(1.30)]{bismutzhang1992cm} for more exact description). 
}

Let $V^k$\index{Vk@$V^k$} be the dual bundle of ${V}_k$, $\p: V^k\to V^{k+1}$\index{partial@$\p$} be the dual of $\tilde\p$ and $V:=\oplus_{k}V^k$. Let $\nabla^V$ be the connection on $V$ induced by $\nabla^{\cF}$, as in \cite[\S 5.5]{bismut2001families}. The metric $h^V$ on $V$ given by $h_\s^{{V}^k}:=\oplus_{p\in\Sigma^{0,k}(\bbf),\ppi(p)=\s} g^{o^{\ru,*}_p}\otimes h^{\cF_p}$\index{hV@$h^V$}, where $g^{o^{\ru,*}_p}$ is the metric on $o_p^{\ru,*}$ induced by $g^{T\bZ}.$

Let \begin{equation}
    \label{def-sc-V}
    \A'\index{Aprime@$\A'$}=\p+\nabla^V.
\end{equation}
Then, as in Definition \ref{defn26} and Remark \ref{remeq}, we can define the associated analytic torsion form
$\T(\A',h^V)\index{TAhV@$\T(\A',h^V)$}\in \Omega^{\bullet}(S')$.

Consider $\s_1, \s_2 \in S' \cap \ppi(\V)$, which belong to different connected components of $S' \cap \ppi(\V)$ but are close to each other in $S$. Let $\gamma: [0, 1] \to S$ be a path connecting $\s_1$ and $\s_2$, and suppose that $\gamma$ crosses $L = \ppi(\Sigma^{(1)}(\bbf))$ exactly once, at the point $\s' \in L$.
Suppose $\gamma(\half)=\s'$. If there is only one birth-death point $p$ corresponding to $\s'$,  there is only one function $t_1$ as in Lemma \ref{lem441} corresponding to it, so we can assume that $t_1(\s_1)<0$ and $t_1(\s_2)>0$. 
Then as $\gamma$ passes through $L$ from $\s_1$ to $\s_2$, two new critical points are created. It is straightforward to see that
\be\label{crosscplx} V^{\s_2}\cong V^{\s_1}\bigoplus\Big(\left(o_p^{\ru,1,*}\otimes\cF_p\right)\oplus \left(o_p^{\ru,2,*}\otimes\cF_p\right)\Big).\ee
{Here, $o_p^{\ru,1,*}$ and $o_p^{\ru,2,*}$ can be viewed as the orientation bundle associated with the 2 collapsing critical points.} By item (1) in Condition \ref{5a'}, if $\s_1$ and $\s_2$ are close enough, the differential map $\p_{\s_2}$ on $\V_{\s_2}$ is given by
\begin{equation}\label{d-bd-point}
	\begin{pmatrix}
		\p_{\s_1} &0 &0\\
		0 & 0& \pm1\\
		0 & 0& 0
	\end{pmatrix}
\end{equation}
with respect to the decomposition \eqref{crosscplx}.

{Since the metrics $g^{T\bZ}$ and $h^\cF$ are standard near critical points (in the fiber on $\Omega$), \cite[item (c) in Theorem A.1.1]{bismut1995flat} imply that \begin{obs}\label{obs5}there exists $\epsilon_0>0$, s.t. if $|s-\half|<\epsilon_0$, the contribution of $\left(o_p^{\ru,1,*}\otimes\cF_p\right)\oplus \left(o_p^{\ru,2,*}\otimes\cF_p\right)$ to the torsion form at $\gamma(s)$ equal to $0$.
\end{obs} }

Similar facts hold if $\s'$ corresponds to more than one birth-death point.	As a result, we get
\begin{thm}\label{thm53}We can extend $\T(\A',h^V)$\index{TAhV@$\T(\A',h^V)$} to a smooth form on $S.$\end{thm}

Let $H^V\index{HV@$H^V$}\to S'$ be the cohomological bundle with respect to $V\to S'$ and $\p$.  By (1) in  \Cref{5a'}, {using \eqref{crosscplx}-\eqref{d-bd-point}}, the subbundle of harmonic elements $\H^V\subset V$ can be naturally extended to a bundle over $S$. By Hodge theory,  we can extend $H^V\to S'$ to $H^V\to S$ naturally. 

\def\cL{{\mathcal{L}}}
By \cite[Theorem 1.3]{DY2020cohomology}, we know that 
\be\label{H^V=bfH}
\bfH\simeq H^V.
\ee
In the sequel, we would like to define an explicit isomorphism  $\J:\bfH\to H^V$. First, we define a map $$\cL_T\index{LT@$\cL_T$}:\bfH_{T}\to H^V$$ as follows. Let $w$ be a harmonic form of associated with $(d^{\bZ},g^{T\bZ},e^{-T\bbf}h^{\cF})$, then for $\s\in\Omega_0''$, similar to \cite[Definition 5.2]{bismut2001families}, we set \[L^0_T(w):=\sum_{p\in\Sigma(f),\ppi(p)=\s}\int_{\rW^{\ru}(p)}w,\]
and for $\s\in\Omega_1$, we set \[L^1_T(w):=\sum_{p\in\Sigma(f)-\V,\ppi(p)=\s}\int_{\rW^{\ru}(p)}w.\]
Here unstable manifolds $\rW^{\ru}(p)$ are defined with respect to the vertical gradient flow of $\bbf$ for the metric $g^{T\bZ}$, and over each of them the bundle  $\cF$ is identified with the trivial bundle of fiber $\cF_p$ thanks to $\nabla^\cF$. The well-definedness of $L^a_T$ for $a=0,1$ (i.e., the integrations mentioned are all finite) is discussed in \cite[p. 659]{DY2020cohomology}. More details will be explained in \cref{proofthm54}. Furthermore, \cite[Corollary 7.2]{DY2020cohomology} asserts that the image of $L^{a}_T(w)$  is $\p$-closed for $a=0,1$. Also, according to (1) in Condition \ref{5a'}, $L^0_T(w)$ and $L^1_T(w)$ represent the same class in $H^V$ for $\s\in\Omega\cap\Omega_0''$.

Thus, we can define 
\begin{equation}
    \label{def-cL_T}
       \text{For } \alpha=[w]\in\bfH_{T}, \quad \cL_T(\alpha)=[L^a_T(w)]\in H^V, \quad\text{ for }a=0\text{ or }1.
\end{equation}

Recall that we assume that Condition \ref{assum32} holds. To prove the next theorem, we will need the analytic techniques developed in Part \ref{PartII}. However, in order to apply them, we need to add a condition to the metric $g^{T\bZ}$ near the sphere $\V(\frac{r_1+r_2}{2},\frac{r_1+r_2}{2})$, see Assumption~\ref{product-type}.  By the same reasoning as for the second part of Lemma \ref{flow-lines-avoid-neighborhoods}, we see that if we perform a small enough perturbation of $g^{T\bZ}$ near $\V(\frac{r_1+r_2}{2},\frac{r_1+r_2}{2})$, the independence condition (1) of Condition \ref{5a'} will be preserved. Thus, if $r_2$ is small enough, we can deform  $g^{T\bZ}$ so that the following holds.
\begin{cond}\label{5a}
	\begin{enumerate}[(1)]
		\item Fiberwisely, for $\s\in\Omega$, birth-death points, or critical points of $\bbf_\s$ which converge to birth-death points as $t_1 \to 0$ (where $t_1$ is the function in Lemma \ref{lem441}), are independent of other critical points under $g^{T\bZ}$.
		\item we have $g^{T\bZ}=g^{T\bZ'}$ on $\M-\V$. 
		\item the metric $g^{T\bZ}$ is standard inside $\V(r_1/2)$ in the sense of \eqref{standardmetric}.
        \item Let $Y$ be the sphere $\V(\frac{r_1+r_2}{2},\frac{r_1+r_2}{2})$, then on 
        $\V(r_1,r_2) \simeq [-r_1,r_2]\times Y $, the metrics $g^{T\bZ}$ and $h^F$ are of product form  in the sense of \cite[(2.1)-(2.3)]{MR3030691}.
	\end{enumerate}
\end{cond}
Note that, again using the same reasoning as for the second part of Lemma \ref{flow-lines-avoid-neighborhoods}, after this deformation,  $(\bbf, g^{T\bZ})$ still satisfies the fiberwise Thom-Smale transversality condition.

We can now state our result.
\begin{thm}\label{thm54}
	The map $\cL_T\K_T\J_T^{-1}\K_1^{-1}:\bfH\to H^V$ is independent of $T$, so we will denote it by $\J$\index{J@$\J$}. Moreover, the map $\J$ is an isomorphism of bundles over $S$.
\end{thm}
Let \(h^{\mathbf H}\) and \(h^{\mathbf H_{\mathbb C^m}}\) be arbitrary metrics on \(\mathbf H \to S\) and \(\mathbf H_{\mathbb C^m} \to S\), respectively.
 Similarly to Definition \ref{assotor}, we have
\begin{defn}\label{assocomtor}
	The modified higher combinatorial torsion $\tiT(\A',h^V,h^\bfH)$ is defined to be 
	$$\tiT(\A',h^V,h^\bfH)\index{TAhVhHtilde@$\tiT(\A',h^V,h^\bfH)$}=\T(\A',h^V)+\wch(\nabla^{\bfH},h^{\check{\bfH}}),$$
	where $\check{\bfH}\to \chS$ is the pullback of the bundle $\bfH\to S$ with respect to the canonical projection $\chS\to S$,  and $h^{\check{\bfH}}$ is the metric on $\check{\bfH}\to\chS$ given by $h^{\check{\bfH}}|_{S\times\{l\}}=(1-l)h_{L^2}^V+lh^{\bfH}$. Here $h_{L^2}^V$ is the metric on $\bfH$ given by $h^V$ and Hodge theory as well as $\J.$
	
	We also define $\btiT(\A',h^V,h^\bfH,h^{\bfH_{\Cb^m}})$\index{TAhVhHtildebar@$\btiT(\A',h^V,h^\bfH,h^{\bfH_{\Cb^m}})$} applying the bar convention (c.f. \cref{sec21}) to $\tiT(\A',h^V,h^\bfH)$.
	
\end{defn}

\def\infact{0}
\if\infact1
In fact, \eqref{goodmetric} holds if we replace $h^V$ with some generalized metric $\bh^V$ with the following nice properties:
\begin{defn}[Good metric]
	A generalized metric $\bh^V$ on $V \to S'$ is called good if, for any $\s \in S'$ and for any $v, w \in V_\s$ corresponding to different critical points of $\bbf_\s$, we have $\bh^V(v, w) = 0$.
\end{defn}

Then similar to \eqref{goodmetric}, we can show that for any $\tau>0$, and good metric $\bh^V$,
\be\label{goodmetric1}\lim_{s\to\left(\frac{1}{2}\right)^+}\T^{\mL}_\tau(\A',\bh^V)(\gamma(s))=\lim_{s\to\left(\frac{1}{2}\right)^-}\T^{\mL}_\tau(\A',\bh^V)(\gamma(s)).\ee
As a result, we can extend $\T^{\mL}_\tau(\A',\bh^V)$ to a smooth form on $S.$

\fi

\subsubsection{Proof of the first part of Theorem \ref{thm54}}\label{proofthm54}
Here, we give {the proof of the first part of Theorem \ref{thm54}, i.e., the fact that the map $\cL_T\K_TJ_T^{-1}\K_1^{-1}$ is independent of $T$}. The proof of Theorem \ref{thm54} will be completed in \cref{sec11}.

First, we will give a brief introduction to the Agmon estimate derived in \cite{DY2020cohomology}. More details will be covered in \cref{conag} and \cref{sec9}.

\def\AG{{\mathrm{AG}}}
For simplicity, to state the Agmon estimate, we assume $S$ is simply a point. A fiberwise version of the Agmon estimate is straightforward to see. We choose a compact subset $K \subset \bZ$ containing all critical points of $\bbf$, such that $g^{T\bZ}$ and $g^{T\bZ'}$ agree outside $K$. Then outside $K$, the gradient of $\bbf$ with respect to $g^{T\bZ}$ and $g^{T\bZ'}$ are the same, so it will not cause any confusion to denote both of them as $\nabla \bbf$ out side $K$. Then we have the Agmon metric outside $K$: $g_{\AG}^{T\bZ}:=|\nabla \bbf|^2 g^{T\bZ}$. Let $d$ be the distance function induced by $g_{\AG}^{T\bZ}$, and set $\rho(x):=d(x,K)$ for any $x \in \bZ - K$ and $\rho(x)=0$ if $x\in K$.

Classically in Witten deformation we know that if $w$ is a harmonic form associated with $(d^{\bZ},g^{T\bZ},e^{-T\bbf}h^{\cF})$, then $\tilde{w}:=e^{-T\bbf}w$ is a harmonic form associated with $(d_T, g^{T\bZ}, h^{\cF})$, where $d_T := d^{\bZ} + T d \bbf \wedge$ is the Witten deformation of $d^{\bZ}$. Then Agmon estimate \cite[Theorem 1.1]{DY2020cohomology} tells us that {there are some $T$-independent constant $c,C > 0$ such that, } if $w$ is a harmonic form associated with either $(d^{\bZ}, g^{T\bZ}, e^{-T\bbf}h^{\cF})$ or $(d^{\bZ}, g^{T\bZ'}, e^{-T\bbf}h^{\cF})$, $\tilde{w}:=e^{-T\bbf}w$ satisfies the following exponential decay:
\be\label{eq39}
|\tilde{w}| \leq Ce^{-cT\rho} \quad \text{near infinity}.
\ee
Here the norm is induced by $g^{T\bZ}$ and $h^\cF.$

Let $F$ be a positive smooth function on $\bZ$ such that outside a compact neighborhood of $K$, $F \equiv \frac{1}{|\nabla f|^2}$. Let $p$ be a Morse point of $\bbf$ with Morse index $k$, and let $B \subset \rW^{\ru}(p)$ be a small ball of dimension $k$ centered at $p$. Let $\Phi^t$ be the flow generated by $-F \nabla f$. Then for any differential form $\a$, by the substitution rule of integrals, we have
\begin{equation}\label{eq40}
	\begin{aligned}
		& \left|\int_{\rW^\ru(p)}  \a\right| = \left|\lim_{t \rightarrow \infty} \int_{\left({\Phi}^t\right)(B)} \a\right| \\
		& \leq C \exp (T \bbf(p)) \lim_{t \rightarrow \infty} \int_{B} |\tilde{\a}| \circ {\Phi}^t \left|\operatorname{det}\left(\left({\Phi}^t\right)_*\right)\right| \mathrm{dvol}_{\rW^\ru_p},
	\end{aligned}
\end{equation}
where $\tilde{\a} = e^{-T\bbf} \a$ and the norm on the differential form is induced by $g^{T\bZ}$ and $h^\cF$.

For any harmonic form $w$ associated with $(d^{\bZ}, g^{T\bZ}, e^{-T\bbf}h^{\cF})$, $\tilde{w} := e^{-T\bbf} w$ satisfies \eqref{eq39}. By \eqref{eq40} and \cite[Lemma 4.4 and Lemma 4.5]{DY2020cohomology}, all integrals involved in the definition of $\cL_T$ are finite.

Next, we would like to give another description of $\J_T$ in \eqref{defjt}. We will construct a map similar to one described in \cite[p. 673]{DY2020cohomology}.

For any $T > 1$, let $w$ be a harmonic form associated with $(d^{\bZ}, {g^{T\bZ'}}, e^{-T\bbf} h^{\cF})$. By \cite[Theorem 1.1]{DY2020cohomology}, for any $T' \in \left[\frac{7T}{8}, \frac{9T}{8}\right]$, $w$ is also $L^2$-integrable with respect to the measure induced by ${g^{T\bZ'}}$ and $e^{-T'\bbf} h^{\cF}$, so $w$ represents a class in $\bfH_{T'}'$. We denote by $\J_{T,T'}$\index{JTTprime@$\J_{T,T'}$} the map that sends the element in $\bfH_{T}'$ represented by $w$ to the element in $\bfH_{T'}'$ represented by $w$.

\begin{lem}\label{lem51}
	For any $ T > 1 $ and $ T' \in [\frac{7T}{8}, \frac{9T}{8}] $, we have $ \J_{T,T'} = \J_{T'}^{-1}\J_T $. Moreover, for any harmonic form $ w_1 $ associated with $ (d^{\bZ}, g^{T\bZ'}, e^{-T\bbf}h^{\cF}) $, there exists a harmonic form $ w_2 $ associated with $ (d^{\bZ}, g^{T\bZ'}, e^{-T'\bbf}h^{\cF}) $ and a differential form $ w_3 $ such that $ w_1 = w_2 + d^{\bZ} w_3 $.
	
	Furthermore, $ \tilde{w}_3 := e^{-T\bbf} w_3 $ satisfies \eqref{eq39}. Therefore, according to the discussions above, $ w_3 $ is integrable over all unstable manifolds in the definition of $ \cL_T $.
\end{lem}
\begin{proof}
	We first establish the lemma when $ Z $ is a point. Then $ \bZ = \R^{N} \times \R^N $ with coordinates $ (x_1, x_2) $ and $ \bbf = -|x_1|^2 + |x_2|^2 $. We can take $ \cF $ to be the trivial line bundle, and $ h^{\cF} $ is the canonical metric. Note that $ \frac{e^{-2T|x_1|^2}}{(\pi T)^{N/2}}dx_1 $ is harmonic associated with $ (d^{\bZ}, g^{T\bZ'}, e^{-T\bbf}h^{\cF}) $. Note that $ \dim(\bfH_{T}') = \dim(\bfH_{T'}') = 1 $. Proceeding as in \cite[p. 673]{DY2020cohomology}, we can see that there exists $ \a \in  \Omega^{\bullet}(\R^{2N}) $, such that
	\be\label{alpha is here} \frac{e^{-2T|x_1|^2}}{(2\pi T)^{\frac{N}{2}}}dx_1 = \frac{ce^{-2T'|x_1|^2}}{(2\pi T')^{\frac{N}{2}}}dx_1 + d^{\R^{2N}}\a \ee
	for some constant $ c $. Moreover, $ \tilde\a:=e^{-T\bbf}\a $ satisfies \eqref{eq39}. 
	
	So by Stokes' formula (see \cite[Corollary 7.2]{DY2020cohomology} and Remark \ref{rem56} below), we should have
	\[
	\int_{\R^N\times\{0\}}\frac{e^{-2T|x_1|^2}}{(2\pi T)^{\frac{N}{2}}}dx_1=\int_{\R^N\times\{0\}}\frac{ce^{-2T'|x_1|^2}}{(2\pi T')^{\frac{N}{2}}}dx_1.
	\]
	So the constant $c$ must be $1.$
	
	Thus we have established the lemma for $Z$ being a point.
	
	For simplicity, we will not distinguish between a closed form and the cohomology class it represents.
	
	In the general case, suppose $ w_1 = \gamma_1 \wedge \frac{e^{-2T|x_1|^2}}{(2\pi T)^{N/2}}dx_1 $ is harmonic associated with $ (d^{\bZ}, g^{T\bZ'}, e^{-T\bbf}h^{\cF}) $, where $ \gamma_1 $ is harmonic associated with $ (d^{Z}, g^{TZ}, e^{-Tf}h^{F}) $. Then $ \J_{T'}^{-1}\J_T $ maps $ w_1 $ to $ \gamma_1 \wedge \frac{e^{-2T'|x_1|^2}}{(2\pi T')^{N/2}}dx_1 $.
	
	By the discussion above, $ \J_{T,T'} $ maps $ w_1 $ to $ w_2 := \gamma_2 \wedge \frac{e^{-2T'|x_1|^2}}{(2\pi T')^{N/2}}dx_1 $, where $ \gamma_2 $ is harmonic associated with $ (d^{Z}, g^{TZ}, e^{-T'f}h^{F}) $, and $ \gamma_1 = \gamma_2 + d^Z \b $ for some $ \b \in  \Omega^{\bullet}(Z,F|_Z) $.
	
	As a result, we have 
    \be
    \label{JTT'=JT'invJT}
    \J_{T,T'} = \J_{T'}^{-1}\J_T.
    \ee
    Moreover,
	\[ w_1 = w_2 + d^{\bZ} \left( (-1)^{\deg(\gamma_1)} \gamma_1 \wedge \a + \b \wedge \frac{e^{-2T'|x_1|^2}}{(2\pi T')^{\frac{N}{2}}}dx_1 \right), \]
	where $\a$ is described in \eqref{alpha is here}.
	So we can take $w_3$ in the statement of this lemma to be $(-1)^{\deg(\a)}\gamma_1\wedge\a+\b\wedge \frac{e^{-2T'|x_1|^2}}{(2\pi T')^{N/2}}dx_1$. 
	
	It is straightforward to see that $\tilde{w}_3:=e^{-T\bbf}w_3$ satisfies \eqref{eq39} (Note that $2T'-T\geq3T/4$). By \eqref{eq40} and \cite[Lemma 4.4 and Lemma 4.5]{DY2020cohomology}, the integral of $w_3$ over all unstable manifold in the definition of $\cL_T$ is finite.
\end{proof}

\begin{lem}\label{lem52}
	For any $ T > 1 $, any harmonic form $ w_1 $ associated with $ (d^{\bZ}, g^{T\bZ'}, e^{-T\bbf}h^{\cF}) $, there exists a harmonic form $ w_2 $ associated with $ (d^{\bZ}, g^{T\bZ}, e^{-T\bbf}h^{\cF}) $ and a differential form $ w_3 $ such that $ w_1 = w_2 + d^{\bZ} w_3 $. Moreover, $ \tilde{w}_3 := e^{-T\bbf} w_3 $ satisfies \eqref{eq39}. Therefore, based on the discussions above, $ w_3 $ is integrable over all unstable manifolds in the definition of $ \cL_T $.
\end{lem}
\begin{proof}
	Following the approach in \cite[p. 673]{DY2020cohomology}, one can select $ w_3 $ such that $ \tilde{w}_3 := e^{-T\bbf} w_3 $ satisfies \eqref{eq39}. According to \eqref{eq40} and \cite[Lemma 4.4 and Lemma 4.5]{DY2020cohomology}, the integral of $ w_3 $ over all unstable manifolds in the definition of $ \cL_T $ is finite.
\end{proof}

We can construct a map $ J_{T,T'}:\bfH_{T}\to\bfH_{T'} $\index{JTTprime@$J_{T,T'}$} in the same way as $ \J_{T,T'}$ above Lemma \ref{lem51}, and set $ J_T := \K_1 \J_T \K_T^{-1}:\bfH_{T}\to\bfH $. By Lemma \ref{lem51} and Lemma \ref{lem52}, we have
\begin{lem}\label{lem53}
	For any $ T > 1 $ and $ T' \in [\frac{7T}{8}, \frac{9T}{8}] $, we have $ J_{T,T'} = J_{T'}^{-1} J_T $. Moreover, for any harmonic form $ w_1 $ associated with $ (d^{\bZ}, g^{T\bZ}, e^{-T\bbf}h^{\cF}) $, there exists a harmonic form $ w_2 $ associated with $ (d^{\bZ}, g^{T\bZ}, e^{-T'\bbf}h^{\cF}) $, and a differential form $ w_3 $ such that $ w_1 = w_2 + d^{\bZ} w_3 $. Furthermore, $ \tilde{w}_3 := e^{-T\bbf} w_3 $ satisfies \eqref{eq39}. Therefore, based on the discussions above, $ w_3 $ is integrable over all unstable manifolds in the definition of $ \cL_T $.
\end{lem}

Fix any $ T > 1 $. As $\K_T \J_T^{-1} \K_1^{-1} = J_T^{-1}$, we can apply Lemma  \ref{lem53} and use a version of Stokes' formula (see \cite[Corollary 7.2]{DY2020cohomology} and Remark \ref{rem56} below) to obtain $\cL_T \K_T \J_T^{-1} \K_1^{-1} = \cL_{T'} \K_{T'} \J_{T'}^{-1} \K_1^{-1}$ if $ T' \in \left[\frac{7T}{8}, \frac{9T}{8}\right] $. Replacing $ T $ with $\frac{9T}{8}$ and then with $\frac{7T}{8}$, we have $\cL_T \K_T \J_T^{-1} \K_1^{-1} = \cL_{T'} \K_{T'} \J_{T'}^{-1} \K_1^{-1}$ if $ T' \in \left[\frac{7^2 T}{8^2}, \frac{9^2 T}{8^2}\right] $. Continuing this process, we show that $\cL_T \K_T J_T^{-1} \K_1^{-1}$ is independent of $ T \geq 1 $.

\begin{rem}\label{rem56}
	Although the Stokes' formula \cite[Corollary 7.2]{DY2020cohomology} is stated for eigenforms with respect to small eigenvalue, by tracking the proof of \cite[Corollary 7.2]{DY2020cohomology}, we can see that in fact  \cite[Corollary 7.2]{DY2020cohomology} holds for $\a$ as long as both $\a$ and $(d^{\bZ}+Td^{\bZ}\bbf\wedge)\a$ satisfy $\eqref{eq39}$.
\end{rem}

\textbf{To simplify and streamline our discussion, we assume throughout this paper, starting from this point onward, that $\bbf$ has at most one birth-death point in each fiber, i.e., the stratification  $S_k$ of Section \ref{sec-fGMF} satisfies $S_k=\emptyset$ for $k\geq 2$, except in \cref{ineresult1} where we deal with the general case.} In this case, $\Omega_1=\ppi(\V)$.

In \cref{sec121}, we will show that $ \J $ is an isomorphism in $ \Omega_0'' $, and in \cref{sec122}, we will establish that $ \J $ is an isomorphism in $ \Omega_1'' $. This will complete the proof of Theorem \ref{thm54}, assuming each fiber has at most one birth-death point. The general case of Theorem \ref{thm54} can be proved using the same argument.

\subsection{Combinatorinal torsion forms associated with one ball removed}\label{diskremov}\label{sec53}

Recall that $(\M,\bbf)$ is a double suspension of $(M,f)$ introduced in \cref{suspen}.

Since $\ppi:\M\to S$ is a fibration with fiber $\bZ:=Z\times\R^N\times\R^N$, we can extend the coordinate $(u_0,\cdots,u_n)$ of $Z$ in Lemma \ref{lem441} to $(u_0,\cdots,u_n,\cdots,u_{n+2N})$ of $\bZ$.

Since $ g^{TZ} $ is standard with respect to the coordinates given in Lemma \ref{lem441} near $ \Sigma(f) $, $ g^{T\bZ} $ is standard with respect to the coordinates given above near $ \Sigma(\bbf)=\Sigma(f)\times\{0\}\times\{0\} $. Let $ \V $ be a small tubular neighborhood of $ \Sigma^{(1)}(\bbf) $ as in \cref{combinatorial-complex}.

	


Recall that $r_1$, $r_2$ are the constants appearing in the definition of $ q_A $ in \cref{assotwo}, and that in Section \ref{combinatorial-complex} we choose $r_2$ to be small enough (i.e., smaller that $\rho_0/3$, and such that Condition   \ref{5a} holds). 

{
	Let $\delta\in(0,\frac{r_1}{24})$ be small enough, whose value is to be determined later (See Definition \ref{defr}). We choose $\V$  such that the function $ t_1 $ in Lemma \ref{lem441} satisfies $ |t_1| \leq \delta^2 $ in $\Omega_1= \ppi(\V) $. We deform $\bbf$ inside $\V|_{\Omega_1}$  to replace it by $\tilde{\bbf}$, where $\tilde{\bbf}$ is defined in a similar way as $\tilde{f}_y$ in Lemma \ref{nocritical}. More precisely, let $\varrho$ be a smooth cut-off function on $S$, which is 1 on $\Omega''_1$ and supported in $\Omega_1$, then
	\begin{equation}
		\tilde{\bbf}=\bbf(u)-\varrho(\theta)\Big(\eta(|u|)u_{n+1}+t_1\tilde{\eta}(|u|)u_0\Big).
	\end{equation}
	We see that $\tilde{\bbf}=\bbf$ on $\M-\V$ or if $\s\in\Omega_1''$ and $|u|\in (0,r_1/6)\cup(5r_2/2,\infty)$. Moreover, by Lemma \ref{nocritical} and Remark \ref{rem41}, $\tilde{\bbf}$ and $\bbf$ have the same critical points, with same indices.	Finally, because $\tilde{\bbf}$ is fiberwisely a deformation of  $\bbf$ only in the annulus $\{r_1/6<|u|<5r_2/2\}$, applying again the same reasoning as for the second part of Lemma \ref{flow-lines-avoid-neighborhoods}, we see that   $(\tilde{\bbf}, g^{T\bZ})$ still satisfies the fiberwise Thom-Smale transversality condition.
	
	\textbf{From now on, we will replace $\bbf$ by this new function $\tilde{\bbf}$, so we will denote this latter function simply by $\bbf$.} We are taking this liberty because of Remark \ref{torsion-indep-perturbation}. Moreover, everything we have proved so far in  \cref{suspen} is still valid for this new function.
	\begin{rem}
		Note that in Lemma \ref{nocritical}, we deform the original function $ f_y $ by the terms $ \eta(|u|)u_1 $ and $ y\tilde{\eta}(|u|)u_0 $. According to Remark \ref{strgp}, the coordinate $ u_0 $ remains unchanged in different open sets $ V_\alpha $ in Lemma \ref{lem441}. However, the coordinate $ u_1 $ may vary for different open sets $ V_\alpha $. On the other hand, the coordinate $ u_{n+1} $ corresponds to the first component of the first copy of $ \mathbb{R}^N $, this is why we use it here instead of $u_1$.
	\end{rem}
	
	
}

Because we have assumed that $\bbf$ has at most one birth-death point in each fiber, the only non-Morse function in Definition \ref{assononmor} we need to  consider is $\bbf_{1,T,R}$. We let $\bbf_{T,R}:=\bbf_{1,T,R}$\index{fTR@$\bbf_{T,R}$} for short.

Let $\M^-\index{Mminus@$\M^-$}:=(\ppi)^{-1}(\Omega_1)-\V(\frac{r_1+r_2}{2}).$ Then $\M^-\to \Omega_1$ is a fibration of fiber $\bZ^-$\index{Zminus@$\bZ^-$}, where $\bZ^-_{\s}:=\bZ_\s-\V(\frac{r_1+r_2}{2})_\s,\s\in\Omega_1$.

We define a vector bundle $\tilde{V}_k \to \Omega_1$ as follows: for each $\s \in \Omega_1$,
\be
\label{def-Vtilde-on-Omega1}
(\tilde{V}_k)_\s := \bigoplus_{\substack{p \in \Sigma^{(0,k)}(\bbf) - \V \\ \ppi(p) = \s}} o^\ru_p \otimes \cF^*_{p}.
\ee
Then we can define $\tilde{V}^k$\index{Vtilde@$\tilde{V}$} on $\Omega_1$ to be the dual of $\tilde{V}_k$. We endow $\tilde{V}$ with the restriction of connection $\nabla^V$. Moreover, by Condition \ref{5a}, critical point inside $\V_\s$ are independent of any other critical points so we can also endow $\tilde{V}$ with the restriction of the differential $\partial$ of $V$. As in \eqref{def-sc-V}, this gives rise to a superconnection $\A^{\prime, \tilde{V}}$ on $\tilde{V}$.

It follows from Lemma \ref{sixcrits} that there exists \be\label{r1omega} R_1=R_1(\Omega_1''),\ee such that if $R>R_1$, fiberwisely, the critical points of $\bbf_{1,R}:=\bbf_{T=1,R}$\index{f1R@$\bbf_{1,R}$} are critical points of $\bbf|_{\M^-}$ with three more points added,  $w_+(R,\s),v_{1,+}(R,\s),v_{2,+}(R,\s)\in \bZ^{-}_\s,\s\in\Omega_1''$. We will abbreviate $w_+(R,\s),v_{1,+}(R,\s),v_{2,+}(R,\s)$ as $w_+,v_{1,+},v_{2,+}$ if it will not cause any confusion. We define a bundle $V^-(R)\to \Omega_1''$  by
$$V^-_\s(R):=\bigoplus_{p\in\Sigma(\bbf_{1,R}|_{\M^-},\ppi(p)=\s)}o_p^{\ru,*}\otimes \cF_p,$$
then, 
\be\label{decvr} V^-_\s(R)=\tilde{V}_\s\oplus (o_{v_{1,+}}^{\ru,*}\otimes \cF_{v_{1,+}})\oplus (o_{v_{2,+}}^{\ru,*}\otimes \cF_{v_{2,+}})\oplus (o_{w_+}^{\ru,*}\otimes \cF_{w_+}).\ee 
When we do not need to emphasise the dependence on $R$, we will denote $V^-_\s(R)$ simply by $V^-$\index{Vminus@$V^-$} for short.

\begin{rem}
\label{rem-construction-V^-}
    Here the pair of critical points converging to the birth-death point do not appear because we have assumed that $ |t_1| \leq \delta^2 $ in $\Omega_1$, so they are not in $\M^-.$

    Moreover, note that by our construction of $\bbf_{T,R}$, we have $\nabla_{\nu} \bbf_{T,R}|_{\partial \bZ^-} > 0$ for $R$ large enough, so that in the construction of the Thom-Smale-Witten complex for manifolds with boundary under relative boundary conditions (see \cite[Section 2]{lu2017thom} and \cite{LaudenbachFrançois2011AMco}) no critical points  appear in $\partial \bZ^-$.
\end{rem}

As $r_2<\rho_0/3$, with $\rho_0$ in Lemma \ref{flow-lines-avoid-neighborhoods}, by Lemma \ref{sixcrits}, we get the following.
\begin{obs}
\label{5c}
If $ g^{T\bZ} $ satisfies Condition \ref{5a}, then the critical points $w_+(R,\s)$, $v_{1,+}(R,\s)$ and $v_{2,+}(R,\s)$ are independent of {the critical points in $\Sigma^{(0,k)}(\bbf) - \V$} under $g^{T\bZ}$ for $R>R_1$ and $\s\in\Omega_1''$;
\end{obs}

\textbf{ Throughout the paper, we choose metric $ g^{T\bZ} $ satisfying Condition \ref{5a}}. 

Let $\p^-_{R}$ be the Thom-Smale differential of $V^-(R)$, defined as in \cref{combinatorial-complex} for $V$. It will be abbreviated as $\p^-$\index{partialminusR@$\p^-$} if this entails no confusion. As the deformation of $\bbf$ into $\bbf_{1,R}$ only take place inside $\{p\: : \: d_\bZ(p,\Sigma^1)\leq \rho_0/3\}$, by Lemma \ref{flow-lines-avoid-neighborhoods}, we see that the gradient lines of $\bbf_{1,R}$ linking critical points in $\tilde{V}$ will not change with $R$.
By Lemma \ref{sixcrits} and Observation \ref{5c}, we can thus see that with respect to the decomposition \eqref{decvr}, 
\be \label{decompo-partial_R}
\p^-_{R} = \begin{pmatrix}
	\p|_{\tilde{V}} &0 &0 &0\\
	0 & 0& c &0\\
	0 & 0& 0 &0\\
	0 &0 &0 &0
\end{pmatrix}
\ee
for some $c.$ 

As in \cref{combinatorial-complex}, we also have a connection $\nabla^{V^-(R)}$ on $V^-(R)$, induced by embedding $\Sigma(f_{1,R}|_{\M^-}) \subset \M^-$, $\nabla^{\cF}$, and $g^{T\bZ}$. Also, there is a metric $h_{1,R}^{V^-}$ on $V^-$, given by $h_{1,R,\s}^{{V}^{-,k}} := \bigoplus_{p \in \Sigma^k(\bbf_{1,R}|_{\M^-}), \pi(p) = \s} g^{o^{\ru,*}_p} \otimes h^{\cF_p}$. For simplicity, we will abbreviate $\nabla^{V^-(R)}$ and $h_{1,R}^{V^-}$  as $\nabla^{V^-}$\index{nablaVminus@$\nabla^{V^-}$} and $h^{V^-}$\index{hVminus@$h^{V^-}$}, if no confusion arises. Define
\be\A^{\prime,-}\index{Aprimeminus@$\A^{\prime,-}$}=\nabla^{V^-}+\p^-.\ee

Using the data above, as in \cref{combinatorial-complex}, we define a torsion form $\T({\A^{\prime,-}}, h^{V^-})$\index{TAminushVminus@$\T({\A^{\prime,-}}, h^{V^-})$} on $\Omega_1'' \subset S$. We have the following theorem:
\begin{thm}\label{int3}
	On $\Omega_1''$,
	\begin{align*}
		&\ \ \ \ \mathcal{\overline{T}}\left({\A^{\prime,-}},h^{V^{-}}\right)=\mathcal{\overline{T}}\left(\A',h^V\right).
	\end{align*}
	Here the bar convention (c.f. \cref{sec21}) is applied.
\end{thm}
\begin{proof}
	\def\U{\mathcal{U}}
	We rewrite the bundle $ V \to S $ as $ V^{\cF} \to S $ or $ V^{\Cb^m} \to S $ to emphasize that our bundle is related to $ \cF \to \M $ or $ \Cb^m \to \M $. Recall that we take $ \Omega_1'' $ to be a small enough tubular neighborhood of $ L $, such that fiberwise, {the critical points {of $\bbf$} inside $ \V_\s $} are independent of any other critical points for $ \s \in \Omega_1'' $.

	We just need to show that, on any simply connected subset $\tilde\Omega\subset\Omega_1''$, $\mathcal{\overline{T}}\left({\A'}^{-},h^{V^{-}}\right)=\mathcal{\overline{T}}\left(\A',h^V\right)$. By our choice of $\V$, $\U=(\ppi)^{-1}(\tilde\Omega)\cap\V$ is also simply connected. So there is a unitary trivialization of the bundle $\Phi:\cF|_{\U}\to \Cb^m|_{\U}$.
	
	Let $t_1$ be the function in Lemma \ref{lem441}. For $\s\in\Omega$ such that $t_1(\s)>0$, there exists $p^+,p^-\in\U\cap \Sigma^{(0)}(\bbf_\s)$, such that
	\[V_{\s}=\tilde{V}_{\s}\oplus\Big(\big(o_{p^+}^{\ru,*}\otimes\cF_{p^+}\big)\oplus \big(o_{p^-}^{\ru,*}\otimes\cF_{p^-}\big)\Big).\]
	
	We let $W\to {\tilde{\Omega}}^+:=\{\s\in\tilde\Omega:t_1(\s)>0\}$, such that $W_\s:=\left(o_{p^+}^{\ru,*}\otimes\cF_{p^+}\right)\oplus \left(o_{p^-}^{\ru,*}\otimes\cF_{p^-}\right).$ By Condition \ref{5a}, critical point inside $\V_\s$ are independent of any other critical points, so we have a canonical decomposition \be\label{dec48} \A'=\A^{\prime, \tilde{V}}\oplus\A^{\prime, W} \mbox{ and } h^V=h^{\tilde{V}}\oplus h^{W}.\ee
	
	Now by the definition of torsion forms,
	at the point $\s\in\Omega^+$, for ``$\bullet$''=``$\cF$'' or ``$\Cb^m$'',
	\be\label{bullet1}
	\T\Big(\A',h^{V^{\bullet}}\Big)_\s=\T\Big(\A^{\prime,\tilde{V}^{\bullet}},h^{\tilde{V}^{\bullet}}\Big)_\s+ \T\Big(\A^{\prime,W^{\bullet}},h^{W^{\bullet}}\Big)_\s.
	\ee		
	However, the unitarily trivialization $\cF|_{\U}\to \Cb^m|_{\U}$ induces an isometry of the bundle $W^{\cF}\to W^{\Cb^m}$ on ${\tilde\Omega}^+$. As a result, {at $\s\in\Omega^+$},
	\be\label{bullet2}
	\T\Big(\A^{\prime,W^{\cF}},h^{W^{\cF}}\Big)_\s=\T\Big(\A^{\prime,W^{\Cb^m}},h^{W^{\Cb^m}}\Big)_\s.
	\ee
	
	By \eqref{bullet1} and \eqref{bullet2},
	\be\label{bullet3}
	\overline{\T}\Big(\A',h^V\Big)=\overline{\T}\Big(\A^{\prime,\tilde{V}},h^{\tilde{V}}\Big) \text{ on } \Omega^+.
	\ee
	
	Proceeding in the same way, using {Observation \ref{5c}, \eqref{decvr} and} the corresponding decomposition of superconnections as in \eqref{dec48}, we also get
	\be\label{bullet4}
	\overline{\T}\Big(\A^{\prime,-},h^{V^-}\Big)=\overline{\T}\Big(\A^{\prime,\tilde{V}},h^{\tilde{V}}\Big) \text{ on } \Omega^+.
	\ee
	Hence the theorem holds for $\s\in\tilde{\Omega}^+$.
	
	Similarly, the theorem holds for $\s\in{\tilde\Omega}^-:=\{\s\in\tilde\Omega:t_1(\s)<0\}$.
	
	It follows from the continuity of torsion forms that the theorem holds on $\tilde\Omega.$
\end{proof}

\begin{rem}
	\label{torsion-indep-perturbation}
	Given the decomposition in \eqref{dec48} and equation \eqref{bullet3}, it is clear that $\overline{\T}(\A', h^V)$ is independent of the choice of perturbations {made to $(f,g^{TZ})$ or} $(\bbf, g^{T\bZ})$ made in this section {and \cref{secfff}, because they are only taking place inside $\V$} and we may assume that $\V$ is chosen small enough so that no flow line coming from or going to a critical point outside of $\V$ intersects $\V$. Thus, these flow lines will not be affected by our perturbations, and neither is the complex $\tilde{V}$.
\end{rem}

\begin{rem}\label{rem516}
	Let $\bfH^-_T \to \Omega_1''$\index{HTminus@$\bfH_T^-$} be the $L^2$-cohomology bundle associated with $$\Big(d^{\bZ}|_{\bZ^-}, g^{T\bZ}|_{\bZ^-}, e^{-2T\bbf} h^{\cF}|_{\bZ^-}\Big)$$ with relative boundary conditions, and let $\bfH^- := \bfH^-_{T=1}$\index{Hminus@$\bfH^-$}. Let $H^{V^-} \to \Omega_1''$\index{HVminus@$H^{V^-}$} be the cohomological bundle associated with $(V^-, \p^-)$. In Section \ref{sec122} we will show that $\bfH^-$ and $H^{V^-}$ are isomorphic and construct an explicit isomorphism $\J^-$ between them.
\end{rem}

\subsection{An anomaly formula}\label{sec54}
Let $\bbf^V\in\End(V)$\index{fV@$\bbf^V$} be given by \be\label{def-bbf^V}\bbf^V|_{o^{\ru,*}_p\otimes \cF_p}=\bbf(p). \mathrm{Id}_{ o^{\ru,*}_p\otimes \cF_p} \quad \text{for }p\in\Sigma(\bbf).\ee 
Then $\T\Big(\A',\big(e^{-T\bbf^V}\big)^*h^V\Big)$ is a smooth form on $ S' $. Since the metric $\big(e^{-T\bbf^V}\big)^*h^V$ does not satisfy the conditions required for \cite[item (c) in Theorem A.1.1]{bismut1995flat} to hold, the analogue of Observation \ref{obs5} is not true for $\T\Big(\A', \big(e^{-T\bbf^V}\big)^*h^V\Big)$. However, instead of Observation \ref{obs5}, by \cite[item (d) in Theorem A.1.1]{bismut1995flat}, we have
\be\label{goodmetric}\lim_{s\to\left(\frac{1}{2}\right)^+}\T\Big(\A',\big(e^{-T\bbf^V}\big)^*h^V\Big)(\gamma(s))=\lim_{s\to\left(\frac{1}{2}\right)^-}\T\Big(\A',\big(e^{-T\bbf^V}\big)^*h^V\Big)(\gamma(s)).\ee
So $\T\Big(\A',\big(e^{-T\bbf^V}\big)^*h^V\Big)$ can be extended to a (continuous) form on $S$.

Let $ \bbf_{1,R}^{V^-} \in \End(V^-) $\index{f1RVminus@$\bbf_{1,R}^{V^-}$} such that 
\be
\bbf_{1,R}^{V^-}|_{o^{\ru,*}_p \otimes \cF_p} = \bbf_{1,R}(p) . \mathrm{Id}_{o^{\ru,*}_p \otimes \cF_p} \quad \text{for } p \in \Sigma(\bbf_{1,R}|_{\M^-}).
\ee
Defining for $ (V^-,\bbf_{1,R}^{V^-},h^{V^-},\A^{\prime,-}) $ all the objects we defined above for $ (V,\bbf^V,h^V,\A') $, we {get a smooth form on $\Omega''_1$}
$$\T\Big(\A',\big(e^{-T\bbf_{1,R}^{V^-}}\big)^*h^{V^-}\Big).$$
Proceeding as in the proof Theorem \ref{int3}, one can see that on $\Omega_1''$,
\[\overline{\T}\Big(\A',\big(e^{-T\bbf^V}\big)^*h^V\Big)=\overline{\T}\Big(\A',\big(e^{-T\bbf_{1,R}^{V^-}}\big)^*h^{V^-}\Big).\]
As a result,\begin{obs}\label{obs6} while $\T\Big(\A', \big(e^{-T\bbf^V}\big)^*h^V\Big)$ can only be extended to $S$ as a continuous form, $\overline{\T}\Big(\A', \big(e^{-T\bbf^V}\big)^*h^V\Big)$\index{TAexphVbar@$\overline{\T}\Big(\A', \big(e^{-T\bbf^V}\big)^*h^V\Big)$} can be extended on  $S$ as a smooth form.\end{obs}
\def\cona{1}
\if\cona0
In fact, we can also show that 

\begin{obs}\label{obs519}
	On $\Omega_1''-L$, we have the decomposition $V=\tilde{V}\oplus W$ as in the proof of \Cref{int3}. Then the contribution of $W$ to $\T\Big(\A', \big(e^{-T\bbf^V}\big)^*h^V\Big)^{>0}$ vanishes. As a result,
	$\T\Big(\A', \big(e^{-T\bbf^V}\big)^*h^V\Big)^{>0}$ could be extended to a smooth form on $S$. Here $\T\Big(\A', \big(e^{-T\bbf^V}\big)^*h^V\Big)^{>0}$ is the positive degree component of $\T\Big(\A', \big(e^{-T\bbf^V}\big)^*h^V\Big).$
\end{obs}
\begin{proof}
	Note that $\tilde{V} \to \Omega_1''-L$ is a bundle of constant rank, which can be naturally extended to a bundle $\tilde{V} \to \Omega_1''$. However, $W \to \Omega_1''-L$ has different ranks in different connected components of $\Omega_1''-L$.
	
	Define $\nabla^W := \nabla^V|_{W}$, $\p^W := \p|_{W}$, $\bbf^W := \bbf^V|_{W}$, and $h^W := h^V|_{W}$. Define $\A^{W,\prime} := \A'|_{W}$, and let $\A^{W,\prime\prime}_T$ be the adjoint superconnection of $\A^{W,\prime}$ with respect to $\Big(\big(e^{-T\bbf^W}\big)\Big)^* h^W$. Set $2\A^W_T := \A^{W,\prime\prime}_T - \A^{W,\prime}_T$. Then, in the region where $t_1 > 0$, where $t_1$ is the function in \Cref{lem441},
	\[ 
	2e^{-T\bbf^W}\A^W_T e^{T\bbf^W} = \nabla^{W,*} - \nabla^W - 2T[\nabla^W, \bbf^W] + e^{-2T\left(1-3^{-\frac{3}{2}}\right)t_1^{\frac{3}{2}}}(\p^{W,*} - \p^W), 
	\]
	where $\nabla^{W,*}$ (resp. $\p^{W,*}$) is the adjoint connection (resp. adjoint operator) of $\nabla^W$ (resp. $\p^W$) with respect to $h^W$.
	
	Let $2\o^W_{T} := \nabla^{W,*} - \nabla^W - 2T[\nabla^W, \bbf^{W}]$ and $2V^W_T := e^{-2T\left(1-3^{-\frac{3}{2}}\right)t_1^{\frac{3}{2}}}(\p^{W,*} - \p^W)$. Because $\o^W_{T}|_{\tilde{V}}$ is odd and commutes with $(V_T^W)^2=-e^{-4T\left(1-3^{-\frac{3}{2}}\right)t_1^{\frac{3}{2}}}\mathrm{Id}$ and $N^W (:= N^V|_{W})$, while $h^{\prime}$ is an even function, and the supertraces of even powers of odd operators vanish, we have
	\[ 
	\Tr_s\left(N^W h'(\o^W_T + \sqrt{t}V^W_T)\right) = \Tr_s\left(N^W h'(\sqrt{t}V^W_T)\right). 
	\]
	
	As a result, the positive degree of $\Tr_s\left(N^W h'(\o^W_T + \sqrt{t}D^W_T)\right)$ vanishes in the region where $t_1 > 0$. Since it is clear that $\Tr_s\left(N^W h'(\o^W_T + \sqrt{t}D^W_T)\right)$ vanishes in the region where $t_1 < 0$, the observation is proven.
\end{proof}\fi

Replacing $h^V$ by $\big(e^{-T\bbf^V}\big)^*h^V$ in Definition \ref{assocomtor}, we define the continuous form $\tiT\Big(\A',\big(e^{-T\bbf^V}\big)^*h^V,h^{\bfH}\Big)$  and the smooth form $\btiT\Big(\A',\big(e^{-T\bbf^V}\big)^*h^V,h^{\bfH},h^{\bfH_{\Cb^m}}\Big)$\index{TAexphVhHhH0tildebar@$\btiT\Big(\A',\big(e^{-T\bbf^V}\big)^*h^V,h^{\bfH},h^{\bfH_{\Cb^m}}\Big)$} for any metrics $ h^\bfH $ on $ \bfH \to S $ and $ h^{\bfH_{\Cb^m}} $ on $ \bfH_{\Cb^m} \to S $, in the same manner as $\tiT(\A',h^V,h^{\bfH})$ and $\btiT(\A',h^V,h^{\bfH},h^{\bfH_{\Cb^m}})$.

Let $\hchS{}':=S'\times[0,1]\times(0,\infty)$ with coordinates $(l,t)\in[0,1]\times(0,\infty)$, $\hat{\check{\A}}'=\A'+dl\wedge\frac{\p}{\p l}+dt\wedge\frac{\p}{\p t}$, $\hat{\check{V}}\to \hchS{}'$ be the pullback of $V\to S'$ with respect to canonical projection $\hchS{}'\to S.$ Let $h^{\hat{\check{V}}}$ be a metric on $\hat{\check{V}}$, s.t. $$h^{\hat{\check{V}}}(\cdot,\cdot)|_{S'\times \{l\}\times\{t\}}=t^{-\frac{\dim{\bZ}}{2}} \big(e^{-lT\bbf^V}\big)^*h^V\left(t^{\frac{N^{\hat{\check{V}}}}{2}} \cdot, t^{\frac{N^{\hat{\check{V}}}}{2}} \cdot\right).$$

As in \cref{sec2}, we set $\B\Big(\A',h^{\hat{\check{V}}}\Big):=-\int_0^\infty \int_0^1h\Big(\hat{\check{\A}}',h^{\hat{\check{V}}}\Big)^{dtdl}$, where as before, for $w\in \Omega^{\bullet}(\hchS),$ $w^{dtdl}:=dt\wedge dl\wedge i_{\frac{\p}{\p l}}i_{\frac{\p}{\p t}}w$ is the $dt\wedge dl$-component of $w$. By \cite[Proposition 2.13]{bismut1995flat}, $$h\Big(\hat{\check{\A}}',h^{\hat{\check{V}}}\Big)=\frac{\big(\chi'(\bfH)+O(t^{-\half})\big)dt}{2t} \mbox{ as } t\to\infty\mbox{ and } h\Big(\hat{\check{\A}}',h^{\hat{\check{V}}}\Big)=\frac{\big(\chi'(V)+O(t)\big)dt}{2t}\mbox{ as }t\to0,$$ so $\B\Big(\A',h^{\hat{\check{V}}}\Big)$ is well defined and smooth on $S'$. 

Similarly, we can extend $\B\Big(\A',h^{\hat{\check{V}}}\Big)$ to a continuous form on $S$ and extend $ \overline{\B}\Big(\A',h^{\hat{\check{V}}}\Big)$  to a smooth form on $S$.


Let $\check{S}':=S'\times[0,1]$ with coordinates $l$ in $[0,1]$ component.
Let $\check{V}\to \check{S}'$ be the pullback bundle of $V\to S'$ with respect to the canonical projection $\chS'\to S'$.

Recall that $\check{S}:=S\times[0,1]$ with coordinates $l$ in $[0,1]$ component.
Let $\check{\bfH}\to \check{S}$ be the pullback bundle of $\bfH\to S$ with respect to the canonical projection $\chS\to S$.

Let $h_{L^2}^{\check{V}}$ be the metric on $\check{\bfH}$ induced from $h^{\check{V}}|_{S\times\{l\}}:=\big(e^{-lT\bbf^V}\big)^*h^V$ by Hodge theory and $\J$.

We have the following anomaly formula, in the spirit of \cite[Theorem 2.24]{bismut1995flat}:
\begin{thm}\label{thm57}
	On $S$,
	$$
	\overline{\T}(\A',h^V)-\overline{\T}\Big(\A',\big(e^{-T\bbf^V}\big)^*h^V\Big)=-\overline{\wch}\left(\nabla^{\bfH}, h_{L^2}^{\check{V}}\right)+d^S\overline{\B}\Big(\A',h^{\hat{\check{V}}}\Big).
	$$
	Here the bar convention is applied.
	
	As a result, in $ \Omega^{\bullet}(S)/d^S \Omega^{\bullet}(S)$,
	\[\btiT(\A',h^V,h^{\bfH},h^{\bfH_{\Cb^m}})=\btiT\Big(\A',\big(e^{-T\bbf^V}\big)^*h^V,h^{\bfH},h^{\bfH_{\Cb^m}}\Big)\]
	for any metrics $h^\bfH$ on $\bfH\to S$ and $h^{\bfH_{\Cb^m}}$ on $\bfH_{\Cb^m}\to S.$
\end{thm}

\begin{proof}
	{As in Theorem \ref{thm109} (see also \cite[Theorem 2.24]{bismut1995flat})}, we have
	$$
	\overline{\T}(\A',h^V)-\overline{\T}\Big(\A',\big(e^{-T\bbf^V}\big)^*h^V\Big)=\overline{\wch}\left(\A', h^{\check{V}}\right)-\overline{\wch}\left(\nabla^{\bfH}, h_{L^2}^{\check{V}}\right)+d^S\overline{\B}\Big(\A',h^{\hat{\check{V}}}\Big)
	$$
	on $ S' $. 
	
	Note that $ \bbf^V $ commutes with $ \nabla^V-\nabla^{V,*} $ and $ [\nabla^V,\bbf^V] $, proceeding as in \cite[Proposition 2.13]{bismut1995flat},  $ {\wch}\left(\A', h^{\check{V}}\right) $ has only degree $0$ term, so $ \overline{\wch}\left(\A', h^{\check{V}}\right)=0 $ by our bar convention.
	
	As a result, 
	$$
	\overline{\T}(\A',h^V)-\overline{\T}(\A',(e^{-T\bbf^V})^*h^V)=-\overline{\wch}\left(\nabla^{\bfH}, h_{L^2}^{\check{V}}\right)+d^S\overline{\B}\Big(\A',h^{\hat{\check{V}}}\Big)
	$$
	on $ S' $.
	
	Since $ \overline{\T}(\A',h^V) $, $ \overline{\T}\Big(\A',\big(e^{-T\bbf^V}\big)^*h^V\Big) $, $ \overline{\wch}\left(\nabla^{\bfH},h_{L^2}^{\check{V}}\right) $ and $ \overline{\B}\Big(\A',h^{\hat{\check{V}}}\Big) $ can be extended smoothly to $ S $, we have proved the Theorem. 
\end{proof}

\section{Intermediate Results}\label{ineresult}

In this section, we will present some intermediate results. Assuming these intermediate results hold, we will prove our main result. The remainder of the paper will be dedicated to proving these intermediate results.

In this section, we use the notations of \cref{combitor}.
Let $f:M\to\R$ be a fiberwise generalized Morse function.  Let $(\M,\bbf)$ a double suspension of $(M,f)$ for a large enough even number $N$ (see Definition \ref{doublesuspension}). By slightly perturbing $\bbf$ in the $Z$ component within $\V$ if needed, we assume $\bbf$ provides a stratification of $S$ as described in \cref{secfff}. We then use the modified version of $\bbf$, still denoted by $\bbf$, as in \cref{diskremov}. By Remark \ref{torsion-indep-perturbation}, the torsions under consideration in this paper do not depend of the particular deformations described above.

Let $\{\bbf_{1,T,R},\bbf_{2,T,R,R'}\}$ be the associated non-Morse functions in Definition \ref{assononmor} associated with $\bbf$ for $T\geq1$. 

{Recall that we have assumed in the end of \cref{proofthm54} has at most one birth-death point in each fiber, i.e., the stratification  $S_k$ of Section \ref{sec-fGMF} satisfies $S_k=\emptyset$ for $k\geq 2$. In this case, $\Omega_1=\ppi(\V).$ The general case will be treated in \cref{ineresult1}. Thus, the only non-Morse function in Definition \ref{assononmor} we need to  consider is $\bbf_{1,T,R}$. We let $\bbf_{T,R}:=\bbf_{1,T,R}$\index{fTR@$\bbf_{T,R}$} for short.


Now let $\cE^k=\Omega^k(\bZ,\cF|_{\bZ})$\index{E@$\cE$}, which is the bundle over $S$ such that for $\s\in S$, $\cE^k=\Omega^k(\bZ_\s,\cF_\s)$, and set $\cE=\oplus_{k=0}^{\operatorname{dim} \bZ} \cE^k$\index{E@$\cE$}. As in \eqref{def-nablabfE}, we have connection $\nabla^{\cE}$ on $\cE.$


Let $d^\M: \Omega^{\bullet}(\M,\cF)\to \Omega^{\bullet+1}(\M,\cF)$ be induced by $\nabla^\cF.$
As before, $d^\M$ is a flat superconnection of total degree 1 on $\cE$. Let $\cC_t^{\prime}:=t^{N^\bZ / 2} d^\M t^{-N^\bZ / 2}.$


\def\mbB{{\mathbb{B}}}
\def\mcB{{\mathcal{B}}}
\def\mcD{{\mathcal{D}}}
\def\mcU{{\mathcal{U}}}
\def\mcL{{\mathcal{L}}}
\def\mcR{{\mathcal{R}}}
\def\mcF{{\mathcal{F}}}


We set $h_{T,R}^\cF\index{hFTR@$h_{T,R}^\cF$}:=e^{-2\bbf_{T,R}}h^{\cF}$ and $h^\cF_{T}\index{hFT@$h_{T}^\cF$}:=e^{-2T\bbf}h^{\cF}$. Let $g^{T\bZ}$
be a metric on $T\bZ$.  We assume that Condition \ref{5a} holds. Thus we can also use Observation \ref{5c}. 




Let $h^{\cE}_{T,R}\index{hETR@$h_{T,R}^\cE$}$ be the metric on $\cE$ induced by $T^H\M$, $g^{T\bZ}$ and $h_{T,R}^{\cF}$, and $\cC''_{t,T,R}$ be the adjoint connection of $\Cc'_t$ with respect to $h^\cE_{T,R}$. Set 
\begin{equation}
	\label{def-DtTR}
	\cD_{t,T,R}\index{DtTR@$\cD_{t,T,R}$}:=\frac{1}{2}\left(\Cc''_{t,T,R}-\Cc'_{t}\right).
\end{equation}

\def\tiM{\widetilde{\M}}
\def\tiS{\widetilde{S}}
\def\tiE{\widetilde{\cE}}


Let $ h^H $\index{hH@$h^H$} (resp. $ h^{H_{\Cb^m}} $\index{hHCm@$h^{H_{\Cb^m}}$})  be a metric on $ H$  (resp. $ H_{\Cb^m}$).   As  in \cref{suspen}, let $ \bfH$ (resp. $\bfH_{\Cb^m}$) be the fiberwise  $ L^2 $-cohomology of $\cF$ (resp. $\Cb^m$)  with respect to the $ L^2 $ metric induced by $ g^{T\bZ} $ and $ e^{-2\bbf}h^{\cF} $ (resp. $ e^{-2\bbf}h^{\Cb^m}$), and let $h^{\bfH} $ (resp. $ h^{\bfH_{\Cb^m}} $) be the metric on  $ \bfH$ (resp. $\bfH_{\Cb^m}$) induced by $ h^H $\index{hH@$h^H$} (resp. $ h^{H_{\Cb^m}}$) and \eqref{isom-H_bfH}. By Theorem \ref{thm52} applied to $\tilde{\bbf}:=\bbf_{T=1,R}$, 
  to show Theorem \ref{main1},  it suffices to show that in $ \Omega^{\bullet}(S)/d^S \Omega^{\bullet}(S)$,
\be\label{eq47}\btiT(T^H\M,g^{T\bZ},h_{T,R}^\cF,h^\bfH,h^{\bfH_{\Cb^m}})-\btiT(\A',h^V,h^\bfH,h^{\bfH_{\Cb^m}})=0,\ee
for $R>0$ and $T>0$, noting that the LHS is actually independent of $R,T>0$ (module an exact form on $S$).

We will set $R_0 > 0$\index{R0@$R_0$} to be sufficiently large, the precise meaning of which being determined in Definition \ref{def-of-R0}. Let $h^{\bfH,\prime}$ and $h^{\bfH_{\Cb^m},\prime}$ be another metrics on $\bfH$ and $\bfH_{\Cb^m}$ respectively, then in $ \Omega^{\bullet}(S)/d^S \Omega^{\bullet}(S)$, 
\begin{multline}
\label{torsion-with-prime-metric-and-nonprime-metric}
\btiT(T^H\M,g^{T\bZ},h_{T,R}^\cF,h^\bfH,h^{\bfH_{\Cb^m}})-\btiT(\A',h^V,h^\bfH,h^{\bfH_{\Cb^m}})\\
=\btiT(T^H\M,g^{T\bZ},h_{T,R}^\cF,h^{\bfH,\prime},h^{\bfH_{\Cb^m},\prime})-\btiT(\A',h^V,h^{\bfH,\prime},h^{\bfH_{\Cb^m},\prime}).
\end{multline}
Moreover, note that $\tiT(T^H\M,g^{T\bZ},h_{T,R}^\cF,h^{\bfH}_{T,R,L^2})=\T(T^H\M,g^{T\bZ},h_{T,R}^\cF)$, where $h^{\bfH}_{T,R,L^2}\index{hHTRL2@$h^{\bfH}_{T,R,L^2}$}$ is the metric on $\bfH$ induced by Hodge theory, $h^{\cE}_{T,R}$ and the isomorphism $J_T=\K_1\J_T\K_T^{-1}\colon \bfH_{T}\overset{\sim}{\rightarrow}\bfH$ defined in \cref{combitor}. Hence, to prove \eqref{eq47}, it suffices to show that in $ \Omega^{\bullet}(S)/d^S \Omega^{\bullet}(S)$, for $R=R_0$,
\begin{align}\begin{split}\label{eqsplit481}
		&\ \ \ \ \overline{\T}(T^H\M,g^{T\bZ},h_{T,R}^\cF)=\btiT(\A',h^V,h^{\bfH}_{T,R,L^2},h^{\bfH_{\Cb^m}}_{T,R,L^2}).
\end{split}\end{align}
Here $h^{\bfH_{\Cb^m}}_{T,R,L^2}$ is the corresponding Hodge metric on $\bfH_{\Cb^m}.$ 

Recall that $\bbf^V\in\End(V)$  is defined in \eqref{def-bbf^V}.
By Theorem \ref{thm57}, we can also replace $ h^V $ with $ (e^{-T\bbf^V})^*h^V $. That is, \eqref{eqsplit481} is equivalent to showing that in $ \Omega^{\bullet}(S)/d^S \Omega^{\bullet}(S)$, for $ R = R_0 $, 

\begin{align}
	\begin{split}\label{eqsplit48}
		\overline{\T}(T^H\M,g^{T\bZ},h_{T,R}^\cF) - \btiT\Big(\A',\big(e^{-T\bbf^V}\big)^*h^V,h^{\bfH}_{T,R,L^2},h^{\bfH_{\Cb^m}}_{T,R,L^2}\Big)=0.
	\end{split}
\end{align}
Again, note that the LHS is actually independent of $R,T>0$ (module an exact form on $S$).

For any $\tau > 0$, let $\mathcal{T}_\tau^{\mL}(T^H\M, g^{T\bZ}, h_{T,R}^\cF)$ \index{TLtau@$\mathcal{T}_\tau^{\mL}$, $\tiT_\tau^{\mL}$, $\overline{\T}_\tau^\mL$ and $\btiT_\tau^\mL$} denotes the torsion form obtained by replacing the integral $\int_0^\infty$ in the definition of the torsion form with $\int_\tau^\infty$. Similarly, we introduce the notations $\tiT_\tau^{\mL}$, $\overline{\T}_\tau^\mL$ and $\btiT_\tau^\mL$.

For an $(T,R)$-\textbf{dependent} differential form $w\in \Omega^{\bullet}(S)$, by saying $w=O(\epsilon)$, we mean that $|w|_{g^{TS}}\leq C\epsilon$ for some $(T,R)$-\textbf{independent} constant $C;$ by saying $w=O(\epsilon)$ in $ \Omega^{\bullet}(S)/d^S \Omega^{\bullet}(S)$, we mean that there exists a smooth family of exact forms $dw'_{{T,R}}$, s.t. $|w-dw'_{{T,R}}|_{g^{TS}}\leq C\epsilon$ for some $(T,R)$-\textbf{independent} constant $C.$ By saying $\lim_{T\to\infty}w=0$, we mean $\lim_{T\to\infty}|w|_{g^{TS}}=0.$ Here $g^{TS}$ is some metric on $TS\to S$, and $|\cdot|_{g^{TS}}$ is the pointwise norm on $ \Omega^{\bullet}(S)$ induced by $g^{TS}.$

We will prove that:
\begin{prop}\label{prop62}
	For any $\epsilon>0$, there exists $\tau=\tau(\epsilon,R_0)>0$, s.t. for  $R\in[0,R_0]$, 
	\[\overline\T(T^H\M,g^{T\bZ},h_{T,R}^\cF)=\overline\T_\tau^{\mL}(T^H\M,g^{T\bZ},h_{T,R}^\cF)+O(\epsilon);\]
	\[\btiT\Big(\A',\big(e^{-T\bbf^V}\big)^*h^V,h^{\bfH}_{T,R,L^2},h^{\bfH_{\Cb^m}}_{T,R,L^2}\Big)=\btiT_\tau^\mL\Big(\A',\big(e^{-T\bbf^V}\big)^*h^V,h^{\bfH}_{T,R,L^2},h^{\bfH_{\Cb^m}}_{T,R,L^2}\Big)+O(\epsilon).\]
\end{prop}

Recall that $S':=\{\s\in S:f_{\s} \mbox{ is Morse}\}$. Let $\hS'=S'\times(0,+\infty)$ and $\hV=V\times(0,+\infty)\to \hS'$\index{Vhat@$\hV$}. 
In \cref{sec11} {(see \eqref{hvt} and the proof of Proposition \ref{prop63} on $\Omega_1''$ in Section \ref{sec122})}, we will introduce{, for $T$ large enough,} a generalized metric $\fh_T^V$\index{hVT@$\fh^V_T$ and $\fh^{\hV}_T$} on $V\to S'$ and a generalized metric $\fh^{\hV}_T$ on $\hV\to \hS'$ satisfying Condition \ref{assum21} with $\hat{\I}_1^*\fh^{\hV}_T=\fh_T^V$ (where $\hat{\I}_t:s\in S'\to(s,t)\in \hS'$). Let  $\T(\A',\fh^\hV_T)$ be the associated torsion form, and let  $\btiT(\A',\fh^{\hV}_T,h^{\bfH}_{T,R,L^2},h^{\bfH_{\Cb^m}}_{T,R,L^2})$\index{TAfVhathHTRtildebar@$\btiT(\A',\fh^{\hV}_T,h^{\bfH}_{T,R,L^2},h^{\bfH_{\Cb^m}}_{T,R,L^2})$} be its modified version according to Definition \ref{assocomtor} and our bar convention (c.f. \cref{sec21}). 

\begin{prop}\label{prop63}
	
	The form $\btiT_\tau^\mL(\A',\fh^{\hV}_T,h^{\bfH}_{T,R,L^2},h^{\bfH_{\Cb^m}}_{T,R,L^2})$ can be extend to a smooth form on $S$. Moreover, there exists $T_0=T_0(\tau,\epsilon)$, such that if $T>T_0$, on $ \Omega^{\bullet}(S)/d^S \Omega^{\bullet}(S)$, 
	\[\btiT_\tau^\mL(\A',\fh^\hV_T,h^{\bfH}_{T,R,L^2},h^{\bfH_{\Cb^m}}_{T,R,L^2})=\btiT_\tau^\mL\Big(\A',\big(e^{-T\bbf^V}\big)^*h^V,h^{\bfH}_{T,R,L^2},h^{\bfH_{\Cb^m}}_{T,R,L^2}\Big)+O(\epsilon).\]
\end{prop}

As a result, by Proposition \ref{prop62} and Proposition \ref{prop63}, to establish \eqref{eqsplit48}, we just need to show that:
\begin{thm}\label{prop64}
	For $R=R_0$, we have
	\be\label{eq49}
	\lim_{T\to\infty}\left(\overline{\T}_\tau^{\mL}\big(T^H\M,g^{T\bZ},h_{T,R}^\cF\big)-\btiT_\tau^\mL\big(\A',\fh_T^{\hV},h^{\bfH}_{T,R,L^2},h^{\bfH_{\Cb^m}}_{T,R,L^2}\big)\right)=0.
	\ee
\end{thm}

	
	

To prove Theorem \ref{prop64}, we need the following three theorems.

Recall that we defined the covering $ \{\Omega_k''\}_{k=0}^\fk $ of $S$ in Definition \ref{def-covering-Omegak}, and that here $\fk=1$. 

First, on $\Omega_0''$, although $\bbf_{T=1,R}$ may have very singular critical points, the original function $\bbf$ is Morse. We have the following theorem, which compares torsion forms associated with $\bbf_{T,R}$ and $T\bbf$, respectively:
\begin{thm}\label{int1}
	On $\Omega_0''$, for $R=R_0$, we have
	\[\lim_{T\to\infty}\left(\overline{\T}_\tau^{\mL}\big(T^H\M,g^{T\bZ},h_{T,R}^\cF\big)-\overline{\T}_\tau^{\mL}\big(T^H\M,g^{T\bZ},h_{T}^\cF\big)\right)=0.\]	
\end{thm}

On $\Omega_1''$, both the original function $\bbf$ and $\bbf_{T=1,R}$ are generalized Morse functions, but $\bbf_{T=1,R}|_{\M^-}$ is fiberwise Morse. We have the following theorem, which compares the torsion forms associated with $\bbf_{T,R}|_{\M^-}$ on $\M^-$ and $\bbf_{T,R}$ on $\M$, respectively.

\begin{thm}\label{int20}
	On $\Omega_1''$, for $R=R_0$, we have
	\begin{align*}
		\lim_{T\to\infty}\left(	\mathcal{\overline{T}}_{\tau}^\mL\left(T^H \M, g^{T \bZ}, h^{\cF}_{T,R}\right)-\mathcal{\overline{T}}_\tau^{\mL}\left(T^H\M^-, g^{T \bZ}|_{\bZ^-}, h^\cF_{T,R}\right)\right)=0
	\end{align*}
	
	Here, $\T_{\tau}^\mL\left(T^H\M^-, g^{T \bZ}|_{\bZ^-}, h^{\cF}_{T,R}\right)$ is the torsion form for the fibration $\M^- \to \Omega_1''$ with  boundary conditions. See \cref{sec53} for the description of $\M^-.$
\end{thm}
\begin{rem}\label{remthminr20}
	In fact, the theorem above resembles \cite[Theorem 1.1]{Yanforms} but extends to the level of differential forms. This extension is facilitated by the fact that $ \bZ-\bZ^- $ is a ball, which is topologically trivial, allowing for a more refined construction than in \cite{Yanforms}, as detailed in \cref{proofint20}.
\end{rem}

The following theorem is essentially \cite[Theorem 9.8]{bismut2001families}.

\begin{thm}\label{int2}
	For $R=R_0,$ on $\Omega_0''$, we have
	\[\lim_{T\to\infty}\left(\overline{\T}_\tau^{\mL}\big(T^H\M,g^{T\bZ},h_{T}^\cF\big)-\btiT_\tau^\mL\big(\A',\fh_T^{\hat{V}},h^{\bfH}_{T,R,L^2},h^{\bfH_{\Cb^m}}_{T,R,L^2}\big)\right)=0;\]	
	on $\Omega_1''$,   \[\lim_{T\to\infty}\left(\overline{\T}_\tau^{\mL}\big(T^H\M^-,g^{T\bZ}|_{\bZ^-},h_{T,R}^\cF\big)-\btiT_\tau^\mL\big(\A',\fh_T^{\hat{V}},h^{\bfH}_{T,R,L^2},h^{\bfH_{\Cb^m}}_{T,R,L^2}\big)\right)=0.\]	
\end{thm}

One might expect that for the second equality in Theorem \ref{int2}, the  term on the left-hand side should be the combinatorial torsion for one ball removed. However, in view of Theorem \ref{int3},  a second term of this form is not surprising.

Theorem \ref{prop64} then follows from Theorem \ref{int1}-Theorem \ref{int2} easily (recall that we have assumed  in the end of \cref{proofthm54} that for now $S_k,k\geq2$ is empty, the general case will be dealt with in \cref{ineresult1}).

\def\thmint{0}
\if\thmint1
\begin{thm}\label{int30}
	On $\Omega_1''$,  for any $\epsilon,T_0,R_0>0$, there exists $t_1=t_1(\epsilon)>0$ small enough, $t_2=t_2(\epsilon,T_0,R_0)$ large enough, s.t.
	\begin{align*}
		&\ \ \ \ \mathcal{\overline{T}}\left(T^H M, g^{T Z}, h^F\right)(1,0,\infty,T_0,R_0)-\mathcal{\overline{T}}\left(T^H M, g^{T Z}, h^F\right)(1,0,\infty,0,R_0)\\
		&=\int_{\{t_2\}\times[0,T_0]\times\{R_0\}} \overline{\gamma}_1-\int_{\{t_1\}\times[0,T_0]\times\{R_0\}} \overline{\gamma}_1
		-d\mcU_1(t_1,t_2,T_0,R_0)
		+O(\epsilon).
	\end{align*}
	
\end{thm}

\begin{thm}\label{int4}
	On $\Omega_1''$, for any $\epsilon,T_0,R_0>0$, there exists $t_1=t_1(\epsilon)>0$ small enough, $t_2=t_2(\epsilon,T_0,R_0)$ large enough, s.t.
	\begin{align*}
		&\ \ \ \ \mathcal{\overline{T}}\left(T^H M, g^{T Z}, h^F\right)(1,0,\infty,T_0,R_0)-\mathcal{\overline{T}}\left(T^H M, g^{T Z}, h^F\right)(1,0,\infty,T_0,0)\\
		&=\int_{\{t_2\}\times\{T_0\}\times[0,R_0]} \overline{\gamma}_1-\int_{\{t_1\}\times\{T_0\}\times[0,R_0]} \overline{\gamma}_1
		-d\mcR_1(t_1,t_2,T_0,R_0)
		+O(\epsilon).
	\end{align*}
	
	On $\Omega_1''$, for any $\epsilon,R_0>0$, there exists $t_1=t_1(\epsilon)>0$ small enough, $t_2=t_2(\epsilon,R_0)$ large enough, s.t.
	\begin{align*}
		&\ \ \ \ \mathcal{\overline{T}}\left(T^H M, g^{T Z}, h^F\right)(1,0,\infty,0,R_0)-\mathcal{\overline{T}}\left(T^H M, g^{T Z}, h^F\right)(1,0,\infty,0,0)\\
		&=\int_{\{t_2\}\times\{0\}\times[0,R_0]} \overline{\gamma}_1-\int_{\{t_1\}\times\{0\}\times[0,R_0]} \overline{\gamma}_1
		-d\mcL_1(t_1,t_2,T_0,R_0)
		+O(\epsilon).
	\end{align*}
	
\end{thm}

\begin{thm}\label{int5}
	For any fix $T_0>0$,
	\[\lim_{t\to0^+}\int_{\{t\}\times[0,T_0]\times\{0\}} \overline{\gamma}_a=0\]
	in $\Omega_0''\cup\Omega_1''.$
\end{thm}

It follows from Theorem \ref{int1}, Theorem \ref{int2} and Theorem \ref{int5} that
\begin{prop}\label{int6}
	On $\Omega_0''$, for any $\epsilon>0$, there exists $t_1=t_1(\epsilon)>0$ small enough, $t_2=t_2(\epsilon),T_0=T_0(\epsilon)>0$ large enough,
	\begin{align*}&\ \ \ \ \mathcal{\overline{T}}\left(T^H M, g^{T Z}, h^F\right)-\mathcal{\overline{T}}\left(\A',h^V\right)\\
		&=\int_{\{t_2\}\times[0,T_0]\times\{0\}} \overline{\gamma}_1+d\mcD_1(t_1,t_2,T_0,R_0)+O(\epsilon).\end{align*}
\end{prop}

It follows from Proposition \ref{contours}, Theorem \ref{int20}, Theorem \ref{int2}, Theorem \ref{int3}, Theorem \ref{int30}, Theorem \ref{int4}, and Theorem \ref{int5} that
\begin{prop}\label{int7}
	On $\Omega_1''$, for any $\epsilon>0$, there exists $t_1=t_1(\epsilon)>0$ small enough, $t_2=t_2(\epsilon),T_0=T_0(\epsilon),R_0=R_0(\epsilon)>0$ large enough,
	\begin{align*}&\ \ \ \ \mathcal{\overline{T}}\left(T^H M, g^{T Z}, h^F\right)-\mathcal{\overline{T}}\left(\A',h^V\right)\\
		&=\int_{\{t_2\}\times[0,T_0]\times\{0\}} \overline{\gamma}_1-d\mcB_1(t_1,t_2,T_0,R_0)-d\mcF_1(t_1,t_2,T_0,R_0)\\
		&-d\mcU_1(t_1,t_2,T_0,R_0)-d\mcL_1(t_1,t_2,T_0,R_0)-d\mcR_1(t_1,t_2,T_0,R_0)+O(\epsilon).\end{align*}
\end{prop}

\def\con{1}
\ifx\con\undefined
undef
\else 
\if\con1

It follows from the definition of $ \overline{\gamma}_a$, the fact that $F\to M$ is unitarily flat, and \cite[Theorem 9.6]{bismut2001families} that
\begin{lem}
	On $S$, for any $\epsilon>0$, there exists $t_1>0$ small enough, $t_2,T_0>0$ large enough 
	\[\lim_{t\to\infty}\int_{\{t\}\times[0,T_0]\times\{0\}}{ \overline\gamma}_a\]
	exists. This limit represents $\overline\wch\left(\nabla^H,h^{H}_{L^2,T_0},h^{H}_{L^2}\right)$. 
	
	\def\ch{{{h}}}
	
\end{lem}
\fi\fi\fi


As a result, we need to show Proposition \ref{prop62}, Proposition \ref{prop63}, Theorem \ref{int1}, Theorem \ref{int20} and Theorem \ref{int2}. {Proposition \ref{prop62} will be proved in \cref{sec7}, Proposition \ref{prop63} will be proved in \cref{sec121}, Theorem \ref{int1} will be proved in \cref{sec10}, Theorem \ref{int20} will be proved in \cref{proofint20} and Theorem \ref{int2} will be proved in \cref{sec11}.}

\def\thmint{0}
\if\thmint1
\begin{thm}\label{int30}
	On $\Omega_1''$,  for any $\epsilon,T_0,R_0>0$, there exists $t_1=t_1(\epsilon)>0$ small enough, $t_2=t_2(\epsilon,T_0,R_0)$ large enough, s.t.
	\begin{align*}
		&\ \ \ \ \mathcal{\overline{T}}\left(T^H M, g^{T Z}, h^F\right)(1,0,\infty,T_0,R_0)-\mathcal{\overline{T}}\left(T^H M, g^{T Z}, h^F\right)(1,0,\infty,0,R_0)\\
		&=\int_{\{t_2\}\times[0,T_0]\times\{R_0\}} \overline{\gamma}_1-\int_{\{t_1\}\times[0,T_0]\times\{R_0\}} \overline{\gamma}_1
		-d\mcU_1(t_1,t_2,T_0,R_0)
		+O(\epsilon).
	\end{align*}
	
\end{thm}

\begin{thm}\label{int4}
	On $\Omega_1''$, for any $\epsilon,T_0,R_0>0$, there exists $t_1=t_1(\epsilon)>0$ small enough, $t_2=t_2(\epsilon,T_0,R_0)$ large enough, s.t.
	\begin{align*}
		&\ \ \ \ \mathcal{\overline{T}}\left(T^H M, g^{T Z}, h^F\right)(1,0,\infty,T_0,R_0)-\mathcal{\overline{T}}\left(T^H M, g^{T Z}, h^F\right)(1,0,\infty,T_0,0)\\
		&=\int_{\{t_2\}\times\{T_0\}\times[0,R_0]} \overline{\gamma}_1-\int_{\{t_1\}\times\{T_0\}\times[0,R_0]} \overline{\gamma}_1
		-d\mcR_1(t_1,t_2,T_0,R_0)
		+O(\epsilon).
	\end{align*}
	
	On $\Omega_1''$, for any $\epsilon,R_0>0$, there exists $t_1=t_1(\epsilon)>0$ small enough, $t_2=t_2(\epsilon,R_0)$ large enough, s.t.
	\begin{align*}
		&\ \ \ \ \mathcal{\overline{T}}\left(T^H M, g^{T Z}, h^F\right)(1,0,\infty,0,R_0)-\mathcal{\overline{T}}\left(T^H M, g^{T Z}, h^F\right)(1,0,\infty,0,0)\\
		&=\int_{\{t_2\}\times\{0\}\times[0,R_0]} \overline{\gamma}_1-\int_{\{t_1\}\times\{0\}\times[0,R_0]} \overline{\gamma}_1
		-d\mcL_1(t_1,t_2,T_0,R_0)
		+O(\epsilon).
	\end{align*}
	
\end{thm}

\begin{thm}\label{int5}
	For any fix $T_0>0$,
	\[\lim_{t\to0^+}\int_{\{t\}\times[0,T_0]\times\{0\}} \overline{\gamma}_a=0\]
	in $\Omega_0''\cup\Omega_1''.$
\end{thm}

It follows from Theorem \ref{int1}, Theorem \ref{int2} and Theorem \ref{int5} that
\begin{prop}\label{int6}
	On $\Omega_0''$, for any $\epsilon>0$, there exists $t_1=t_1(\epsilon)>0$ small enough, $t_2=t_2(\epsilon),T_0=T_0(\epsilon)>0$ large enough,
	\begin{align*}&\ \ \ \ \mathcal{\overline{T}}\left(T^H M, g^{T Z}, h^F\right)-\mathcal{\overline{T}}\left(\A',h^V\right)\\
		&=\int_{\{t_2\}\times[0,T_0]\times\{0\}} \overline{\gamma}_1+d\mcD_1(t_1,t_2,T_0,R_0)+O(\epsilon).\end{align*}
\end{prop}

It follows from Proposition \ref{contours}, Theorem \ref{int20}, Theorem \ref{int2}, Theorem \ref{int3}, Theorem \ref{int30}, Theorem \ref{int4}, and Theorem \ref{int5} that
\begin{prop}\label{int7}
	On $\Omega_1''$, for any $\epsilon>0$, there exists $t_1=t_1(\epsilon)>0$ small enough, $t_2=t_2(\epsilon),T_0=T_0(\epsilon),R_0=R_0(\epsilon)>0$ large enough,
	\begin{align*}&\ \ \ \ \mathcal{\overline{T}}\left(T^H M, g^{T Z}, h^F\right)-\mathcal{\overline{T}}\left(\A',h^V\right)\\
		&=\int_{\{t_2\}\times[0,T_0]\times\{0\}} \overline{\gamma}_1-d\mcB_1(t_1,t_2,T_0,R_0)-d\mcF_1(t_1,t_2,T_0,R_0)\\
		&-d\mcU_1(t_1,t_2,T_0,R_0)-d\mcL_1(t_1,t_2,T_0,R_0)-d\mcR_1(t_1,t_2,T_0,R_0)+O(\epsilon).\end{align*}
\end{prop}

\def\con{1}
\ifx\con\undefined
undef
\else 
\if\con1

It follows from the definition of $ \overline{\gamma}_a$, the fact that $F\to M$ is unitarily flat, and \cite[Theorem 9.6]{bismut2001families} that
\begin{lem}
	On $S$, for any $\epsilon>0$, there exists $t_1>0$ small enough, $t_2,T_0>0$ large enough 
	\[\lim_{t\to\infty}\int_{\{t\}\times[0,T_0]\times\{0\}}{ \overline\gamma}_a\]
	exists. This limit represents $\overline\wch\left(\nabla^H,h^{H}_{L^2,T_0},h^{H}_{L^2}\right)$. 
	
	\def\ch{{{h}}}
	
\end{lem}
\fi\fi\fi

\newpage
\part{Agmon estimates and eigenforms of the deformed Laplacian}\label{PartII}

This part consists of \cref{setting-partII}--\cref{sec9}. It aims to provide the necessary estimates to prove Theorems \ref{int1} to \ref{int2}. This part closely follows some parts of the papers \cite{Yanforms,Yantorsions} of J.Y.. First, we will establish a trace formula-like estimate for eigenforms (Lemma \ref{limit1}). Using Lemma \ref{limit1}, we then obtain Lemma \ref{lem87}, which will be crucial for estimating the lower bound of non-zero eigenvalues in \cref{sec9}. Finally, we will introduce the Agmon estimate, with more sophisticated applications discussed in \cref{sec9}. 

\textbf{To avoid introducing many new notations, the reader is advised that the notations in this part are independent of those used in the rest of the paper: they carry similar—but not necessarily identical—meaning. A separate index of notations is provided at the end for reference.}

\section{Setting of Part \ref{PartII}}\label{setting-partII}

Let $(Z,g^{TZ})$\index[p2]{Z@$Z$} be a complete noncompact Riemannian manifold with {bounded} geometry, i.e., 
\begin{enumerate}[(1)]
	\item  The injectivity radius of $Z$ is positive.
	\item {For any $m\in \mathbb{N}$, there is $C_m>0$ such that} $\left|\nabla^m R\right| \leq C_m$, where $\nabla^m R$ is the $m$-th covariant derivative of the curvature tensor.
\end{enumerate}

Fix $p_0 \in Z$ and let $d^{g^{TZ}}$ be the distance function induced by $g^{TZ}$. In the sequel, by $p \rightarrow \infty$ we mean that $d^{g^{TZ}}\left(p, p_0\right) \rightarrow \infty$.

Let $f:Z\to \R$\index[p2]{f@$f$} be a smooth function, we say $(Z,g^{TZ},f)$ is strongly tame\index[p2]{Strongly tame} \cite[Definition 1.1]{DY2020cohomology}(see also \cite{DY2020index,fan2011schr,fan2021torsion}), if 
\begin{enumerate}[(1)]
	\item if $(Z, g^{TZ})$ has bounded geometry;
	\item 	$$
	\limsup _{p \rightarrow \infty} \frac{\left|\Hess f\right|(p)}{|\nabla f|^2(p)}=0;
	$$
	\item
	$$
	\liminf_{p \rightarrow \infty}|\nabla f| \rightarrow \infty.
	$$
	Here $\nabla f$ and $\Hess f$ are the gradient and Hessian of $f$, respectively.
\end{enumerate}

Let $F\to Z$\index[p2]{F@$F$} be a flat vector bundle with the flat connection $\nabla^F$ and a Hermitian metric $h^F$ on $F$. Here we don't require that $\nabla^F h^F=0.$

Let $\nabla^{F,*}$ be the adjoint connection of $\nabla^F$ with respect to $h^F$, and
\def\Mc{\mathcal{C}}
let $\o(\nabla^F,h^F):=\nabla^{F,*}-\nabla^F\in \Omega^1(Z,\End(F)).$ We will also abbreviate $\o(\nabla^F,h^F)$\index[p2]{omegaF@$\o(\nabla^F,h^F)$ (abbr. $\o$)}  as $\o$ for simplicity. We assume that 

$$|\o|\leq \Mc_0 \text{ for some constant } \Mc_0,$$ 
where $|\cdot|$ is the pointwise norm induced by $g^{TZ}$ and $h^F.$ 

Let $Y\subset Z $\index[p2]{Y@$Y$} be a hypersurface, such that
\begin{enumerate}[(1)]
	\item   $Y$ is compact;
	\item $Y$ cuts $Z$ into two disconnected pieces $Z_1$ and $Z_2$. We regard $Z_1$ and $Z_2$\index[p2]{Z1@$Z_1$ and $Z_2$} as manifolds with boundary $Y$ instead of open manifolds, i.e. $Z_1\cap Z_2=Y$.
\end{enumerate}
Then there exists a unit normal vector field $\nu$ on $Y$. 
For each point $y\in Y$, there exists a unique geodesic $\gamma_y:\R\to Z$, s.t. $\gamma_y(0)=y$ and $\gamma_y'(0)=\nu(y).$

The exponential map $\exp_{Y}:\R\times Y\to Z$ is defined as $\exp_Y(s,y)=\gamma_y(s)$. Also, assume that $\gamma_y(s)\in Z_1$ if $s<0.$ It is clear that there exists $r>0$, s.t. the image
$\exp_Y( (-4r,4r)\times Y)$ is diffeomorphic to a tubular neighborhood of $Y.$ For $s\in(-4r,4r)$, set $Y^s:=\exp_Y(s,Y)$ and let $H(s)$ be the mean curvature on $Y^s$. For simplicity, we will sometimes identify $Y^s$ with $\{s\}\times Y$ for $s\in(-4r,4r).$

We assume that
\begin{assum} \label{product-type} The metrics $ g^{TZ} $ and $ h^F $ are of product form on $ (-4r,4r) \times Y $ in the sense of \cite[(2.1)-(2.3)]{MR3030691}. Moreover,
    \be\label{8e} |\nabla f|^2(p)\geq e_f>0\ee for some constant $e_f$\index[p2]{ef@$e_f$} whenever $p\in (-4r,4r)\times Y$. 
\end{assum}
\
 
Also, set $$E_f\index[p2]{Ef@$E_f$}:=\sup_{x\in(-4r,4r)\times Y}\frac{|\Hess(f)|}{|\nabla f|^2}<\infty$$ and $$e_f'\index[p2]{efprime@$e_f'$}:=\sup_{x\in(-4r,4r)\times Y}|\nabla f|^2.$$

		If $(Z,g,f)$ is strongly tamed, we can fix a relatively compact open set $U\subset Z$ such that all the critical points of $f$ are contained in $U$, and set
		\[ c_f\index[p2]{cf@$c_f$}:=\inf_{Z-U}|\nabla f|>0.\] 	
        
On $Z-U$, we consider the Agmon metric 
\be\label{def-Agmon-metric}
g_{T}:=T^2|\nabla f|^2g^{TZ}\index[p2]{gT@$g_T$},
\ee
which is a Riemannian metric on $Z-U$. We also define $\dist_{T}$\index[p2]{distT@$\dist_{T}$} to be the distance on $Z-U$ induced by $g_{T}$, and $\rho_{T }(x):=\dist_{T }\big(x,U\big)$\index[p2]{rhoT@$\rho_{T}$} for $x\in Z-U$. One can see that $\rho_{T}$ is Lipschitz {for the Agmon metric} and that $\rho_T=T\rho_1$. Finally, let \be\label{defn KT} K_T:=\{x\in Z-U: \rho_{T}(x)\leq 2\}\index[p2]{KT@$K_T$}\ee.
		
		We have the following nice property for $\rho_T$:
		\begin{prop}\label{prop84}
			\begin{enumerate}[(1)]
				\item $T^2|\nabla f|^2=|\nabla \rho_T|^2$ a.e., where the gradient and norm is taken with respect to original metric $g^{TZ}$. 
				\item For $x,y\in Z-U$, $T|f(x)-f(y)|\leq \dist_T(x,y)$. Moreover, if there exists a flow generated by $\nabla f$ connecting $x$ and $y$, then
				$T|f(x)-f(y)|= \dist_T(x,y)$. 
			\end{enumerate}
		\end{prop}
		
		\begin{proof}
			For the first point, just observe that the norm of gradient of distance function for a metric is 1 a.e, so if $\nabla^T$ denotes the gradient for the Agmon metric, then $|\nabla^T \rho_T|_{g_T}^2=1$ a.e. On the other hand,  $\nabla^Th=\frac{\nabla h}{T^2|\nabla f|^2}$ for any smooth function $h$, so we get the equality.
			
			The second point follows from \cite[Lemma 4.2]{DY2020cohomology}
		\end{proof}

		
		Now, we consider a family $f_A$, $A\geq 0$, of function such that $(Z,g^{TZ},f_A)$ is strongly tame and $f_{A=0}=f$. Note that the precise definition of $f_A$ will be different, however similar, in \cref{conag} and \cref{sec9}, but we will keep the same notation.



Let $\|\cdot\|_{L^2}$ be the $L^2$-norm on $ \Omega^{\bullet}(Z,F)$ induced by $h^F$ and $g^{TZ}$.
Let $d_{T,A}:=d+Tdf_A\wedge.$\index[p2]{dTA@$d_{T,A}$} Let $d_{T,A}^*$ be the formal adjoint of $d_{T,A}$ with respect to $\|\cdot\|_{L^2}.$ We define
\be
\begin{aligned}
 &   D_{T,A}:=d_{T,A} + d_{T,A}^*\index[p2]{DTA@$D_{T,A}$} &\text{the Witten Dirac operator}\\
 & \Delta_{T,A} \index[p2]{DeltaTA@$\Delta_{T,A}$}:= d_{T,A}^* d_{T,A} + d_{T,A} d_{T,A}^* &\text{the Witten Hodge Laplacian} \\
 & V_{T,A}:=d_{T,A} - d_{T,A}^*\index[p2]{VTA@$V_{T,A}$}
\end{aligned}
\ee
When $A=0$, so that $f_A=f$, these operators will be abbreviated to $d_T$, $D_T$\index[p2]{DT@$D_T$}, $V_T$ and $\Delta_T$\index[p2]{DeltaT@$\Delta_T$} respectively.



Finally, we denote by $D_{T,A,1}$ (resp. $\Delta_{T,A,1}$)  the restriction of  $D_{T,A}$ (resp. $\Delta_{T,A}$) to $Z_1$ with absolute boundary conditions, and by $D_{T,A,2}$\index[p2]{DTAi@$D_{T,A,1}$ and $D_{T,A,2}$} (resp. $\Delta_{T,A ,2}$\index[p2]{DeltaTAi@$\Delta_{T,A,1}$ and $\Delta_{T,A,2}$})   the restriction of  $D_{T,A}$ (resp. $\Delta_{T,A }$) to $Z_2$ with relative boundary conditions.

\begin{rem}
	Since $(Z,g^{TZ} ,f_A)$ is strongly tame, all Laplacian operators above do have discrete eigenvalues (c.f. \cite[Theorem 2.2]{DY2020cohomology}).
\end{rem}

\section{Estimates of eigenforms}\label{conag}
\def\cn{\mathbf{c}}
\def\Cn{\mathbf{C}}

\subsection{Estimate of eigenforms on the {cylinder}}
\label{sect7.2}

For simplicity, in this section, we set $$d:=d^{\F}.$$

Let $p_A$\index[p2]{pA@$p_A$} be a smooth family of odd smooth functions on $[-2r,2r]$, such that 
\begin{enumerate}[(1)]
	\item $p_{{A=0}}\equiv0.$
	\item $p_A|_{[r,2r]}\equiv Ar^2/2$,
	\item $p_A|_{[0.02r,r]}(s)=-A\rho\big(e^{A^2}(r-s)\big)(s-r)^2/2+Ar^2/2$ , where $\rho\in C_c^\infty([0,\infty))$ is such that $0\leq\rho\leq1$, $\rho_{[0,r/2]}\equiv0,$ $\rho_{[3r/4,\infty]}\equiv1,$ $|\rho'|\leq \cn_1r^{-1}$ and $|\rho''|\leq \cn_2r^{-2}$ for some universal constants $\cn_1$ and $\cn_2$.
	\item $\Cn_1Ar\leq p_A'(s)\leq 2\Cn_1 Ar$, $|p_A''|\leq \Cn_2A$ for some universal constants $\Cn_1$ and $
	\Cn_2$ whenever $s\in[0,0.02r].$
	
\end{enumerate}
Then we can see that $|p_A'|(s)\leq \Cn_3A\big||s|-r\big|$ and $|p_A''|(s)\leq \Cn_4A$ whenever $\big||s|-r\big|\leq e^{-A^2}r$ for some universal constant $\Cn_3$ and $\Cn_4$. Moreover, we can see $p_A$ as a function on $Z$ as follows: if $Z_1':=Z_1-(-r,0]\times Y$ and $Z_2':=Z_2-[0,r)\times Y$, then $$p_A|_{Z_2'}=Ar^2/2, p_A|_{Z_1'}=-Ar^2/2, p_A|_{[-r,r]\times Y}(s,y)=p_A(s).$$ 

In this sub-section, we consider the objects defined in \cref{setting-partII} associated with
\be\label{fA-sect-7.1-7.2}
f_A:=f+p_A.\index[p2]{fA@$f_A$!in  \cref{sect7.2} and \cref{sect7.3}}
\ee

Let $\{e_j\}$ be a local frame of $TZ$ and $\{e^j\}$ its dual frame. Also, we assume that near $Y$, $e_1=\frac{\p}{\p s}$. Let $c(e_j):=e^j\wedge-i_{e_j}, \hat{c}(e_j):=e^j\wedge+i_{e_j}.$ Let  $\nabla$ be the connection on $\Lambda^\bullet T^*Z\otimes F$ induced by $\nabla^{F}$ and $g^{TZ}$, and let 
\[{D}\index[p2]{D@$D$}=\sum_{j}c(e_j)\nabla_{e_j}-i_{e_j}\o(e_j).\]
Then
\[{D}_{T }={D}+T\hat{c}(\nabla f),\]
and \[{D}_{T,A }={D}_{T }+T\hat{c}(\nabla p_A)={D}+T\hat{c}(\nabla (f+p_A)).\]


	Let  $\lan\cdot,\cdot\ran$ be the pointwise metric on $\Lambda^*(Z)\otimes F$ induced by $g ^{TZ}$ and $h^{F}$, and $|\cdot| :=\sqrt{\lan\cdot,\cdot\ran }$.
Set $${L}_{A }\index[p2]{LA@$L_A$}:=\frac{\p^2 p_A}{\p s^2}c\Big(\frac{\p}{\p s}\Big)\hat{c}\Big(\frac{\p}{\p s}\Big)+\sum_{j,k}\mathrm{Hess} (f)(e_j,e_k)c(e_j)\hc(e_k).$$
	We can see that in $(-2r,2r)\times Y$,
	\be\label{bardelta}
	{\Delta}_{T,A }= {\Delta}+ T{L}_{A }+T^2|\nabla  f+\nabla p_A| ^2.
	\ee

	\def\mycmd{0}
	\ifx\mycmd\undefined
	undefed
	\else
	\if\mycmd1
	
	\begin{lem}\label{limit1}
		Let $u\in \Omega^\bullet(\bm;\F)$ be a unit eigenform with respect to an eigenvalue $\leq\l$. Then for $s\in [-2,-1+\sqrt{\frac{2}{T}}]\cup [1-\sqrt{\frac{2}{T}},2]$
		\be\label{new10}\int_Y|u|^2(s,y)\dvol_Y\leq C(\lambda+1).\ee
		\be\label{new1}\int_Y|D_Au|^2(s,y)\dvol_Y\leq C(\lambda^2+1).\ee

		Let $w_i\in \Omega^\bullet(\bm_i;\F_i)(T),i=1,2$ be a unit eigenform w.r.t eigenvalues $\leq\l$. Then for $s\in [-2,-1+\sqrt{\frac{2}{T}}]$ or $s\in [1-\sqrt{\frac{2}{T}},2]$
		\be\label{new20}\int_Y|w_i|^2(s,y)\dvol_Y\leq C(\lambda+1)\ee
		\be\label{new2}\int_Y|D_Aw_i|^2(s,y)\dvol_Y\leq C(\lambda^2+1)\ee
		
		\be\label{new30}\int_Y|w_i|^2(0,y)\dvol_Y\leq C(\lambda+1).\ee
		\be\label{new3}\int_Y|w_i|^2(0,y)+|D_Aw_i|^2(0,y)\dvol_Y\leq C(\lambda^2+1).\ee
		
	\end{lem}
	\begin{proof}\ \\
		$\bullet$ {\textit{We only prove \eqref{new10} and \eqref{new1} for $s\in[-2,-1+\sqrt{\frac{2}{T}}].$} }\\
		First, by trace formula and Garding's inequality,
		\begin{align*}
			\int_Y |u|^2(-1,y)\dvol_Y&\leq C\int_{M_1}|u|^2+|\nabla u|^2\dvol_{M_1}\leq C\int_{M_1} |u|^2+|(d+d^*)u|^2\dvol_{M_1} \\
			&=C\int_{M_1}|u|^2+|(d_T+d_T^*)u|^2\dvol_{M_1}\leq C(\lambda+1).
		\end{align*}
		Similarly, we still have for $s\in [-2,-1]$,
		\be\label{lnew}\int_Y |u|^2(s,y)\dvol_Y\leq C(\lambda+1).\ee
		\def\c{\gamma}
		
		Next for $\c\in(0,1)$ to be determined, suppose $s_0\in [-1,-1+\sqrt{\frac{\c}{T}}]$ achieves the supreme of $$A_T:=\sup_{s\in [-1,-1+\sqrt{\frac{\c}{T}}]}\int_{Y}|u|^2(s,y)\dvol_Y,$$
		then 
		\begin{align}\begin{split}\label{uuu1}
				\int_{Y}|u(s_0,y)-u(-1,y)|^2\dvol_{Y}&\leq \int_{Y}\left|\int_{-1}^{-1+\sqrt{\frac{\c}{T}}}|\frac{\partial}{\partial s'}u(s',y)|ds'\right|^2\dvol_{Y}\\
				&\leq\sqrt{\frac{\c}{T}}\int_{Y}\int_{-1}^{-1+\sqrt{\frac{\c}{T}}}|d u(s',y)|^2+|d^* u(s',y)|^2ds'\dvol_{Y}.\\
		\end{split}\end{align}
		Integration by parts,
		\begin{align}
			\begin{split}\label{uuu2}
				&\ \ \ \ \l\geq\int_{\bm}|d_T u|^2+|d_T^* u|^2\dvol_{\bm}\geq \int_{-1}^{-1+\sqrt{\frac{\c}{T}}}\int_{Y}|d_T u|^2+|d^*_A u|^2\dvol_{Y}ds\\
				&\geq \int_{-1}^{-1+\sqrt{\frac{\c}{T}}}\int_{Y}|d u|^2+|d^* u|^2+(L_{f_T}u,u)+|\nabla f_T|^2|u|^2\dvol_Yds\\
				&-\int_Y|df_T||u|^2 (-1+\sqrt{\frac{\c}{T}},y)\dvol_Y\\
				&\geq \int_{-1}^{-1+\sqrt{\frac{\c}{T}}}\int_{Y}|d u|^2+|d^* u|^2-\Cn_4T|u|^2\dvol_Yds-\sqrt{T}A_T\\
				&\geq\int_{-1}^{-1+\sqrt{\frac{\c}{T}}}\int_{Y}|d u|^2+|d^* u|^2\dvol_Yds-(1+\Cn_4\sqrt\gamma)\sqrt{ T}A_T.\\
			\end{split}
		\end{align}
		See \cref{aw} for the definition of $\Cn_4.$
		By (\ref{uuu1}) and (\ref{uuu2}), we can see that
		\[A_T\leq (\l+1) (\sqrt{\frac{\c}{T}}+C)+\sqrt{\c}(1+\Cn_4\sqrt{\c})A_T.\]
		Fix $\c\in(0,1)$, such that $\sqrt{\c}(1+\Cn_4\sqrt{\c})\leq \half.$ Thus, whenever $s\in [-1,-1+\sqrt{\frac{\c}{T}}]$,
		\[\int_Y|u(s,y)|^2\dvol_Y\leq 3C(\l+1).\]
		Similarly, we can show that for $s\in [-1+\sqrt{\frac{\c}{T}}, -1+2\sqrt{\frac{\c}{T}}]$
		\[\int_Y|u(s,y)|^2\dvol_Y\leq 3^2C(\l+1).\]
		Let $m=[\frac{2}{\gamma}]+1$, then repeating the arguments above for $m$ times, we can see that
		\be\label{new11}\int_Y|u(s,y)|^2\dvol_Y\leq 3^mC(\l+1)\ee
		whenever $s\in [-1,-1+\sqrt{\frac{2}{T}}].$
		
		Replace $u$ with $D_Au$ and repeat the arguments above, we have 
		\[\int_Y|D_Au(s,y)|^2\dvol_Y\leq C(\l^2+1)\] 
		whenever $s\in [-1,-1+\sqrt{\frac{2}{T}}].$\\
		$\bullet$ {\textit{Similarly, we have \eqref{new20} and \eqref{new2}.}}\\
		$\bullet$ {\textit{We only prove \eqref{new30} and \eqref{new3} for $w_1$.}}\\
		Assume that on $[-2,0]\times Y$, $w_1=\a_1+\b_1 ds$, and set
		\[h(s):=\int_Y|\a_1|^2(s,y)\dvol_Y.\]
		Then $h(0)=\int_Y|w_1|^2(0,y)\dvol_Y$.
		
		We first show that $h(0)\leq C(\l+1).$
		
		Notice that for $s\in[-1,0]$, $f_T$ is increasing, we compute
		\begin{align}\begin{split}\label{new4}
				&\ \ \ \ h'(s)=2\int_Y\lan\partial_s\a_1,\a_1\ran\dvol_Y\\
				&=2\int_Y\lan d_T\a_1,ds\wedge\a_1\ran\dvol_Y-2\int_Y\lan \partial_sf_Tds\wedge\a_1,ds\wedge\a_1\ran\dvol_Y\\
				&\leq 2\int_Y\lan d_T\a_1,ds\wedge\a_1\ran\dvol_Y\leq 2\int_Y|d_T\a_1|^2+|\a_1|^2\dvol_Y.
		\end{split}\end{align}
		
		By (\ref{new4}) and (\ref{new20}), we can see that
		\[h(0)=h(-1)+\int_{-1}^0h'(s)ds\leq C(1+\l).\]
		
		Similarly, we can show that
		\[\int_Y|D_Aw_i|^2(0,y)\dvol_Y\leq C(1+\l^2).\]
	\end{proof}
	\else

	From this point onward, the statements will hold  for sufficiently large $T$, and we will not restate this condition each time. All constants appeared are at least $(T,A)$-independent. 
	
	
	\def\contro{0}
	\if\contro0

	As in \cite{Yantorsions}, we have
	\begin{lem}\label{limit1} 
		Let $A\geq A_0$, $u\in \Omega^\bullet(Z;F)$. Then for $s\in [-5r/4,5r/4]$,
		\be\label{new10}\int_{Y}|u|^2(s,y)\dvol_Y\leq C\left(\int_{Z}|D_{T,A}u|^2+T|u|^2\right).\ee
		
		Let $w_a\in \Omega_{\bd}^\bullet(Z_a;F)$, $a=1,2$. Then for $s\in [-5r/4,0]$ (if $a=1$) or $s\in [0,5r/4]$ (if $a=2$),
		\be\label{new20}\int_{Y}|w_a|^2(s,y)\dvol_{Y}\leq C\left(\int_{Z_a}|D_{T,A}w_a|^2+T|w_a|^2\right)\ee
		for some constant $C=C\Big(r,g^{TZ}|_{(-4r,4r)\times Y},h^F|_{(-4r,4r)\times Y},e_f'\Big)$.
	\end{lem}

	\begin{proof}
		Assume that on $[-2r,2r]\times Y$, $u=\a+ds\wedge\b$, $w_a=\a_a+ds\wedge\b_a(a=1,2)$, and $i_{\frac{\p}{\p s}}\a=i_{\frac{\p}{\p s}}\a_a=i_{\frac{\p}{\p s}}\b=i_{\frac{\p}{\p s}}\b_a=0$. 
		\\ \ \\
		$\bullet$ {\textit{We prove \eqref{new10} for $\a.$} }\\
		
		Let $\rho$ be a bump function, s.t. $\rho|_{[-9r/8,9r/8]}\equiv1$, $\rho(s)=0$ if $|s|\geq 5r/4.$ We can view $\rho$ as a smooth function on $Z.$

		Next, for $s\in[-5r/4,5r/4]$, set \[h(s):=\int_{Y^s}|\rho {\a}|^2(s,y)\dvol_{Y^s}.\]

		Note that by \Cref{product-type}, the metrics are of product-type. We compute
		\begin{align}\begin{split}\label{eq65}
				&\ \ \ \ h'(s)=\int_{Y^s}2\Re\big\lan\nabla_{\ps}\rho \a,\rho \a\big\ran(s,y)\dvol_{Y^s}\\
				&=\int_{Y^s}2\Re\big\lan d_{T,A}\rho {\a},ds\wedge\rho {\a}\big\ran(s,y)-2T\big\lan (\partial_sp_A)ds\wedge\rho {\a},ds\wedge\rho {\a}\big\ran(s,y)\dvol_{Y^s}\\
				&+\int_{Y^s}-2T\lan df\wedge\rho {\a},ds\wedge\rho {\a}\ran(s,y)\dvol_{Y^s}\\
		\end{split}\end{align}
		Notice that for $s\in[-5r/4,5r/4]$, $p'_A(s)\geq0$ and $T|df|\leq CT$, so
		\begin{align}\begin{split}\label{new4}	
				&\ \ \ \ h'(s)\leq \int_{Y^s}2\Re\lan d_{T,A}\rho {\a},ds\wedge\rho {\a}\ran(s,y)+CT|\rho  {u}|^2(s,y)\dvol_{Y^s}\\&\leq \int_{Y^s}| {D}_{T,A }\rho {\a}|^2(s,y)+CT|\rho  {u}|^2(s,y)\dvol_{Y^s}\\
		\end{split}\end{align}
		
		For $s\in(-5r/4,5r/4)$, $ \Omega^{\bullet}(Z,F)|_{Y^s}= \Omega^{\bullet}(Y^s,F|_{Y^s})\oplus  \Omega^{\bullet}(Y^s,F|_{Y^s})ds$. Moreover, $ \Omega^{\bullet}(Y^s,F|_{Y^s})$ is{orthogonal} to $ \Omega^{\bullet}(Y^s,F|_{Y^s})ds$.
		
		Note that $\Delta$, $\frac{\p^2 p_A}{\p s^2}c\Big(\frac{\p}{\p s}\Big)\hat{c}\Big(\frac{\p}{\p s}\Big)$ as well as $|\nabla f+\nabla p_A|^2$ preserves both $ \Omega^{\bullet}(Y^s,F|_{Y^s})$ and $ \Omega^{\bullet}(Y^s,F|_{Y^s})ds$. Moreover, $|\sum_{j,k}\Hess(f)(e_j,e_k)c(e_j)
        \hat{c}(e_k)|\leq C$. By \eqref{bardelta}  we can see easily that
		\be\label{bochnerformula1}\lan \Delta_{T,A } \rho \a,\rho ds\wedge\b\ran\leq CT |\rho \a||\rho ds\wedge\b|\leq CT|u|^2.\ee

		By (\ref{new4}) and \eqref{bochnerformula1}, we can see that
		\begin{align}\begin{split}\label{mod12}
				h(s)&=\int_{-2r}^sh'(s)ds\leq C\int_{{Z}}|D_{T,A }\rho\a|^2+T|u|^2=C\int_{{Z}}\lan \Delta_{T,A }\rho\a,\rho\a\ran+T|u|^2\\
				&\leq C\int_{{Z}}\lan \Delta_{T,A }\rho u,\rho u\ran+T|u|^2=C\int_{{Z}}|D_{T,A } \rho u|^2+T|u|^2\\
				&\leq C'\int_{{Z}}\rho^2|D_{T,A } u|^2+|c(\nabla \rho)u|^2+T|u|^2
		\end{split}\end{align}
		By \eqref{mod12}, we have \eqref{new10} for $\a.$\\
		$\bullet$ {\textit{Similarly, we have \eqref{new10} for $\b.$} }\\

		Repeating the arguments above and noticing that $\b_1(0,y)=0, \a_2(0,y)=0$, we have \eqref{new20}  for $w_1$ and $w_2$.
	\end{proof}
	\fi
	\fi
	\fi

	\def\mycmde{0}
	\ifx\mycmde\undefined
	undefed
	\else
	\if\mycmd1
	\begin{lem}\label{limadd}
		Assume $u,w_i,i=1,2$ meet the same conditions as in Lemma \ref{limit1}.
		Then 
		\be\label{new5}\int_{-\half}^{\half}\int_{Y^s}|u(s,y)|^2\dvol_{Y^s}ds\leq \frac{C(\l+1)}{T^{2}},\ee
		\be\label{new6}\int_{-\half}^{0}\int_{Y^s}|w_1(s,y)|^2\dvol_{Y^s}ds\leq \frac{C(\l^2+1)}{T^{2}},\ee
		\be\label{new7}\int_{0}^{\half}\int_{Y^s}|w_2(s,y)|^2\dvol_{Y^s}ds\leq \frac{C(\l^2+1)}{T^{2}}.\ee
		
		Moreover, if $u,w_i,i=1,2$ are harmonic, then for any $l\in\Z^+$, there exists $T$-independent $C_l,$ s.t.
		\be\label{new51}\int_{-\half}^{\half}\int_{Y^s}|u(s,y)|^2\dvol_{Y^s}ds\leq \frac{C_l}{T^{l}},\ee
		\be\label{new61}\int_{-\half}^{0}\int_{Y^s}|w_1(s,y)|^2\dvol_{Y^s}ds\leq \frac{C_l}{T^{l}},\ee
		\be\label{new71}\int_{0}^{\half}\int_{Y^s}|w_2(s,y)|^2\dvol_{Y^s}ds\leq \frac{C_l}{T^{l}}.\ee
		
	\end{lem}
	\begin{proof}\ \\
		$\bullet$ {\textit{Proof of \eqref{new5} and \eqref{new51}}.}\\
		Let $0\leq\eta\in C_c^\infty((-2,2))$, such that $\eta|_{[-\frac{1}{2},\frac{1}{2}]}\equiv1$, $\eta|_{[\frac{5}{8},2)\cup [-2,-\frac{5}{8}]}\equiv 0$, $|\nabla \eta|\leq 64$. We can regard $\eta$ as a smooth function on ${Z}.$
		
		Integrate by parts, 
		\begin{align}\begin{split}\label{new121}
				&\ \ \ \ \l\geq \int_{{Z}}\lan\Delta_Au,\eta^2u\ran\dvol_{{Z}}=\int_{{Z}}\lan Du,D\eta^2u\ran+\lan L_{p_A}u,\eta^2u\ran+|\nabla p_A|^2\eta^2|u|^2\dvol_{{Z}}\\
				&\geq\int_{{Z}}\eta^2\lan Du,Du\ran-|\eta'\eta||\lan Du,u\ran|+\lan L_{p_A}u,\eta^2u\ran+|\nabla p_A|^2\eta^2|u|^2\dvol_{{Z}}\\
				&\geq\int_{{Z}}\eta^2\lan Du,Du\ran/2-4|\eta'|^2|u|^2+C'T^2\eta^2|u|^2\dvol_{{Z}}\\
		\end{split} \end{align}
		(\ref{new121}) implies that
		\be\label{new131}
		\int_{-\half}^{\half}\int_{Y^s}|u|^2\dvol_{Y^s}ds\leq \frac{C(\l+1)}{T^2}.\ee
		
		Moreover, if $u$ is harmonic, (\ref{new121}) implies that
		\be\label{new132}
		\int_{-\half}^{\half}\int_{Y^s}|u|^2\dvol_{Y^s}ds\leq \frac{C}{T^2} \int_{-5/8}^{5/8}\int_{Y^s}|u|^2\dvol_{Y^s}ds.
		\ee
		Repeating the arguments above, we can see that
		\be\label{new133}
		\int_{-5/8}^{5/8}\int_{Y^s}|u|^2\dvol_{Y^s}ds\leq \frac{C'}{T^2}.
		\ee
		(\ref{new132}) and (\ref{new133}) then imply that
		\be\label{new1320}
		\int_{-\half}^{\half}\int_{Y^s}|u|^2\dvol_{Y^s}ds\leq \frac{C_4}{T^4}.
		\ee
		Repeating the arguments above, we have (\ref{new51}) for any $l\in\Z^+.$\\
		$\bullet$ {\textit{Proof of \eqref{new6} and \eqref{new61}}.}\\ 
		Let $\eta_1$ be a nonnegative bounded smooth function on ${Z}_1$, such that $\eta_1|_{[-\frac{1}{2},0]\times Y}\equiv1$, $\eta_1|_{{Z}_1-[-3/4,0]\times Y}\equiv 0$, $|\nabla \eta_1|\leq 64$. Let $w_1=\a+\b ds$, then
		by Lemma \ref{limit1} and integration by parts,
		\begin{flalign}\begin{split}\label{new12}
				&\ \ \ \ \l\geq \int_{{Z}_1}\lan\Delta_Aw_1,\eta_1^2w_1\ran\dvol_{{Z}_1}\\
				&=\int_{{Z}_1}\lan Dw_1,D\eta_1^2w_1\ran+\lan L_{p_A}w_1,\eta^2_1w_1\ran+|\nabla p_A|^2\eta_1^2|w_1|^2\dvol_{{Z}_1}\\
				&-\int_{Y^s} \lan ds\wedge w_1, dw_1\ran(0,y)\dvol_{Y^s}\\
				&=\int_{{Z}_1}\lan Dw_1,D\eta_1^2w_1\ran+\lan L_{p_A}w_1,\eta^2_1w_1\ran+|\nabla p_A|^2\eta_1^2|w_1|^2\dvol_{{Z}_1}\\
				&-\int_{Y^s} \lan ds\wedge w_1,d_Aw_1\ran(0,y)\dvol_{Y^s}+\int_{Y^s}\partial_s p_A\lan\a,\a\ran(0,y)\dvol_{Y^s}\\
				&\geq \int_{{Z}_1}\lan Dw_1,D\eta_1^2w_1\ran+\lan L_{p_A}w_1,\eta^2_1w_1\ran+|\nabla p_A|^2\eta_1^2|w_1|^2\dvol_{{Z}_1}-C(\l^2+1)\\
				&\geq\int_{{Z}_1}\eta^2\lan Dw_1,Dw_1\ran/2-4|\eta_1'|^2|w_1|^2+C'T^2\eta_1^2|w_1|^2\dvol_{{Z}_1}-C(\l^2+1)\\
		\end{split} \end{flalign}
		(\ref{new12}) implies that
		\be\label{new13}
		\int_{-\half}^{0}\int_{Y^s}|w_1|^2\dvol_{Y^s}ds\leq \frac{C(\l^2+1)}{T^2}.
		\ee
		Similarly, when $w_1$ is harmonic, since $d_Aw_1=0$, (\ref{new12}) also implies (\ref{new61}).
		
	\end{proof}
	\fi\fi

	

		

	
	\def\letome{1}
	\if\letome0
	Let $\Omega_c^\bullet(Z,F)$ denote the space of compact supported differential forms.
	\begin{lem}\label{limit0}
		Let $u\in  \Omega^{\bullet}_c(Z;F)$, and $w_1\in\Omega_c^\bullet(Z_1,F)$ satisfies the absolute boundary conditions, and $w_1\in\Omega_c^\bullet(Z_2,F)$ satisfies the relative boundary conditions.
		Then there exist constants $T_0=T_0(\Mc_0,E_f,e_f,e_f',r), A_1=A_1(r,e_f,e_f',E_f,\Mc_0)\geq1$, such that if $A\geq A_1,T\geq T_0$,
		\[\int_{-r/2}^{r/2}\int_{Y^s}|u|^2(s,y)\dvol_{Y^s}ds\leq \frac{C_l(\|D_{T,A }u\|_{L^2,A}^2+\|u\|_{L^2,A}^2)}{T^2A^2},\]
		\[\int_{-r/2}^{0}\int_{Y^s}|w_1|^2(s,y)\dvol_{Y^s}ds\leq \frac{C_l(\|D_{T,A }w_1\|_{L^2,A}^2+\|w_1\|_{L^2,A}^2)}{T^2A^2},\]
		and 	\[\int_{0}^{r/2}\int_{Y^s}|w_2|^2(s,y)\dvol_{Y^s}ds\leq \frac{C_l(\|D_{T,A }u\|_{L^2,A}^2+\|w_2\|_{L^2,A}^2)}{T^2A^2}\]
		for some constant $C=C(r,g^{TZ}|_{(-4r,4r)\times Y},h^F|_{(-4r,4r)\times Y},e_f',e_f,E_f)$.
	\end{lem}
	\begin{proof}
		First, it follows from the construction of $p_A$ that there exists universal constant $\Cn_5>0,$ s.t. $|\nabla p_A|^2(s,y)\geq \Cn_5(|s|-|r|)^2$ for $s\in(-r+e^{-A^2}r,r-e^{-A^2}r).$ 
		
		We show inequalities for $w_1$ first.
		
		Let $\eta\in C^\infty_c(-3r/4,0]$, such that $\eta\equiv1$ on $[-r/2,0]\times Y$. First, 
		\begin{align}\begin{split}\label{lem871}
				&\ \ \ \ 	\int_{Z_1}\lan {\Delta}_{T,A } {w_1},\eta^2 {w_1}\ran\dvol_Z=\int_{Z_1}\lan D_{T,A } {w_1},D_{T,A}\eta^2 {w_1}\ran\dvol_Z\\&=\int_{Z_1}\left\lan D_{T,A } {w_1},\left(\eta^2D_{T,A}+c(\nabla \eta^2)\right) {w_1}\right\ran\dvol_Z\leq C\left(\|D_{T,A }w_1\|_{L^2,A}^2+\|w_1\|_{L^2,A}^2\right).
		\end{split}\end{align}
		Next, integrate by parts,
		\begin{align}
			\begin{split}\label{mod21}
				&\ \ \ \int_{ Z_1}\lan {\Delta}_{T,A } {w}_1,\eta^2 {w}_1\ran\dvol_{ Z_1}\\
				&=\int_{ Z_1}\lan  {D } {w}_1, {D }\eta^2 {w}_1\ran+\lan  {L}_{T,A } {w}_1,\eta^2 {w}_1\ran+T^2| \nabla p_A+\nabla f|^2\eta^2| {w}_1| ^2\dvol_{ Z_1}\\
				&-\int_{Y^s} \lan ds\wedge  {w}_1, d {w}_1\ran (0,y)\dvol_{Y^0}.\\
			\end{split}
		\end{align}
		
		Also, note that $i_{\ps} {d}_{T,A} {w}_1=0$,
		\begin{align}
			\begin{split}\label{mod22}
				&\ \ \ \ 	\int_{Y^s} \lan ds\wedge  {w}_1, d {w}_1\ran (0,y)\dvol_{Y^0}	=
				\int_{Y^0} \lan ds\wedge  {w}_1, {d}_{T,A} {w}_1\ran (0,y)\dvol_{Y^0}\\&-T\int_{Y^0}(\partial_s p_A+\p_s f)|\a|^2 (0,y)\dvol_{Y^0}\leq 0,
			\end{split}
		\end{align}
		since if $A$ is large, for some $(T,A)$-independent constant $c$, $\partial_sp_A(0)-|\p_sf|\geq cA-e_f'\geq 0.$ 
		
		Lastly, if $T$ and $A$ are large,
		\begin{align}
			\begin{split}\label{mod23}
				&\ \ \ \ \int_{ Z_1}\lan  {D}   {w}_1, {D}  \eta^2 {w}_1\ran +\lan   TL_{A} {w}_1,\eta^2 {w}_1\ran +T^2| \nabla p_A+\nabla f|^2\eta^2| {w}_1| ^2\dvol_{ Z_1}	\\
				&\geq\int_{ Z_1}\eta^2\lan  {D}   {w}_1, {D}   {w}_1\ran -2|\eta'\eta||\lan  {D}   {w}_1, {w}_1\ran |+\lan   {L}_{T,A } {w}_1,\eta^2 {w}_1\ran \\
				&+T^2|\nabla p_A+\nabla f|^2\eta^2| {w}_1| ^2\dvol_{ Z}\\
				&\geq\int_{ Z_1}\eta^2\lan  {D}  {w}_1, {D}  {w}_1\ran /2-4|\eta'|^2| {w}_1| ^2+C'T^2A^2\eta^2| {w}_1| ^2\dvol_{ Z_1}.
			\end{split}
		\end{align}
		Here the last inequality follows from Cauchy-Swartz inequality and the fact that in $[-3r/4,0]\times Y$, $T^2|\nabla p_A+\nabla f|^2/2\geq C_1T^2A^2$ and $
		T^2|\nabla p_A+\nabla f|^2/2\geq C_2TA\geq C_3|TL_{A}|$ 
		for some $C_1,C_2,C_3>0$ if $A$ is large enough.
		It follows from \eqref{mod21},\eqref{mod22},\eqref{mod23}, that
		\be\int_{-r/2}^0\int_{Y^s}|{w}_1|^2\dvol_{Y^s} ds\leq \frac{C\left(\|D_{T,A }w_1\|_{L^2,A}^2+\|w_1\|_{L^2,A}^2\right)}{T^2{A^2}}.\ee
		
		We could show inequalities for $u$ and $w_2$ similarly.

		\def\con{0}
		\if\con{1}
		There exists $(T,A)$-independent $\kappa=\kappa(r,e_f,E_f,\Mc_0)\geq1$, $A_1=A_1(r,e_f,E_f,\Mc_0)\geq1)$ s.t. if $A\geq A_1$, on $ [-r+\sqt,r-\sqt]$, $T^2|\nabla p_A+\nabla f|^2-| TL_{A}|\geq TA$. Also note that $| p_A'| (\pm \big(r- \sqt\big))=C\sqrt{A}$. Hence, 
		\begin{align*}
			&\ \ \ \ \int_{-r+\sqt}^{r-\sqt}\int_{Y^s} {TA}| {u}|^2(s,y)\dvol_{Y^s}ds\\
			&\leq \int_{-r+\sqt}^{r-\sqt}\int_{Y^s} | {D}  {u}|^2+T^2|\nabla  p_A+\nabla f|^2| {u}|^2+\lan  TL_{A} {u}, {u}\ran\dvol_{Y^s}ds.\\
		\end{align*}
		Integrate by parts,
		\begin{align*}
			&\ \ \ \ \int_{-r+\sqt}^{r-\sqt}\int_{Y^s} | {D}  {u}|^2+T^2|\nabla  p_A+\nabla f|^2| {u}|^2+\lan  TL_{A} {u}, {u}\ran\dvol_{Y^s}ds.\\
			&\leq \int_{-r+\sqt}^{r-\sqt}\int_{Y^s} | {D}_{T,A} {u}|^2\dvol_{Y^s}ds+CT\int_{Y^s}(|\nabla p_A|+|\nabla f|| {u}|^2\dvol_{Y^s}\Big|_{s=\pm(-r+\sqrt{\frac{\kappa}{A}})}\\
			&\leq \int_{ Z} | {D}_{T,A} {u}|^2\dvol_{ Z}+CT\sqrt{A}\int_{Y^s}| {u}|^2\dvol_{Y^s}\Big|_{s=\pm(-r+\sqrt{\frac{\kappa}{A}})}\\
			&\leq CT\sqrt{A}(\l+T)\mbox{ (By Lemma \ref{limit1})}.
		\end{align*}
		Thus,
		\[\int_{-r+\sqrt\frac{2}{{A}}}^{r-\sqt} \int_{Y^s}| {u}(s,y)|^2\dvol_{Y^s}du\leq\frac{C(\l+T)}{\sqrt{A}}.\]
		By Lemma \ref{limit1}, we can show that
		\[\int_{||s|-r|\leq \sqt}\int_{Y^s}| {u}|^2(s,y)\dvol_{Y^s}ds\leq \frac{C(\lambda+T)}{\sqrt{A}}.\]

		Now we prove the lemma for $w_1$. Assume that $w_1=\a+ds\wedge\b$ on the cylinder, and let $ {w}_1:=e^{-Tf}w_1$. Let $\eta\in C^\infty(-\infty,0]$, s.t. $\eta|_{(-\infty,-r/2]}\equiv1$, $\eta(s)=0$ if $s\in[-r/4,0].$ We can think $\eta$ as a smooth function on $ Z_1.$
		
		Proceeding as before, we can show that $\int_{-r}^0\int_{Y^s}|\eta w_1|_T^2\dvol_{Y^s} ds\leq \frac{C(\l+1)}{\sqrt{A}}.$
		\fi
	\end{proof}
	\fi
	
	\def\proofofthm{0}
	\if\proofofthm1	
	\begin{proof}[Proof of  Theorem \ref{eigencon}]
		\def\span{\mathrm{span}}
		Let $u_j$ be eigenforms for $j$-th eigenvalue of $\Delta_{T,A }$, such that $\dim\span\{u_j:1\leq j\leq k\}=k$.
		
		Let $0\leq \eta\in C^\infty(Z)$, such that $\eta\equiv1$ on $Z-(-r/2,r/2)\times Y$, $\eta\equiv0$ on $(-r/4,r/4)\times Y$. Set $w_{1,j}:=\eta u_j|_{Z_1}$, and $w_{2,j}=\eta u|_{Z_2}.$
		
		Then $w_{1,j}$ and $w_{2,j}$ obvious satisfy boundary conditions. Let $w_j:=(w_{1,j},w_{2,j})\in \Omega^{\bullet}_{\abs}(Z_1,F)\oplus\Omega^\bullet_{\rel}(Z_2,F)$. Let $V_{k}:=\span\{w_j:1\leq j\leq k\}$, then it is obvious that $\dim(V_k)=k.$
		
		It follows from Lemma \ref{limit0} and Lemma \ref{o1} that for some $(T,A )$-independent $C,$
		\[\tilde{\l}_{k}(T,A )\leq\sup_{w\in V_k}\frac{\|D_{T,A }w\|^2_{L^2,A}}{\|w\|^2_{L^2,A}}\leq \l_k(T,A )+\frac{C(\l_k(T,A )^2+1)}{T^2A^2}\leq\l_k(T,A )+\frac{C(\Lambda_k(T)^2+1)}{T^2A^2}.
		\]

		Similarly, we can show that
		\[\l_{k}(T,A )\leq\tilde{\l}_k(T,A )+\frac{C(\Lambda_k(T)^2+1)}{T^2A^2}.\]
		
		The result then follows.

	\end{proof}\fi
	\def\estimates{0}
	\if\estimates1
	Estimates in Lemma \ref{limit0} could be improved if $u$ is an eigenform w.r.t eigenvalue $\leq cA^2T^2$. More explicitly,
	\begin{lem}\label{agmonapp}
		For any $b\in(0,\half),$,
		there exists $T_3=T_3(\Mc_0,E_f,e_f,e_f',r,b),A_3=A_3(\Mc_0,E_f,e_f,e_f',r,b)$, s.t. when $T\geq T_3,A\geq A_3$, if $u$ is an eigenform of $\Delta_{T,A }$ with respect to eigenvalue $\leq cT^2A^2$, 
		\[\int_{(-r/2,r/2)\times Y}|u|^2 \leq e^{-br^2TA}\|u\|_{L^2 }^2.\]
		A similar statements holds for eigenfoms of $\Delta_{T,A,1}$ and $\Delta_{T,A ,2}.$
		
		As a result, proceeding as in the proof of Theorem \ref{eigencon}, if $\l_{k}(T,A )\leq cT^2A^2$, there exists $(T,A )$-independent constant $c''$, such that 
		\[|\tilde{\l}_k(T,A )-\l_{k}(T,A )|\leq e^{-c''TAr^2}.\]
	\end{lem}
	The proof of Lemma \ref{agmonapp} would be given in the next subsection.
	\fi
	
	\def\cont{0}
	\if\cont0
	\begin{lem}\label{lem87}
		Let $0\leq\varphi\in C_c\big((-r,r)\big)$, such that $\varphi(s)\leq \big||s|-r\big|^2$. We can view $\varphi$ as continuous function on $Z$ with support inside $(-r,r)\times Y$. Then there exists $A_2=A_2(r,e_f')$, such that when $A\geq A_2$, for any $u\in \Omega^{\bullet}(Z,F)$ with compact support,
		\[\int_{Z}\varphi|u|^2\dvol_Z\leq \frac{C}{TA^{\frac{3}{2}}}\left(\int_ZT|u|^2+|D_{T,A}u|^2\dvol_{Z}\right);\]
		for any $w_1\in \Omega^{\bullet}(Z_1,F)$ satisfies the absolute boundary conditions,
		\[\int_{Z_1}\varphi|w_1|^2\dvol_Z\leq \frac{C}{TA^{\frac{3}{2}}}\left(\int_{Z_1}T|w_1|^2+|D_{T,A}w_1|^2\dvol_{Z}\right);\]
		for any $w_2\in \Omega^{\bullet}(Z_2,F)$  with compact support satisfies the relative boundary conditions,
		\[\int_{Z_2}\varphi|w_2|^2\dvol_Z\leq \frac{C}{TA^{\frac{3}{2}}}\left(\int_{Z_2}T|w_2|^2+|D_{T,A}w_2|^2\dvol_{Z}\right)\]
		for some constant $C=C\Big(r,g^{TZ}|_{(-4r,4r)\times Y},h^F|_{(-4r,4r)\times Y},f\Big).$
	\end{lem}
	\begin{proof}
		
		First, it follows from the construction of $p_A$ that there exists universal constant $\Cn_5>0,$ s.t. $|\nabla p_A|^2(s,y)\geq \Cn_5(|s|-|r|)^2$ for $s\in(-r+e^{-A^2}r,r-e^{-A^2}r).$ 
		
		We first show inequalities for $w_1$.
		
		Let $\eta\in C_c^\infty(-4r,4r)$, s.t. $0\leq\eta\leq1$ and $\eta|_{[-2r,2r]}\equiv1$. We can regard $\eta$ as a smooth function on $Z$.
		
		We have 
		\begin{align}\begin{split}\label{lem871}
				&\ \ \ \ 	\int_{Z_1}\lan {\Delta}_{T,A } {w_1}, \eta^2{w_1}\ran\dvol_Z=\int_{Z_1}\lan D_{T,A } {w_1},D_{T,A} \eta^2{w_1}\ran\dvol_{Z_1}\leq C\int_{Z_1}|w_1|^2+|D_{T,A}w_1|^2\dvol_{Z_1}.
		\end{split}\end{align}
		Next, integrate by parts,
		\begin{align}
			\begin{split}\label{mod21}
				&\ \ \ \int_{ Z_1}\lan {\Delta}_{T,A } {w}_1, \eta^2{w}_1\ran\dvol_{ Z_1}\\
				&=\int_{ Z_1}\lan  {D } {w}_1, {D } \eta^2{w}_1\ran+\lan  {L}_{T,A } {w}_1, \eta^2{w}_1\ran+T^2| \nabla p_A+\nabla f|^2| \eta{w}_1| ^2\dvol_{ Z_1}\\
				&-\int_{Y^s} \lan ds\wedge  {w}_1, d {w}_1\ran (0,y)\dvol_{Y^0}.\\
			\end{split}
		\end{align}
		
		Also, note that $i_{\ps} {d}_{T,A} {w}_1=0$,
		\begin{align}
			\begin{split}\label{mod22}
				&\ \ \ \ 	\int_{Y^s} \lan ds\wedge  {w}_1, d {w}_1\ran (0,y)\dvol_{Y^0}	=
				\int_{Y^0} \lan ds\wedge  {w}_1, {d}_{T,A} {w}_1\ran (0,y)\dvol_{Y^0}\\&-T\int_{Y^0}(\partial_s p_A+\p_s f)|\a|^2 (0,y)\dvol_{Y^0}\leq 0,
			\end{split}
		\end{align}
		since if $A$ is large, for some $(T,A,r)$-independent constant $c>0$, $\partial_sp_A(0)-|\p_sf|\geq cAr-e_f'\geq 0.$

		By \eqref{lem871}-\eqref{mod22}, and noting that $\lan Dw_1,D\eta^2 w_1\ran= |D\eta w_1|^2-|\nabla \eta|^2|w_1|^2\geq -C|w_1|^2$ we have
		\begin{align}
			\begin{split}\label{lem872}
				\int_{Z_1}T^2|\nabla p_A+\nabla f|^2| \eta{w_1}|^2+\eta^2\lan TL_{A} {w_1}, {w_1}\ran\dvol_{Z_1}\leq C\int_{Z_1} |w_1|^2+| {D}_{T,A} {w_1}|^2\dvol_{Z_1}.
			\end{split}
		\end{align}
		Also, the proof of Lemma \ref{limit1} also implies that for $s\in[-5r/4,0]\times Y$,
		\be\label{lem873}
		\int_{Y^s}| {w_1}|^2\dvol_{Y^s}\leq C\left(\int_{Z_1}T| {w_1}|^2+| {D}_{T,A} {w_1}|^2\dvol_{Z_1}\right).
		\ee
		
		Note that on $\Big[-r+\frac{2\sqrt{e_f'}}{A},r-\frac{2\sqrt{e_f'}}{A}\Big]\times Y$, $4|\nabla p_A+\nabla f|^2\geq |
		\nabla  p_A|^2\geq CA^2\varphi$ for some universal constant $C$. There exists $c=c(r,E_f,e_f,\Mc_0)$, such that when $T$ is large enough, $\frac{1}{2}T^2|\nabla p_A+\nabla f|^2-| TL_{A}|\geq0$ if $s\in\big\{s'\in(-4r,4r):||s'|-r|\geq\frac{c}{\sqrt{A}}\big\}.$ As a result, by \eqref{lem872}, \eqref{lem873}, and the fact that $|L_{A}|\leq CTA$,
		\begin{align}\begin{split}\label{lem83eq1}
			&\ \ \ \ \int_{\Big[-r+\frac{2\sqrt{e_f'}}{A},0\Big]\times Y}T^2A^2\varphi| {w_1}|^2\dvol_{Z}\\
            &\leq  C\left(\int_{Z_1}| {w_1}|^2+| {D}_{T,A} {w_1}|^2\dvol_{Z_1}+\int_{||s|-r|\leq\frac{c}{\sqrt{A}}}\int_{Y^s}| TL_{A}|| {w_1}|^2\dvol_{Y^s}ds\right)\\
			&\leq C'\left(\int_{Z_1}T^2\sqrt{A}| {w_1}|^2+T\sqrt{A}| {D}_{T,A} {w_1}|^2\dvol_{Z_1}\right).
		\end{split}\end{align}

        Moreover, by the property of $\varphi,$
\be\label{lem83eq2}
\int_{\Big[-r,-r+\frac{2\sqrt{e_f'}}{A}\Big]\times Y}\varphi| {w_1}|^2\dvol_{Z_1}\leq \frac{C\int_{Z_1}|w_1|^2\dvol_{Z_1}}{A^2}
\ee

		By \eqref{lem83eq1} and \eqref{lem83eq2}, we obtained the estimate for $w_1$.
		
		Following the same procedure as 
		above, we derived the estimates for $u$ and $w_2$.
	\end{proof}
	\fi
	
	\def\controll{0}
	\ifx\controll\undefined
	undef
	\else\if\controll1
	\begin{proof}[Proof of Theorem \ref{eigencon}]\ \\
		$\bullet$ $\lim\sup_{T\to\infty}\lambda_{k}(T)\leq \lambda_k.$\\
		Let $u_j=(u_{j,1},u_{j,2})$ be the $j$-th eigenvalue of $\Delta_1\oplus\Delta_2$ on $\Omega^\bullet_{\rel}(M_1;F_1)\oplus \Omega^\bullet_{\abs}(M_2;F_2)$ ($1\leq j\leq k$). 
		
		Let $\eta\in C_c^\infty([0,1])$, such that $\eta|_{[0,1/4]}\equiv0$, $\eta|_{[\half,1]}\equiv1.$

		For any $u=(v_1,v_2)\in \Omega^\bullet_{\abs}(M_1;F_1)\oplus \Omega^\bullet_{\rel}(M_2;F_2)$, let $Q_T:\Omega^\bullet_{\abs}(M_1;F_1)\oplus\Omega^\bullet_{\rel}(M_2,F_2)\to W^{1,2}\Omega( Z;\F)$, s.t.,
		\begin{equation*}
			Q_T({u})(x)=\begin{cases}
				v_1(x), \mbox{ if  $x\in M_i$;}\\
				\eta(-s)v_1(-r,y)e^{- p_A(s)-T/2}, \mbox{if $x=(s,y)\in [-r,0]\times Y$;}\\
				\eta(s)v_2(1,y)e^{ p_A(s)-T/2}, \mbox{if $x=(s,y)\in [0,1]\times Y$.}
			\end{cases}
		\end{equation*}
		Suppose $v_i=\a_i+\b_i\wedge ds(i=1,2)$ on the cylinder $[-2,2]\times Y$.
		It follows from the construction of $Q_T$ and the boundary conditions that ,\begin{equation}\label{qtdt}D_AQ_T(u)(x)=\begin{cases}D_Av_i, \mbox{ if }x\in M_i;\\
				D_Y\a_1(-r,y)e^{- p_A-T/2},\mbox{ if } x=(s,y)\in[-r,-r/2]\times Y; \\
				D_Y\b_2(1,y)\wedge ds e^{ p_A-T/2}, \mbox{ if }x=(s,y)\in[\half,1]\times Y.
		\end{cases}\end{equation}
		
		Let $ \overline{u}=Q_T(u)$, then we can see that $\dim span\{ \overline{u}_j\}_{j=1}^k=k$. Moreover, by trace formula and proceeding as in Lemma \ref{limit1}, we can show that for any $u=(v_1,v_2)\in span\{u_j\}_{j=1}^k,$ there exists $C_1>0$
		\begin{equation}\label{trace}\int_Y|v_i((-r)^{i},y)|^2+|D_Y v_i((-r)^i,y)|^2\dvol_Y\leq C_1(1+\lambda_k^2)\int_{M_i}|v_i|^2\dvol_{ Z}.\end{equation}
		
		Then by (\ref{trace}) and the construction of $ \overline{u},$ we have
		\begin{align*}
			&\ \ \ \ \int_{ Z} |D  \overline{u}|^2+|\nabla  p_A|^2| \overline{u}|^2+\lan L_{ p_A} \overline{u}, \overline{u}\ran\dvol_M\\
			&=\int_{M_1\cup M_2}|D u|^2\dvol+\int_{-r}^0\int_Y |D  \overline{u}|^2+|\nabla  p_A|^2| \overline{u}|^2+\lan L_{ p_A} \overline{u}, \overline{u}\ran\dvol_Yds\\
			&+\int_{0}^1\int_Y |D  \overline{u}|^2+|\nabla  p_A|^2| \overline{u}|^2+\lan L_{ p_A} \overline{u}, \overline{u}\ran\dvol_Yds=I+II+III.
		\end{align*}
		First, notice that $I\leq\lambda_k\int_{M_1\cup M_2}|u|^2\dvol\leq \lambda_k\int_{ Z}| \overline{u}|^2\dvol .$\\
		By a straightforward computation, \eqref{qtdt} and Lemma \ref{limit2},
		\begin{align*}
			&\ \ \ \ II\leq\int_{-r}^0\int_Y|D_Y u_1|^2e^{-2 p_A(s)-T}\dvol_Yds+C\int_{-r/4}^{-r/2}\int_YA^2|e^{-T/16}u_1|^2\dvol_Yds\\
			&\leq \frac{C_2(1+\lambda_k^2)}{\sqrt{A}}\int_{M_2\cup M_2}| \overline{u}|^2\dvol_{ Z}\leq \frac{C_2(1+\lambda_k^2)}{\sqrt{A}}\int_{ Z}| \overline{u}|^2\dvol_{ Z}.
		\end{align*}
		Similarly, $III\leq \frac{C_2(1+\lambda_k^2)}{\sqrt{A}}\int_{ Z}| \overline{u}|^2\dvol_{ Z}.$
		Hence, $\l_{k}(T)\leq \l_k+\frac{C_2(1+\lambda_k^2)}{\sqrt{A}}$.
		
		Consequently, as $T\to \infty$, $\lim\sup_{T\to\infty}\lambda_{k}(T)\leq \lambda_k.$\ \\ \ \\
		$\bullet$ $\lim\inf_{T\to\infty}\lambda_{k}(T)\geq \lambda_k.$\\
		Let $\{A_l\}$ be a sequence, such that $\lim_{k\to\infty}\lambda_{k}(A_l)=\lim\inf_{T\to\infty}\lambda_{k}(T).$
		Let $u_{i,l}$ be an eigenform of $\Delta_{A_l}$ with norm $1$, with respect to eigenvalue $\lambda_{j}(A_l) (1\leq j\leq k).$  By Lemma \ref{o1}, $\lambda_{j}(A_l)\leq\Lambda_j$ for some $\Lambda_j>0$. Proceeding as in Lemma \ref{limit1},
		\[\left\|u_{j,l}\left|_{M_1\cup M_2}\right.\right\|_{W^{N,2}(M_1)\oplus W^{N,2}(M_2)}\leq C_3(1+\Lambda_k)^N\|u_{j,l}\|_{L^2( Z)}= C_3(1+\Lambda_k)^N\]
		for any $N>0$.
		By Sobolev's embedding theorem, we may as well assume that $\{u_{j,l}|_{M_1\cup M_2}\}$ converges in $W^{2,2}(M_1)\oplus W^{2,2} (M_2)$-topology, and assume $u_{j,\infty}:=\lim_{l\to\infty}u_{j,l}.$ 
		
		Moreover, we can also have $\dim span\{u_{j,\infty}\}_{j=1}^k=k.$ Hence, for any $u_\infty\in span\{u_{j,\infty}\}_{j=1}^k$ with $\int_{M_1\cup M_2}|u_\infty|^2=1$, we can find $u_l\in span\{u_{j,l}\}_{j=1}^k$, such that $u_l|_{M_1\cup M_2}\to u_{\infty}$ in $W^{2,2}(M_1)\oplus W^{2,2}(M_2)$ topology.
		
		By Lemma \ref{limit2}, we can see that $u_{\infty}\in \Omega_{\abs}(M_1;F_1)\oplus\Omega_{\rel}(M_2;F_2)$.
		
		By Lemma \ref{limit0}, we have
		\begin{equation}\label{uinfty}\lim_{l\to\infty}\int_{ Z} |u_{l}|^2\dvol_{ Z}=\lim_{l\to\infty}\int_{M_1\cup M_2}|u_{l}|^2\dvol_{ Z}=\int_{M_1\cup M_2}|u_{\infty}|^2\dvol_{ Z}=1.\end{equation}

		As a result, 
		\begin{align*}
			&\ \ \ \ \lim_{l\to\infty}\lambda_{k}(A_l)\geq\lim_{l\to\infty}\int_{ Z}(d_{A}u_l,d_{A}u_l)+(d_{A}^* u_l,d_{A}^*u_l)\dvol\\
			&\geq \lim_{l\to\infty}\int_{M_1\cup M_2}(du_l,du_l)+(d^* u_l,d^*u_l)\dvol_{ Z}\\
			&= \int_{M_1\cup M_2}(d u_{\infty},du_\infty)+(d^*u_{\infty},d^*u_\infty)\dvol_{ Z}\geq \l_k.\\
		\end{align*}
		
		Hence
		$\lim\inf_{T\to\infty}\lambda_{k}(T)\geq \lambda_k$.\ \\ \ \\
		$\bullet$ Similarly, we can show that $\lim_{T\to\infty}\tl_{k}(T)=\l_k.$
	\end{proof}
	\fi\fi

	\def\mycmd{0}
	\ifx\mycmd\undefined
	undefed
	\else
	\if\mycmd1
	\appendix
	\section{Existence of perturbation}

	Fix $\{r_1,\cdots,r_6\}\subset(0,r/2)$, s.t. $r_1<\cdots<r_6$, while $|r_2-r_3|$ and $|r_4-r_5|$ are small enough.

	Let $\rho=\sqrt{u_0^2+u_1^2+\cdots+u_n^2}, \quad \sigma_k=\frac{u_k}{p}, 0 \leq k \leq n$.
	Then \be\label{derrho}\frac{\partial \rho}{\partial u_k}=\frac{u_k}{\rho}\mbox{ and } \frac{\partial \sigma_k}{\partial u_l}=\frac{\delta_{k l}-\sigma_k \sigma_l}{\rho}.\ee
	
	Let $\frac{\p}{\p\rho}$ be the vector field $\frac{\p}{\p\rho}:=\sum_{k=0}^n\sigma_k\frac{\p}{\p u_k}$, then by \eqref{derrho}
	\be\label{gradirho}
	(\frac{\p}{\p \rho},\frac{\p}{\p u_k})=\frac{\p \rho}{\p u_k}\mbox{ and }\frac{\p}{\p\rho}\sigma_k=0.
	\ee
	
	\begin{lem}
		Fix $\{\delta_1,\cdots,\delta_n\}$ be small enough, s.t. $\delta_l\pm\delta_m\neq0$ for any $l,m\in\{1,2,\cdots,n\},l\neq m$.
		For $r_2>r_1>0$, let
		$$
		D\left(r_1, r_2\right):=\left\{\left(u_0, \cdots, u_n\right) \in \mathbb{R}^{n+1}: r_1^2 \leq u_0^2+\ldots+u_n^2 \leq r_2^2\right\}.
		$$
		
		Suppose $f \in C^{\infty}\left(D\left(r_1, r_2\right)\right)$ is given by
		$$
		f=\eta_1(\rho) \sigma_0^3+\eta_2(\rho) \sigma_0-\sum_{k=1}^i(\eta_3(\rho)+\delta_k)\sigma_k^2+\sum_{k=i+1}^n\left(\eta_3(\rho)+\delta_k\right)\sigma_k^2
		$$
		for some $\eta_1, \eta_2, \eta_3 \in C^{\infty}\left(r_1, r_2\right), \delta>0$ small enough, s.t.
		\be\label{etacond}
		\eta_3>3\left|\eta_1\right|+\left|\eta_2\right|+2\sum_{k=1}^n\delta_k,\quad	\eta_3'>2\left|\eta'_1\right|+2\left|\eta'_2\right|, \quad \eta_1^{\prime}+\eta_2^{\prime} \neq 0, \quad \eta_3^{\prime} \neq 0
		\ee
		then $|\nabla f| \neq 0$ on $D\left(r_1, r_2\right)$.
	\end{lem}
	\begin{proof}
		By \eqref{gradirho},
		\be\label{derrhof}\frac{\p}{\p\rho}f=	f=\eta_1'(\rho) \sigma_0^3+\eta_2'(\rho) \sigma_0-\eta_3'(\rho)\left\|\left(\sigma_1, \cdots, \sigma_i\right)\right\|^2+\eta_3'(\rho)\left\|\left(\sigma_{i+1}, \cdots, \sigma_n\right)\right\|^2.\ee
		\begin{enumerate}
			\item  If $\frac{\partial}{\partial \rho} f \neq 0$, then we are done.
			\item Now assume \be\label{vanrho}\frac{\partial}{\partial \rho} f=0.\ee
			\begin{enumerate}[(1)]
				\item If $u_0\neq0$. By \eqref{derrhof} and \eqref{vanrho},
				$$
				\frac{\p}{\p u_l}f=\frac{\sigma_l}{\rho}\left(-3\eta_1(\rho) \sigma_0^3-\eta_2(\rho) \sigma_0-2\sum_{k=1}^i(\eta_3(\rho)+\delta_k)\sigma_k^2+2\sum_{k=i+1}^n\left(\eta_3(\rho)+\delta_k\right)\sigma_k^2\pm2(\eta_3(\rho)+\delta_l)\right)
				$$ for $l\neq0.$

				\item So we may as well assume $u_1 \neq 0$.
				
				By \eqref{vanrho}, \eqref{derrho} and \eqref{gradirho},
				\begin{align*}
					\frac{\partial}{\partial u_1} f&=\frac{\sigma_1}{\rho}\left(-3 \eta_1(\rho) \sigma_0^3-\eta_2(\rho) \sigma_0-2 \eta_3(\rho)\left(1-\left\|\left(\sigma_1, \ldots, \sigma_i\right)\right\|^2+\left\|\left(\sigma_{i+1}, \ldots \sigma_n\right)\right\|^2\right)\right)\\&-\frac{2 \delta \sigma_1}{\rho}\left\|\left(\sigma_{i+1}, \ldots, \sigma_n\right)\right\|^2
				\end{align*}\begin{enumerate}[(a)]
					\item If $-\left\|\left(\sigma_1, \cdots \sigma_i\right)\right\|^2+\left\|\left(\sigma_{i+1}, \cdot \cdot \sigma_n\right)\right\|^2 \geq-\frac{1}{2}$, then 
					$$\frac{1}{\sigma_1} \frac{\partial}{\partial u_1} f \leq \frac{1}{\rho}\left(3\left|\eta_1(\rho)\right|+\left|\eta_2(\rho)\right|-\eta_3(\rho)\right)<0.$$
					
					\item If $-\left\|\left(\sigma_1, \cdots, \sigma_i\right)\right\|^2+\left\|\left(\sigma_{i+1}, \cdots, \sigma_n\right)\right\|^2<-\frac{1}{2}$, by \eqref{derrhof} and \eqref{etacond},
					$\frac{\p}{\p \rho}f\neq0$
				\end{enumerate}

			\end{enumerate}

		\end{enumerate}
	\end{proof}

	\section{a}
	\fi

			
			
		\subsection{Agmon estimates}
        \label{sect7.3}
		
		In this subsection, we will derive Agmon estimates for eigenforms of $\Delta_{T,A }$ under certain assumptions.

		In this sub-section, as in \cref{sect7.2}, we consider the objects defined in \cref{setting-partII} associated with $f_A$ in \eqref{fA-sect-7.1-7.2}. Moreover, from now on in this subsection, we assume that $(Z,g^{TZ})=(\mathbb{R}^{n+1},g)$, where $g|_{\R^{n+1}-D(12r)}$ agree with the standard metric on $\mathbb{R}^{n+1}$, and that $Y$ is the Euclidean sphere of radius $8r$ centered at $0$ for some $r>0$. Here $D(12r)$ is the Euclidean ball of radius $12r.$ Moreover, we assume that all critical points of $f$ are contained inside  $U=D(r):=\left\{x \in \mathbb{R}^{n+1}:|x| \leq r\right\}$. Then if $T$ is large enough, $K_T\subset D(12r,40r):=\left\{x \in \mathbb{R}^{n+1}:12r\leq |x| \leq 40r\right\}$ (with $U$ in \eqref{defn KT} being $D(12r)$). Here $|x|$ is the Euclidean norm` .
		
		

		\begin{lem} \label{agmon}Fix $b\in(0,1)$. Let $u$ be an eigenform of $\Delta_{T,A }$ w.r.t an eigenvalue $\leq \frac{(b-b^2)c_f^2}{4} T^2$. Then there exists $T_0=T_0(r,e_f,e_f',E_f,\Mc_0,b)$, such that if $T\geq T_0$,  there exists $C=C(b)$, such that
			\[	\int_{D(12r,\infty)-K_T}\exp(2b\rho_{T})|u|^2 \leq C\int_{\R^{n+1}}|u| ^2.\]
			In particular,
			\[	\int_{D(40r,\infty)}\exp(2b\rho_{T})|u|^2 \leq C\int_{\R^{n+1}}|u| ^2.\]
			
		\end{lem}
		\begin{proof}
			Our proof is adapted from that of Theorem 1.5 in \cite{agmon2014lectures}, see also \cite[Lemma 3.1]{DY2020cohomology}.
			
			\def\ek{\eta_k}
			For $k\geq 2$ let $\eta_k \in C^\infty_c(\mathbb{R})$ be a smooth bump function such that
			\[
			\eta_k(t) =
			\begin{cases}
				0, & \text{if } |t| < 1 \text{ or } |t| > k+1; \\
				1, & \text{if } |t| \in (2, k),
			\end{cases}
			\]
			and $|\eta_k'(t)| \leq 2$, $\eta_k(t) \in [0, 1], \forall t \in \mathbb{R}$.
			
			\def\rjt{\rho_{T,j}}
			\def\phikj{\varphi_{k,j}}
			\def\ltj{\lambda_{T,j}}
			\def\lt{\lambda_{T}}
			For simplicity, let $\lambda_{T} = T^2 |\nabla f|^2$. Set 
			\[
			\rho_{T,j} = \min\{\rho_{T}, j\} \quad \text{ and } \quad
			\lambda_{T,j} =
			\begin{cases}
				\lambda_{T}, & \text{if } \rho_{T} < j; \\
				0, & \text{otherwise}
			\end{cases}.
			\]
			Then by Proposition \ref{prop84} and the above definitions, we have
			\be
			\label{eq-lambdaTj-and-rhoTj}
			\lambda_{T} \geq \lambda_{T,j}= |\nabla \rho_{T,j}|^2 \qquad \text{a.e.}
			\ee
			
			Now set $\varphi_{k,j} = (\eta_k \circ \rho_{T}) \exp(b \rho_{T,j})$. Then by assumption, we have
			\be
			\label{eq-assum-on-u}
			\int_{\mathbb{R}^{n+1}} \langle \Delta_{T,A} u, \varphi_{k,j}^2 u \rangle \leq \frac{(b - b^2) c_f^2}{4} T^2 \int_{\mathbb{R}^{n+1}} \varphi_{k,j}^2 |u|^2.
			\ee
			
			{Recall formula \eqref{bardelta} for $\Delta_{T,A}$. Observe that, $\varphi_{k,j}$ (resp. $p_A$) vanishes on (resp. outside of) $D(12r)$, and that by definition, on $D(12r,\infty)$ we have $|\nabla f|^2 \geq c_f^2>0$. Thus, for $T$ large enough, $\langle TL_{A}u,\varphi_{k,j}^2 u\rangle \leq (1-b)T^2|\nabla f|^2 |\varphi_{k,j}u|^2$ on $\R^n$.} As a consequence, for $T$ large enough, an integration by parts entails
			\be
			\label{eq-IbP-D}
			\int_{\mathbb{R}^{n+1}} \big\langle D u, D (\varphi_{k,j}^2 u) \big\rangle + b \lambda_{T} |\varphi_{k,j} u|^2 \leq \int_{\mathbb{R}^{n+1}} \langle \Delta_{T,A} u, \varphi_{k,j}^2 u \rangle.
			\ee

			Noting that $\big\langle D u, D (\varphi_{k,j}^2 u) \big\rangle = \big|D (\varphi_{k,j} u)\big|^2 - |\nabla \varphi_{k,j}|^2 |u|^2 \geq -|\nabla \varphi_{k,j}|^2 |u|^2$, \eqref{eq-assum-on-u} and \eqref{eq-IbP-D} imply
			\begin{equation}\label{eq7}
				\int_{\mathbb{R}^{n+1} - D(12r)} (b \lambda_{T} |\varphi_{k,j} u|^2 - |u|^2 |\nabla \varphi_{k,j}|^2) \leq \int_{\mathbb{R}^{n+1}} \langle \Delta_{T,A} u, \varphi_{k,j}^2 u \rangle.
			\end{equation}
			
			Equations \eqref{eq-assum-on-u} and \eqref{eq7} yield
			\be
			\label{eq7-bis}
			\int_{\mathbb{R}^{n+1} - D(12r)} (b \lambda_{T} |\varphi_{k,j} u|^2 - |u|^2 |\nabla \varphi_{k,j}|^2) \leq \frac{b - b^2}{4} \int_{\mathbb{R}^{n+1} - D(12r)} \lambda_{T} |u|^2 \varphi_{k,j}^2.
			\ee
			
			Using
			\begin{align*}
				|\nabla \varphi_{k,j}|^2 & \leq \frac{b^2 + b}{2} (\eta_k \circ \rho_{T})^2 |\nabla \rho_{T,j}|^2 \exp(2b \rho_{T,j}) + \frac{b + b^2}{b - b^2} (\eta_k' \circ \rho_{T})^2 |\nabla \rho_{T}|^2 \exp(2b \rho_{T,j}) \\
				& = \frac{b^2 + b}{2} (\eta_k \circ \rho_{T})^2 \lambda_{T,j} \exp(2b \rho_{T,j}) + \frac{b + b^2}{b - b^2} (\eta_k' \circ \rho_{T})^2 \lambda_{T} \exp(2b \rho_{T,j}),
			\end{align*}
			along with \eqref{eq-lambdaTj-and-rhoTj} and \eqref{eq7-bis}, we find that for $k\geq j$:
			\def\ek{\eta_k \circ \rho_{T}}
			\def\ekp{\eta_k' \circ \rho_{T}}
			\def\tb{\tilde{B}}
			
			\begin{equation}
				\begin{aligned}
					\label{sim}
					\frac{b - b^2}{4} \int_{\mathbb{R}^{n+1} - D(12r)} \lambda_{T}& (\eta_k \circ \rho_{T})^2 |u|^2 \exp(2b \rho_{T,j})  \\
					&\leq \frac{b + b^2}{b - b^2} \int_{\mathbb{R}^{n+1} - D(12r)} |u|^2 (\ekp)^2 \lambda_{T} \exp(2b \rho_{T,j}) \\
					& \leq \frac{b + b^2}{b - b^2} \int_{K_T} |u|^2 \lambda_{T} \exp(2b \rho_{T,j}) + \frac{b + b^2}{b - b^2} \int_{\tilde{B}_{k+1} - \tilde{B}_k} |u|^2 \lambda_{T} \exp(2bj),
				\end{aligned}
			\end{equation}
			where $\tilde{B}_k := \{ x \in \mathbb{R}^{n+1} : \rho_{T}(x) \leq k \}$, and the constant $\frac{b - b^2}{4}$ comes from $b - \frac{b + b^2}{2} - \frac{b - b^2}{4}$.
			
			Now, as $u \in \text{Dom}(D_{T,A})$, we know that $\int_{\mathbb{R}^{n+1}} \lambda_{T} |u|^2 < \infty$. Thus, letting $k \to \infty$, and then $j \to \infty$, and using the monotone convergence theorem, we have
			\be
			\label{proofagmon3}
			\frac{b - b^2}{4} \int_{D(12r,\infty) - K_T} \lambda_{T} |u|^2 \exp(2b \rho_{T}) \leq \frac{b + b^2}{b - b^2} \int_{K_T} |u|^2 \lambda_{T} \exp(4b).
			\ee
			This implies our Lemma.
		\end{proof}

		\fi
		
		\def\sol{0}
		\if\sol1
		
		\subsection{Solving Dirichlet boundary condition for $D_{T,A}$}
		In this subsection, we would like to solve the following boundary problem on $Z_i,i=1,2:$
		\be\label{dirichlet}
		\begin{cases}
			D_{T,A}u=0,u\in \Omega^{\bullet}(Z_i,F);\\
			u=0 \mbox{ on $\p Z_i$.}\end{cases}
		\ee
		
		
		We will prove that
		\begin{prop}\label{dirichletpro}
			Equation \eqref{dirichlet} has only trivial solution.
		\end{prop}
		
		In fact, we only need to study \eqref{dirichlet} for $Z_2$. First, we will prove the following Carleman-type estimate:
		\begin{lem}\label{carleman}
			For any $\tau>0$ sufficiently small and $\Lambda>0$, we have
			\[ 2\Lambda \int_{0}^\tau \int_{Y} e^{\Lambda(\tau-s)^2}|v(s, y)|^2 \mathrm{dvol}_{Y^s} ds 
			\leq  \int_{0}^\tau \int_{Y} e^{\Lambda(\tau-s)^2}|D_{T,A} v(s, y)|^2 \mathrm{dvol}_{Y^s} ds,\]
			where $v(s,y)=\eta(s)u(s,y)$ such that $\eta(s)=1$ if $s<\frac{8\tau}{10}$, $\eta(s)= 0$ if $s\in(\frac{9\tau}{10},\tau)$. Here, $u$ is a solution to \eqref{dirichlet}.
		\end{lem}
		
		\begin{proof}
			Let $v_0=e^{\Lambda(\tau-s)^2/2}v$, then it suffices to show that
			\[2\Lambda \int_{0}^\tau \int_{Y}|v_0(s, y)|^2 \mathrm{dvol}_{Y^s} ds 
			\leq  \int_{0}^\tau \int_{Y} |D_{T,A} v_0(s, y)+\Lambda c(\frac{\partial}{\partial s})(\tau-s)v_0(s,y)|^2 \mathrm{dvol}_{Y^s} ds.\]
			
			Note that $D_{T,A}^*=D_{T,A}$, $\left(c(\frac{\partial}{\partial s})(\tau-s)\right)^*=-c(\frac{\partial}{\partial s})(\tau-s)$, and $[D_{T,A},c(\frac{\partial}{\partial s})(\tau-s)]=1$ (here $[\cdot,\cdot]$ is the super-communicator). We can see that
			\begin{align*}
				&\int_{-\tau}^0 \int_{Y} |D_{T,A} v_0(s, y)+\Lambda c(\frac{\partial}{\partial s})(\tau-s)v_0(s,y)|^2 \mathrm{dvol}_{Y^s} ds\\
				&=\int_{-\tau}^0 \int_{Y} |D_{T,A} v_0(s, y)|^2+|\Lambda c(\frac{\partial}{\partial s})(\tau-s)v_0(s,y)|^2 \mathrm{dvol}_{Y^s} ds+2\Lambda \int_{0}^\tau \int_{Y}|v_0(s, y)|^2 \mathrm{dvol}_{Y^s} ds.
			\end{align*}
			The estimate then follows.
		\end{proof}
		
		By Lemma \ref{carleman}, proceeding as in \cite[Lemma 8.5]{}, we can see that $u=0$ near $\{0\}\times Y=\p Z_2$. Then, using a similar estimate as in Lemma \ref{carleman} (by replacing $Y$ with the boundary of some sufficiently small ball), and proceeding as in \cite[Lemma 8.5]{}, we have $u=0$ in the interior. As a result, Proposition \ref{dirichletpro} is established.
		
		\fi

		\section{Estimate of Schauder Norms}\label{sec9}

		This section aims to estimate the Schauder norms (see Definition \ref{schauder}) of certain operators. We will begin by focusing on the case where $ S $ is a point, specifically addressing single manifold cases. We will provide estimates for the resolvent of some Dirac/Laplacian-type operators, discussing separately the cases when $ A $ is large and when $ A $ is not large. To achieve this, we need to estimate {various lower bounds for} the eigenvalues of these Laplacian operators. Using the Agmon estimate, we will extend the techniques from \cite[$\S12$ and $\S13$]{bismut1991complex} to handle more singular critical points.
		Lastly, we extend our estimates to family case.

		Let $(Z,g^{TZ})$ be an $(n+1)$-dimensional complete non-compact Riemannian manifold with bounded geometry. Let $f_y\in C^\infty(Z)$ be a smooth family of generalized Morse function parameterized by $y\in(-\delta^2,\delta^2)$ for some small $\delta,$ such that $(Z,g^{TZ},f_y)$ is strongly tame. We will abbreviate $f_y$ as $f$ if it will not cause any confusion.

		We assume that the following holds.
		\begin{assum}\label{9a}
        There exist disjoint small open sets $U_0,\cdots,U_k$ and metric $g^{TZ'}$ on $TZ\to M$, fulfilling the following properties:
			\begin{enumerate}[(1)]
				\item All critical points of $f$ are contained in $\cup_{l=0}^kU_l$. Moreover, outside $\cup_{l=0}^kU_l$, $g^{TZ}=g^{TZ'}.$
				\item Each $U_l$ ($l\geq1$) contains precisely one non-degenerate critical point $p_l$ of Morse index $i_l$. Furthermore, there exist {$c_l\in \R$ and} {($y$-independent)} coordinates $(u_0,\cdots,u_n)$ on $U_l$ such that $u_0(p_l)=\cdots=u_n(p_l)=0$, and 
				\be\label{g-morse-lem1} f_{{y}}(q)={c_l}-u_0^2(q)-\cdots -u_{i_l-1}^2(q)+u_{i_l}^2(q)+\cdots u_n^2(q),\quad q\in U_l.\ee
				Also, $g^{TZ}=g^{TZ'}=(du_0)^2+\cdots (du_n)^2$ in $U_l$. 
				\item There exist {$c_0\in \R$ and} {($y$-independent)} coordinates $(u_0,\cdots,u_n)$ on $U_0$, such that for some $p_0\in U_0$, $u_0(p_0)=\cdots=u_n(p_0)=0$, and
				\be\label{g-morse-lem} f_{{y}}(q)={c_0}+u_0^3(q)-yu_0(q)-u_1^2(q)\cdots -u_{i-1}^2(q)+u_i^2(q)+\cdots u_n^2(q),\quad q\in U_0\ee
				for some $i\in\{1,2,\cdots,n\}$. Moreover, $g^{TZ'}=(du_0)^2+\cdots (du_n)^2$ in $U_0$.

\textbf{We fix $ r \in (0,1) $ small, whose value will be determined later (See Definition \ref{defr}), and let $r_1=7r$ and $r_2=9r$\index[p2]{r1@$r_1$ and $r_2$}.}\\
            We further assume that 
            \begin{enumerate}
                \item  $g^{TZ}=g^{TZ'}$ if $u_0^2+\cdots u_n^2\leq (7r/2)^2$.
                \item $g^{TZ}$ is product type as in \Cref{product-type} with $Y=\{u_0^2+\cdots u_n^2=(8r)^2\}.$ I.e., $g^{TZ}$ is of product type as in \Cref{product-type} if $\{r_1^2\leq u_0^2+\cdots+u_n^2\leq r_2^2\}$.
            \end{enumerate}
                
                \item $U_l$ is a ball of radius $80$ (w.r.t. $g^{TZ'}$) centered at $p_l$, the bundle $F|_{U_l}\to U_l$ is trivial, and the metric $h^F|_{U_l}$ is standard for $l\in\{0,1,\cdots,k\}$.
		\end{enumerate}

        \end{assum}

		Let  $f_{A,y}$\index[p2]{fAy@$f_{A,y}$} be the function as described in \eqref{defn fay}, with $r_1=7r$, $r_2=9r$. It is worth noting that 
\[
    f_{A=0,\, y} \neq f_y,
\]
although it is only a small perturbation of \( f_y \).
 {Recall that the critical points of $f_{A,y}$ have been studied in \cref{model}}. Let $(f_A)_y$ be a smooth family of function parameterized by $y\in\big(-\delta^2,\delta^2\big)$ defined by
		\be\label{fa} (f_{A})_y\index[p2]{fAy@$(f_{A})_y$}=f_y \text{ in } Z-U_0 \text{ and } (f_A)_y(q)=f_y(p_0)+f_{A,y}\big(u_0(q),\cdots,u_n(q)\big), \quad q\in U_0.\ee
		 We will abbreviate $(f_A)_y$ as $f_A$\index[p2]{fA@$f_A$!in  \cref{sec9}} if it causes no confusion.

        	Let $U_l'\subset U_l$ be a ball of radius $12$ (w.r.t. $g^{TZ'}$ in \Cref{9a}) centered at $p_l$. 	We consider in this section the objects introduced in \cref{setting-partII} associated with $f_A$ and $U=\cup_{l}U_l'$. When we want to emphasize on the parameter $y$ for these objects, we add a subscript `$y$' to them, e.g., $(D_{T,A})_y$ or $(\Delta_{T,A})_y$, but if no confusion arises, we will omit this subscript.

		        \textbf{All constants appearing in this section are at least $(T,A,y)$-independent.}

		Observe that, if $T$ is sufficiently large, $K_T\subset \cup_l (U_l-U_l')$ (with $U$ in \eqref{defn KT} being $\cup_{l}U_l'$). As in Lemma \ref{agmon}, we have the following lemma.
		\begin{lem} \label{agmonz}
			For any $b\in(0,1)$. Let $u$ be an eigenform of $\Delta_{T,A }$ with eigenvalues $\leq \frac{(b-b^2)c_f^2}{4} T^2$. Then, there exists $T_0=T_0(f,g^{TZ},h^F,b)$ such that if $T\geq T_0$, there exists $C=C(f,g^{TZ},h^F,b)$ satisfying
			\[	\int_{Z-\left(\cup_lU_l'\right)-K_T}\exp(2b\rho_{T})|u|^2 \leq C\int_{Z}|u| ^2.\]
			Thus,
			\[	\int_{Z-\cup_l U_l}\exp(2b\rho_{T})|u|^2 \leq C\int_{Z}|u| ^2.\]
		\end{lem}

		\subsection{Estimate of resolvents when $A$ is large}\label{Sect-estimate-A-large}

		\subsubsection{Lower bounds for small non-zero eigenvalues when $A$ is large}
		Let $\lambda_k(T,A)$\index[p2]{lambdakTA@$\lambda_k(T,A)$} be the $k$-th eigenvalue of $\Delta_{T,A}$. According to Hodge theory, there exists $k_0$\index[p2]{k0@$k_0$} such that if $k\leq k_0$, then $\lambda_k(T,A)\equiv0$, and if $k>k_0$, then $\lambda_{k}(T,A)\neq0$.

        Let $\mu_k(T,A)$\index[p2]{lambdakTA@$\mu_k(T,A)$} be the $k$-th eigenvalue of $D_{T,A}$. Here, the ordering of eigenvalues of $D_{T,A}$ is as follows: if $|\mu_k|>|\mu_l|$, then $k>l$; if $|\mu_k|=|\mu_l|$, but $\mu_k>0$ and $\mu_l<0$, then $k>l$. Then $\mu_k(T,A)$ is a Lipchitz function in $(T,A)$ and $\mu_k^2(T,A)=\lambda_k(T,A)$.

		Recall that $(f_{A})_y$ is described in \eqref{fa} (with $r_1 = 7r$, $r_2 = 9r$). 
        Let $(f_{A=0})_y=(f_{A})_y$ for $A=0$. When no confusion arise, we will drop the $y$ subscript and only write $f_{A=0}=(f_{A=0})_y$. Let
        \be C_{f_{A=0}}:=\sup_{x\in U'_l}|f_{A=0}|.\ee
		\begin{lem}\label{lem810}

            There exists $c=c(C_{f_{A=0}})$ such that for $k>k_0$, there are $T_k=T_k(f,g^{TZ},h^F)$ and $C_k=C_k(g^{TZ},h^F,f)$ such that for any $T>T_k$, we have $\lambda_k(T,0)\geq C_ke^{-cT}$.
		\end{lem}
		\begin{proof}
			{Fix $k>k_0$.}   Let $u_{k,T,A}$ be the $k$-th unit eigenform of $D_{T,A}$. Since $\int_Z|u_{k,T,A}|^2\equiv1$, 
			\be\label{lem810eq0}
			\frac{\partial}{\partial T}\int_Z|u_{k,T,A}|^2=2\Re\left(\int_Z\Big\langle \frac{\partial}{\partial T}u_{k,T,A},u_{k,T,A}\Big\rangle\right)=0.
			\ee
			
			Thus,
			\begin{align}
				\begin{split}\label{lem810eq1}
					&\ \ \ \ \left|\frac{\partial}{\partial T}\mu_k(T,0)\right|=\left|\frac{\partial}{\partial T}\int_Z\langle D_{T,0}u_{k,T,0} ,u_{k,T,0}\rangle\right|\\
					&=\left|\int_Z\bigg\langle \Big(\frac{\partial}{\partial T}D_{T,0}\Big)u_{k,T,0} ,u_{k,T,0}\bigg\rangle\right|\mbox{ (by \eqref{lem810eq0})}\\
					&=\left|\int_Z\Big\langle [V_{T,0},f_{A=0}]u_{k,T,0} ,u_{k,T,0}\Big\rangle\right|\leq 2\left(\int_{Z}|V_{T,0}u_{k,T,0}|^2\right)^{\frac{1}{2}}\left(\int_Z |f_{A=0}u_{k,T,0}|^2\right)^{\frac{1}{2}}\\
					&\leq 2|\mu_k(T,0)|\left(\int_Z |f_{A=0}u_{k,T,0}|^2\right)^{\frac{1}{2}}.
				\end{split}
			\end{align}

			By Proposition \ref{prop84}, 
			\be\label{lem810eq2}
			|f_{A=0}|(x)\leq C_{f_{A=0}}+\rho_1(x),\forall x\in Z.
			\ee
			
			Using { \Cref{agmonz} and} \eqref{lem810eq2}, we get

            \begin{equation}
                \label{lem810eq3}
                \begin{aligned}
                  \int_Z |f_{A=0}u_{k,T,0}|^2 &\leq  \int_{\left(\cup_{l}U_l'\right)\cup K_T}(C_{f_{A=0}}+2T^{-1})^{{2}}|u_{k,T,0}|^2+\int_{Z-\left(\cup_{l}U_l'\right)-K_T}|\rho_1|^2|u_{k,T,0}|^2\\
					&\leq \int_{Z}(C_{f_{A=0}}+2)^{{2}}|u_{k,T,0}|^2+\frac{1}{T^2}\int_{Z-\left(\cup_{l}U_l'\right)-K_T}\exp(\rho_T)|u_{k,T,0}|^2\\
                    &\leq C.
                \end{aligned}
            \end{equation}

			Hence, the lemma follows from the standard theory of ODE, \eqref{lem810eq1} and \eqref{lem810eq3}.
		\end{proof}
		
		Next, we aim to estimate lower bounds for $\lambda_k(T,A)$.

		Let $Z_1$ be the (Euclidean) ball of radius $8r$, centered at $p_0$, and $Z_2 := \overline{Z - Z_1}$\index[p2]{Z1@$Z_1$ and $Z_2$}. We will apply the notations of the end of \cref{setting-partII} to $Z=Z_1\cup Z_2$, cutted along $Y$, the sphere of radius $8r$ with respect to $g^{TZ'}$ (See \Cref{9a}).
        
        Let $\lambda_k(T,A,a)$\index[p2]{lambdakTAa@$\lambda_k(T,A,a)$, $a=1$, $2$} be the $k$-th eigenvalue of $\Delta_{T,A,a}$, $a = 1, 2$, and $\tilde{\lambda}_k(T,A)$\index[p2]{lambdakTAtilde@$\tilde{\lambda}_k(T,A)$} be the $k$-th eigenvalue of $\Delta_{T,A,1} \oplus \Delta_{T,A,2}$.
		
		By Hodge theory, for $a=1$, $2$, there exist $k_a \in \mathbb{Z}^+$\index[p2]{ka@$k_a$, $a=1$, $2$} such that if $k \leq k_a$, then $\lambda_k(T,A,a) \equiv 0$, and if $k > k_a$, then $\lambda_k(T,A,a) \neq 0$.
		
		\begin{lem}\label{lem811}
			There exist $A_4=A_4(r,g^{TZ},h^F,f)$ and $T_0=T_0(r,f,g^{TZ},h^F,k)$ such that if $A>A_4$ and $T>T_0$, we have for $k> k_a$, $a=1,2$,
			\[\lambda_{k}(T,A,a)\geq C_ke^{-c\left(T+TA^{\frac{1}{4}}\right)},\]
			where $c=c(r,g^{TZ},h^F,f)$ and $C_k=C_k(r,g^{TZ},h^F,f)$ are constants. Moreover, $ c $ and $A_4$ only depend  on the restrictions of the metrics $ g^{TZ} $ and $ h^F $ to the ball of radius $ 12r $ centered at $ p_0 $.
		\end{lem}
		\begin{proof}
			
			Let $\mu_k(T,A,a)$ be the $k$-the eigenvalue of {$D_{T,A,a},a=1,2$}. We prove the lemma for $a=1$, the proof is similar for $a=2$.
			
				Proceeding as in Lemma \ref{lem810}, we can observe that
			\be\label{lem812eq1}
			| \frac{\partial}{\partial A} \mu_k(T,A,1) | \leq 2 |\mu_k(T,A,1)| \left( \int_Z T^2\bigg|  \frac{\partial}{\partial A} q_A  u_{k,T,A} \bigg|^2 \right)^{\frac{1}{2}}.
			\ee

			Observe that on $D(7r, 8r)$, $|\frac{\partial}{\partial A} q_A|^2$ meets the same conditions in the present situation as $\varphi$ in Lemma \ref{lem87} did in the situation of Section \ref{conag} (up to some scalar multiplication by $(T,A)$-independent constant). Therefore, using \eqref{lem812eq1} and Lemma \ref{lem87}, we obtain
			\be\label{lem812eq2}
			\left|\frac{\partial}{\partial A}\mu_k(T,A,1)\right|\leq C\big(TA^{-\frac{3}{4}}|\mu_k(T,A,1)|+\sqrt{T}A^{-\frac{3}{4}}|\mu_k(T,A,1)|^2\big).
			\ee
			
			The estimate then follows standard theory of Bernoulli-type ODE.
		\end{proof}
		From now on, we will fix the following constants:			
		\begin{defn}\label{deffix}
			Let $b_0=0.1$, $b=0.8$, $b_1=0.95$, $b'=0.98$ and $b''=1.04$.\index[p2]{b@$b_0$, $b$, $b_1$, $b'$ and $b''$}
		\end{defn}

		For each $r' \in (0, 12r)$, let $D(r')$ be the closed ball, with respect to $g^{TZ'}$ in \Cref{9a}, of radius $r'$ and centered at $p_0$. For any $0 < r'' < r'$, define $D(r', r'') := \overline{D(r'') - D(r')}$.
		
		{For $\beta\in(0,1)$, let $U_{\beta} := D\big((7+\beta)r, (9-\beta)r\big) \subset Z$, and $U_{\beta, a} := U_{\beta} \cap Z_a$, $a = 1, 2$\index[p2]{Ubeta@$U_{\beta}$ and $U_{\beta, a}$, $a=1$, $2$}. }
		
		On $U_{b_0}$, consider the metric $g_{T,A} := T^2 |\nabla f_A|^{{2}} g^{TZ}$\index[p2]{gTA@$g_{T,A}$} associated to $f_A$ as in \cref{setting-partII}, and let $\dist_{T,A}$\index[p2]{distTA@$\dist_{T,A}$} be the distance induced by $g_{T,A}$. Define $\rho_{T,A}(x)\index[p2]{rhoTA@$\rho_{T,A}$} := \dist_{T,A}(x, Z - U_{b_0})$ the distance function associated with $U=Z - U_{b_0}$. Let $K_{T,A} := \{ x \in U_{b_0} : \rho_{T,A} \leq 2 \}$\index[p2]{KTA@$K_{T,A}$}.

		Applying the same arguments in the proof of Lemma \ref{agmon}, we derive the following lemma:
		\begin{lem} \label{agmonub}

            {
             There exists $A_5 = A_5(r, g^{TZ}, h^F, f)$ such that the following hold:
\begin{enumerate}
    \item Let $u$ be an eigenform of $\Delta_{T,A}$ with eigenvalues $\leq \frac{(b_1 - b_1^2)(b_0)^2 T^2 A^2 r^2}{32}$. Then there exists $C>0$ such that for $A>A_5$ and $T\gg1$
			\[
			\int_{U_{b_0} - K_{T,A}} \exp(2b_1 \rho_{T,A}) |u|^2 \leq C \int_{Z} |u|^2.
			\]
            \item Let $w_a (a = 1, 2)$ be an eigenform of $\Delta_{T,A,a}$ with eigenvalues $\leq \frac{(b_1 - b_1^2) b_0^2 T^2 A^2 r^2}{32}$. Then, if $\{ x \in U_{b_0} \cap Z_a : \rho_{T,A} \leq j \} \cap \partial D(8r) = \emptyset$ for some $j  \geq2$, there is $C>0$ such that for $A \geq A_5$ and $T \gg 1$,
			\[
			\int_{U_{b_0} \cap Z_a - K_{T,A}} \exp(2b_1 \rho_{T,A,j}) |w_a|^2 \leq C \int_{Z_a} |u|^2,
			\]
            where $\rho_{T,A,j} := \min\{\rho_{T,A}, j\}$.
\end{enumerate}
Moreover, $A_5$ depends only on the restrictions of the metrics $g^{TZ}$ and $h^F$ to the ball of radius $12r$ centered at $p_0$.}
		\end{lem}
        \begin{rem}
        For the upper bound of the eigenvalues $\frac{(b_1 - b_1^2)(b_0)^2 T^2 A^2 r^2}{32}$, the term $ b_0^2 A^2 r^2 / 8 $ corresponds to $ c_f^2 $ in the upper bound $\frac{(b - b^2)c_f^2}{4} T^2$ from \Cref{agmon}. This is used to obtain an estimate analogous to \eqref{eq-IbP-D}. Also, $ b_1 $ is playing the role of $ b $.
        \end{rem}
		\begin{proof}
            {The first part of the statement is proved in the same way as Lemma \ref{agmon}. For the second part on $Z_a$, we explain below how to adapt the proof for $a=1$, the case $a=2$ being similar.}
            
			Let $\eta\in C^\infty_c(\R)$ be a smooth bump function such that
			\begin{equation*}
				\eta(t)=
				\begin{cases}
					0, \mbox{ If $|t|<1$;}\\
					1, \mbox{ If $|t|\in(2,\infty)$},
				\end{cases}
			\end{equation*}
			and $|\eta'(t)|\leq 2,$ $\eta(t)\in[0,1],\forall t\in\R.$
			
			\def\rjt{\rho_{T,j}}
			\def\phij{\varphi_{j}}
			\def\ltaj{\lambda_{T,A,j}}
			\def\lta{\lambda_{T,A}}
			For simplicity, let $\l_{T,A}=T^2|\nabla f_A|^2,$ and
			\[\ltaj=\begin{cases}
				\lta, \mbox{ if } \rho_{T,A}<j,\\
				0, \mbox{ otherwise}
			\end{cases}.\]Clearly $|\nabla \rho_{T,A,j}|^2=\ltaj$ a.e. and $\lta\geq \ltaj.$

			Now set $\varphi_{j}=(\eta\circ\rho_{T,A})\exp(b_1\rho_{T,A ,j})$. Since $\{x\in U_{b_0}:\rho_{T,A}\leq j\}\cap {\partial D(8r)}=\emptyset,$ $\varphi_j^2w_1$ still satisfies the boundary conditions.
			
			Following the approach in the proof of Lemma \ref{agmon} and using the boundary condition when proving the analogue of \eqref{eq-IbP-D} as done in \eqref{mod21} and \eqref{mod22}, we obtain {the following analogue of \eqref{sim}:} if $A$ is large enough,
			\begin{align}
				\begin{split}
					\label{sim1}
					\int_{U_{b'}\cap Z_1}\lta (\eta\circ\rho_{T,A})^2|w_1|^2\exp(2b_1\rho_{T,j}) \leq C(b_1)\int_{K_T}|w_1|^2\lta\exp(2b_1\rho_{T ,j}).
			\end{split}\end{align}
			The estimate then follows for $w_1$. 
		\end{proof}
		
	Let $x\in U_{b'}$, then for $x'\in \partial U_{b_0}$ such that $\rho_ {T,A}(x)=\dist_ {T,A}(x,x')$, we get from 	(2) in Proposition \ref{prop84} that $\rho_ {T,A}(x) \geq T|f_A(x)-f_A(x')|$. Using the fact that $q_A$ is non-decreasing and its expression on $[7r;7.98r]$ (see Section \ref{assotwo}) we find that there is a constant $C>0$ such that  $|f_A(x)-f_A(x')|\geq q_A(7.98r)-q_1(7.1r) - C = 0.9504Ar^2/2-C$. Thus, if $\eta\in C^\infty(Z)$ is supported in {$U_{b'}$}, by  Lemma \ref{agmonub}, we get that if $A$ is large enough and $u$ is as in Lemma \ref{agmonub}, 
    \be\label{eq93} \int_Z|\eta u|^2\leq Ce^{-(0.4b+0.6b_1)TAr^2}\int_Z|u|^2.\ee 
    A similar statements hold for $w_1,w_2$ as in Lemma \ref{agmonub}.
          
		\begin{lem}\label{lem813}
			There exists $A_5'=A_5'(r,g^{TZ},h^F,f)$, $c=c(r,g^{TZ},h^F,f)$ and $C_k=C_k(g^{TZ},h^F,f)$ such that if $A\geq A_5'$ and $T$ is large enough:
            \begin{enumerate}
                \item For $k_0\leq k\leq k_1+k_2$, \be\label{lem8131}C_ke^{-2b''TAr^2}\leq\l_k(T,A)\leq C_ke^{-(b_1+b)TAr^2/2}.\ee 
                \item For $k>k_1+k_2$,
                \be\label{lem8132}\l_k(T,A)\geq C_ke^{-c\left(T+TA^{\frac{1}{4}}\right)}.\ee
            \end{enumerate}
			Moreover,  $ c $ and $A_5'$ depend only on the restrictions of the metrics $ g^{TZ} $ and $ h^F $ to the ball of radius $ 12r $ centered at $ p_0 $.
		\end{lem}
		\begin{proof}
			We use the same notation as in the proof of Lemma \ref{lem810}.

			

			We  first prove \eqref{lem8132} and the second inequality in \eqref{lem8131}.
            
			We claim that for some $(T, A, r)$-independent constant $C'$, if 
			\[
			\lambda_k(T, A) \leq C' T^2 A^2 r^2 \quad \text{or} \quad |\tilde{\lambda}_k(T, A)| \leq C' T^2 A^2 r^2,
			\]
			then 
			\[
			|\lambda_k(T, A) - \tilde{\lambda}_k(T, A)| \leq C_k e^{-\frac{(b + b_1) T A r^2 }{2}}.
			\]
			Hence, either the assumption of the claim does not hold, in which case  \eqref{lem8132} and the second inequality in \eqref{lem8131} are automatically true, or it hold, in which case these equations follow from the claim and Lemma \ref{lem811}.
			
			To prove the claim, let $W_k$ be the space spanned by the $j$-th eigenform of $\Delta_{T,A,1} \oplus \Delta_{T,A,2}$ for all $j \leq k$, and assume that $|\tilde{\lambda}_k(T, A)| \leq C' T^2 A^2 r^2$. Let $\eta \in C^\infty(\mathbb{R})$ be such that $0 \leq \eta \leq 1$, $\eta|_{(-\infty, 7r + b'r) \cup (9r - b'r, \infty)} \equiv 1$, and $\eta|_{\Big(7r + \frac{(1 + b')r}{2}, 9r - \frac{(1 + b')r}{2}\Big)} \equiv 0$. We can view $\eta$ as a smooth function on $Z$. Then, for each $u \in W_k$, we can regard $\eta u$ as a smooth form on $Z$. Moreover, by \eqref{eq93},
			
			\[\ba
			&\ \ \ \ \frac{\int_Z |D_{T,A} \eta u|^2}{\int_Z |\eta u|^2} \leq \frac{\int_Z \eta^2 |D_{T,A} u|^2 + 2\Big|\big\langle c(\nabla \eta) u, \eta D_{T,A} u \big\ran\Big| + \big|c(\nabla \eta) u\big|^2}{\int_Z |\eta u|^2} \\&\leq \tilde{\lambda}_k(T, A) + C T^2 A^2 r^2 e^{-(0.6b_1+0.4b) T A r^2},
\ea			\]
			where $D_{T,A}$ in the second term above should be understood as $D_{T,A,1} \oplus D_{T,A,2}$. Thus, {by a standard Rayleigh quotient argument, we get}
			\[
			\lambda_k(T, A) \leq \tilde{\lambda}_k(T, A)+ C e^{-\frac{(b_1 + b) T A r^2} {2}}
			\]
			if $A$ is large. 
						
			Similarly, we can show $\tilde{\lambda}_k(T, A) \leq \lambda_k(T, A) + C e^{-(b_1 + b) T^2 A^2 r^2 / 2}$ if $\lambda_k(T, A) \leq C' T^2 A^2 r^2$. {Hence the claim is proved.}
			
			{We now prove the first inequality in \eqref{lem8131}.} Let $b_2 = 1.02 < b''$. Take $b''' = 2 - \sqrt{b_2}$, then $b''' > b'$. Let $\eta_1 \in C^\infty(\mathbb{R})$, such that $0 \leq \eta_1 \leq 1$, $\eta_1 \equiv 1$ on $(-\infty, 7r + b'''r)$ and $\eta_1 \equiv 0$ on $(9r - b'''r, \infty)$. One can regard $\eta_1$ as a smooth function on $Z$, and $\eta_2 := 1 - \eta_1$.
			
			As in the proof of Lemma \ref{lem810}, let $u_{k, T, A}$ be a unit eigenform of the $k$-th eigenvalue $\mu_k(T, A)$ with respect to $D_{k, T, A}$. Then as in \eqref{lem810eq1},
			
			\be\label{lem813eq11}
			\frac{\partial}{\partial A} \mu_k(T, A) = \int_Z \left\langle \Big(\frac{\partial}{\partial A} D_{T, A}\Big) u_{k, T, A}, u_{k, T, A} \right\rangle
			\ee
			
			Note that $\Big[V_{T, A}, \frac{\partial}{\partial A} q_A \Big]=\Big[V_{T, A}, C + \frac{\partial}{\partial A} q_A \Big]$ for any constant $C$. In the following equation, our choice of $C$ is designed to minimize $\sup_{\mathrm{Supp}(\eta_1)} |C + \frac{\partial}{\partial A} q_A|$. By Lemma \ref{agmonub},
			
			\begin{align}\begin{split}\label{lem813eq12}
					& \left|\int_Z \bigg\langle \Big(\frac{\partial}{\partial A} D_{T, A}\Big) u_{k, T, A}, \eta_1 u_{k, T, A} \bigg\rangle \right| = \left|\int_Z \bigg\langle T \Big[V_{T, A}, 2r^2 - \frac{b_2 r^2}{2} + \frac{\partial}{\partial A} q_A \Big] u_{k, T, A}, \eta_1 u_{k, T, A} \bigg\rangle \right| \\
					& \leq 2 \left(\int_Z \eta_1 |V_{T, A} u_{k, T, A}|^2 \right)^{\frac{1}{2}} \left(\int_Z \eta_1 \bigg|T^2 \Big(\frac{r^2}{2} - \frac{b_2 r^2}{4} + \frac{\partial}{\partial A} q_A \Big)^2 u_{k, T, A}\bigg|^2 \right)^{\frac{1}{2}} + C e^{-b_1 T A r^2} \\
					& \leq b_2 r^2 T \left(\int_Z \eta_1 |V_{T, A} u_{k, T, A}|^2 \right)^{\frac{1}{2}} \left(\int_Z \eta_1 |u_{k, T, A}|^2 \right)^{\frac{1}{2}} + C e^{-b_1 T A r^2}.
			\end{split}\end{align}
			Here, the term $C e^{-b_1 T A r^2}$ arises from the following considerations. Note that $V_{T, A} \eta_1 u_{k, T, A} = \hat{c}(\nabla \eta_1) u_{k, T, A} + \eta_1 V_{T, A} u_{k, T, A}$ and that $\hat{c}(\nabla \eta_1)$ is supported within $(7r + b'r, 9r - b'r)$. Therefore, by \eqref{eq93},
			\[
			\int_Z \big\langle \hat{c}(\nabla \eta_1) u_{k, T, A}, u_{k, T, A} \big\rangle
			\]
			yields the term $C e^{-b_1 T A r^2 / 2}$.
			
			Similarly, 
			\begin{align}\begin{split}\label{lem813eq13}
					&\ \ \ \ \left|\int_Z\bigg\langle \Big(\frac{\partial}{\partial A}D_{T,A}\Big)u_{k,T,A} ,\eta_2u_{k,T,A}\bigg\rangle\right|\leq b_2r^2T\left(\int_Z \eta_2|V_{T,A}u_{k,T,A}|^2\right)^{\frac{1}{2}}\left(\int_Z \eta_2|u_{k,T,A}|^2\right)^{\frac{1}{2}}+Ce^{-b_1TAr^2}.
			\end{split}\end{align}
			
			It follows from \eqref{lem813eq11}-\eqref{lem813eq13} and the Cauchy-Schwarz inequality that
			\be
			\left|\frac{\partial}{\partial A}\mu_k(T,A)\right|^2\leq b_2r^2T|\mu_k(i,T,A)|+Ce^{-b_1TAr^2}.
			\ee
			By standard ODE theory, for $k_0\leq k\leq k_1+k_2$, if $A$ is large,
			\[\sqrt{\l_{k}(T,A)}=|\mu_k(T,A)|\geq C_ke^{-cT}e^{-b_2TAr^2}\geq C_ke^{-b''TAr^2}.\]
            
            The proof is complete.\end{proof}
		\subsubsection{Schauder norm of $(\l-\sqrt{t}D_{T,A})^{-1}$}
		\textbf{From now on, we assume that $T\gg1$.} 
		
To estimate the Schauder norm of $ (\lambda - \sqrt{t}D_{T,A})^{-1} $, we first use the domain monotonicity of eigenvalues (see \cite[Lemma 3.1]{tachizawa1992eigenvalue}) to show  {in Proposition \ref{prop96} that under Assumption \ref{ass91} below, then eigenvalues of $(\Delta_{T,A})_y$ satisfy the following.}
\begin{cond}\label{schauder-cond}\ When $T\geq1$ is large enough,
    \begin{itemize}
        \item There is a $(T,y)$-independent constant $C$ (possibly depending on $A$) such that \be\label{schauder-cond1} k_1:=\sup\{k:\l_{k}(T,A,y)\leq 1\}\leq C\sqrt{T}\ee 
        \item For any $k>k_1$, there exists $(T,A,y)$-independent $C_1,\nu_0>0$, s.t. \be\label{schauder-cond2}\l_k(T,A,y)\geq C_1k^{\nu_0}.\ee
    \end{itemize}
\end{cond}

		Firstly, we assume that the following holds.
		\begin{assum}\label{ass91}
			Let $\Delta_y\index[p2]{Deltay@$\Delta_y$}:=\Delta+|\nabla f_y|^2$ and $\Delta^{\mathrm{N}}_y$\index[p2]{DeltayN@$\Delta^{\mathrm{N}}_y$} be the restriction of $\Delta_y$ on $Z-\cup_{l=0}U_l$ with Neumann boundary conditions. Let $\l_k(y)$\index[p2]{lambdaky@$\l_k(y)$} be the $k$-th eigenvalue of $\Delta^{\mathrm{N}}_y$. Then there exists $y$-independent constant $c>0$ and $\nu\in(0,1)$, s.t. if $\l_k(y)\neq 0$, then $\l_k(y)\geq ck^{\nu}.$
		\end{assum}
		
		Let $(\Delta^{\mathrm{N}}_{T,A})_y$\index[p2]{DeltaTAyN@$(\Delta^{\mathrm{N}}_{T,A})_y$} be the restriction of $(\Delta_{T,A})_y$ on $Z-\cup_{l=0}U_l$ with Neumann boundary conditions. Let $\l^{\mathrm{N}}_k(T,A,y)$\index[p2]{lambdakTAyN@$\l^{\mathrm{N}}_k(T,A,y)$} be the $k$-th eigenvalue of $(\Delta^{\mathrm{N}}_{T,A})_y$.
		
		If $T\geq 1$ is large enough, using \eqref{bardelta} and noting that $f_y=f_{A,y}$ outside of the union of the $U_l$, $(\Delta^{\mathrm{N}}_{T,A})_y\geq \Delta^{\mathrm{N}}_y$, that is, for any $w\in \Omega^{\bullet}_c\big(Z-{\cup_{l=0}U_l}\big)$ satisfying Neumann boundary condition, we have
		\begin{equation}
        \label{DeltaN_TA>DeltaN}
		          \left((\Delta^{\mathrm{N}}_{T,A})_yw,w\right)_{L^2}\geq \left(\Delta^{\mathrm{N}}_yw,w\right)_{L^2}.
		\end{equation}

		Thus, if $ \lambda^{\mathrm{N}}_k(T,A,y) \neq 0 $, we also have  
\begin{equation}\label{eq100}  
\lambda^{\mathrm{N}}_k(T,A,y) \geq c k^{\nu}.  
\end{equation}  

This implies that the eigenvalues of $ (\Delta^{\mathrm{N}}_{T,A})_y $ satisfy \Cref{schauder-cond}.  

Next, let $ (\Delta^{\mathrm{N}}_{T,A,l})_y $\index[p2]{DeltaTAylN@$(\Delta^{\mathrm{N}}_{T,A,l})_y$} denote the restriction of $ (\Delta_{T,A})_y $ to $ U_l $ for $ l \neq 0 $ with Neumann boundary conditions. By \Cref{9a}, it follows that the eigenvalues of $ (\Delta^{\mathrm{N}}_{T,A,l})_y $ satisfy \Cref{schauder-cond}.  

Similarly, let $ (\Delta^{\mathrm{N}}_{T,A,0,1})_y $\index[p2]{DeltaTAy01N@$(\Delta^{\mathrm{N}}_{T,A,0,1})_y$} denote the restriction of $ (\Delta_{T,A})_y $ to $ U_0 - D(12r) $ with Neumann boundary conditions. Again, by \Cref{9a}, the eigenvalues of $ (\Delta^{\mathrm{N}}_{T,A,0,1})_y $ satisfy \Cref{schauder-cond}.  

Now, consider $ (\Delta^{\mathrm{N}}_{T,A,0,2})_y $\index[p2]{DeltaTAy02N@$(\Delta^{\mathrm{N}}_{T,A,0,2})_y$}, the restriction of $ (\Delta_{T,A})_y $ to $ D(4r) \subset U_0 $ with Neumann boundary conditions. We establish the following lemma:  

\begin{lem}  \label{lem911}
The eigenvalues of $ (\Delta^{\mathrm{N}}_{T,A,0,2})_y $ for $ |y| \leq \delta^2 $ satisfy \Cref{schauder-cond}.  
\end{lem}  

\begin{proof}  
We first analyze the one-dimensional model.  

Define $ f^{\mathbb{R}}_y(s) := \frac{1}{3}s^3 - ys $ for $ |y| \leq \delta^2 $, and let $ \Delta_{T,y}^{\mathbb{R},\mathrm{N}} $ be the Neumann Witten Laplacian on $ \Omega^{\bullet}(-4r,4r) $ associated with $ T f_y^{\mathbb{R}} $. Let $ \Delta_{T,y,1}^{\mathbb{R},\mathrm{N}} $ be the restriction of $ \Delta_{T,y}^{\mathbb{R},\mathrm{N}} $ to $ (-2\delta,2\delta) $ with Neumann boundary conditions.  

Introducing the new coordinate $ \tilde{s} = T^{1/3} s $, we transform $ \Delta_{T,y,1}^{\mathbb{R},\mathrm{N}} $ into $ \tilde{\Delta}_{T,y,1}^{\mathbb{R},\mathrm{N}} $, which satisfies  
\begin{equation}  
\tilde{\Delta}_{T,y,1}^{\mathbb{R},\mathrm{N}} \geq T^{2/3} \left( -\frac{\partial^2}{\partial \tilde{s}^2} - L \right)  
\end{equation}  
for some $ (T,y) $-independent constant $ L > 0 $. That is, for any $ u \in C^\infty(-2\delta T^{1/3},2\delta T^{1/3}) $ with Neumann boundary conditions,  
\[
\int_{-2\delta T^{1/3}}^{2\delta T^{1/3}} \langle \tilde{\Delta}_{T,y,1}^{\mathbb{R},\mathrm{N}} u, u \rangle d\tilde{s} \geq T^{3/2} \int_{-2\delta T^{1/3}}^{2\delta T^{1/3}} \left\langle \left( -\frac{\partial^2}{\partial \tilde{s}^2} - L \right) u, u \right\rangle d\tilde{s}.  
\]  

For the one-dimensional Neumann Laplacian $ -\frac{\partial^2}{\partial \tilde{s}^2} $ on $ (-2\delta T^{1/3},2\delta T^{1/3}) $, there are at most $ C\sqrt{L\delta}T^{1/3} $ eigenvalues $ \leq L $. Thus, the eigenvalues of $ \Delta_{T,y,1}^{\mathbb{R},\mathrm{N}} $ satisfy \Cref{schauder-cond}.  

Since $ (f^{\mathbb{R}}_y)'(s) = 0 $ has at most two solutions $ \pm\sqrt{y} $, the restriction $ \Delta_{T,y,2}^{\mathbb{R},\mathrm{N}} $  of $ \Delta_{T,y}^{\mathbb{R},\mathrm{N}} $ to $ (-1,1) \setminus (-2\delta,2\delta) $ with Neumann boundary conditions satisfies $ \Delta_{T,y,2}^{\mathbb{R},\mathrm{N}} \geq -\frac{\partial^2}{\partial s^2} $, and consequently, the eigenvalues of $ \Delta_{T,y,2}^{\mathbb{R},\mathrm{N}} $ satisfy \Cref{schauder-cond}.  

By the domain monotonicity of eigenvalues,  the eigenvalues of $ \Delta_{T,y,2}^{\mathbb{R},\mathrm{N}} $ satisfy \Cref{schauder-cond}.  

{Using the formula for $f_{A,y}={f}_y$ on $D(4r)$ and Lemma \ref{nocritical}, we can} extend this argument to higher dimensions to conclude, as in the one-dimensional case, that the eigenvalues of $ (\Delta^{\mathrm{N}}_{T,A,0,2})_y $ for $ |y| \leq \delta^2 $ satisfy \Cref{schauder-cond}.  
\end{proof}

{The following lemma primarily deals with the restriction of $ (\Delta_{T,A})_y $ to $ D(4r,12r) \subset U_0 $.}

\begin{lem}  \label{lem9.12}
Let $ (Z, g^{TZ})$ be an $ n $-dimensional complete Riemannian manifold, and let $ f_y: Z \to \mathbb{R} $ be a family of smooth function parametrized by $y,|y|\leq \delta^2$ such that $(Z,g^{TZ},f_y)$ is strongly tame. Moreover, {assume that} the spectrum of $ \Delta + |\nabla f_y|^2 $ satisfies the estimates in \Cref{ass91} and that there exist finitely many disjoint embeddings $ \gamma_j: [0,1] \to Z $ for $ j = 1, 2, \dots, d $ such that:  
\begin{itemize}  
    \item 

All critical points of $ f_y $ are either contained in $\cup_j \gamma_j\big((0,1)\big) $, or if a critical point $ p $ is not in the union of the $ \gamma_j $'s, there exists a sufficiently small, $ y $-independent neighborhood $ U $ around $ p $ such that on $ U $, the function behaves as in \eqref{g-morse-lem1} or as in \eqref{g-morse-lem}.
    \item There exists pairwise disjoint neighborhoods $ N_j $ of $ \gamma_j $ and a diffeomorphisms 
    \be\label{tubular}  
    N_j \cong [0,1] \times B_1(0),  
    \ee  
    where $ B_r(0) $ is an $(n-1)$-dimensional ball of radius $ r $ centered at $ 0 $. Moreover, in some local coordinates $ \{x_2, \dots, x_n\} $ in the $B_1(0)$ components, the Hessian $ \left(\frac{\partial^2 f}{\partial x_j \partial x_l}\right)_{2\leq j,l\leq n} $ is non-degenerate along $ \gamma_j $. 
    \item $f_y|_{N_j}$ is independent of $y$.  
\end{itemize}  
Then the eigenvalues of $ \Delta_T $\index[p2]{DeltaT@$\Delta_T$}, the Witten Laplacian associated with $Tf$ on $ Z $,  satisfy \Cref{schauder-cond}. 
\end{lem}

\begin{proof}  
First, we assume that there is only one curve $ \gamma = \gamma_1 $ as in the statement and that all critical points of $ f_y$ lie within $ \gamma $.

First, suppose that the metric $ g^{TZ} $ is of product type on $ N = N_1 $ with respect to the coordinates \eqref{tubular}. Identifying $ N $ with $ [0,1] \times B_1(0) $, we define for some $ c > 0 $ the smaller region  
\[
N_T = [0,1] \times B_{c  T^{-1/2}}(0).
\]  
Let $ \Delta_{T,1} $ be the restriction of $ \Delta_T $ to $ N_T $ with Neumann boundary conditions, and let $ \Delta_{T,2} $ be the restriction of $ \Delta_T $ to $ Z \setminus N_T $ with Neumann boundary conditions. By our assumptions, there exists $ c > 0 $ such that for sufficiently large $ T $, we have by \eqref{bardelta}
\[
\Delta_{T,2} \geq \Delta+|\nabla f_y|^2.
\]  

Note that $f|_{N}$ is independent of $y$.
On $ N_T $, we introduce the rescaled coordinates $ \tilde{x}_j = \sqrt{T} x_j $ for $ j \geq 2 $, while keeping the $[0,1]$ component unchanged. In these coordinates, $ \Delta_{T,1} $ transforms into $ \tilde{\Delta}_{T,1} $ on $[0,1]\times B_c(0)$, which satisfies  for some $T$-independent constant $L>0$,
\[
\tilde{\Delta}_{T,1} \geq \left(-\frac{\partial^2}{\partial s^2} - LT\right) + T \Delta^{\mathbb{R}^{n-1}}.
\]  
Since the Neumann Laplacian $ -\frac{\partial^2}{\partial s^2} $ on $ [0,1] $ has at most $ C \sqrt{LT} $ eigenvalues below $ LT $, it follows that $ \Delta_{T,1} $ satisfies \Cref{schauder-cond}.  

By domain monotonicity, we conclude that $ \Delta_T $ also satisfies \Cref{schauder-cond}.

Now consider the case where $ g^{TZ} $ is not a product metric on $ N $. We deform $ g^{TZ} $ to a new metric $ g^{TZ,\prime} $ such that $ g^{TZ,\prime} $ is of product type on $ N $ and satisfies  
\[
C^{-1} g^{TZ} \leq g^{TZ,\prime} \leq C g^{TZ}  
\]  
for some $ C > 1 $. Let $ \lambda_k(T) $ denote the $ k $th eigenvalue of $ \Delta_T $, and let $ \lambda_k'(T) $ denote the $ k $th eigenvalue of the Witten Laplacian associated with $ g^{TZ,\prime} $. Then there exists a constant $ C' >1 $ such that  
\[
(C')^{-1} \lambda_k'(T) \leq \lambda_k(T) \leq C' \lambda_k'(T).
\]  
Thus, the result follows in the case where there is only one $ \gamma_j $ and that all critical points of $ f_y$ lie within $ \gamma $.   

{Now, if there is more than one $ \gamma_j $, we use the same reasoning on each $\cup_j N_j$ and still conclude by domain monotonicity. Lastly, if $f_y$ has critical points outside $\cup_j\gamma_j$, as these critical point are either Morse or birth-death, we can study the restrictions of the Laplacian (with Neuman boundary condition) respectively inside and outside  of neighborhoods of these critical points as done above \Cref{lem911} and in the proof of \Cref{lem911} to conclude.}
\end{proof}

        Combining \Cref{lem9.12} and with \Cref{threecrits} {and \Cref{threecrit-rmk}}, we obtain that  for $ |y| \leq \delta^2 $, we have
		\begin{cor}\label{prop96}
			If Assumption \ref{ass91} holds, then eigenvalues of $(\Delta_{T,A})_y$ satisfy \Cref{schauder-cond}.
		\end{cor}

		\def\bfn{{\mathbf{n}}}	
		Next we would like to estimate Schauder $\bfn$-norm of $(\l-\sqrt{t}D_{T,A})^{-1}$.
		
		\begin{defn}[Schauder $\bfn$-norm]\label{schauder} For $\bfn \geq 1$ and an operator $B$ on a Hilbert space, the Schauder $\bfn$-norm of $B$ is defined as follows,
			$$
			\|B\|_\bfn \index[p2]{norm@norms!Schauder norm $\vert\vert\cdot\vert\vert_\bfn$}=\bigg(\operatorname{Tr}\Big[\left(B^* B\right)^{\bfn / 2}\Big]\bigg)^{\frac{1}{\bfn}} .
			$$
		\end{defn}
		
		Let $\|B\|_{\infty}$ be the operator norm of $B$. It is standard to see that these norms satisfy the H\"older's inequality and Minkowski inequality: for $\bfn_1, \bfn_2, \bfn_3 \in[1,+\infty]$ with $1 / \bfn_1+1 / \bfn_2=1 / \bfn_3$, we have
		\be\label{holdsch}
		\|B_1 B_2\|_{\bfn_3} \leq\|B_1\|_{\bfn_1}\|B_2\|_{\bfn_2} \mbox{ and } 	\|B_1 +B_2\|_{\bfn}\leq 	\|B_1\|_{\bfn}+ \|B_2\|_{\bfn} .
		\ee
		
		Moreover, if $B$ is of finite rank, we have
		\be\label{schfinite}
		\|B\|_\bfn \leq\big(\operatorname{rk}(B)\big)^{\frac{1} {\bfn}}\|B\|_{\infty} .
		\ee
		
		We have the following lemma for the Schauder norm, whose proof follows from H\"older's inequality and Minkowski's inequality trivially:
		\begin{lem}\label{triviallem}
			The following estimate holds for some constant $C$ that depends only on $\bfn>1$: $$\|B_1^{\bfn+1}-B_2^{\bfn+1}\|_1\leq C\|B_1-B_2\|_\infty\max\big\{\|B_1\|_\bfn^\bfn,\|B_2\|_\bfn^\bfn\big\}.$$
		\end{lem}

 Recall that $D_{T,A,a}$ and $\Delta_{T,A,a}$, for $a=1$, $2$, have been defined before Lemma \ref{lem811}.
		Let 
        \be\label{dtabd}D_{T,A}^{\bd}\index[p2]{DTAbd@$D_{T,A}^{\bd}$}:=D_{T,A,1}\oplus D_{T,A,2} \mbox{ and } \Delta_{T,A}^{\bd}\index[p2]{DeltaTAbd@$\Delta_{T,A}^{\bd}$}:=\Delta_{T,A,1}\oplus \Delta_{T,A,2}.\ee
        
		\begin{lem}\label{lem97}
			If $\bfn\geq1$ is large enough and Assumption \ref{ass91} holds, for $t\geq \tau>0$, {and $|\Im(\l)|=1$}
			\[ \left\| (\lambda - \sqrt{t} D_{T,A})^{-1} \right\|_{\bfn} \leq C \left( \sqrt{T} \right)^{\frac{1}{\bfn}} |\lambda| \left( 1 + \frac{1}{\sqrt{t}} \right) \]
			for some $(T,\l)$-independent but $(A,\tau,\bfn,y)$-dependent positive constant $C.$ 
			
			Moreoever, a similar estimate holds for $(\l-D_{T,A}^{\bd})^{-1}.$
		\end{lem}
		\begin{proof}

			Let $\l\in\C$ such that $|\Im(\l)|=1$, then {for $s\in\R$,} 
			\be\label{trivial}
			|(\l-s)^{-1}|\leq \min\left\{1,\frac{|\l|}{|s|}\right\} .
			\ee 
			
		We claim that the lemma follows from \eqref{schfinite} and \eqref{trivial}.

			\def\cp{{\check{p}}}
			To prove the claim, let $\check{p}$ be the orthogonal projection to the space generated by the $k$-th eigenform of $D_{T,A}$ for all $k \leq k_1$ ($k_1$ being described in \Cref{schauder-cond}). Let $B_1=((\l-\sqrt{t}D_{T,A})^{-1})\cp$ and $B_2=(\l-\sqrt{t}D_{T,A})^{-1}(1-\cp).$
			
			By \eqref{schauder-cond1}, \eqref{schfinite} and \eqref{trivial},  for any $\bfn>1$, we have
			\[
			\|B_1\|_\bfn \leq C \left(\sqrt{T}\right)^{\frac{1}{\bfn}} \leq C \left(\sqrt{T}\right)^{\frac{1}{\bfn}} |\lambda| \left(1 + \sqrt{t}^{-1}\right).
			\]
			
			By \eqref{schauder-cond2} and \eqref{trivial}, if $\bfn > \frac{1}{\nu_0}$, we have	
			\[
			\|B_2\|_\bfn \leq C |\lambda| \left(1 + \sqrt{t}^{-1}\right) \leq C \left(\sqrt{T}\right)^{\frac{1}{\bfn}} |\lambda| \left(1 + \sqrt{t}^{-1}\right).
			\]
			The lemma then follows from Minkowski inequality.
		\end{proof}
		
		\def\osc{{\mathrm{osc}}}

		\subsubsection{Estimate of  resolvents when $A$ is large}\label{sec913}

		 In this section, one of the key point to control the resolvant of $\Delta_{T,A}$ is to use Agmon estimates and the fact that as $A$ grows, the Agmon distance form $\p D(7r)$ to $\partial Z_1$ grows.

		By Lemma \ref{sixcrits}, there exists $A_0=A_0(r)$, such that when $|y|<\delta^2$ and $A\geq A_0$, $f_{A}$ is generalized Morse. Let \be\label{A6} A_6=2\max\{A_0,A_4,A_5,A_5'\},\ee 
where $A_4$, $A_5$ and $A_5'$ are given as in Lemmas \ref{lem811}, \ref{agmonub} and \ref{lem813} respectively.
        
        \textbf{In this subsection, we require that $A\geq A_6.$}
		
		
		
		\def\contro{1}
		\if\contro0
		Let $f_{l,A}$ be a generalized Morse function on $\R^{n+1}$ given by
		\[f_{l,A}(u_0,\cdots,u_n)=-u_0^2-\cdots -u_{i_l-1}^2+u_{i_l}^2+\cdots u_n^2,l\geq 1;\]
		\[f_{0,A}=f_{A,y}.\]
		(So $f_{l,A}$ is independent of $A$ if $l\geq1.$)
		
		Let $\Delta_{l,T,A}$ be the Witten Laplacian for $f_{l,A}$ w.r.t the standard metric $g$, acting on $ \Omega^{\bullet}(\R^{n+1};\Cb^m)$ (Here $m=\rank(F)$). Then Lemma \ref{agmon} implies that there exists $(T,A)$-independent $c,c'>0$, such that if $w$ is an eigenform of $\Delta_{l,T,A}$ with respect to eigenvalue $\l\leq cT^2$, then 
		\be\label{Agmon1}\|w\|_{L^2(\R^{n+1}-D(40r))}\leq e^{-c'T}\|w\|_{L^2(\R^{n+1})}.\ee
		
		Let $\eta_{b'}\in C_c^\infty(\R^{n+1})$, s.t. $\eta\equiv1$ on $D(50r)$ and $\eta\equiv0$ if on $D(55r,\infty)$. Then for any eigenform $w$ of $\Delta_{l,T,A}$, $\eta w$ could be view as a differential form supported inside $U_l.$
		\fi
		
		Recall that we have fix $b=0.8$, $b_1=0.95$ and $b'=0.98$ in Definition \ref{deffix}.
		Let $\eta_{b'}\in C^\infty(\R)$, such that $\eta_{b'}\equiv 1$ on $(-\infty,7r+b'r)\cup(9r-b'r,\infty)$ and $\eta_{b'}\equiv 0$ on $\Big(7r+\frac{(1+b')r}{2},9r-\frac{(1+b')r}{2}\Big).$ Then $\eta_{b'}$ could be viewed as a smooth function on $Z.$

        Let $E_{T,A}\subset \Omega^{\bullet}(Z,F)$ be the space generated by $\eta w$, for all eigenform $w$ of $\Delta_{T,A,1}\oplus\Delta_{T,A,2}$ with eigenvalue $\leq \frac{(b_1-b_1^2)(b_0)^2T^2A^2r^2}{32}$, and $E^\perp_{T,A}$\index[p2]{ETA@$E_{T,A}$ and $E^\perp_{T,A}$} be its orthogonal complement. Let        
         \be\label{pta}
             p_{T,A}:  L^2 \Omega^{\bullet}(Z,F) \to E_{T,A} \quad \text{and}\quad
             p_{T,A}^\perp:  L^2 \Omega^{\bullet}(Z,F) \to E_{T,A}^\perp\index[p2]{pTA@$p_{T,A}$ and $p^\perp_{T,A}$}
         \ee be the orthogonal projections.

		For $w\in  \Omega^{\bullet}_c(Z,F)$, we define a norm $\|\cdot\|_{H^1_{T,A}}$ as follows:
		\be\label{h1ta}\|w\|_{H^1_{T,A}}^2\index[p2]{norm@norms!$H^1$ norm $\vert\vert \cdot \vert\vert_{H^1_{T,A}}$}:=\|w\|_{L^2}^2+\|D_{T,A}w\|_{L^2},\ee
		where $\|\cdot\|_{L^2}$ is the $L^2$-norm induced by $g^{TZ}$ and $h^F$.
		Let $H^1_{T,A}( \Omega^{\bullet}(Z,F))$ be the completion of $ \Omega^{\bullet}_c(Z,F)$ with respect to $\|\cdot\|_{H^1_{T,A}}.$ Let $H^{-1}_{T,A}( \Omega^{\bullet}(Z,F))$ be the anti-dual of $H^{1}_{T,A}( \Omega^{\bullet}(Z,F))$ with a norm $\|\cdot\|_{H^{-1}_{T,A}}$ associated with  $\|\cdot\|_{H^1_{T,A}}$. Finally, for convenience we set $\|\cdot\|_{H^0_{T,A}}=\|\cdot\|_{L^2}$.
		
		Then we have the following continuous dense embeddings with norms small than 1,
		\[H^1_{T,A}( \Omega^{\bullet}(Z,F))\subset L^2( \Omega^{\bullet}(Z,F))\subset H^{-1}_{T,A}( \Omega^{\bullet}(Z,F)).\]
		
		For simplicity, we let $H^1:=H^1_{T,A}( \Omega^{\bullet}(Z,F)),$ $H^0:=L^2( \Omega^{\bullet}(Z,F))$ and $H^{-1}:=H^{-1}_{T,A}( \Omega^{\bullet}(Z,F)).$

		\begin{prop}\label{agmon2}
			For any $ w \in E_{T,A} $,  {$ w' \in E^\perp_{T,A} \cap H^1 $ and $ w'' \in E^\perp_{T,A} \cap H^2 $} :
			\[ \|p_{T,A}^\perp D_{T,A}w\|_{L^2} \leq C e^{-\frac{b T A r^2}{2}} \|w\|_{L^2} \mbox{ and }
			\|p_{T,A} D_{T,A}w'\|_{L^2} \leq C e^{-\frac{b T A r^2}{2}} \|w'\|_{L^2}; \]
			\[ \|p_{T,A}^\perp \Delta_{T,A}w\|_{L^2} \leq C e^{-\frac{b T A r^2}{2}} \|w\|_{L^2} \mbox{ and } \|p_{T,A} \Delta_{T,A}w''\|_{L^2} \leq C e^{-\frac{b T A r^2}{2}} \|w''\|_{L^2}. \]
			Here $ C  $ is a $(T,A,y)$-independent constant. 
			
			A similar statements holds for $D_{T,A}^\bd$ and $\Delta_{T,A}^\bd$.
		\end{prop}
		
		\begin{proof}
			We may as well assume that $w$ has support inside $Z_1$. Then $w=\eta_{b'} w_0$, where $w_0\in \Omega^{\bullet}(Z_1,F)$ lies in the space spanned by eigemforms of $\Delta_{T,A,1}$ with eigenvalues $\leq \frac{(b_1-b_1^2)(b_0)^2T^2A^2r^2}{32}$.
			
			Then by \eqref{eq93}, \be\label{Agmon11}
			\|w_0\|_{L^2}^2\leq2\|w\|^2_{L^2}.
			\ee
			
			It is clear that $\eta_{b'} D_{T,A,1}w_0\in E_{T,A}$,\ so
			\[p_{T,A}^\perp D_{T,A}w=p_{T,A}^\perp c(\nabla \eta_{b'})w_0+p_{T,A}^\perp\eta_{b'} D_{T,A,1}w_0=p_{T,A}^\perp c(\nabla \eta_{b'})w_0. \]

{Moreover, by \eqref{eq93}
\be\label{eq933}
\int_{\supp(\nabla \eta_{b'})}|w_0|^2\leq Ce^{-(0.4b+0.6b_1)TAr^2}\int_{Z}|w_0|^2.
\ee}
Thus by \eqref{Agmon11} and \eqref{eq933} 
\[\ba&\ \ \ \ \|p_{T,A}^\perp D_{T,A}w\|_{L^2}=\|p_{T,A}^\perp c(\nabla \eta_{b'})w_0\|_{L^2}\leq \|c(\nabla \eta_{b'})w_0\|_{L^2}\\
&\leq Ce^{-(0.2b+0.3b_1)TAr^2}\|w_0\|_{L^2}\leq C'e^{-(0.2b+0.3b_1)TAr^2}\|w\|_{L^2}.\ea\]

			\def\si{0}
			\if\si0
			Similarly, we can see
			\[\Delta_{T,A}w=\Delta_{T,A}\eta_{b'} w_0=\eta_{b'}\Delta_{T,A,1}w_0+2c(\nabla\eta_{b'})D_{T,A,1}w_0+\Delta(\eta_{b'})w_0.\]
			{By \eqref{eq93}} $$\|c(\nabla\eta_{b'})D_{T,A,1}w_0\|^2_{L^2}\leq Ce^{-(b_1+b)TAr^2/2 }\|D_{T,A,1}w_0\|^2_{L^2}\leq  e^{-(b_1+b)TAr^2/2 } {\frac{(b_1-b_1^2)(b_0)^2}{32}}T^2A^2r^2\|w_0\|^2_{L^2},$$ 
for the same reason as above, the result follows.
			\fi
			
			{Note that $p_{T,A}D_{T,A}p_{T,A}^\perp$ is the formal adjoint of $p^\perp_{T,A} D_{T,A}p_{T,A}$, so } all the remaining estimates follow easily.
		\end{proof}
		
		\begin{prop}\label{Agmon3}
			For any $w\in E^\perp_{T,A}\cap H^2$, $$\| D_{T,A}w\|_{L^2}\geq C TA \|w\|_{L^2}
			\mbox{ and } \| \Delta_{T,A}w\|_{L^2}\geq C T^2A^2 \|w\|_{L^2}$$ 
			for some $(T,A)$-independent constant $C.$	
			
			Similar estimates hold for $D_{T,A}^{\bd}$ and $\Delta_{T,A}^{\bd}.$
		\end{prop}
		\begin{proof}
			Using Proposition \ref{agmon2}, and proceeding as in \cite[Proposition 4.12]{zhang2001lectures}, the result follows.
		\end{proof}

		For any bounded linear operator $B$ from $H^l$ to $H^j$, $l,j\in\{-1,0,1\}$. Set
        \begin{equation}
        \label{def-op-norm-TA}
           \|B\|^{l,j}_{T,A}\index[p2]{norm@norms!operator norm $\vert\vert\cdot\vert\vert^{l,j}_{T,A}$} \quad=\text{the operator norm of $B$ w.r.t } \|\cdot\|_{H^l_{T,A}} \text{ and } \|\cdot\|_{H^j_{T,A}}.
        \end{equation}
		\begin{prop}\label{Agmon31}
			We have $\|p_{T,A}^{\perp}\|_{T,A}^{l,l-1}\leq\frac{C}{TA},l\in\{0,1\}$ for some  $C=C(f,g^{TZ},h^F,r).$
		\end{prop}
		\begin{proof}
			By Proposition \ref{Agmon3}, for any smooth form $w$ with compact support,  $$\begin{aligned}&\ \ \ \ T^2A^2\int_Z|p_{T,A}^\perp w|^2\leq C\int_{Z}|D_{T,A}p_{T,A}^\perp w|^2
            \leq C\int_{Z}|D_{T,A}p_{T,A}^\perp w|^2+|D_{T,A}p_{T,A} w|^2\\
				&= C\left(\int_{Z}|D_{T,A}w|^2-2\Re\int_{Z}\lan D_{T,A}p_{T,A}^\perp w, D_{T,A}p_{T,A}w\ran \right)	
			\end{aligned}$$
			While by Proposition \ref{agmon2},
			\[
			\begin{aligned}
				&\ \ \ \ \left|\int_{Z}\big\langle D_{T,A}p_{T,A}^\perp w, D_{T,A}p_{T,A}w\big\rangle\right| \\
				&\leq \left|\int_{Z}\big\langle p_{T,A}D_{T,A}p_{T,A}^\perp w, D_{T,A}p_{T,A}w\big\rangle\right| + \left|\int_{Z}\big\langle p_{T,A}^\perp D_{T,A}p_{T,A}^\perp w, D_{T,A}p_{T,A}w\big\rangle\right| \\
				&\leq \left|\int_{Z}\big\langle p_{T,A}D_{T,A}p_{T,A}^\perp w, D_{T,A}p_{T,A}w\big\rangle\right| + \left|\int_{Z}\big\langle D_{T,A}p_{T,A}^\perp w, p_{T,A}^\perp D_{T,A}p_{T,A}w\big\rangle\right| \\
				&\leq C e^{-bTAr^2} \int_Z |w|^2 \leq C \int_Z |w|^2.
			\end{aligned}
			\]
			The result then follows from the two estimates above.		\end{proof}
		
		

		
		\def\U{\mathfrak{U}}
		\def\Lf{\mathfrak{L}}
		\def\Rf{\mathfrak{R}}

		According to our convention, where we set $\|\cdot\|_{H^0}$ to be $\|\cdot\|_{L^2}$, the norm $\|B\|_{T,A}^{0,0}$ for a bounded operator $B:H^0 \to H^0$ is independent of $T$, $A$, and $D_{T,A}$. Therefore, we will simply denote it as $\|B\|$\index[p2]{norm@norms!operator norm $\vert\vert\cdot\vert\vert$}.

		{\begin{prop}\label{Agmon4}
				There exists $(\l,t,T,A,y)$-independent constant $C$ such that for $|\Im(\lambda)|=1$ {and $t\in(0,\infty)$}
				\be\ba\label{prop920eq1} \big\|(\l-\sqrt{t}D_{T,A})^{-1}-(\l-\sqrt{t}p_{T,A}D_{T,A}p_{T,A})^{-1}p_{T,A}-(\l-\sqrt{t}p_{T,A}^{\perp}D_{T,A}p_{T,A}^{\perp})^{-1}p^\perp_{T,A}\big\|\\
                \leq C\sqrt{t}e^{-\frac{bTA r^2}{2}};
				\ea\ee
				\def\pr{0}
				\if\pr0
				and
				\be\ba\label{prop920eq2}\big\|(\l^2-{t}\Delta_{T,A})^{-1}-(\l^2-{t}p_{T,A}\Delta_{T,A}p_{T,A})^{-1}p_{T,A}-(\l^2-tp^\perp_{T,A}\Delta_{T,A}p^\perp_{T,A})^{-1}p^\perp_{T,A}\big\| \\
                \leq C{t}e^{-\frac{bTAr^2}{2}}.
				\ea\ee
				
				\fi
				We have similar estimates for $D_{T,A}^{\bd}$ and $\Delta_{T,A}^{\bd}.$
				
				Moreover, use the following identification:
        \begin{equation*}
            \begin{aligned}
               L^2 \Omega^\bullet(Z,F) &\simeq L^2\Omega^\bullet(Z_1,F)\oplus L^2\Omega^\bullet(Z_2,F) \\
                u &\mapsto (u|_{Z_1}, u|_{Z_2}),
            \end{aligned}
        \end{equation*}
				and let
                \begin{equation}
                \label{def-Rlambdat1-Rlambdat2}
                    \begin{aligned}
                        &R_1^{\l,t}:=(\l^2-t p_{T,A}^{\perp}\Delta_{T,A}p_{T,A}^{\perp})^{-1}p^\perp_{T,A}-(\l^2-{t}p_{T,A}^{\perp}\Delta^{\bd}_{T,A}p_{T,A}^{\perp})^{-1}p^\perp_{T,A},\\
                        &R_2^{\l,t}:=(\l^2-{t}\Delta_{T,A})^{-1}-(\l^2-{t}\Delta^\bd_{T,A})^{-1}-R_1^{\l,t}.
                    \end{aligned}
                \end{equation}
                Then
                \begin{equation}
                \label{estimation-Rlambdat1-Rlambdat2}
                \begin{aligned}
                     &\|R_1^{\l,t}\|\leq \frac{C|\l|^2}{{t}T^2A^2},\\
                    &{p_TR_2^{\l,t}p_T=0 \quad\text{and}}\quad  \|R_2^{\l,t}\|\leq C{t}e^{-bTAr^2/2}.
                \end{aligned}
                \end{equation}
			\end{prop}
			
			\begin{proof}
				Let $\Lf_{T,A}^{\l,t}:=(\l-\sqrt{t}p_{T,A}D_{T,A}p_{T,A})^{-1}p_{T,A}+(\l-\sqrt{t}p_{T,A}^{\perp}D_{T,A}p_{T,A}^{\perp})^{-1}p^\perp_{T,A}$, 
				
				
				It is clear that,
				\be\label{Agmon421}
				\|\Lf_{T,A}^{\l,t}\|\leq 2 \mbox{ and } \|(\l-\sqrt{t}D_{T,A})^{-1}\|\leq 1.
				\ee

				Note that 	\be(\l-\sqrt{t}D_{T,A})^{-1}=\Lf_{T,A}^{\l,t}+\Lf_{T,A}^{\l,t}\left(\sqrt{t}p_{T,A}D_{T,A}p_{T,A}^{\perp}+\sqrt{t}p_{T,A}^{\perp}D_{T,A}p_{T,A}\right)(\l-\sqrt{t}D_{T,A})^{-1}.\ee
			Also, note that $p_{T,A}^{\perp}(\l-\sqrt{t}D_{T,A})^{-1}$ maps $H^0$ to $H^1$ as $D_{T,A}$ is elliptic and $E_{T,A}\subset \Omega^{\bullet}(Z,F)$. Thus, we can use Proposition \ref{agmon2} to prove the first estimate. The second estimate can be derived similarly.

By \Cref{Agmon3}, $\|R_1^{\l,t}\|\leq \frac{C|\l|}{{t}T^2A^2}.$

As 
\begin{equation}
    p_{T,A}(\l^2-{t}p_{T,A} \Delta_{T,A}p_{T,A})^{-1}p_{T,A}=p_{T,A}(\l^2-{t}p_{T,A} \Delta^{\bd}_{T,A}p_{T,A})^{-1}p_{T,A},
\end{equation}
we get $p_TR_2^{\l,t}p_T=0$, and by \eqref{prop920eq1} and \eqref{prop920eq2}, we then have $\|R_2^{\l,t}\|\leq C{t}e^{-bTAr^2/2}.$
                \end{proof}
		}

		\subsection{Estimate of  resolvents when $A$ is not large}\label{sec92}

		In \cref{sec92}, we assume that $A\in[0, A_7]$ for some fixed $A_7$\index[p2]{
        A7a@$A_7$} larger than $A_6$ in \eqref{A6}. Then constant appearing here will thus all depend on $A_7$, and we will not repeat it each time.	We also assume that when $y=0$, the birth-death point $p_0$ is independent from all other critical points with respect to the gradient flow of the metric $g^{TZ}$.
		
        {Before continuing, let us explain briefly our strategy in this section. In \cref{Sect-estimate-A-large}, one of the key point to control the resolvant of $\Delta_{T,A}$ is to use Agmon estimates and the fact that as $A$ grows, the Agmon distance form $p_0$ to $\partial Z_1$ grows. Here, $A$ is bounded but we compensate by forcing the Agmon distance to be big enough using a new parameter $\Lambda$, which will be fixed large enough in \eqref{def-Lambda}. This is why in \cref{sec921} we deform the metric $g^{TZ}$ so that it satisfies Condition \ref{assum92}, and in \cref{sec922} we use a new cutting $Z=\tilde{Z}_1\cup\tilde{Z}_2$ rather than $Z=Z_1\cup Z_2$.}
        
        \subsubsection{A deformation of $g^{TZ}$}\label{sec921}

        Recall that the open sets $U_l$ satisfy Assumption \ref{9a}, and that $U'_l\subset U_l$ is a ball of radius 12 (w.r.t. $g^{TZ'}$) centered at $p_l$.
		
		Given that $p_0$ is independent of other critical points, we may as well assume that $(\rW^{\ru}(p_0)\cup \rW^{\rs}(p_0))\cap  \left(\cup_{l \geq 1}  \bar{U}_l \right)=\emptyset$. As a result,
		\begin{obs}\label{obs8}There exists a $\rho_0 \in (0,1)$ such that  $V_{\rho_0}\index[p2]{Vrho0@$V_{\rho_0}$} := \{q \in Z : d(q,p_0) < \rho_0\}$ (where $d$ is the distance induced by $g^{TZ}$), satisfies the following: 
        \begin{itemize}
            \item For any $p\in V_{\rho_0}$, given a gradient line $\gamma$ of $ f_{y=0}$ that passes through $p$, $\gamma$ does not intersect $\bigcup_{l \geq 1}  \bar{U}_l$. 
            \item {For any gradient flow line $\gamma$ of  $f_y$ such that $\gamma(t_i)\in \partial V_{\rho_0}$ for $i=1,2$, then $\gamma([t_1,t_2])\subset U_0$.}
        \end{itemize}
		\end{obs}
		\begin{proof}

            The fact that the first point holds for $\rho_0$ small enough is proved exactly as 
            Lemma \ref{flow-lines-avoid-neighborhoods}, and
           {the fact that the second point holds for $\rho_0$ small enough follows from  Assumption \ref{9a}.}
		\end{proof}
		
		\begin{defn}\label{defr}
			We now set the value of our parameters $r$ and $\delta$ as follows. 
            First, we take $r\leq \min\{\rho_0/27,(14\times 9)^{-1}\}$ so that $r_2=9r$ is small enough for Condition \ref{5a} to hold.
            Next, we will fix $\delta < \frac{7r}{36}$ such that for any $y \in (-\delta^2, \delta^2)$ and any $p \in V_{\rho_0}$, given any flow $\gamma$ with respect to $-\nabla f_y$ that passes through $p$, $\gamma$ does not intersect with $\bigcup_{l \geq 1}  \bar{U}_l$.  Now with $(r,\delta)$ being determined, $A_6$ in \eqref{A6} is also determined. 
		\end{defn}
		
		Let $\phi^t_y$ be the flow generated by $-\nabla (f_A)_y$. In the discussion below, we will omit the subscript $y$ for brevity.
		For any $(A,y)$-independent constant $$\Lambda > \sup_{p \in V_{\rho_0}} |f_A(p) - f_A(p_0)|,$$ set $$J_\Lambda\index[p2]{JLambda@$J_\Lambda$} := \Big\{p \in \bigcup_{t \in \mathbb{R}} \phi^t(V_{\rho_0}) : |f_A(p) - f_A(p_0)| \leq 8\Lambda\Big\}.$$ Then we can see that $J_{\Lambda} \cap \left(\cup_{l \geq 1}  \bar{U}_l \right) = \emptyset$. Note that $A \in [0, A_7]$, so $\Lambda$ should not be too large. Since $\nabla f_A = \nabla f$ outside $V_{\rho_0}$, $J_\Lambda$ is independent of $A$.
		
		Let $\dist$ be the {Agmon} distance induced by $\big(|\nabla f_A|^2g^{TZ}\big)|_{Z-V_{\rho_0}}=\big(|\nabla f|^2g^{TZ}\big)|_{Z-V_{\rho_0}}$. By Proposition \ref{prop84}, $$\sup_{p\in J_\Lambda}\dist(V_{\rho_0},p)\leq \sup_{q\in V_{\rho_0},p\in J_\Lambda}|f_A(p)-f_{A}(q)|<9\Lambda.$$ In particular $J_\Lambda$ is compact. By Observation \ref{obs8}, we can choose a set  $U_{\Lambda}$\index[p2]{ULambda@$U_\Lambda$} 
        such that
	\begin{equation}
	   \begin{aligned}
	       &U_\Lambda \text{ is an  open neighborhood of } J_{\Lambda},\\
            & \bar{U}_{\Lambda} \text{ is compact,}\\
            &\bar{U}_{\Lambda} \cap \left(\cup_{l \geq 1}  \bar{U}_l \right) = \emptyset.
	   \end{aligned}
	\end{equation}
		{In the remainder of this Section \ref{sec921}, we will prove that we can} deform the metric $g^{TZ}$ to $g^{TZ}_1$ satisfying the following. The idea is to enlarge the metric in the direction normal to \(\nabla f\), as in \eqref{def-gTZgamma}.

		\begin{cond}\label{assum92}
			\begin{enumerate}[(1)]
				\item We require that $\nabla^{g_1^{TZ}}f_A = \nabla^{g^{TZ}}f_{A}$, where $\nabla^{g}$ denotes the gradient with respect to $g$. Also, we require that $g^{TZ}(\nabla f_A,\nabla f_A)=g^{TZ}_1(\nabla f_A,\nabla f_A)$. So it will not cause any confusion to write both $\nabla^{g^{TZ}}f_A$ and $\nabla^{g_1^{TZ}}f_A$ as $\nabla f_A$, and both $g^{TZ}(\nabla f_A,\nabla f_A)$ and $g^{TZ}_1(\nabla f_A,\nabla f_A)$ as $|\nabla f_A|^2$.
				\item Let $\dist_1(\cdot,\cdot)$ be the distance with respect to the Agmon metric $|\nabla f_A|^2g_1^{TZ}=|\nabla f|^2g_1^{TZ}$ on $U_\Lambda-V_{\rho_0/2}$. Then we require that $\dist_1(\partial U_\Lambda,\partial V_{\rho_0})> 6\Lambda$.
			\end{enumerate}
		\end{cond}

		Note that according to our convention for abbreviating $ y $, $ g^{TZ}_1 $ is actually a family of metrics parameterized by $ y \in (-\delta^2, \delta^2) $.

		From the construction of $U_\Lambda$, we can see that:
		\begin{prop}\label{prop99}
			If $\gamma$ is a flow line of $-\nabla f$ that pass through a point in $V_{\rho_0}$, and if  $q\in \partial U_{\Lambda}\cap \gamma$, then $|f_A(q)-f_{A}(p_0)|\geq 8\Lambda.$ As a result, $\inf_{p\in V_{\rho_0}}|f_A(q)-f_A(p)|\geq 7\Lambda.$
		\end{prop}

		\def\VG{{\varGamma}}

		Let $0\leq\eta\in\C^\infty(Z)$, s.t. $$\eta|_{V_{\frac{\rho_0}{2}}\cup\bigcup_{l\geq1}\left(U'_l\right)}\equiv0\mbox{ and }\eta|_{Z-V_{\rho_0}-\cup_{l\geq1}U_l}\equiv1.$$
		
		For any vector field $X\in\Gamma(TZ)$, consider the orthogonal decomposition $X_1+X_2$ with respect to the metric $g^{TZ}$ on $Z-V_{\frac{\rho_0}{2}}-\cup_{l\geq1}U_l'$, where $X_1$ is parallel to $-\nabla f_A=-\nabla f$.
		
		For any $\VG>0$, we define the metric $g^{TZ}_\VG$ on $Z$ as follows:
	\begin{equation}
    \label{def-gTZgamma}
	    \begin{aligned}
	        &g^{TZ}_\VG(X,X):=g^{TZ}(X_1,X_1)+(1+\Gamma\eta)g^{TZ}(X_2,X_2)\mbox{ on $Z-V_{\frac{\rho_0}{2}}-\cup_{l\geq1}U_l'$;}\\
            &g^{TZ}_\VG(X,X):=g^{TZ}(X,X)\mbox{ on $V_{\frac{\rho_0}{2}}\bigcup\left(\cup_{l\geq1}U_l'\right)$.}
	    \end{aligned}
	\end{equation}

		Then it's straightforward to verify that $g^{TZ}_\VG$ satisfies item (1) in Condition \ref{assum92}.
		
		\begin{lem}\label{lem910}
			There exists a sufficiently large $\VG$ such that $g^{TZ}_\VG$ satisfies item (2) in Condition \ref{assum92}.
		\end{lem}
		\begin{proof}
			\def\AG{{\mathrm{AG}}}
			Let $g^{\AG}_{\VG}:=|\nabla f|^2g^{TZ}_\VG$ and $g^{\AG}:=|\nabla f|^2g^{TZ}$ on $U_{\Lambda}-V_{\rho_0{/2}}.$
			
			Assume that Lemma \ref{lem910} is false. To avoid introducing heavy notations for subsequences, and as it will not impact our proof, we will even assume that for any $\VG>0$, there exists a geodesic $\gamma_{\VG}:[0,1]\to U_{\Lambda}-V_{\rho_0}$ with respect to the metric $g^\AG_{\VG}$ connecting $\partial V_{{\rho_0}}$ and $\partial U_\Lambda$, such that its length with respect to the metric $g^\AG_\VG$ is smaller than $6\Lambda$. Consequently, {as $g^{TZ}_\Gamma \geq g^{TZ}$}, the length of $\gamma_\VG$ with respect to the metric $g^\AG$ is also smaller than $6\Lambda$.
			
			We will see in Observation \ref{obs4} that we can choose $U_{\Lambda}$ so that it embeds into a bounded open set in $\mathbb{R}^{n+1}$. Thus, we can regard $\gamma_{\VG}$ as a uniformly bounded $\mathbb{R}^{n+1}$-valued function on $[0,1]$.
			
			Given that \be\label{lem910eq1}g^{\AG}(\gamma_\VG',\gamma_\VG')\leq g^\AG_{\VG}(\gamma_\VG',\gamma_\VG')\leq 36\Lambda^2,\ee it follows that $\gamma_\VG':[0,1]\to \mathbb{R}^{n+1}$ is also uniformly bounded. Hence, $\gamma_\VG$ is equicontinuous.
			
			Let $\nabla^{\AG,\VG}$ and $\nabla^{\AG}$ denote the Levi-Civita connections for $g^\AG_\VG$ and $g^\AG$ respectively.
			
			Then $\omega_\VG:=\nabla^{\AG,\VG}-\nabla^{\AG}$ is an $\End(TZ)$-valued 1-form.
			
			Let $\{e_l\}_{l=0}^n$ be a local orthonromal frame{ with respect to $g^{\AG}$.} in $U_{\Lambda}-V_{\rho_0{/2}}$ with respect to $g^{\AG}$, such that $e_0=\frac{\nabla f}{|\nabla f|^2}$. 
			
			Then it follows from the definition of the Levi-Civita connection that:\begin{equation}\label{lem910eq2}
				\begin{cases}
					\omega_{\VG}(e_0)e_0=\sum_{l\geq1}a_le_l \text{ with $a_l=O(1)$}, \omega_{\VG}(e_0)e_l=\sum_{j\geq1}b_{lj}e_j\text{ with $b_{lj}=O(\sqrt{\VG})$},l\geq1;\\
					\omega_{\VG}(e_l)e_0=\sum_{j\geq1}c_{lj}e_j \text{ with $c_{lj}=O(\sqrt{\VG})$},l\geq1;\\ \omega_{\VG}(e_l)e_k=d_{lk}e_0+\sum_{j\geq1}d_{lkj}e_j \text{ with $d_{lkj}=O(\VG)$ and $d_{lk}=O(\sqrt{\VG})$,}l,k\geq1.
				\end{cases}
			\end{equation}
			
			Now suppose $\gamma_\VG'=\sum_{i=0}^n\a_le_l$, by {\eqref{def-gTZgamma} and}  \eqref{lem910eq1}, 
			
			\begin{equation}\label{lem910eq3}
				\a_0=O(1)\text{ and }\a_j=O\left(\frac{1}{\sqrt{\VG}}\right), j\geq1.
			\end{equation}
			
			Note that $\nabla^{\AG,\VG}_{\gamma_{\VG}'}\gamma_{\VG}'=0$, by \eqref{lem910eq2} and \eqref{lem910eq3},
			
			\begin{equation}\label{lem910eq4} 
				g^{\AG}\big(\nabla^{\AG}_{\gamma_{\VG}'}\gamma_{\VG}',\nabla^{\AG}_{\gamma_{\VG}'}\gamma_{\VG}'\big)=O(1).
			\end{equation}
			
			Note that $U_{\Lambda}$ is embedded into $\mathbb{R}^{n+1}$, by \eqref{lem910eq4}, $\gamma_\VG'':[0,1]\to\mathbb{R}^{n+1}$ is uniformly bounded. So $\gamma_\VG':[0,1]\to \mathbb{R}^{n+1}$ is also equicontinuous.
			
			By the Arzel\`a–Ascoli theorem,  to avoid introducing heavy notations for subsequences, we may as well assume that $\gamma_{\VG}$ converges uniformly in the $C^1$-topology to $\gamma_{\infty}:[0,1]\to U_{\Lambda}-V_{\rho_0}$. So the length of $\gamma_\infty$ with respect to the metric $g^{\AG}$ is less than $6\Lambda$. Moreover, $\gamma_\infty$ connects $\partial U_{\Lambda}$ and $\partial V_{\rho_0}$.
			
			Consider the  orthogonal decomposition $\gamma_\Gamma':=X_{1,\Gamma}+X_{2,\Gamma}$ with respect to $g^{\AG}$, where $X_{1,\Gamma}$ is parallel to $\nabla f$.
			By {\eqref{def-gTZgamma} and} \eqref{lem910eq1}, $g^{\AG}(X_{2,\Gamma},X_{2,\Gamma})\leq \frac{36\Lambda^2}{\Gamma}$, so $\gamma'_{\infty}$ is parallel to $\nabla f$, so $\gamma_\infty$ is a reparameterization of a flow line with respect to the vector field $-\nabla f$.
			
			It follows from the property of the Agmon metric (e.g., Proposition \ref{prop84}) that $\big|f\big(\gamma_\infty(0)\big)-f\big(\gamma_\infty(1)\big)\big|$ is bounded above by the length of $\gamma_\infty$ with respect to the metric $g^\AG$, i.e.,
			\begin{equation}\label{lem910eq5}
				\big|f\big(\gamma_\infty(0)\big)-f\big(\gamma_\infty(1)\big)\big|\leq 6\Lambda.
			\end{equation}
			But \eqref{lem910eq5} contradicts Proposition \ref{prop99}.
		\end{proof}

		\subsubsection{Lower bounds for small non-zero eigenvalues when $A$ is not large}\label{sec922}
		
		\def\tZ{{\tilde{Z}}}
		\def\tp{{\tilde{p}}}
		\def\tE{{\tilde{E}}}
		
		According to Lemma \ref{lem910}, there exists a deformation $ g^{TZ}_1 $ of the metric $ g^{TZ} $ within a compact neighborhood of $ p_0 $, satisfying the two conditions in Condition \ref{assum92}, for some $ \Lambda $ to be determined later (see \eqref{def-Lambda}).

		Define $\Delta_{T,A}$ (or $D_{T,A}$) as the Witten Laplacian (or Dirac operator) for $f_A$, induced by $g^{TZ}_1$ (not by $g^{TZ}$), and $h^F$, as in \cref{setting-partII}.

		We fix $\tZ_1\subset U_{\Lambda}$ to be a {compact} smooth manifold of dimension $n+1$ containing $V_{\rho_0}$, s.t. \be\label{eq115}\forall p\in\p\tZ_1,\:\dist_1(p, V_{\rho_0})\in(3\Lambda,4\Lambda).\ee
   {The existence of such a manifold follows from Condition \ref{assum92}.} Set $ \tZ_2:=Z-\tZ_1$\index[p2]{Z1tilde@$\tZ_1$ and $\tZ_2$}.
		
		Let $D_{T,A,a}'$\index[p2]{DTA1@$D_{T,A,1}'$ and $D_{T,A,2}'$}, $a=1,2$, be the restriction of $D_{T,A}$ on $\tZ_a$ with absolute/relative boundary conditions, and $\Delta_{T,A,a}'$\index[p2]{DeltaTA1@$\Delta_{T,A,1}'$ and $\Delta_{T,A,2}'$}, $a=1,2$, be the restriction of $\Delta_{T,A}$ on $\tZ_a$ with absolute/relative boundary conditions.
		
		Recall that $\lambda_{k}(T,A)$ denotes the $k$-th eigenvalue of $\Delta_{T,A}$. 
		
		Let $\lambda^{\bd}_k(T,A)$ be the $k$-th eigenvalue of $\Delta_{T,A,1}'\oplus \Delta_{T,A,2}'.$
		
		Let $ f^{\R^{n+1}}_y $ be the function on $\R^{n+1}$ defined by 
		\[ f^{\R^{n+1}}_y(u) =  u_0^3 - y u_0 - \sum_{j=1}^i u_j^2 + \sum_{j=i+1}^n u_j^2. \] 
		
		{Note that the topology of $U_\Lambda$ may be complicated, but we can still trivialize our data on it, as proven in the following observation.}
		\begin{obs}\label{obs4}
			By choosing a suitable $ U_{\Lambda} $, there exists an embedding $\Phi: U_{\Lambda} \to \mathbb{R}^{n+1} $, which coincide with the one given by the coordinates in Assumption \ref{9a} on $U_0\cap U_\Lambda$, and such that $ f_y|_{U_{\Lambda}} = (f_y^{\mathbb{R}^{n+1}}) \circ \Phi $. 
            
            Moreover, if $F\to Z$ is unitarily flat, the bundle $ F|_{U_{\Lambda}} \to U_{\Lambda} $ admits a unitarily trivialization. That is, there exists an isomorphism of vector bundles $\Psi: F|_{U_{\Lambda}} \to \Cb^m|_{U_{\Lambda}} $ such that $ h^{{F}} = (\Psi)^* h^{\Cb^m} $ and $ \nabla^{{F}} = (\Psi)^* \nabla^{\Cb^m} $. Here $ m = \mathrm{rank}(F) $, $ \Cb^m \to Z $ is the trivial bundle with its canonical trivial connection $\nabla^{\Cb^m}$ and canonical Hermitian metric $ h^{\Cb^m} $.
		\end{obs}
		\begin{proof}
			By selecting an appropriate $ U_{\Lambda} $, we can ensure the following condition:
			\begin{cond}\label{assum94}
				For each $ p \in U_{\Lambda}-V_{\rho_0} $, there exists a $ q \in V_{\rho_0}-\{p_0\} $ such that $ p $ and $ q $ are connected by a path $\gamma$ with respect to flow generated by $-\nabla f_y$.  
			\end{cond}
			
			Note that by the construction, $ J_{\Lambda} $ satisfies Condition \ref{assum94} if we replace $ U_\Lambda $ in the statement of Condition \ref{assum94} by $ J_\Lambda $. Therefore, {as $V_{\rho_0}$ is open,} we can select a sufficiently small open neighborhood $ U_{\Lambda} $ of $ J_\Lambda $ to guarantee that Condition \ref{assum94} is satisfied.
			\ \\			\ \\
			$\bullet$	Now the embedding $\Phi$ is given as follows: $V_{\rho_0}$ is mapped to an Euclidean ball of radius $\rho_0$ by the coordinates given in Assumption \ref{9a}. For each $p \in U_{\Lambda} - V_{\rho_0}$, let $q \in V_{\rho_0} - \{p_0\}$ be a point that can be connected to $p$ via the flow generated by $-\nabla f_y$. 
			
			Let $q^{\mathbb{R}} \in \mathbb{R}^{n+1}$ be the image of $q$ by the above embedding of $V_{\rho_0}$. Then there exists a unique $p^{\mathbb{R}} \in \mathbb{R}^{n+1}$ such that $p^{\mathbb{R}}$ and $q^{\mathbb{R}}$ are connected by the flow generated by $-\nabla f^{\mathbb{R}^{n+1}}_y$ (with respect to canonical metric $g^{T\R^{n+1}}$), and \[f^{\mathbb{R}^{n+1}}_y(p^{\mathbb{R}}) - f^{\mathbb{R}^{n+1}}_y(q^{\mathbb{R}}) = f_y(p) - f_y(q).\] Then we map $p$ to $p^{\mathbb{R}}$.
			\ \\ \ \\	
			$\bullet$ Next, we would like to construct $\Psi$. To this end, we show that for any vector $v_0\in F_{p_0}$, there exists a parallel section $s$ on $U_\Lambda$, s.t. $s(p_0)=v_0$:
			\begin{enumerate}[(a)]
				\item
				If $ p \in V_{\rho_0} $, we select a path $ \gamma $ within $ V_{\rho_0} $ connecting $ p $ and $ p_0 $. Then, $ s(p) \in F_p $ is determined by parallel transporting $ v_0 $ along $ \gamma $. Due to the simple connectedness of $ V_{\rho_0} $, $ s(p) $ is independent of the specific path $ \gamma $ chosen. It is clear that $\nabla^F s=0$ inside $V_{\rho_0}.$
				\item 	If $ p \in U_{\Lambda} - V_{\rho_0} $, let $ \gamma_1 $ denote the gradient flow generated by $ -\nabla f_y $ with $\gamma(0)=p$, and $ q \in \gamma_1 \cap \partial V_{\rho_0} $. Then $q=\gamma(t_0)$  for some $t_0$, and we may assume without loss of generality that $t_0>0$. Choose $ \gamma_2 \subset V_{\rho_0}$ as a smooth path going from $ q $ to $ p_0 $. Then, $ s(p) \in F_p $ is defined by the parallel transport of $ v_0 $ along $ \gamma_2 $ and $ \gamma_1 $. Let $ q' \in \gamma_1 \cap \partial V_{\rho_0} $. {By Observation \ref{obs8}}, the segment of $ \gamma_1 $ connecting $ q $ and $ q' $ must lie within $ U_0 $. Since $ U_0 $ is simply connected, $ s(p) $ does not depend on the specific choice of $ q \in \gamma_1 \cap \partial V_{\rho_0} $ or $ \gamma_2 $. 
				
				If $\tilde{p} \in U_{\Lambda}$ is sufficiently close to $p$, let $\gamma_3$ be any smooth path going from $p$ to $\tilde{p}$, ensuring that the entire path remains close to $p$. Let $\tilde{\gamma}_1$ denote the gradient flow generated by $-\nabla f_y$ with $\tilde{\gamma}(0) = \tilde{p}$. Choose $\tilde{q} \in \tilde{\gamma}_1 \cap \partial V_{\rho_0}$ to be close to $q$ with $\tilde{\gamma}_1(\tilde{t}_0) = \tilde{q}$. Select a smooth path $\tilde{\gamma}_2 \subset V_{\rho_0}$ going from $\tilde{q}$ to $p_0$. Then, there exists a (piece-wise) smooth homotopy of paths between $\gamma_2 \gamma_1|_{[0,t_0]}$ and $\tilde{\gamma}_2 \tilde\gamma_1|_{[0,\tilde{t}_0]}\gamma_3$. From this, we conclude that $\nabla^F s = 0$ at $p$. 
			\end{enumerate}
		\end{proof}

		{Let $ f_{A,y}^{\R^{n+1}} $ be defined from $f_y^{\R^{n+1}} $ as $f_{A,y}$ was defined from $f_y$ (see the beginning of Section \ref{sec9}). As usual, } for simplicity, we will denote $ f^{\R^{n+1}}_y $ (resp. $ f_{A,y}^{\R^{n+1}} $) as $ f^{\R^{n+1}} $ (resp. $ f_{A}^{\R^{n+1}} $) when there is no ambiguity. Let $ g^{T\R^{n+1}}_1 $ be a metric on $\R^{n+1}$ that coincides with $ g^{TZ}_1 $ when restricted to $\Phi(U_{\Lambda})$ and equal the standard metric $ g^{T\R^{n+1}} $ outside some compact set.
		
		Let $D_{T,A}^{\R^{n+1}}$ be the Witten deformed Dirac operator with respect to $f_{A}^{\R^{n+1}}$ and $\Delta_{T,A}^{\R^{n+1}}$ be the Witten Laplacian with respect to  $f_{A}^{\R^{n+1}}$. Let $\l_k^{\R^{n+1}}(T,A)$ be the $k$-th eigenvalues with respect to $\Delta_{T,A}^{\R^{n+1}}.$ 
		
		\begin{lem}\label{lem914}
			There exists $C_k=C_k(g^{TZ}_1,h^F,f)$, $\Lambda_0=\Lambda_0(A_7,r,f)$\index[p2]{Lambda0@$\Lambda_0$} , s.t. when $A\in[0,A_7]$, $k> k_0$, we have
			\[\l_k(T,A)\geq C_ke^{-\frac{\Lambda_0 T}{2}},\]
			and for $k>k_1+k_2 $
			\[\l^\bd_k(T,A)\geq C_ke^{-\frac{\Lambda_0 T}{2}}.\]
		  Moreover, $\Delta_{T,A}^{\R^{n+1}} $ admits no non-trivial harmonic forms, and more precisely for any $k$, we have
			\[\l_{k}^{\R^{n+1}}(T,A)\geq C_ke^{-\frac{\Lambda_0 T}{2}}.\]
			
		\end{lem}

        \begin{proof}
            The first two estimates are proved following the same procedure as in the proof of Lemma \ref{lem810}. The last one follows from \cite[Theorem 1.3]{DY2020cohomology}.
            \end{proof}
		
		We will \textbf{deform $g^{TZ}$ to a metric  {$g_1^{TZ}$} that satisfies Condition \ref{assum92}} for
		\begin{equation}\label{def-Lambda}
		\Lambda = \max\Big\{\Lambda_0, \sup_{\substack{y \in (-\delta^2, \delta^2) \\ p \in V_{\rho_0}, A \in [0, A_6]}} \big|(f_A)_y(p) - (f_{A})_y(p_0)\big|\Big\}.
		\end{equation}\index[p2]{Lambda@$\Lambda$} 
		\begin{obs}\label{obs3} Note that $g_1^{TZ} = g^{TZ}$ in $V_{\rho_0/2}$, and that constants $A_4$, $A_5 $, and $A_5'$ (given as in Lemmas \ref{lem811}, \ref{agmonub} and \ref{lem813} respectively.) only depend on the restriction of the metrics on $ V_{\rho_0/2} $. Therefore, the constant $ A_6$ in \eqref{A6} remains unchanged when performing our deformation, and the choice of $A_7$ as in the beginning of \cref{sec92} has not to be changed.\end{obs}
		
		\textbf{From now until the end of this section, We will still denote the deformed metric by $g^{TZ}$.}

		Before moving on, we will state a lemma, whose proof is easy.
		\begin{lem}\label{trilem}
			Let $B:H\to H$ be a positive self-adjoint operator on Hilbert space $H$ with discrete eigenvalues. Let $P$ be the orthogonal projection onto the space generated eigenvectors corresponding to eigenvalues $\leq C_1$ for some $C_1>0$. Let $u\in H$ such that
			$|(Bu,u)|\leq C_2\|u\|^2$ for some $C_2>0$, then
			\[\|u-Pu\|^2\leq \frac{C_2\|u\|^2}{C_1}.\]
		\end{lem}
		
		Let $\P_{T,A}$\index[p2]{PTA@$\P_{T,A}$} be the orthogonal projection to the space of harmonic form of $D_{T,A}$.

		Lastly, the following estimate will be needed later.

		\begin{lem}\label{lem915}
			If $A\in[0,A_7]$, $\|\P_{T,A}-\P_{T,0}\|=O(e^{-cT})$ for some $(T,A)$-independent constant $c>0$.
		\end{lem}
		\begin{proof}

			Let $\tilde{Z}'_1:=\{p\in Z-V_{\rho_0}:\dist_1(p,\p V_{\rho_0})\leq \Lambda\}\cup V_{\rho_0}$, then $\tilde{Z}_1'\subset\tilde{Z}_1$. 
			
			Let $\eta\in C^\infty(Z)$, s.t. $\eta|_{Z-\tilde{Z}_1}\equiv0$ and $\eta|_{\tilde{Z}_1'}\equiv1.$

			Let $ u $ be a unit harmonic form with respect to $ D_{T,A} $. First, we show that
			\be\label{lem915eq1}
			\int_{Z} |\eta u|^2 \leq Ce^{-\Lambda_0 T}.
			\ee
			
			Note that $ D_{T,A} \eta u = c(\nabla \eta) u $.  {Moreover, for $l\geq 1$, $\dist_1(\tilde{Z}_1',U_l)\geq \dist_1(\tilde{Z}_1',\p\tilde{Z}_1)\geq 2\Lambda$ and $\dist_1(\supp(\nabla \eta),V_{\rho_0})\geq\Lambda$.} Thus, by Lemma \ref{agmonz} and as $\Lambda\geq \Lambda_0$, we get
			\be\label{lem915eq2}
			\int_Z |D_{T,A} \eta u|^2 \leq Ce^{-1.5T \Lambda_0}.
			\ee

			However, regarding $\eta u$ as a smooth form on $\R^{n+1}$ via $\Phi$ in Observation \ref{obs4}, the last estimate in Lemma \ref{lem914} implies that
			\be\label{lem915eq3}
			\int_Z |D_{T,A} \eta u|^2 \geq Ce^{-\frac{T \Lambda_0}{2}} \int_Z |\eta u|^2.
			\ee
			Then \eqref{lem915eq1} follows from \eqref{lem915eq2} and \eqref{lem915eq3}.
			
			Let $ w $ be a unit harmonic form for $ D_{T,0} $. Then we have $ D_{T,A} w = T \hat{c}(\nabla  q_A) w $ and $\supp(\nabla q_A)\subset {\tilde{Z}_1'}\subset \{\eta=1\}$, so by \eqref{lem915eq1} {for $A=0$},
			\be\label{lem915eq4}
			\int_{Z} |D_{T,A} w|^2 =\int_Z| T \hat{c}(\nabla  q_A) w|^2\leq CT^2 A^2\int_{Z}| \eta w|^2\leq CT^2 A^2 e^{-\Lambda_0 T}\leq CT^2 A_7^2 e^{-\Lambda_0 T}.
			\ee

			Using Lemma \ref{lem914} and \eqref{lem915eq4}, we can apply Lemma \ref{trilem} for the constant $C_1=Ce^{-\Lambda_0T/2}$ and $C_2=CT^2 A_7^2 e^{-\Lambda_0 T}$ to get $ \|\P_{T,A} w - w\|^2 = O(e^{-cT})$. 			
		\end{proof}

		\subsubsection{Estimate of  resolvents when $A$ is not large}\label{sec923}
		
		Recall that we have replaced the metric $g^{TZ}$ in \cref{sec922} with a metric that satisfies Condition \ref{assum92} for some sufficiently large $\Lambda$ (see \eqref{def-Lambda}).
		
		Recall that
		$D_{T,A,a}'$, $a=1,2$, is the restriction of $D_{T,A}$ on $\tZ_a$ with relative/absolute boundary conditions, and $\Delta_{T,A,a}'$, $a=1,2$, be the restriction of $\Delta_{T,A}$ on $\tZ_a$ with relative/absolute boundary conditions. 		Let \be\label{dtabd1}D_{T,A}^{\bd,\prime}\index[p2]{DTAbdprime@$D_{T,A}^{\bd,\prime}$}:=D_{T,A,1}'\oplus D_{T,A,2}'\mbox{ and } \Delta_{T,A}^{\bd,\prime}\index[p2]{DeltaTAbdprime@$\Delta_{T,A}^{\bd,\prime}$}:=\Delta_{T,A,1}'\oplus \Delta_{T,A,2}'.\ee

		\def\ab{0}
		\if\ab{1}
		Let $f_{l,A}$ be a generalized Morse function on $\R^{n+1}$ given by
		\[f_{l,A}(u_0,\cdots,u_n)=-u_0^2-\cdots -u_{i_l-1}^2+u_{i_l}^2+\cdots u_n^2,l\geq 1;\]
		\[f_{0,A}=f_{A,y}.\]
		For $l\geq1$, let $\Delta_{l,T,A}$ be the Witten Laplacian for $f_{l,A}$ w.r.t the standard metric $g$ on $ \Omega^{\bullet}(\R^{n+1};C^m)$ (Here $m=\rank(F)$). Then Lemma \ref{agmon} implies that there exists $(T,A)$-independent $c,c'>0$, such that if $w$ is an eigenform of $\Delta_{l,T,A}$ with respect to eigenvalue $\l\leq cT^2(A^2+1)^2$, then 
		\be\label{Agmonrn}\|w\|_{L^2(\R^{n+1}-D(40r))}\leq e^{-c'T(A^2+1)}\|w\|_{L^2(\R^{n+1})}.\ee
		
		For $l=0$, Since $W_\Lambda$ could be embedded in $\R^{n+1}$, we choose a metric $g_1$ on $\R^{n+1}$, s.t. $g_1|_{U_\Lambda}=g_1^{TZ}$, and $g_1|_{\R^{n+1}-W_\Lambda}$ is standard.  Let $\Delta_{0,T,A}$ be the Witten Laplacian for $f_{l,A}$ w.r.t the metric $g_1$ on $ \Omega^{\bullet}(\R^{n+1};C^m)$.
		
		Let $\eta\in C_c^\infty(\R^{n+1})$, s.t. $\eta\equiv1$ on $D(50r)$ and $\eta\equiv0$ if on $D(55r,\infty)$. Then for any eigenform $w$ of $\Delta_{l,T,A}$, $\eta w$ can be view as a differential form supported inside $U_l,l\geq1.$

		\fi
		
		Let $\eta\in C^\infty(Z)$, s.t. 
        \begin{equation}
            \begin{aligned}
                &\eta(x)=0  \qquad\text{if}\quad x\in\{x'\in Z{-V_{\rho_0}}:\dist_1(x',\p V_{\rho_0})\in(3\Lambda,5\Lambda)\}, \\
                &\eta(x)=1 \qquad\text{if}\quad x\in\{x'\in Z-V_{\rho_0}:\dist_1(x',\p V_{\rho_0})\in[0,2.5\Lambda]\cup[6\Lambda,\infty)\}\cup V_{\rho_0}.
            \end{aligned}
        \end{equation}
         Then by the construction of $\tZ_1$ and $\tZ_2$, for any eigenform $w$ of $\Delta'_{T,A,a}$, $\eta w$ can be view as a differential form supported inside $\tZ_a$, $a=1,2.$

Let $\tE_{T,A} \subset \Omega^{\bullet}(Z,F)$ (resp. $\tE^\perp_{T,A}$\index[p2]{ETAtilde@$\tE_{T,A}$ and $\tE^\perp_{T,A}$}) be the subspace spanned by $\eta w$, where $w$ is an eigenform of $\Delta_{T,A,a}'$ corresponding to an eigenvalue $\leq cT^2$, for $a=1,2$ (resp. its orthogonal complement). {Here, $c$ is the constant defined in Lemma \ref{agmonz} with $b=0.8$, but with $U_0'$ replaced by $V_{\rho_0/2}$ and $U_0$ replaced by $V_{\rho_0}$.}  Let $\tp_{T,A}$ (resp. $\tp_{T,A}^\perp$\index[p2]{pTAtilde@$\tp_{T,A}$ and $\tp^\perp_{T,A}$})) be the orthogonal projection from $L^2 \Omega^{\bullet}(Z,F)$ to $\tE_{T,A}$ (resp. $\tE^\perp_{T,A}$).

{Furthermore, let $\dist_T(\cdot,\cdot)$ denote the Agmon distance. By our construction, for any $x \in \supp \nabla\eta$, we have  
\[
0.8\times \dist(x, V_{\rho_0} \cup (\cup_{l\geq1} U_1)) \geq 2\Lambda \geq 2\Lambda_0.
\]  
Thus, if $w$ is an eigenform of $\Delta_{T,A,a}'$ with eigenvalue $\leq cT^2$ for $a=1,2$, it follows that  
\be\label{eq956}
\int_{\supp(\nabla\eta)} |w|^2 \leq C e^{-2\Lambda_0 T} \int_Z |w|^2.
\ee  
}

    	Recall that we have defined Sobolev-type space $H^{l}$ and norm $\|\cdot\|_{H^l_{T,A}}$, $l\in\{-1,0,1\}$, in \cref{sec913}. Recall also that for any bounded linear operator $B$ from $H^l$ to $H^j$, $l,j\in\{-1,0,1\}$, $\|B\|^{l,j}_{T,A}$ be the operator norm with respect to $\|\cdot\|_{H^l_{T,A}}$ and $\|\cdot\|_{H^j_{T,A}}$. In particular, the norm $\|B\|_{T,A}^{0,0}$ for a bounded operator $B:H^0 \to H^0$ is independent of $T$, $A$, and $D_{T,A}$. Therefore, we will simply denote it as $\|B\|$.

		By employing the same line of reasoning as in the proof of Proposition \ref{agmon2}, Proposition \ref{Agmon3}, and Proposition \ref{Agmon4}, {with the constant $\Lambda_0$ playing the role of $A$ in these results, and using \eqref{eq956} instead of \eqref{eq933}} we get the following three propositions.
		\begin{prop}\label{agmon2t}
			For any $w\in \tE_{T,A}$ and  $w'\in \tE^\perp_{T,A}\cap H^1$, $$\|\tp_{T,A}^\perp D_{T,A}w\|_{L^2}\leq Ce^{-2\Lambda_0T }\|w\|_{L^2}
			\mbox{ and }\|\tp_{T,A} D_{T,A}w'\|_{L^2}\leq Ce^{-2\Lambda_0T }\|w'\|_{L^2};$$	
			\def\foranynew{0}
			\if\foranynew0
			$$\|\tp_{T,A}^\perp \Delta_{T,A}w\|_{L^2}\leq Ce^{-2\Lambda_0T }\|w\|_{L^2}	
			\mbox{ and } \|\tp_{T,A} \Delta_{T,A}w'\|_{L^2}\leq Ce^{-2\Lambda_0T }\|w'\|_{L^2}.$$\fi
			Here $C$ is a $(T,A,y)$-independent constant. 
			
			Similar estimates hold for $D_{T,A}^{\bd,\prime}$ and $\Delta_{T,A}^{\bd,\prime}.$
		\end{prop}

		\begin{prop}\label{Agmon3t}
			For any $w\in \tE^\perp_{T,A}$, $$\| D_{T,A}w\|_{L^2}\geq C T \|w\|_{L^2}\mbox{ and }\| \Delta_{T,A}w\|_{L^2}\geq C T^2\|w\|_{L^2}$$ 
			for some $(T,A,y)$-independent constant $C.$
			
			Similar estimates hold for $D_{T,A}^{\bd,\prime}$ and $\Delta_{T,A}^{\bd,\prime}.$
		\end{prop}

		
		
		\def\tDd{\tilde{D}}
		
		\def\U{\mathfrak{U}}
		\def\Lf{\mathfrak{L}}
		\def\Rf{\mathfrak{R}}
		\def\tR{{\tilde{R}}}

		\begin{prop}\label{Agmon4t}
			There exists a $(\l,t, T,A,y)$-independent  constant $C$, s.t. {for $|\Im(\lambda)|=1$ and $t\in(0,\infty)$,}            
			\begin{align*}&\ \ \ \ \big\|(\l-\sqrt{t}D_{T,A})^{-1}-(\l-\sqrt{t}\tp_{T,A}D_{T,A}\tp_{T,A})^{-1}\tp_{T,A}-(\l-\sqrt{t}\tp_{T,A}^\perp D_{T,A}\tp_{T,A}^\perp)^{-1}p^\perp_{T,A}\big\|\leq C\sqrt{t}e^{-2\Lambda_0 T};
			\end{align*}
			\def\prnew{0}
			\if\prnew0
			and
			\begin{align*}&\ \ \ \ \big\|(\l^2-{t}\tD_{T,A})^{-1}-(\l^2-{t}\tp_{T,A}\Delta_{T,A}\tp_{T,A})^{-1}\tp_{T,A}-(\l^2-t\tp_{T,A}^\perp \Delta_{T,A}\tp_{T,A}^\perp)^{-1}\tp^\perp_{T,A}\big\|\leq C{t}e^{-2\Lambda_0 T };
			\end{align*}\fi
			We have similar estimates for $D_{T,A}^{\bd,\prime}$ and $\Delta_{T,A}^{\bd,\prime}.$

            Moreover, use the following identification:
        \begin{equation}
            \begin{aligned}
               L^2 \Omega^\bullet(Z,F) &\simeq L^2\Omega^\bullet(\tilde{Z}_1,F)\oplus L^2\Omega^\bullet(\tilde{Z}_2,F) \\
                u &\mapsto (u|_{\tilde{Z}_1}, u|_{\tilde{Z}_2}),
            \end{aligned}
        \end{equation}and let 
                \begin{equation}
                    \begin{aligned}
                        &\tR_1^{\l,t}:=(\l^2-t p_{T,A}^{\perp}\Delta_{T,A}p_{T,A}^{\perp})^{-1}p^\perp_{T,A}-(\l^2-{t}p_{T,A}^{\perp}\Delta^{\bd}_{T,A}p_{T,A}^{\perp})^{-1}p^\perp_{T,A},\\
                        &\tR_2^{\l,t}:=(\l^2-{t}\Delta_{T,A})^{-1}-(\l^2-{t}\Delta^\bd_{T,A})^{-1}-R_1^{\l,t}.
                    \end{aligned}
                \end{equation}
                Then
                \begin{equation}
                \begin{aligned}
                     &\|\tR_1^{\l,t}\|\leq \frac{C|\l|^2}{{t}T^2A^2},\\
                    &{p_T\tR_2^{\l,t}p_T=0 \quad\text{and}}\quad  \|\tR_2^{\l,t}\|\leq C{t}e^{-bTAr^2/2}.
                \end{aligned}
                \end{equation}
		\end{prop}

		\def\tp{{\tilde{p}}}

		\subsection{Estimates in family cases}\label{sec94}
		In this subsection, let $L$ be a closed manifold, and $S=L\times (-\delta^2,\delta^2)$\index[p2]{S@$S$}  for some small $\delta>0.$
		
		Let $\pi:M\to S$ be a fibration with fiber $n+1$ dimensional $Z$. Here we do not require $Z$ to be compact. We fix a splitting \be\label{splittm}TM=T^HM\oplus TZ.\ee
        
		Suppose there exists $f\in C^\infty(M)$, such that for any $\s\in S-L$, $f|_{Z_\s}$ is Morse and  for $\s\in L$, $f|_{Z_{\s}}$ is generalized Morse with exactly one birth-death point. 	As before, let $TZ \to M$ denotes the vertical bundle. Suppose that there exists a metric $g^{TZ}$ on $TZ \to M$ such that fiberwise, $(Z, g^{TZ}, f|_Z)$ is strongly tame. Also, assume that fiberwisely, birth-death point (and Morse points converges to birth-death point when $y\to0$) is independent from all other critical points under vertical gradient flow with respect to $g^{TZ}.$
		
		{As $L$ is compact, we can chose} $\rho_0$ in Observation \ref{obs8} and $(\delta, r)$ in Definition \ref{defr} uniformly for each fiber $f|_{\pi^{-1}(\s)}$, $\s \in S =L\times(-\delta^2,\delta^2)$.

        {We also deform the metric $g^{TZ}$ in the same way as in \cref{sec921} so that Lemma \ref{lem914} hold in each fiber, and that $\Lambda_0$ in this lemma can be chosen uniformly on $S$. Then Proposition \ref{Agmon4t} holds for this $\Lambda_0$ in each fiber, and with constants uniform on $S$.}

        Since the gradient of $f$, as well as $f_{A}$, are unchanged under the 
        deformation in \cref{sec921}, deforming $g^{TZ}$ will not affect our choice of constants $(r, \delta)$ above.
		
		Moreover, we assume that the following holds.
		\begin{assum}\label{9c}
        There is a neighborhoods $V$ of $\Sigma^{(1)}(f)$ in $M$ and smooth embeddings
		\be\label{embedding Valpha}
		\left(\pi, \phi_\alpha\right)=\left(\pi, u_{0,\a}, u_{1,\a}, \cdots, u_{n,\a}\right): V_\alpha \rightarrow S \times \mathbb{R}^{n+1}
		\ee
		defined on open sets $V_\alpha$ covering the entire critical set $\Sigma(f)$ and $V$ so that (We will write $u_{k,\a}$ as $u_{k}$ if it will not cause any confusion):
			\begin{enumerate}[(1)]
				\item $\pi(V)=S.$
				\item\label{9c2}  On $V_\alpha \cap V_\beta, \phi_\alpha$ and $\phi_\beta$ differ by a rotation $g_{\alpha \beta}: \pi\left(V_\alpha \cap V_\beta\right) \rightarrow O(n+1)$
				\item\label{9c3} For some $i\in\mathbb{N}$, $f |_{\left(V_\alpha \cap V\right)}=c\circ \pi+u_0^3- y u_0-\left\|\left(u_1, \cdots, u_i\right)\right\|^2+\left\|\left(u_{i+1}, \cdots, u_n\right)\right\|^2$, here $y:M\to S=L\times(-\delta^2,\delta^2) \to (-\delta^2,\delta^2)$ is the composition of $\pi$ with the projection onto the second factor.
				\item\label{9c4} $f|_{\left(V_\alpha-V\right)}=c\circ\pi+\sum_{j=0}^n \pm u_j^2$.
			\end{enumerate}
		\end{assum}
		Assume that metrics $g^{TZ}$ and $g^{TZ'}$ satisfies \Cref{9a} fiberwisely. Moreover, identify $V_\alpha$ with an open subset of $S\times\R^{n+1}$ via \eqref{embedding Valpha},
        \be\label{standard S indep}
g^{TZ}=\sum_{i,j}g_{ij}du_{i,\alpha}\otimes du_{j,\alpha}
        \ee
        for some $S$-independent $g_{ij}.$
        
        Let $V_{\a,\s}:=V_{\a}\cap \pi^{-1}(\s)$, we may as well assume that $V_{\a,\s}$ contains a ball of radius $100$ (w.r.t. $g^{TZ'}$), with center satisfying $u_{0,\a}=\cdots=u_{n,\a}=0.$ Let $i_\a:S\cap \pi(V_\a)\to V_\a$ be the map defined by sending $\s$ to the point in the corresponding fiber in $V_\a$ with $u_{k,\a}=0,k=0,\cdots,n$, then  we require that the metric $h^F$ satisfies
		\be\label{standard1}
		h^F|_{V_\a}=\pi^*h^{i_\a^*F} \mbox{ for some metric $h^{i_\a^* F}$ on $i_\a^*(F|_{V_\a})\to S\cap\pi(V_\a).$}
		\ee
		
		One can see easily that fiberwisely, $g^{TZ}$ satisfies Assumption \ref{9a}. Let $V_\s:=V\cap \pi^{-1}(\s),$ we may as well assume that $V_\s$ is a ball of radius $80$ with respect to $g^{TZ'}$. Then for each $A\in [0,\infty)$, we can {fiberwisely} deform  $f$ to $f_A$ as in \eqref{fa}.   
		
		Let $(F\to M,\nabla^F,h^F)$ be a unitarily flat bundle. $d^M:\Omega^{\bullet}(M,F)\to   \Omega^{\bullet}(M,F)$ be the de Rham differential induced by $\nabla^F$. As before, we have infinite dimensional bundle $\bfE\to S$, where $\bfE:= \Omega^{\bullet}(Z;F|_Z).$\index[p2]{E@$\bfE$} Let $h^F_{T,A}:=e^{-2Tf_{A}}h^F$\index[p2]{hFTA@$h^F_{T,A}$}, then $h^F_{T,A}$ (resp. $h^F$) and $g^{TZ}$ induce a metric $h_{T,A}^{\bfE}$\index[p2]{hETA@$h^{\bfE}_{T,A}$} (resp. $h^{\bfE}$\index[p2]{hE@$h^{\bfE}$}) on  $\bfE$. Let $d^{M,*}_{T,A}$ be the formal adjoint operator of $d^M$ with respect to $h_{T,A}^{\bfE}.$ 
		
		\def\bfT{{\mathbf{T}}}
		As in \eqref{decdm}, we have decomposition $d^M=d^Z+\nabla^{\bfE}+i_{\bfT}$. Let $\nabla^{\bfE,*}$ be the formal adjoint connection of $\nabla^{\bfE}$ with respect to $h^{\bfE}$ (not $h^{\bfE}_{T,A}$). Let $d^{Z,*}_{T,A}$ be the formal adjoint of $d^Z$ with respect to $h^{\bfE}_{T,A}$. For $Y\in\Gamma(TS)$, let $Y^H\in \Gamma(T^HM)$ be the lifting of $Y$ with respect to splitting \eqref{splittm}, and define $d^Sf_A(Y):=Y^Hf_A$. Finally, for $U\in \Gamma(TZ)$, define $U^\flat\in\Omega^1(Z)$ to be the metric dual of $U$. Then as in \eqref{decdmstar}, 
        \begin{equation}
		    d^{M,*}_{T,A}=d^{Z,*}_{T,A}+\nabla^{\bfE,*}-2Td^Sf_A+\bfT^{ \flat}\wedge,
		\end{equation}

		For $U\in\Gamma(TZ)$, we set $c(U):=U^\flat\wedge-i_U$, $\hc(U):=U^\flat\wedge+i_U$.
		
		Let $\cD_{t,T,A}\index[p2]{DtTA@$\cD_{t,T,A}$}:=-\frac{1}{2}\left(t^{N^Z/2}d^Mt^{-N^Z/2}-t^{-N^Z/2}d_{T,A}^{M,*}t^{N^Z/2}\right),$ then
		\be
		\cD_{t,T,A}=\frac{1}{2}\left(\sqrt{t}(d^{Z,*}_{T,A}-d^Z)+\o^{\bfE}-2Td^S f_A+\frac{c(\bfT)}{\sqrt{t}}\right),
		\ee
		where $\o^{\bfE}:=\nabla^{\bfE,*}-\nabla^{\bfE}.$\index[p2]{omegaE@$\o^{\bfE}$} 
		
		Set \be\label{conjugation} B_{t,T,A}=e^{-Tf_A}\cD_{t,T,A}e^{Tf_A}.\ee\index[p2]{BtTA@$B_{t,T,A}$} Recall that $d_{T,A}=d^Z+Tdf_A\wedge,$ that $d^*_{T,A}$ is the formal adjoint of $d_{T,A}$ with respect to $h^{\bfE}$, and that  $V_{T,A}:=d^*_{T,A}-d_{T,A}$. Then 
		\be\label{eq128}
		B_{t,T,A}=\frac{1}{2}\left(\sqrt{t}V_{T,A}+\o^{\bfE}-2Td^S f_A+\frac{c(\bfT)}{\sqrt{t}}\right).
		\ee
		Recall that $\Delta_{T,A}^{\bd}$ (resp. $\Delta_{T,A}^{\bd,\prime}$) was defined in \eqref{dtabd} (resp. \eqref{dtabd1}). We can define in the same way $B_{t,T,A}^{\bd}$ and $B_{t,T,A}^{\bd,\prime}$\index[p2]{BtTAbd@$B_{t,T,A}^{\bd}$ and $B_{t,T,A}^{\bd,\prime}$} as well as $V^{\bd}_{T,A}$  and $V^{\bd,\prime}_{T,A}$\index[p2]{VtTAbd@$V_{t,T,A}^{\bd}$ and $V_{t,T,A}^{\bd,\prime}$}. Then we have
		\be
		B_{t,T,A}^{\bd}=\frac{1}{2}\left(\sqrt{t}V^{\bd}_{T,A}+\o^{\bfE}-2Td^S f_A+\frac{c(\bfT)}{\sqrt{t}}\right)
		\ee
		and 	
		\be\label{eq122}
		B_{t,T,A}^{\bd,\prime}=\frac{1}{2}\left(\sqrt{t}V^{\bd,\prime}_{T,A}+\o^{\bfE}-2Td^S f_A+\frac{c(\bfT)}{\sqrt{t}}\right).
		\ee

		\subsubsection{When $A$ is large}\label{sec941}

		In \cref{sec941}, we require that $A\geq A_6$ with $A_6$ chosen as in \eqref{A6}, uniformly on $S$.
		
		As in \eqref{h1ta} and \eqref{def-op-norm-TA}, we have the fiberwise norm $\|\cdot\|_{H^l_{T,A}}$ and $\|\cdot\|_{T,A}^{l,j}$, $l,j\in\{-1,0,1\}$ associated with the fiberwise operator $D_{T,A}:=d_{T,A}+d_{T,A}^*$. Also, recall that $\|\cdot\|_{T,A}^{0,0}$ is independant on $(T,A)$ and is simply denoted by $\|\cdot\|$.

        {Let $g^{TS}$ be a metric on $S$. Using \eqref{splittm} we extend these norms to $\Omega^\bullet(S,\bfE)$ and $\Omega^\bullet(S,\End(\bfE))$ respectively, by taking the $C^0$ norm for the $\Omega^\bullet(S)$ factor.

		\def\inthis{0}
		\if\inthis0
		Recall that $S=L\times(-\delta^2,\delta^2)$ and that $y$ denotes the coordinate of $(-\delta^2,\delta^2)$.  Recall also that $V_\a$ is described in \Cref{lem441}.
		\begin{prop}\label{uvbounded}{Fix $\tau>0$. There exist three operators $U^1_{t,T,A}$, $U^2_{t,T,A}$ and $dy\wedge U^3_{t,T,A}$\index[p2]{UitTA@$U^i_{t,T,A}$, $i=1,2,3$} in $\Omega^{\geq1}(S,\End(\bfE))$
        such that 
        	\be
            \label{uvbounded-eq}
                {B}^2_{t,T,A}=\frac{-t}{4}\Delta_{T,A}+U^1_{t,T,A}+U^2_{t,T,A}+dy\wedge U^3_{t,T,A},
			 \ee
        and such that when $t\geq\tau$, 
        \begin{equation}
        \label{est.-00norm-U1tTa-U3tTA}
            \|U^1_{t,T,A}\|+\|U^3_{t,T,A}\|\leq C\sqrt{t}T 
			\quad \text{and} \quad \|U^2_{t,T,A}\|\leq C,
        \end{equation}
        for some $(T,A,t)$-independent constant $C$. Moreover, $U^1_{t,T,A}$ vanishes in $\cup_{\a}V_\a$. }
        
      {A similar statement holds for $B_{t,T,A}^{\bd}$. } 
		\end{prop}
		\begin{proof}
			Since $S=L\times (-\delta^2,\delta^2)$, $d^S=d^L+dy\wedge\frac{\p}{\p y}.$

			Let 
            \begin{equation}
                \label{def-U1tTa-U3tTA}
                U^1_{t,T,A} := \frac{1}{4}[\sqrt{t}V_{T,A},\omega^{\bfE} - 2T d^L f_A] \quad \text{and} \quad U^3_{t,T,A} := \frac{1}{4} i_{\frac{\partial}{\partial y}} [\sqrt{t}V_{T,A}, -2T dy \frac{\partial}{\partial y} f_A].
            \end{equation}
            Then by item \eqref{9c3} of Assumption \ref{9c}, 
            \begin{equation}
                \label{formula-U3tTA}
               U^3_{t,T,A} =-\frac{\sqrt{t}T}{2}\hat{c}(\nabla^{\mathrm{v}} \frac{\p}{\p y}\tilde{f}_y),
            \end{equation}  where $\tilde{f}_y$ is constructed in \cref{model} and $\nabla^{\mathrm{v}}$ denotes the fiberwise gradient. We see from the definition above that the first part of \eqref{est.-00norm-U1tTa-U3tTA} holds.

			By \eqref{standard S indep}, \eqref{standard1}, item (\ref{9c3}) and item (\ref{9c4}) of Assumption \ref{9c}, $ U^1_{t,T,A} $ vanishes in $ \cup_{\alpha} V_{\alpha} $.
			
			
			Let 
            \begin{equation}
                \label{def-U2tTA}
                U^2_{t,T,A} := \frac{1}{4} \Big[\sqrt{t} V_{T,A} + \omega^{\bfE} - 2T d^S f_A, \frac{c(\bfT)}{\sqrt{t}}\Big] + \frac{1}{4} (\omega^{\bfE} - 2T d^S f_A)^2+\frac{1}{4t} c(\bfT)^2.
            \end{equation}
            Then as $ \|[V_{T,A}, c(\bfT)]\| \leq C $, $ \|[\omega^{\bfE}, c(\bfT)]\| \leq C $, and $ [d^S f_A, c(\bfT)] = [\omega^{\bfE}, d^S f_A] = 0 $, we get  the second part of \eqref{est.-00norm-U1tTa-U3tTA}.

            Now, the first line of \eqref{uvbounded-eq} follows from \eqref{eq128} and the definition of $ U^i_{t,T,A} $ for $i=1,2,3$.

            Proposition \ref{uvbounded} is proved.
			\end{proof}
		
		As in \eqref{pta}, we have a fiberwise orthogonal projection $p_{T,A}$ and $p^\perp_{T,A}$ defined form 
		$\Delta_{T,A}$.	\fi

     {In the reminder of this section, we will use the following identification:
        \begin{equation}
        \label{identification-L2(Z)=L2(Z1)+L2(Z2)}
            \begin{aligned}
               L^2 \Omega^\bullet(Z,F) &\simeq L^2\Omega^\bullet(Z_1,F)\oplus L^2\Omega^\bullet(Z_2,F) \\
                u &\mapsto (u|_{Z_1}, u|_{Z_2})
            \end{aligned}
        \end{equation}}
		
		\begin{lem}\label{lem916}
			For any $\tau\in(0,1)$, there exists a $(\l,t,T,A)$-independent constant $C>0$, such that if $t\in [\tau,e^{bTAr^2/8}]$ and $|\Re(\l)|=1$, {under the identification \eqref{identification-L2(Z)=L2(Z1)+L2(Z2)},}
			\[\left\|dy\wedge(\l^2-B^2_{t,T,A})^{-1}-dy\wedge\left(\l^2-\left(B^{\bd}_{t,T,A}\right)^2\right)^{-1}\right\|\leq \frac{C|\l|^{2\dim(S)}}{\sqrt{t}T^3A^2},\]
			and
			\[\left\|i_{\frac{\p}{\p y}}(\l^2-B^2_{t,T,A})^{-1}-i_{\frac{\p}{\p y}}\left(\l^2-\left(B^{\bd}_{t,T,A}\right)^2\right)^{-1}\right\|\leq \frac{C|\l|^{2\dim(S)}}{\sqrt{t}T^2A^2}.\]
			
		\end{lem}
		\begin{proof}
			We emphasize that the constants appearing in the proof are all independent of $\lambda$, $T$, $A$, and $t$ but may depend on $\tau$.
			
			For $w\in  \Omega^{\bullet}_c(Z,F)$, $t>\tau$ and $l\in\mathbb{N}$, we define a norm $\|\cdot\|_{H^l_{T,A,t}}$ as follows:
			\be\label{h1ta1}\|w\|_{H_{T,A,t}^l}=\sum_{l'=0}^l\|(\sqrt{t}D_{T,A})^{l'}w\|_{L^2},\ee
			where $\|\cdot\|_{L^2}$ is the $L^2$-norm induced by $g^{TZ}$ and $h^F$.
			Let $H^l_{T,A,t}( \Omega^{\bullet}(Z,F))$ be the completion of $ \Omega^{\bullet}_c(Z,F)$ with respect to $\|\cdot\|_{H^l_{T,A,t}}.$ Let also $H^{-l}_{T,A,t}( \Omega^{\bullet}(Z,F))$ be the anti-dual of $H^{l}_{T,A,t}( \Omega^{\bullet}(Z,F))$ with a norm $\|\cdot\|_{H^{-l}_{T,A,t}}$ associated with  $\|\cdot\|_{H^l_{T,A,t}}$. Note that  $\|\cdot\|_{H^0_{T,A,t}}=\|\cdot\|_{L^2}$ in fact do not depend on $T,A,t$ and $D_{T,A}$.
			
			For any bounded linear operator $\cL$ from $H^l_{T,A,t}$ to $H^{l'}_{T,A,t}$, $l,l'\in\Z$, $\|\cL\|^{l,l'}_{T,A,t}$ be the operator norm with respect to $\|\cdot\|_{H^l_{T,A,t}}$ and $\|\cdot\|_{H^{l'}_{T,A,t}}$. 	According to our definition, the norm $\|\cdot\|_{T,A,t}^{0,0}$ do not depend on $T$, $A$, $t$, and $D_{T,A}$. Therefore, we will simply denote it as $\|\cdot\|$.
			
			Similarly to Proposition \ref{Agmon31}, we have for {$l\in \Z$ and} $k\in\{0,1,2\}$
			\be\label{eq133}
			\|p_{T,A}^{\perp}\|_{T,A,t}^{l,l-k}\leq C(\sqrt{t}TA)^{-k}\leq C(\sqrt{t}T)^{-k}.
			\ee
			
			Let $ \Omega^{\bullet}_{\abs}(Z_1,F)$ denote the space of differential forms on $Z_1$ that satisfy the absolute boundary condition, and $ \Omega^{\bullet}_{\rel,c}(Z_2,F)$ denote the space of compactly supported differential forms on $Z_2$ that satisfy the relative boundary condition. We can then define Sobolev-type norms on $ \Omega^{\bullet}_{\abs}(Z_1,F) \oplus  \Omega^{\bullet}_{\rel,c}(Z_2,F)$ as in \eqref{h1ta1} but with $D_{T,A}$ replaced by $D_{T,A}^\bd$. These norm are still denoted by $\|\cdot\|_{H^l_{T,A,t}}$. We also define the associated Sobolev-type spaces and Sobolev-type operator norm as above. 
            
            As a result, whenever we use the norm $\|\cdot\|_{H^l_{T,A,t}}$ or $\|\cdot\|_{H^{l,l'}_{T,A,t}}$, it may refer to the norm constructed from either $D_{T,A}$ or $D_{T,A}^\bd$, depending on the context. 
			
			\def\itiseasy{0}
			\if\itiseasy1	
			It is easy to see that
			\be\label{lem811eq0}
			p_{T,A}(\l-\sqrt{t}p_{T,A} V_{T,A}p_{T,A})^{-1}p_{T,A}=p_{T,A}(\l-\sqrt{t}p_{T,A} V^{\bd}_{T,A}p_{T,A})^{-1}p_{T,A}
			\ee
			By Proposition \ref{Agmon3}, \be\label{lem811eq1} \|p_{T,A}^{\perp}(\l-\sqrt{t}p_{T,A}^\perp V_{T,A}p_{T,A}^\perp)^{-1}p_{T,A}^\perp\|_{T,A}^{0,0}\leq \frac{C|\l|}{\sqrt{t}AT },\ee
			and \be\label{lem811eq11} \|p_{T,A}^{\perp}(\l-\sqrt{t}p_{T,A}^\perp V^{\bd}_{T,A}p_{T,A}^\perp)^{-1}p_{T,A}^\perp\|_{T,A}^{0,0}\leq \frac{C|\l|}{\sqrt{t}AT }.\ee
			
			\fi

			Observe that in \eqref{uvbounded-eq} and its version with $B_{t,T,A}^{\bd}$, by \eqref{def-U1tTa-U3tTA} and \eqref{def-U2tTA}, we have actually $U_{T,A}^j= U_{T,A}^{j,\mathrm{bd}}$ under the identification \eqref{identification-L2(Z)=L2(Z1)+L2(Z2)}. This is because these operators are induced by smooth morphism of bundles. Thus, we get
			\be\label{lem811eq2}(\l^2-B^2_{t,T,A})^{-1}=\sum_{j=0}^{\dim(S)}(\l^2+{t}\Delta_{T,A}/4)^{-1}\left(\left(U^1_{t,T,A}+U^2_{t,T,A}+dy\wedge U^3_{t,T,A} \right)(\l^2+{t}\Delta_{T,A}/4)^{-1}\right)^j,\ee
			and	
			\be\label{lem811eq3}\left(\l^2-\left(B_{t,T,A}^{\bd}\right)^2\right)^{-1}=\sum_{j=0}^{\dim(S)}(\l^2+{t}\Delta^{\bd}_{T,A}/4)^{-1}\left(\left(U^{1}_{t,T,A}+U^{2}_{t,T,A}+dy\wedge U^{3}_{t,T,A} \right)(\l^2+{t}\Delta^{\bd}_{T,A}/4)^{-1}\right)^j.\ee
			
			Since there is $K>0$ such that for $\mu\in\R$ we have $|\frac{\mu}{\l\pm i\mu}|\leq K|\l|$ and $|\frac{1}{\l\pm i\mu}|\leq1$, we obtain that {for $l\in\Z$, there is $C>0$ such that} 
            \be\label{eq1331}\|(\l\pm i\sqrt{t}D_{T,A}^{\bullet}/2)^{-1}\|_{T,A,t}^{l,l+1}\leq C|\l| \quad\text{and}\quad\|(\l\pm i\sqrt{t}D_{T,A}^{\bullet})/2^{-1}\|_{T,A,t}^{l,l}\leq 1.\ee 
			Here, $\Delta^{\bullet}_{T,A}$ may refer to either $\Delta_{T,A}$ or $\Delta^{\bd}_{T,A}$. We use the same convention when writing $D_{T,A}^\bullet$.

			By \eqref{def-U1tTa-U3tTA}, we see that {for $l\in\Z$, there is $C>0$ such that for $j=$ 1 or 3,}
			\begin{align}\begin{split}\label{lem811eq31}
					\|U^j_{T,A}\|_{T,A,t}^{l,l-1}\leq CT.
			\end{split}\end{align}

			By Proposition \ref{uvbounded}, we also have
			\be\label{lem811eq32}
			\|U^2_{T,A}\|_{T,A,t}^{0,0}\leq C.
			\ee

			Recall that $R^{\l,t}_i$, $i=1,2$, have been defined in \eqref{def-Rlambdat1-Rlambdat2}. By Proposition \ref{Agmon4} and \eqref{eq133}, {for $l\in\Z$, there is $C>0$ such that for  $k\in\{0,1,2\}$ and $t\in[\tau,e^{bTAr^2/8}]$,}
			\be\label{lem811eq41}\|R^{\l,t}_1\|_{T,A,t}^{l,l-k}\leq\frac{C|\l|^2}{{t}T^{2+k}A^2} \quad\text{and}\quad \|R^{\l,t}_2\|_{T,A,t}^{0,0}\leq \frac{Ct}{e^{bTAr^2/2}}\leq\frac{C}{\sqrt{t}e^{bTAr^2/4}}.\ee

            Let $c>0$ be a constant, whose value will be determined later. We will prove  Lemma \ref{lem916} by separating the cases $t\in [e^{cT},e^{\frac{bTAr^2}{8}}]$ and $t\in [\tau,e^{cT}]$.\\
            
			\underline{The case where $t\in[e^{cT},e^{\frac{bTAr^2}{8}}]$.}	
			By \eqref{lem811eq41}, we have in this case that for $k\in\{0,1,2\}$,
			\be\label{lem811eq411}\|R^{\l,t}_1\|_{T,A,t}^{l,l-k}\leq\frac{C|\l|e^{-cT/2}}{\sqrt{t}T^{2+k}A^2}.\ee

           { By the identity
            \be\label{albl}
a_1\cdots a_l - b_1\cdots b_l = \sum_{j=1}^{l} b_1\cdots b_{j-1} (a_j-b_j) a_{j+1}\cdots a_l,
\ee
in each term  of \[
(\lambda^2 - B_{t,T,A}^2)^{-1} - \left(\lambda^2 - (B_{t,T,A}^{\bd})^2\right)^{-1}
\] 
expressed as a sum obtained by taking the difference of \eqref{lem811eq2} and \eqref{lem811eq3}, there is precisely one occurrence of the factor  
\[
(\lambda^2 + t\Delta_{T,A}/4)^{-1} - (\lambda^2 - t\Delta^\bd_{T,A}/4)^{-1} = R_1^{\lambda,t} + R_2^{\lambda,t}.
\]  
Moreover, we have
\[
(\lambda^2 - t\Delta^{\bullet}_{T,A}/4)^{-1} = (\lambda + i\sqrt{t}D^{\bullet}_{T,A}/2)^{-1} (\lambda - i\sqrt{t}D^{\bullet}_{T,A}/2)^{-1}.
\]  
Finally, note that in each summation of \eqref{lem811eq2} and \eqref{lem811eq3}, the copies of $(\lambda^2 + t \Delta^{\bullet})^{-1}$ outnumber the copies of $U_{t,T,A}^j$ by at least one copy. Thus, by Proposition \ref{Agmon4} (applied for $\lambda'$ with $|\Im(\lambda')|=1$ and $\lambda'^2=\lambda^2$) and \eqref{lem811eq2}-\eqref{lem811eq411}, for some $c_0>0$,
			\begin{align}\begin{split}\label{lem811eq4}&\ \ \ \ \left\|(\l^2-B_{t,T,A}^2)^{-1}-\left(\l^2-\left(B_{t,T,A}^{\bd}\right)^2\right)^{-1}\right\|\leq \frac{C|\l|^{2\dim(S)}e^{-c_0T}}{\sqrt{t}A^2}.
			\end{split}\end{align}
            So the Lemma holds when $t\in[e^{cT},e^{\frac{bTAr^2}{2}}].$
            }\\

            \underline{The case where $t\in[\tau,e^{cT}]$.} By Lemma \ref{agmonz}, we can see that if $\eta\in C^\infty(Z)$ is supported away from $\cup_{\a}V_\a$, then there exists $c_1=c_1(f,g^{TZ},h^F,\mathrm{supp}\eta),c_2=c_2(f,g^{TZ},h^F,\mathrm{supp}\eta)$ such that if $w$ is an eigenform of $\Delta^{\bullet}_{T,A}$ with respect to eigenvalue $\leq c_1T^2$, then 
			\be\label{eq1401}\int_{Z}|\eta w|^2\leq e^{-2c_2T}\int_Z|w|^2.\ee 
			
			Let $\hat{p}^{\bullet}_{T,A}$ denote the orthogonal projection onto the space spanned by the eigenforms of $\Delta^{\bullet}_{T,A}$ corresponding to eigenvalues $\leq c_1T^2$, and $\hat{p}^{\bullet,\perp}_{T,A}:=1-\hat{p}^{\bullet}_{T,A}$.
			
			Then since {$U^{1,\bullet}_{T,A}$} is supported away from $\cup_{\a}V_\a$, by {the first part of \eqref{est.-00norm-U1tTa-U3tTA} and} \eqref{eq1401}, we can see that \be\|\hat{p}^{\bullet}_{T,A} U^{1,\bullet}_{T,A}\hat{p}^{\bullet}_{T,A}\|_{T,A,t}^{0,0}\leq \sqrt{t} Te^{-c_2T}.\ee
			
			If $ t \in [\tau, e^{cT}] $ for $ c = c_2 / 2 $,
			\be\label{lem811eq5}\|\hat{p}^{\bullet}_{T,A} U^{1,\bullet}_{T,A}\hat{p}^{\bullet}_{T,A}\|_{T,A,t}^{0,0}\leq \frac{{T}e^{-c_2T/2}}{\sqrt{t}}.\ee
			
			By the first part of \eqref{est.-00norm-U1tTa-U3tTA} and \eqref{eq133},
			\begin{align}\begin{split}\label{lem811eq6}\|\hat{p}^{\bullet,\perp}_{T,A} U^1_{T,A}\hat{p}^{\bullet}_{T,A}+\hat{p}^{\bullet}_{T,A} U^1_{T,A}\hat{p}^{\bullet,\perp}_{T,A}+\hat{p}^{\bullet,\perp}_{T,A} U^1_{T,A}\hat{p}^{\bullet,\perp}_{T,A}\|_{T,A,t}^{l,l-2}&\leq \frac{C}{\sqrt{t}T}.
			\end{split}\end{align}

			Note that in each summation of \eqref{lem811eq2} and \eqref{lem811eq3}, the copies of $(\lambda^2 + t \Delta^{\bullet})^{-1}$ outnumber the copies of $U_{t,T,A}^i$'s by at least one copy. Thus, {by the same reasoning as for \eqref{lem811eq4} but replacing the estimate \eqref{lem811eq31} for $j=1$ with \eqref{lem811eq5} and \eqref{lem811eq6}} we find from \eqref{lem811eq2}-\eqref{lem811eq41} and Proposition \ref{Agmon4} that if $ t \in [\tau, e^{cT}] $ for $ c = c_2 / 2 $, we have
			\begin{equation}
			    \label{lem916-proof-estimate-dy}
                \left\|dy\wedge(\l^2-B_{t,T,A}^2)^{-1}-dy\wedge\left(\l^2-\left(B_{t,T,A}^{\bd}\right)^2\right)^{-1}\right\|\leq \frac{C|\l|^{2\dim(S)}}{\sqrt{t}T^3A^2}.
			\end{equation}

			Next, we would like to estimate $\left\|i_{\frac{\p}{\p y}}(\l^2-B_{t,T,A}^2)^{-1}-i_{\frac{\p}{\p y}}\left(\l^2-\left(B_{t,T,A}^{\bd}\right)^2\right)^{-1}\right\|$ for $ t \in [\tau, e^{cT}]$, with $c=c_2/2$ as above.
			
			Since $ U_{t,T,A}^3 $ doesn't vanish in $ \cup_{\alpha} V_{\alpha} $, estimates \eqref{lem811eq5} and \eqref{lem811eq6} do not hold for $ U_{t,T,A}^3 $. But we the following observation
    
			\begin{obs}\label{obs2}
				Since $dy$ super-commutes with all the operators above, there are at most one copy of $dy\wedge U^3_{t,T,A}$ in each term of the summation in \eqref{lem811eq2} and \eqref{lem811eq3}.
			\end{obs}

			Proceeding as in \eqref{lem916-proof-estimate-dy}, we find
			\[ \left\|i_{\frac{\partial}{\partial y}}\Big(\lambda^2 - B_{t,T,A}^2\Big)^{-1} - i_{\frac{\partial}{\partial y}}\Big(\lambda^2 - \big(B_{t,T,A}^{\bd}\big)^2\Big)^{-1}\right\| \leq \frac{C |\lambda|^{2 \dim(S)}}{\sqrt{t} T^2 A^2}. \]

		The proof of Lemma \ref{lem916} is complete.	
		\end{proof}
		
		\begin{lem}\label{lem924}	There exists a $(\l,t,T,A)$-independent constant $C>0$, such that if $t\in (0,1)$ and $|\Re(\l)|=1$, {under the identification \eqref{identification-L2(Z)=L2(Z1)+L2(Z2)},}
        
        \[\left\|dy\wedge(\l^2-t^{-1}B^2_{t,T,A})^{-1}-dy\wedge\left(\l^2-t^{-1}\left(B^{\bd}_{t,T,A}\right)^2\right)^{-1}\right\|\leq \frac{C|\l|^{2\dim(S)}}{t^{\dim(S)}T^3A^2},\]
			and
			\[\left\|i_{\frac{\p}{\p y}}(\l^2-t^{-1}B^2_{t,T,A})^{-1}-i_{\frac{\p}{\p y}}\left(\l^2-t^{-1}\left(B^{\bd}_{t,T,A}\right)^2\right)^{-1}\right\|\leq \frac{C|\l|^{2\dim(S)}}{t^{\dim(S)}T^2A^2}.\]
		\end{lem}
		\begin{proof}
			Let $N^S\in\End( \Omega^{\bullet}(S))$, s.t. if $w$ is a $k$-form on $S$, $N^Sw=kw.$
			
			Note that
			\[(\l^2-t^{-1}B^2_{t,T,A})^{-1}=t^{-\frac{N^S}{2}}(\l^2-B^2_{1,T,A})^{-1},\]
			the Lemma then follows from Lemma \ref{lem916}.  
		\end{proof}

		We also have estimates for Schauder norms.
		
		\begin{lem}\label{lem920} Assume that Assumption \ref{ass91} holds. Then for $\bfn\geq1$ large enough and $\tau\in(0,1)$, there exists a $(T,\l,t)$-independent but $(\tau,A,\bfn)$-dependent constant $C>0$ such that if $t\in [\tau,+\infty)$ and $|\Re(\l)|=1$,
			\[
			\|dy\wedge(\l^2-B^2_{t,T,A})^{-1}\|_{\bfn}\leq C(t^{\frac{1}{4}}T^{\frac{1}{2}})^{\frac{1}{\bfn}}|\l|^{2\dim(S)}(1+\sqrt{t}^{-1}),
			\]
			and
			\[
			\|i_{\frac{\p}{\p y}}(\l^2-B^2_{t,T,A})^{-1}\|_{\bfn}\leq CT(t^{\frac{1}{4}}T^{\frac{1}{2}})^{\frac{1}{\bfn}}|\l|^{2\dim(S)}(1+\sqrt{t}^{-1}).
			\]

			A similar statement holds for $B_{t,T,A}^{\bd}$.
		\end{lem}
		\begin{proof} 
			We just prove the first estimate. The second one is proved similarly.
			By Observation \ref{obs2}, $dy\wedge U^3_{t,T,A}$ doesn't contribute to $dy\wedge (\l^2-B^2_{t,T,A})^{-1}$, i.e.,
			\begin{equation}
			    \label{dy-times-resolvent}
                dy\wedge(\lambda^2 - B_{t,T,A}^2)^{-1} = dy\wedge\sum_{j = 0}^{\dim(S)} (\lambda^2 + t \Delta_{T,A}/4)^{-1} \left( \left( U_{t,T,A}^1 + U_{t,T,A}^2 \right) (\lambda^2 + t \Delta_{T,A}/4)^{-1} \right)^j.
			\end{equation}

			We will use the norm $\|\cdot\|_{T,A,t}$ introduced in \eqref{h1ta1}. 
			
			As in the proof of Lemma \ref{lem916}, we will show the estimate separately for $ t \in [\tau, e^{cT}] $ and $ t \in [e^{cT}, \infty) $ for some $ c > 0 $.

			Let $K_{T,A}^{\l,t}:=dy\wedge(\l^2-B^2_{t,T,A})^{-1}(\l^2+t\Delta_{T,A}/4)^{\half}$. By \eqref{eq1331}, \eqref{lem811eq31}, \eqref{lem811eq32}  and \eqref{dy-times-resolvent}, we see that for $t\in[e^{cT},\infty)$:
			\begin{equation}
			    \label{estim.-KTAlt}
                \|K_{T,A}^{\l,t}\|\leq C|\l|^{2\dim(S)}T^{\dim(S)}\leq C' |\l|^{2\dim(S)}t^{\frac{1}{4\bfn}}.
			\end{equation}

			Then, using the H\"older estimate for the Schauder norms and Lemma \ref{lem97}, the lemma holds for $t\in[e^{cT},\infty)$.
			
			For $t\in(\tau, e^{cT}]$, we replace estimate \eqref{lem811eq31} for $U^1_{t,T,A}$ by \eqref{lem811eq5} and \eqref{lem811eq6}, we can see that \eqref{estim.-KTAlt} becomes
			\[\|K_{T,A}^{\l,t}\|\leq C|\l|^{2\dim(S)}\leq C|\l|^{2\dim(S)}t^{\frac{1}{4\bfn}}. \]
			
			The lemma then holds for $t\in(\tau, e^{cT}]$.
		\end{proof}
		Proceeding as in the proof of Lemma \ref{lem924}, we have
		\begin{lem}\label{lem926}
        Assume that Assumption \ref{ass91} holds. Then for $\bfn\geq1$ large enough  there exists a $(T,A,\l,t)$-independent but $\bfn$-dependent constant $C>0$ such that if $t\in (0,1)$ and $|\Re(\l)|=1$,
			\[
			\|dy\wedge(\l^2-t^{-1}B^2_{t,T,A})^{-1}\|_{\bfn}\leq C(A)(t^{\frac{1}{4}}T^{\frac{1}{2}})^{\frac{1}{\bfn}}|\l|^{2\dim(S)}t^{-\dim(S)},
			\]
			and
			\[
			\|i_{\frac{\p}{\p y}}(\l^2-B^2_{t,T,A})^{-1}\|_{\bfn}\leq C(A)T(t^{\frac{1}{4}}T^{\frac{1}{2}})^{\frac{1}{\bfn}}|\l|^{2\dim(S)}t^{-\dim(S)}
			\]
			
            A similar statement holds for $B_{t,T,A}^{\bd}$.
		\end{lem}
		
		\subsubsection{When $A$ is not large}\label{sec932}
		In \cref{sec932}, we assume that $A\in[0,A_7]$, with $A_7$ fixed as in the beginning of \cref{sec92} and uniformly on $S$.


		We proceed as in the proof of Lemmas \ref{lem916}, \ref{lem924} and \ref{lem926}, {with the constant $\Lambda_0$ of \cref{sec92} playing the role of $A$ in these Lemmas. 
    {Similarly to \eqref{identification-L2(Z)=L2(Z1)+L2(Z2)}, we will use the following identification:
        \begin{equation}
        \label{identification-L2(Z)=L2(Z1tilde)+L2(Z2tilde)}
            \begin{aligned}
               L^2 \Omega^\bullet(Z,F) &\simeq L^2\Omega^\bullet(\tilde{Z}_1,F)\oplus L^2\Omega^\bullet(\tilde{Z}_2,F) \\
                u &\mapsto (u|_{\tilde{Z}_1}, u|_{\tilde{Z}_2})
            \end{aligned}
        \end{equation}
        where the cutting $Z=\tilde{Z}_1\cup\tilde{Z}_2$ is defined in \cref{sec922}.}
        
        Using \Cref{Agmon4t} instead of \Cref{Agmon4}} we find the following results.
        
		\begin{lem}\label{lem9221}
        For any $\tau\in(0,1)$, there exists a $(\l,t,T,A)$-independent but $(\tau,A_7)$-dependent constant $C>0$, such that if $t\in [\tau,e^{{1.1\Lambda_0}T}]$ and $|\Re(\l)|=1$, {under the identification \eqref{identification-L2(Z)=L2(Z1tilde)+L2(Z2tilde)},}
			\[\left\|dy\wedge(\l^2-B^2_{t,T,A})^{-1}-dy\wedge\left(\l^2-\left(B^{\bd,\prime}_{t,T,A}\right)^2\right)^{-1}\right\|\leq \frac{C|\l|^{2\dim(S)}}{\sqrt{t}T^3},\] and
			\[\left\|i_{\frac{\p}{\p y}}(\l^2-B^2_{t,T,A})^{-1}-i_{\frac{\p}{\p y}}\left(\l^2-\left(B^{\bd,\prime}_{t,T,A}\right)^2\right)^{-1}\right\|\leq \frac{C|\l|^{2\dim(S)}}{\sqrt{t}T^2}.\]
		\end{lem}
		
		\begin{lem}\label{lem9222}
        There exists a $(\l,t,T,A)$-independent but $A_7$-dependent constant $C>0$, such that if $t\in (0,1)$ and $|\Re(\l)|=1$, {under the identification \eqref{identification-L2(Z)=L2(Z1tilde)+L2(Z2tilde)},}
			\[\left\|dy\wedge(\l^2-t^{-1}B^2_{t,T,A})^{-1}-dy\wedge\left(\l^2-t^{-1}\left(B^{\bd,\prime}_{t,T,A}\right)^2\right)^{-1}\right\|\leq \frac{C|\l|^{2\dim(S)}}{t^{\dim(S)}T^3},\] and
			\[\left\|i_{\frac{\p}{\p y}}(\l^2-t^{-1}B^2_{t,T,A})^{-1}-i_{\frac{\p}{\p y}}\left(\l^2-t^{-1}\left(B^{\bd,\prime}_{t,T,A}\right)^2\right)^{-1}\right\|\leq \frac{C|\l|^{2\dim(S)}}{t^{\dim(S)}T^2}.\]
		\end{lem}
		
		\begin{lem}\label{lem922}
        Assume that Assumption \ref{ass91} holds. Then for $\bfn\geq1$ large enough and $\tau\in(0,1)$, there exists a $(T,A,\l,t)$-independent but $(\tau,A_7,\bfn)$-dependent constant $C>0$ such that if $t\in [\tau,+\infty)$ and $|\Re(\l)|=1$,
			\[
			\|dy\wedge(\l^2-B^2_{t,T,A})^{-1}\|_{\bfn}\leq C(t^{\frac{1}{4}}T^{\frac{1}{2}})^{\frac{1}{\bfn}}|\l|^{2\dim(S)}(1+\sqrt{t}^{-1}),
			\] and
			\[
			\|i_{\frac{\p}{\p y}}(\l^2-B^2_{t,T,A})^{-1}\|_{\bfn}\leq CT(t^{\frac{1}{4}}T^{\frac{1}{2}})^{\frac{1}{\bfn}}|\l|^{2\dim(S)}(1+\sqrt{t}^{-1}).
			\]

			A similar statement holds for $B_{t,T,A}^{\bd,\prime}$.
		\end{lem}
		
		\begin{lem}\label{lem9220}
                Assume that Assumption \ref{ass91} holds. Then for $\bfn\geq1$ large enough  there exists a $(T,A,\l,t)$-independent but $(\bfn,A_7)$-dependent constant $C>0$ such that if $t\in (0,1)$ and $|\Re(\l)|=1$,
			\[
			\|dy\wedge(\l^2-t^{-1}B^2_{t,T,A})^{-1}\|_{\bfn}\leq C(t^{\frac{1}{4}}T^{\frac{1}{2}})^{\frac{1}{\bfn}}|\l|^{2\dim(S)}t^{-\dim(S)},
			\] and
			\[
			\|i_{\frac{\p}{\p y}}(\l^2-t^{-1}B^2_{t,T,A})^{-1}\|_{\bfn}\leq CT(t^{\frac{1}{4}}T^{\frac{1}{2}})^{\frac{1}{\bfn}}|\l|^{2\dim(S)}t^{-\dim(S)}.
			\]

			A similar statement holds for $B_{t,T,A}^{\bd,\prime}$.
		\end{lem}
		\begin{rem}\label{rem92}
			If we replace $ A $ with a family of positive functions $\phi_A: S \to \mathbb{R}$ parameterized by $A \in [0, A_7]$, such that $|\phi_A| \leq A$ and $|d^S \phi_A| \leq CA$, then construct $f_A$ as in \cref{assotwo} using $\phi_A$ instead of using the constant $A$ as we did in this section, i.e., at $\s\in S,$ $(f_{A})_\s=f+p_{\phi_A(\s)}$, we still have Lemma \ref{lem9221}-Lemma \ref{lem9220}.
		\end{rem}

  To conclude Part \ref{PartII}, let us note the following important fact.
\begin{rem}\label{Important rmk}
Note that the map $ e^{T f_A} : (\bfE, h^{\bfE}) \to (\bfE, h_{T,A}^{\bfE}) $, given by $ w \mapsto e^{T f_A} w $, is unitary. By \eqref{conjugation}, all the estimates in \Cref{PartII} stated for the object $ B_{t,T,A} $ (or $ B_{t,T,A}^\bd $, etc.) automatically apply to the corresponding object $ \mathcal{D}_{t,T,A} $ (or $ \mathcal{D}_{t,T,A}^\bd $, etc.), and vice versa.
\end{rem}

\newpage
\part{Proof of the intermediate results}\label{PartIII}

This part consists of \cref{sec7}--\cref{ineresult1}. The sections \cref{sec7}--\ref{sec11} are devoted to the proof of the intermediate results given in \cref{ineresult}. As explained there, these results imply Theorem \ref{main1}, which is the main result of our paper, under the assumption that each fiber has at most one birth-death point. Finally,  in  \cref{ineresult1} we  explain how to remove this assumption.

        \section{Proof of Proposition \ref{prop62}}\label{sec7}
\def\cD{{\mathcal{D}}}

The exact value of $R_0$ appearing in the statement of  Proposition \ref{prop62} have not yet be fixed, but in fact here we prove that this proposition holds regardless of the value of $R_0$ (although the constants involved depend on it). We can thus fix $R_0$ later.

In this section, we use the same notations as in \cref{ineresult}. However, here we will make the dependence on the bundle explicit in $\cD_{t,T,R}$, by denoting it by $\cD^\cF_{t,T,R}$ or $\cD^{\Cb^m}_{t,T,R}$ depending on the bundle we consider.

First, we show that
\begin{thm}\label{diffsm}
	Fix a metric $g^{TS}$ on $TS\to S.$ {There exist $t_0\in(0,1)$ and $c,C>0$ such that for $t\in(0,t_0)$, $T\geq 1$ and $R\geq 0$,}
	\begin{equation}
		\label{diffsm-eq}
		\left|\Tr_s\Big(N^{\bZ}h'(\cD^\cF_{t,T,R})\Big)^{>0}-\Tr_s\Big(N^{\bZ}h'(\cD^{\Cb^m}_{t,T,R})\Big)^{>0}\right|\leq Ct. 
	\end{equation}
	Here $|\cdot|$ is the pointwise metric on $ \Omega^{\bullet}(S)$ induced from the metric $g^{TS}$.
\end{thm}
\begin{proof}
	When $\s\in S-\Omega_1$, {the deformation with respect to the parameter $R$ is nonexistent, so the only difference between our situation and \cite{bismut2001families} is that the fiber is not compact. However, } {by Condition \ref{5a}, the metric $g^{T\bZ}$ agrees on $(\pi^{\mathrm{sp}})^{-1} (S-\Omega_1)$ with the standard product metric $g^{T\bZ'}$ introduced in Section \ref{suspen}. Thus, the integral kernels under consideration here are just the product of the integral kernel on $Z$, as considered in \cite{bismut2001families}, and of the heat kernel of the standard harmonic oscillators. In particular, noticing \eqref{eq39}, the techniques of \cite{bismut2001families} can be applied to prove that our heat kernels satisfy versions of \cite[Theorem 9.10 and Theorem 9.12]{bismut2001families} where $f$ is replaced with $N^{\bZ}$ and $g=1$. These estimates imply \eqref{diffsm-eq} when $\s\in S-\Omega_1$.}

	We will prove the Theorem for $\s\in\Omega_1-\Omega_1''$ in \cref{sec10}, and for $\s\in\Omega_1''$ in \cref{proofint20}.
\end{proof}

We also have
\begin{thm}\label{diffsm1}
	If the fiber $Z$ is contractible, then
	for $t\in(0,\infty)$, 
	\[\Tr_s\Big(N^{\bZ}h'(\cD^\cF_{t,T,R})\Big)=\Tr_s\left(N^{\bZ}h'(\cD^{\Cb^m}_{t,T,R})\right).\]	
	Moreover, the same holds if we allow $Z$ to have a boundary and use relative/absolute boundary conditions.
\end{thm}
\begin{proof}
	Let $\tilde\Omega\subset S$ be a contractible open set, s.t. there exists a trivialization $\Psi:\pi^{-1}(\tilde\Omega)\to \tilde\Omega\times Z$.

	Let $U:=\Psi^{-1}(\tilde\Omega\times Z)$, then $U$ is contractible. 
	
	Since $\cF\to \M$ is unitarily flat and $U$ is contractible, there exists a unitary isomorphism $\Phi: \cF|_{U\times\R^N\times\R^N}\to \Cb^m|_{U\times\R^N\times\R^N}$ that preserve the flat connections. That is, $h^\cF=\Phi^*h^{\Cb^m}$ and $\nabla^\cF=\Phi^*\nabla^{\Cb^m}.$
	
	Hence, we clearly have $\Tr_s\left(N^{\bZ}h'(\cD^\cF_{t,T,R})\right)=\Tr_s\left(N^{\bZ}h'(\cD^{\Cb^m}_{t,T,R})\right)$ in $\tilde\Omega.$ 
\end{proof}

The first estimate in Proposition \ref{prop62} then follows from Theorem \ref{diffsm} trivially. Let use now prove the second.
Recall that $\A'=\nabla^V+\p$. Let $\A'_t:=\nabla^V+\sqrt{t}\p$, and $\A''_{t,T}$ be the adjoint of $\A'_t$ with respect to metric $(e^{-T\bbf^V})^*h^V.$ 
Let $\A_{t,T}:=(\A''_{t,T}-\A'_t)/2.$\index{AtT@$\A_{t,T}$}

Let $\cC^V_{t,T}:=e^{-T\bbf^V}\A_{t,T}e^{T\bbf^V}$\index{CVtT@$\cC^V_{t,T}$}, then
\begin{align}\begin{split}\label{eq52}
		2\cC^V_{t,T}&=\nabla^{V,*}-\nabla^V-2T\big[\nabla^V,\bbf^{V}\big]+\sqrt{t}\big(e^{T\bbf^V}\p^*e^{-T\bbf^V}-e^{-T\bbf^V}\p e^{T\bbf^V}\big).
\end{split}\end{align}
Here $\nabla^{V,*}$ and $\p^*$ are the adjoint of $\nabla^V$ and $\p$ with respect to $h^V.$ 

Let \be \p_T:=e^{-T\bbf^V}\p e^{T\bbf^V}\text{ and } \p_T^*:=e^{T\bbf^V}\p^*e^{-T\bbf^V}.\ee

{ Let $\hat{h}(z)=(1+2z)e^z$, so that $h'(z)=\hat{h}(z^2)$. By  a computation similar to \cite[(5.19) and (5.20)]{bismutzhang1992cm}, 
	\begin{equation}
		\label{eq-before-53}
		\frac{\p}{\p t}\Tr_s\Big(N^Vh'(\cC^V_{t,T})\Big)= \frac{\p}{\p b}\Tr_s\Big([N,\cC^V_{t,T}]\hat{h}\big(C^{V,2}_{t,T}+b\frac{\p C^V_{t,T}}{\p t}\big)\Big)\Big|_{b=0}
	\end{equation}
	
	By \eqref{eq52}, we see that $[N,\cC^V_{t,T}]=-\frac{\sqrt{t}}{2}(\p_T^*+\p_T)$ and $\frac{\p C^V_{t,T}}{\p t}= \frac{1}{4\sqrt{t}}(\p_T^*-\p_T)$. On the other hand, notice that if $M=N+R$ with $N$ niloptent, then for $K_N=\max_{i}\{\|N^i\|\}$, we have $\|M^k\|\leq K_N(1+\|R\|)^k$. Thus, writing $\hat{h}$ as a power series, we get that there is $C>0$ such that if $t\in(0,1]$ and $T\geq1$,
	\begin{equation}
		\label{eq53}
		\begin{aligned}
			\left|\frac{\p}{\p t}\Tr_s\Big(N^Vh'(\cC^V_{t,T})\Big)\right| &=  \left|\frac{\p}{\p b}\Tr_s\Big([N,\cC^V_{t,T}]\hat{h}\big(C^{V,2}_{t,T}+b\frac{\p C^V_{t,T}}{\p t}\big)\Big)\Big|_{b=0}\right|\\
			&\leq  C(\|\p_T\|^2+\|\p_T^*\|^2)T^{\dim S}.
		\end{aligned}
	\end{equation}
}
\begin{itemize}
	\item 
	
	On $\Omega_0''$, fiberwisely, when there is a gradient flow relating critical points $p$ and $q$, then $|\bbf(p)-\bbf(q)|\geq c$ for some $T$-independent constant $c>0.$ So in $\Omega_0''$, for some $T$-independent constant $C$,
	\be\label{added55}
	\big\|e^{T\bbf^V}\p^*e^{-T\bbf^V}\big\|\leq Ce^{-cT}\mbox{ and } \big\|e^{-T\bbf^V}\p e^{T\bbf^V}\big\|\leq Ce^{-cT}.
	\ee

	Note that $\Tr_s(N^Vh'(\cC^V_{t,T}))^{>0}|_{t=0}=0$, by \eqref{eq53} and \eqref{added55},  for some $T$-independent constant $C$, \be\label{eq54}\Big|\Tr_s\big(N^Vh'(\cC^V_{t,T})\big)^{>0}\Big|=CtT^{\dim(S)}e^{-2cT}.\ee
	
	On $\Omega_0''$, the second inequality in Proposition \ref{prop62} follows from \eqref{eq54} , our bar convention (degree $0$ component is removed), and the fact that
	\[\T(\A',(e^{-T\bbf^V})^*h^V)^{>0}=\int_0^\infty\Tr_s(N^Vh'(\cC^V_{t,T}))^{>0}\ \frac{dt}{2t}.\]

	\item On $ \Omega_1'' $, as in the proof of Theorem \ref{int3}, we have (locally) the decomposition $ V = \tilde{V} \oplus W $. The contribution of $ W $ to torsion forms would cancel out if we subtract the corresponding term for the trivial bundle. For the contribution of $\tilde{V}$, we still have the following estimate similar to \eqref{added55} for some $T$-independent constant $c$ and $C$:
	\[
	\Big\| e^{T\bbf^V} \p^* e^{-T\bbf^V}|_{\tilde{V}} \Big\| \leq Ce^{-cT}\mbox{ and } \quad \Big\| e^{-T\bbf^V} \p e^{T\bbf^V} |_{\tilde{V}}\Big\| \leq Ce^{-cT}.
	\]
	For the same reason as above, the contribution of $\tilde{V}$ to the  {small-time part of the integral in the definition of} torsion will go to $0$ as $T\to\infty$. Thus, the second inequality in Proposition \ref{prop62} also holds on $ \Omega_1'' $.
\end{itemize}
		
		\section{Proof of Theorem \ref{int1}}\label{sec10}

As in \cref{sec7}, the exact value of $R_0$ appearing in the statement of  Theorem \ref{int1} have not yet be fixed, but in fact here we prove that this theorem holds regardless of the value of $R_0$ (although the constants involved depend on it). We can thus fix $R_0$ later.
        
By our construction of $\bbf_{T,R}$ in \cref{ineresult}, we only need to prove Theorem \ref{int1} for $\s\in \Omega_0''\cap \Omega_1$.

		Recall that $(\M, \bbf)$ is a double suspension of $(M, f)$, $\ppi: \M \to S$ is a fibration with fiber $\bZ := Z \times \R^N \times \R^N$, and the critical set of $\bbf$ is $\Sigma(\bbf) = \Sigma(f) \times \{0\} \times \{0\}$. {We will apply the results of \cref{sec9} concerning the case where $A$ is not large to $((\ppi)^{-1}(\Omega_0''\cap \Omega_1), \bZ, \bbf)$.} 
        
        We can chose $\rho_0$ and  $(\delta, r)$ to satisfy the conditions in Observation \ref{obs8}  and Definition \ref{defr} in each fiber for $\bbf|_{(\ppi)^{-1}(\s)}$, $\s \in\Omega_0''\cap \Omega_1$. We can apply Lemma \ref{lem441} on $\Omega_1$, so that Assumption \ref{9a} holds in each fiber over $\Omega_0''\cap \Omega_1$ and that Assumption \ref{9c} holds on the total space, up to shrinking $\Omega_1$ if necessary for the function $t_1$ in Lemma \ref{lem441} to satisfiy $|t_1|\leq \delta^2$. Then we set $r_1=7r,r_2=9r.$

		Moreover, it is easy to see that \Cref{ass91} holds fiberwisely for the double suspension $\M\to S$.

		{We also deform the metric $ g^{T\bZ} $ in a small neighborhood of $ \overline{\pi^{-1}(\Omega_1)} $ in the same way as in \cref{sec921} so that Lemma \ref{lem914} hold in each fiber, and that $\Lambda_0$ in this lemma can be chosen uniformly on $\Omega_0''\cap \Omega_1$. Then Proposition \ref{Agmon4t} holds for this $\Lambda_0$ in each fiber, and with constants uniform on $\Omega_0''\cap \Omega_1$.} Since the vertical gradient of $\bbf$ (as well as $\bbf_{T,R}$) is unchanged under the deformation of \cref{sec921}, deforming $g^{T\bZ}$ will not affect our choice of constants $(\delta,r)$ above.
        
		As a result, we can use all estimates in \cref{sec932} in this section.

		Recall that $\cE=\Omega(\bZ,\cF|_{\bZ})$, which is the bundle over $S$ endowed with the connection $\nabla^{\cE}$ defined as in \eqref{def-nablabfE}. Let $d^{\bZ}:  \Omega^{\bullet}(\bZ, \cF|_{\bZ}) \to \Omega^{\bullet+1}(\bZ, \cF|_{\bZ})$ be exterior differentiation along fibers induced by $\nabla^{\cF}$. We consider $d^{\bZ}$ to be an element of $\Gamma\left(S; \operatorname{Hom}(\cE^{\bullet}, \cE^{\bullet+1})\right)$.
		
		Recall that $d^{\M}:  \Omega^{\bullet}(\M, \cF) \to \Omega^{\bullet+1}(\M, \cF)$ is induced by $\nabla^{\cF}$. As in \eqref{decdm}, $d^{\M}$ can be seen as a flat superconnection of total degree 1 on $\cE$, and $(d^{\M})^2 = 0$ implies that
		$$
		(d^{\bZ})^2 = 0, \quad [\nabla^{\cE}, d^{\bZ}] = 0.
		$$
		
		Recall that $h^{\cE}_{T,R}$ is the metric on $\cE$ induced by $e^{-\bbf_{T,R}} h^{\cF}$ and $g^{T\bZ}$. Sometimes we will also denote $h^{\cE}_{T,R}$ by $(\cdot, \cdot)_{L^2_{T,R}}$, and denote $\sqrt{h^{\cE}_{T,R}(u, u)}$ by $|u|_{L^2_{T,R}}$.
		
		\def\naotr{{\nabla^{\cE,*}_{T,R}}}
		\def\nattr{{\nabla^{\cE,*}_{2,T,R}}}
		\def\dzotr{{d^{\bZ,*}_{T,R}}}
		\def\dzttr{{d^{\bZ,*}_{2,T,R}}}
		\def\dmotr{{d^{\M,*}_{T,R}}}
		\def\dmttr{{d^{\M,*}_{2,T,R}}}
		\def\Cotr{{\mathcal{C}_{t,T,R}}}		\def\Cttr{{\mathcal{C}_{2,t,T,R}}}
		\def\Dotr{{\mathcal{D}_{t,T,R}}}		\def\Dttr{{\mathcal{D}_{2,t,T,R}}}
		\def\Dzotr{{{D}^\bZ_{T,R}}}		\def\Dzttr{{\mathcal{D}^\bZ_{2,T,R}}}
		\def\cC{{\mathcal{C}}}
		\def\cT{{\mathbf{T}}}
		\def\deotr{{\Delta^{\bZ}_{T,R}}}
		\def\dettr{{\Delta^{\bZ}_{2,T,R}}}
		\def\deotop{{\Delta^{\Df'}_{1,T}}}
		\def\deottp{{\Delta^{\bZ^{-,\prime}}_{T}}}
		\def\deoto{{\Delta^{\Df}_{T,R}}}
		\def\deott{{\Delta^{\bZ^{-}}_{T,R}}}
		Let $\naotr, \dzotr, \dmotr$ be the formal adjoints of $\nabla^\cE, d^\bZ, d^\M$ with respect to $h^{\cE}_{T,R}$. Set
		\begin{equation}
        \label{def-DZTR-nablaEuTR}
		    \Dzotr\index{DZTR@$\Dzotr$} = d^{\bZ} + \dzotr, \qquad \nabla_{T,R}^{\cE, u} = \frac{1}{2}(\nabla^\cE + \naotr).
		\end{equation}
		Let $N^\bZ$ be the number operator of $\cE$, i.e., it acts by multiplication by $k$ on the space $\Gamma(S, \Lambda^k(T^* \bZ) \otimes \cF)$. For $t > 0$, set\index{DtTR@$\cD_{t,T,R}$}\index{CtTR@$\cC_{t,T,R}$}
		\begin{align}\begin{split}\label{defdt}
				\cC_t^{\prime} = t^{N^\bZ / 2} d^\M t^{-N^\bZ / 2}, \quad \mathcal{C}_{t,T,R}^{\prime \prime} = t^{-N^\bZ / 2} d^{\M,*}_{{T,R}}t^{N^\bZ / 2}, \\
				\Cotr = \frac{1}{2}(\cC_t^{\prime} + \mathcal{C}_{t,T,R}^{\prime \prime}), \quad \Dotr  = \frac{1}{2}(\mathcal{C}_{t,T,R}^{\prime \prime} - \cC_t^{\prime}),
			\end{split}
		\end{align}
		then $\mathcal{C}_{t,T,R}^{\prime \prime}$ is the adjoint of $\Cc_t^{\prime}$ with respect to $h^{\cE}_{T,R}$. $\Cotr$ is a superconnection and $\Dotr$ is an odd element of $\Omega^\bullet(S, \End(\cE))$, and
		$$
		(\Cotr)^2 = -(\Dotr)^2.
		$$

		For $X \in T \bZ$, let $X^\flat \in T^* \bZ$ correspond to $X$ by the metric $g^{T \bZ}$. Set $c(X) = X^\flat \wedge - i_X$.
		
		Let $\cT$ be the tensor described as in \cref{defnbl}, then
		$$
		\Cotr = \frac{\sqrt{t}}{2} \Dzotr + \nabla^{\cE, u}_{T,R} - \frac{1}{2 \sqrt{t}} c(\cT).
		$$

		For $t>0$, put
		$$
		h^{\wedge}\left(\cC_t^{\prime}, h^{\cE}_{T,R}\right)\index{hwedgeCtHETR@$h^{\wedge}\left(\cC_t^{\prime}, h^{\cE}_{T,R}\right)$}:=\psi \operatorname{Tr}_s\left(\frac{N^\bZ}{2} h^{\prime}\left(\Dotr\right)\right).
		$$

		\def\cB{{\mathcal{B}}}

		Recall that $h$ is the holomorphic function given by $h(a)=a\exp(a^2),a\in\C.$
		Let 
        \begin{equation}
        \label{def-htilde}
            \tilde{h}(a)=(1+2a)h(a), \quad \text{so that for }a\in\C,\:\tilde{h}(a^2)=h'(a).
        \end{equation}	
		Let ${\cB}_{t,T,R}:=e^{-\bbf_{T,R}}\Dotr e^{\bbf_{T,R}}$\index{BtTR@${\cB}_{t,T,R}$}, then we can see that
		\[	h^{\wedge}\left(\cC_t^{\prime}, h^{\cE}_{T,R}\right):=\psi \operatorname{Tr}_s\left(\frac{N^\bZ}{2} \tilde{h}\left({\cB}^2_{t,T,R}\right)\right).\]
		
		Let $\tilde\bZ_1$ and $\tilde{\bZ}_2$ be defined in the present setting as described in \eqref{eq115}. Then we can define, as in \eqref{eq122}, ${\cB}^{\bd,\prime}_{t,T,R}={\cB}^{\bd,\prime}_{t,T,R,1}\oplus {\cB}^{\bd,\prime}_{t,T,R,2}$\index{BtTRbd@${\cB}^{\bd,\prime}_{t,T,R}$}, where ${\cB}^{\bd,\prime}_{t,T,R,i}$ is the restriction of ${\cB}_{t,T,R,i}$ on $\tilde\bZ_i$ with relative/absolute boundary conditions, $i=1,2$.
		
				Let us emphasize that in this section, `` $ \overline{\ \cdot\ }$ " refers to the bar convention  described in \cref{sec21}, not to the complex conjugation.

		\begin{lem}\label{lem101}{Fix $\tau \in (0,1)$.} For $R\in[0,R_0]$
			\[\overline{\int_\tau^\infty\psi \operatorname{Tr}_s\left(\frac{N^\bZ}{2} \tilde{h}\left({\cB}^2_{t,T,R}\right)\right)-\psi \operatorname{Tr}_s\left(\frac{N^\bZ}{2} \tilde{h}\left(({\cB}^{\bd,\prime}_{t,T,R})^2\right)\right)\frac{dt}{t}}=O({T^{-2/3}}).\]
			 Here the bar  convention is used.
		\end{lem}
		\begin{proof}

			Note that the degree 0 component is irrelevant in the bar convention. By Lemma \ref{lem914}, proceeding as in \cite[Theorems 2.13 {and 3.21}]{bismut1995flat}, we have 
			\be\label{lem101eq1}\overline{\psi \operatorname{Tr}_s\left(\frac{N^\bZ}{2} \tilde{h}\left({\cB}^2_{t,T,R}\right)\right)-\psi \operatorname{Tr}_s\left(\frac{N^\bZ}{2} \tilde{h}\left(({\cB}^{\bd,\prime}_{t,T,R})^2\right)\right)}=O(e^{T\Lambda_0/2}t^{-\half}).\ee

			So
			\be\label{lem101eq2}
			\overline{\int_{e^{1.1\Lambda_0T}}^\infty\psi \operatorname{Tr}_s\left(\frac{N^\bZ}{2} \tilde{h}\left({\cB}^2_{t,T,R}\right)\right)-\psi \operatorname{Tr}_s\left(\frac{N^\bZ}{2} \tilde{h}\left(({\cB}^{\bd,\prime}_{t,T,R})^2\right)\right)\frac{dt}{t}}=O(e^{-0.05\Lambda_0T}).
			\ee
			
			Let $\gamma$ be the contour given by $\{\l\in\C:|\Re(\sqrt\l)|=1\}$ (which is also given by the boundary of $\sqrt{U}$ in \cite[(6.32)]{puchol2020adiabatic}). For any $\bfn\in\mathbb{N}$, let $H_{\bfn}$ be the holomorphic function, s.t. $ H_{\bfn}^{(\bfn)}=\tilde{h}.$ Moreover, $H_{\bfn}(0)=H_{\bfn}'(0)=\cdots H_{\bfn}^{(n-1)}(0)=0.$
			
			Then 
			\be\label{lem101eq3}
			\tilde{h}\left({\cB}^2_{t,T,R}\right)=\int_\gamma H_{\bfn}(\l) (\l-{\cB}^2_{t,T,R})^{-\bfn-1}d\l.
			\ee
			
			It follows from Observation \ref{obs2}, Lemmas \ref{triviallem},	\ref{lem9221} and \ref{lem922}, and Remark \ref{rem92}, that
			\be\label{lem101eq4}
			\overline{\int^{e^{1.1\Lambda_0T}}_\tau\psi \operatorname{Tr}_s\left(\frac{N^\bZ}{2} \tilde{h}\left({\cB}^2_{t,T,R}\right)\right)-\psi \operatorname{Tr}_s\left(\frac{N^\bZ}{2} \tilde{h}\left(({\cB}^{\bd,\prime}_{t,T,R})^2\right)\right)\frac{dt}{t}}\leq\int_\tau^{\infty}\frac{C(R_0)}{T^{2/3}t^{5/4}}{dt}\leq \frac{C(R_0)}{T^{2/3}}.
			\ee
			
			The lemma then follows from \eqref{lem101eq2} and \eqref{lem101eq4}.
		\end{proof}

		\begin{proof}[Proof of Theorem \ref{int1}.]
		Note that by the definition of $\tilde{\bZ}_2$, ${\cB}^{\bd,\prime}_{t,T,R,2}$ is independent of $R$, so
		\begin{align}\begin{split}\label{eq139}
				&\ \ \ \ \overline{\psi \operatorname{Tr}_s\left(\frac{N^\bZ}{2} \tilde{h}\left(({\cB}^{\bd,\prime}_{t,T,R})^2\right)\right)-\psi \operatorname{Tr}_s\left(\frac{N^\bZ}{2} \tilde{h}\left(({\cB}^{\bd,\prime}_{t,T,0})^2\right)\right)}
				\\&=\overline{\psi \operatorname{Tr}_s\left(\frac{N^\bZ}{2} \tilde{h}\left(({\cB}^{\bd,\prime}_{t,T,R,1})^2\right)\right)-\psi \operatorname{Tr}_s\left(\frac{N^\bZ}{2} \tilde{h}\left(({\cB}^{\bd,\prime}_{t,T,0,1})^2\right)\right)}
		\end{split}\end{align}
		
		By Observation \ref{obs4}, the same arguments as in the proof of Theorem \ref{diffsm1} show that
		\be\label{eq140}
		\overline{\psi \operatorname{Tr}_s\left(\frac{N^\bZ}{2} \tilde{h}\left(({\cB}^{\bd,\prime}_{t,T,R',1})^2\right)\right)}=0
		\ee
		for all $R'\in[0,R_0].$
		
		Theorem \ref{int1} then follows from Lemma \ref{lem101}, \eqref{eq139} and \eqref{eq140}.
        \end{proof}
		
		\begin{proof}[Proof of Theorem \ref{diffsm} on $\Omega_1-\Omega_1''$.]
			Fix $\varsigma>0$ small enough. 	Let $\eta: \mathbb{R} \rightarrow[0,1]$ be a smooth even function such that
			$$\quad \eta(x)=1  \mbox{ for } |x| \leq \varsigma / 2,  \eta(x)=0  \mbox{ for } |x| \geq 1.$$

			For $z \in \mathbb{C}$, as in \cite[(6.26)]{puchol2020adiabatic},
			$$
			\begin{aligned}
				& F_t(z)=\left(1+2 tz^2\right) \int_{-\infty}^{+\infty} \exp (\sqrt{2} x z) \exp \left(-\frac{x^2}{2t}\right) \eta(x) \frac{d x}{\sqrt{2 \pi t}}, \\
				& G_t(z)=\left(1+2 tz^2\right) \int_{-\infty}^{+\infty} \exp (\sqrt{2} x z) \exp \left(-\frac{x^2}{2t}\right)(1-\eta(x)) \frac{d x}{\sqrt{2 \pi t}}.
			\end{aligned}
			$$
			Then we have
            \begin{equation}
            \label{relation-h-Ft-Gt}
            h'(\sqrt{t}z)=F_t(z)+G_t(z).    
            \end{equation}
            Moreover, since $G_t$ is even, there exists holomorphic function $\tilde{G}_t$, s.t. $\tilde{G}_t(z^2)=G_t(z).$
			
			Proceeding as in the proof of \cite[Lemma 6.3]{Yantorsions}, for any $R'\in[0,R_0]$ (if $\varsigma>0$ is small enough),
			\be\label{eq154}
			\overline{\psi \operatorname{Tr}_s\left(\frac{N^\bZ}{2} F_t\left(t^{-\half}{\cB}_{t,T,R'}\right)\right)-\psi \operatorname{Tr}_s\left(\frac{N^\bZ}{2} F_t\left(t^{-\half}{\cB}^{\bd,\prime}_{t,T,R'}\right)\right)}=0.
			\ee
			
			While by \eqref{eq139} and \eqref{eq140}, 	\be\label{eq155}
			\overline{\psi \operatorname{Tr}_s\left(\frac{N^\bZ}{2} \tilde{h}\left(({\cB}^{\bd,\prime}_{t,T,R})^2\right)\right)-\psi \operatorname{Tr}_s\left(\frac{N^\bZ}{2} \tilde{h}\left(({\cB}^{\bd,\prime}_{t,T,0})^2\right)\right)}=0.
			\ee
			
			Let $\gamma$ be the contour given by $\{\l\in \C:|\Re(\sqrt{\l})|=1\}$, 
			
			By the definition of $G_t$, it is straightforward to see that for any $\bfn\geq1,$ $z\in\gamma$, $t\in(0,1)$ \be\label{eq1571}\left|G_t(z)\right| \leq \frac{C_\bfn e^{- \frac{\varsigma^2}{32t}}}{\left(|z|^2+1\right)^\bfn}.\ee
			Note that
			\be\label{eq1581}
			G_t(\sqrt{t}^{-1}\cB_{t,T,R})=\int_\gamma \tilde{G}_t(\l)(\l-t^{-1}\cB_{t,T,R})^{-1}d\l.
			\ee
			By Lemma \ref{triviallem}, Lemma \ref{lem9222}, Lemma \ref{lem9220}, \eqref{eq1571} and \eqref{eq1581}, 
			\be\label{eq1582}
			\overline{\psi \operatorname{Tr}_s\left(\frac{N^\bZ}{2} G_t\left(t^{-\half}{\cB}_{t,T,R}\right)\right)-\psi \operatorname{Tr}_s\left(\frac{N^\bZ}{2} G_t\left(t^{-\half}{\cB}^{\bd,\prime}_{t,T,0}\right)\right)}=O(e^{-\frac{c}{t}}).
			\ee
			
			It follows from \eqref{def-htilde}, \eqref{relation-h-Ft-Gt}, \eqref{eq154}, \eqref{eq155} and \eqref{eq1582} that
			\be\label{eq159}
			\overline{\psi \operatorname{Tr}_s\left(\frac{N^\bZ}{2} \tilde{h}\left({\cB}^2_{t,T,R}\right)\right)-\psi \operatorname{Tr}_s\left(\frac{N^\bZ}{2} \tilde{h}\left({\cB}^2_{t,T,0}\right)\right)}=O(e^{-\frac{c}{t}}).
			\ee
			
			{As explained in the proof of Theorem \ref{diffsm} on $S-\Omega_1$ given in \cref{sec7}, noticing \eqref{eq39} we can proceed} as in the proof of \cite[Theorem 9.10 and Theorem 9.12]{bismut2001families} to get
			\be\label{eq160}
			\overline{\psi \operatorname{Tr}_s\left(\frac{N^\bZ}{2} \tilde{h}\left({\cB}^2_{t,T,0}\right)\right)}=O(t).
			\ee
			for $s\in\Omega_1-\Omega_1''$. By \eqref{eq159} and \eqref{eq160}, Theorem \ref{diffsm} holds on $\Omega_1-\Omega_1''$.
		\end{proof}

		\section{Proof of Theorem \ref{int20}}\label{proofint20}
		
		\textbf{In this section, we assume that $\s \in \Omega_1''$.}

		\begin{defn}\label{fixr1}
        We fix $(\delta, r)$, as well as the deformation of the metric $ g^{T\bZ} $ in a small neighborhood of $ \overline{\pi^{-1}(\Omega_1)} $ and $\Lambda_0$, as in \cref{sec10}. Recall also that in \cref{sec10}, $\Omega_1$ is chosen to be small enough, such that the function $t_1$ in Lemma \ref{lem441} satisfies $|t_1|\leq \delta^2$. 
        
        We fix $ R_1$ in such that for any fiber $R_1\geq A_6 $, where $A_6$ is defined in each fiber in \eqref{A6}. Also, by Observation \ref{obs3}, if we perform the deformation of $ g^{T\bZ} $  as described in \cref{sec921}, we can keep the same $R_1$. 
        
        The constants $ r_1 $ and $ r_2 $ in the definition of $ q_A $ in \cref{assotwo} will be fixed as $ 7r $ and $ 9r $ respectively.\end{defn}

        \textbf{From now on, in this section, we assume that $R\geq R_1$  and that $T\gg1$.}
         
		It is clear that Assumption \ref{ass91} holds fiberwisely for double suspension $\M\to S$ and that Assumption \ref{9c} holds on the total space, so we can use all estimates in \cref{sec9} and \cref{conag} in this section (see also Remark \ref{Important rmk}).

		For $\s \in \Omega_1''$, let $\Df_\s$ be the ball in each fiber of radius $\frac{r_1 + r_2}{2}=8r$ and centered at the point $p \in \bZ_\s$ such that $u_0(p) = \cdots = u_{n+2N}(p) = 0$ for the coordinates described in \cref{model}. Let $\bZ^-_{\s} = \bZ_{\s} - \Df_{\s}$. We will abbreviate $\bZ_\s$, $\bZ^-_\s$ and $\Df_{\s}$ as $\bZ$, $\bZ^-$\index{Zminus@$\bZ^-$} and $\Df$\index{D@$\Df$}  respectively if it will not cause any confusion.

        		Recall that $h^{\cE}_{T,R}$ is the metric on $\cE$ induced by $e^{-\bbf_{T,R}} h^{\cF}$ and $g^{T\bZ}$. Sometimes we will also denote $h^{\cE}_{T,R}$ by $(\cdot, \cdot)_{L^2_{T,R}}$, and denote $\sqrt{h^{\cE}_{T,R}(u, u)}$ by $|u|_{L^2_{T,R}}$. The function $\bbf_{T,R}=\bbf_{1,T,R}$ is defined in \cref{ineresult} from $\bbf$, following Definition \ref{assononmor}. In particular, in this section, as $\s\in\Omega''_1$, we have
                \begin{equation}
                    \label{bbfTR-on-Omega1''}
                    \bbf_{T,R}|_{\bZ_\s} = T(\bbf + q_R)|_{\bZ_\s}.
                \end{equation}
        
		Let $\deotr\index{DeltaZTR@$\deotr$} = (\Dzotr)^2$. 
		

		Fiberwisely, let $\deoto$\index{DeltaDTR@$\deoto$} be the
		restriction of $\deotr$ on $\Df$ with absolute boundary conditions and $\deott$\index{DeltaZminusTR@$\deott$} be the restriction of $\deotr$ on $\bZ^{-}$ with relative boundary conditions. Let $\l_k( T, R)$ (resp. $\lambda_k(T,R,0)$, $\lambda_k(T,R,1)$) be the $k$-th eigenvalue of $\deotr$ (resp. $\deoto$, $\deott$)\index{lambdakTR@$\l_k( T, R)$, $\l_k( T, R,0)$, $\l_k( T, R,1)$}. Let $\tilde{\l}_k(T,R)$\index{lambdakTRtilde@$\tilde{\l}_k( T, R)$} be the $k$-th eigenvalue of $\deoto\oplus\deott.$ 
		
		\def\Hotr{{\mathcal{H}_{T,R}}}
		\def\Hotro{{\mathcal{H}^{\Df}_{T,R}}}
		\def\Hotrop{{\mathcal{H}^{\Df'}_{T}}}
		\def\Hotrt{{\mathcal{H}^{\bZ^-}_{T,R}}}
		\def\Hotrtp{{\mathcal{H}^{\bZ^{-,\prime}}_{T}}}
		\def\Hotrsm{{\mathcal{H}^{\mathrm{sm}}_{T,R}}}
		
		Let 
        \begin{equation}
        \label{def-Hotr-Hotro-Hotrt}
            \Hotr:=\ker(\deotr), \quad \Hotro:=\ker(\deoto), \quad\text{and}\quad \Hotrt:=\ker(\deott)
        \end{equation}\index{HTR@$\Hotr$, $\Hotro$ and $\Hotrt$} which form vector bundles on $\Omega_1''$. 
		Let $(\Hotr)^l$, $(\Hotro)^l$ and $(\Hotrt)^l$ be the degree $l$ component of these bundles.
		
		
		By Hodge theory, there exists $k_0\in\Z$, such that for $k\leq k_0$, $\l_k(T,R)\equiv0$, and for $k>k_0$, $\l_{k}(T,R)\neq0$. 
         By Hodge theory again, for $a=1,2$, there also exist $k_a\in\Z$, such that for $k \leq k_a$,  $\lambda_k(T,A,a) \equiv 0$, and for $k > k_a$, $\lambda_k(T,A,a) \neq 0$.

		By Lemma \ref{lem813}, there exist $(T,R)$-independent constants $c,c'$ and $C$ such that the space 
        \begin{equation}
            \label{def-Hotrsm}
            \Hotrsm\index{HTRsm@$\Hotrsm$} := \mathrm{Span}\{u\::\: \deotr u=\lambda u, \: \lambda\in [0,Ce^{-cT}e^{-c'TR}]\}
        \end{equation}
        form a vector bundle on $\Omega_1''$. Moreover, when $R$ is large enough, by Lemma \ref{lem813},
		\be\label{dimeq1}
		\dim\left((\Hotrsm)^l\right)=\dim\left((\Hotro)^l\right) +\dim\left((\Hotrt)^l\right). 
		\ee 
		In fact, as vector bundles,
		\be\label{eq156}\Hotrsm\cong\Hotro\oplus\Hotrt.\ee
		By \cite[Theorem 1.3]{DY2020cohomology}, 
		\be\label{eq157}
        \begin{aligned}
        \Hotr|_{\s}&\cong H^\bullet(Z_\s\times D^N_1\times D^N_1,Z_\s\times \p D^N_1\times D^N_1;\cF|_{\bZ_\s})\cong H^\bullet(Z_\s,F|_{Z_\s})[-N], \\
        \Hotrt|_{\s}&\cong H^\bullet\left(Z_\s\times D^N_1\times D^N_1-\Df_\s,(Z_\s\times \p D_1\times D_1)\cup \p \Df_\s;\cF|_{\bZ^-_\s}\right), \\
        \Hotro|_{\s}&\cong H^\bullet(\Df_\s;\cF|_{\bfD_\s}),
         \end{aligned}
        \ee
		 where $D_1^N:=\{x\in\R^N:|x|\leq 1\}$. Hence, \be (\Hotro)^0|_{\s}\cong \C^m \text{ and } (\Hotro)^l|_{\s}=0 ,\forall l\geq 1.\ee
		
		It follows from a standard Mayer-Vietoris sequence argument that
		\be\label{eq158}
		(\Hotrt)^l|_{\s}\cong\begin{cases}
			(\Hotr)^l|_{\s},\mbox{ if $l\geq N$};\\
			0, \mbox{if $1<l<N$ or $l=0$};\\
			\C^m,\mbox{ if $l=1$.}
		\end{cases}
		\ee
		
		So by \eqref{eq156}-\eqref{eq158}, when $R$ is large,
		\begin{equation}
		    \label{Decompo-Hotrsm}
            \Hotrsm|_{\s}\cong\C^m\oplus\C^m \oplus(\Hotrt)^{l\geq N}|_{\s},
		\end{equation}
		where one copy of $\C^m$ is corresponding to $(\Hotro)^0|_{\s}$, the other is corresponding to $(\Hotrt)^1|_{\s}.$ Moreover, the $(\Hotrt)^{l\geq N}|_{\bZ^-_\s}$ part is corresponding to harmonic forms inside $\Hotrsm$. We can make this isomorphism more specific as follows. For a $n$-dimensional submanifold (with boundary) $Z'$ of a $n$-dimensional Riemannian manifold $(Z,g^{TZ})$, we denote by $\I^Z_{Z'}$ the map such that for any $L^2$-form $u$ on $Z'$, $\I^Z_{Z'}(u)$ is the $L^2$-form on $Z$ given by
		\begin{equation}
		    \label{eq-def-Ical}
            \I^Z_{Z'}(u)(x)=\begin{cases}
			u(x), \mbox{ if $x\in Z'$},\\
			0, \mbox{ if $x\notin Z'$.}
		\end{cases}
		\end{equation}
        When $Z'$ and $Z$ are clear from the context, we will just denote $\I^Z_{Z'}$ by $\I$\index{I@$\I$}.        We also denote by $\P$\index{P@$\P$} the orthogonal projection onto $\Hotrsm$, with respect to $h^{\cE}_{T,R}$. Then        
        \begin{equation}
            \label{degree01sm}
            \begin{aligned}
             &{\H_{T,R}^{0}|_{\s}=0 \qquad \text{and}\qquad}(\Hotrsm)^0|_{\s}=\P\I(\Hotro)^0|_{\s},\\ 
            &{\H_{T,R}^{1}|_{\s}=0 \qquad \text{and}\qquad} (\Hotrsm)^1|_{\s}=\P\I(\Hotrt)^1|_{\s},\\
			&	(\Hotrsm)^{l\geq N}|_{\s}=\H_{T,R}^{l\geq N}|_{\s}=\P\I(\Hotrt)^{l\geq N}|_{\s}.   
            \end{aligned}
        \end{equation}
        
		Let 
        \begin{equation}
             \label{def-eb-pb}
             \eb\index{e@$\eb$}:=\P\I:(\Hotro)^0\to (\Hotrsm)^0 \qquad\text{and} \qquad \pb:=\P\I:(\Hotrt)^1\to (\Hotrsm)^1.
        \end{equation}
         Then by looking at  the Mayer-Vietoris sequence, we can see that 
         \be\label{dotdt}d^{\bZ}\eb(\Hotro)^0=(\Hotrsm)^1.\ee
		{From \eqref{degree01sm} and \eqref{dotdt}, we see that \eqref{Decompo-Hotrsm} becomes an identification of complexes:
        \begin{multline}
		    \label{Decompo-Hotrsm-complex}
            (\Hotrsm,d^\bZ)|_{\s}\cong\\
            \C^m\overset{\sim}{\to} \C^m \to 0\cdots\to 0\to (\Hotrt)^{ N}|_{\s}\overset{0}{\to}(\Hotrt)^{N+1}|_{\s}\overset{0}{\to}\cdots \overset{0}{\to}(\Hotrt)^{N+n}|_{\s}.
		\end{multline}}

		\def\ded{{\nabla^\sm}}
		\def\deds{{\nabla^{\sm,*}}}
		\def\ddtr{{\mathcal{D}^\sm_{t,T,R}}}

		On $\Hotrsm$, define 
        \begin{equation}
            \begin{aligned}
                &\nabla^{\sm}:=\P\nabla^{\cE}\P,\\
                &\ddtr:=\deds-\ded+\half\sqrt{t}\P(\dzotr-d^{\bZ})\P.
            \end{aligned}
        \end{equation}
       Here $\deds$ is the adjoint connection of $\ded$ with respect to $h^{\cE}_{T,R}.$ 
		For $t>0$, put
		\begin{equation}
        \begin{aligned}
        &h_{\sm}^{\wedge}\left(\cC_t^{\prime}, h^{\cE}_{T,R}\right):=\psi \operatorname{Tr}_s\left(\frac{N^\bZ}{2} \P h^{\prime}\left(\ddtr\right)\P\right),\\
		&h_{\la}^{\wedge}\left(\cC_t^{\prime}, h^{\cE}_{T,R}\right):=\psi \operatorname{Tr}_s\left(\frac{N^\bZ}{2} h^{\prime}\left(\Dotr\right)\right)-\psi \operatorname{Tr}_s\left(\frac{N^\bZ}{2} \P h^{\prime}\left(\ddtr\right)\P\right).
        \end{aligned}
		\end{equation}
		
		For any $\tau>0$, set
        \begin{equation}\index{TTHMgTZg=hFTRLlatau@$\mathcal{T}^{\mL}_{\la,\tau}\left(T^H \M, g^{T \bZ}, h^{\cF}_{T,R}\right)$, $\mathcal{T}^{\mL}_{\sm,\tau}\left(T^H \M, g^{T \bZ}, h^{\cF}_{T,R}\right)$}
            \begin{aligned}
                \mathcal{T}^{\mL}_{\la,\tau}&\left(T^H \M, g^{T \bZ}, h^{\cF}_{T,R}\right)\\
			&:=-\int_\tau^{\infty}\left(h_{\la}^{\wedge}\left(\cC_t^{\prime}, h^{\cE}_{T,R}\right)- \frac{  \chi(\bZ,\cF)\dim(\bZ)-2\chi'(\Hotrsm)}{4} h^{\prime}\left(\frac{\sqrt{-1} \sqrt{t}}{2}\right)\right)\frac{d t}{t},\\
            \T_{\sm,\tau}^{\mL}&(T^H\M,g^{T\bZ}, h^{\cF}_{T,R})\\
			&:=-\int_\tau^\infty \left(h_{\sm}^{\wedge}\left(\cC_t^{\prime}, h^{\cE}_{T,R}\right)-\frac{\chi'(\bZ,\cF)}{2}\right)+\frac{\chi'(\bZ,\cF)-\chi'(\Hotrsm)}{2}h'(\frac{\sqrt{-1}\sqrt{t}}{2})\frac{dt}{t}.
            \end{aligned}
        \end{equation}

		Recall that we have fixed $b=0.8$, $b_1=0.95$ and $b'=0.98$ in Definition \ref{deffix}.
		Let $\Df_\s''\subset\Df_\s'''$ be the ball of radius $7r+2(1-\sqrt{1-b_1})r$ and $7r+2(1-\sqrt{1-b'})r$ respectively, centered at the same point as $\Df_\s$.  
		
		Each $u\in (\Hotro)^0|_{\bfD_\s}$ is corresponding to a parallel section on $\cF|_{\Df_\s}\to \Df_\s$, so it can be extended to a parallel section on  $\cF|_{\Df'''_\s}\to \Df'''_\s$, which is denoted by $ \underline{u}.$ 
		
		Let $\rho\in C_c^\infty(\Df'''_\s)$, s.t. $\rho|_{\Df_\s''}\equiv1$. Then we can view $\rho \underline{u}$ as a smooth section on $Z_\s.$

		\begin{lem}\label{urhou}
			Assume $T\gg1$. There exists $R_2=R_2(r)$, s.t. if $R>R_2$, for any $u\in (\Hotro)^0|_{\s}$,
			\be\label{urhoua}|\eb(u)-\rho \underline{u}|_{L^2_{T,R}}\leq Ce^{-\frac{bTRr^2}{2}}|\rho \underline{u}|_{L^2_{T,R}}\ee
			and
			\be\label{urhoua1}|d^{\bZ}\eb(u)-d^{\bZ}\rho \underline{u}|_{L^2_{T,R}}\leq  Ce^{-\frac{bTRr^2}{2}}|d^\bZ \rho \underline{u}|_{L^2_{T,R}}\ee
			for some $(T,R)$-independent constant $C$. Although the constant $r$ is already fixed, we still want to emphasize the dependence of $ R_2 $ on {this variable}. 
            
            As a consequence,
            \be\label{urhoua2}|\eb(u)-\rho \underline{u}|_{L^2_{T,R}}\leq Ce^{-\frac{bTRr^2}{2}}|\eb(u)|_{L^2_{T,R}}\ee
			and
			\be\label{urhoua12}|d^{\bZ}\eb(u)-d^{\bZ}\rho \underline{u}|_{L^2_{T,R}}\leq  Ce^{-\frac{bTRr^2}{2}}|d^\bZ \eb(u)|_{L^2_{T,R}}\ee
            
		\end{lem}
		\begin{proof}
            Note that $\underline{u}$ is a parallel section, so by \eqref{def-DZTR-nablaEuTR}, $D^{\bZ}_{T,R} \underline{u}=0$, and moreover as $(\cF,h^\cF)$ is unitarily flat, we get that $|\underline{u}|_{h^\cF}$ is constant on $\Df'''_\s$. Hence, if $R$ is large enough, using \eqref{bbfTR-on-Omega1''}, we get
			\be\label{eq163}
			Ce^{-2b'TRr^2}\leq	\frac{|D^{\bZ}_{T,R}\rho \underline{u}|_{L^2_{T,R}}^2}{|\rho \underline{u}|_{L^2_{T,R}}^2}\leq Ce^{-(b+b_1)TRr^2}.
			\ee
            So by the Rayleigh quotient argument (restricted on $0$-form), one can see that nonzero eigenvalues of $\Delta^\bZ_{T,R}|_{(\H^{\sm}_{T,R})^0}$  have upper bound $Ce^{-(b+b_1)TRr^2}$. By \eqref{dotdt}, we have $(\H^{\sm}_{T,R})^1=d^{\bZ}(\H^{\sm}_{T,R})^0$, so nonzero eigenvalues of $\Delta^\bZ_{T,R}|_{(\H^{\sm}_{T,R})^1}$ still have upper bound $Ce^{-(b+b_1)TRr^2}$. By \eqref{eq163} and explanations below \eqref{Decompo-Hotrsm}, we deduce that \be\label{eq164}\|D^{\bZ}_{T,R}\P\|\leq Ce^{-\frac{(b+b_1)TRr^2}{2}}\mbox{ and } \|d^{\bZ}\P\|\leq Ce^{-\frac{(b+b_1)TRr^2}{2}}.\ee
			Note that these upper bounds refine upper bounds in Lemma \ref{lem813}.
			
			Again using \eqref{bbfTR-on-Omega1''}, if $R$ is large, we find			\be\label{urhou0}|\rho \underline{u}-\I u|_{L^2_{T,R}}\leq  Ce^{\frac{-b_1TRr^2}{2}}|\rho \underline{u}|_{L^2_{T,R}}.\ee
			Since $\|\P\|\leq 1$, by \eqref{urhou0}
			\be\label{urhou1}|\P \rho \underline{u}-\P \I u|_{L^2_{T,R}}=|\P \rho \underline{u}-\eb(u)|_{L^2_{T,R}}\leq Ce^{-\frac{b_1TRr}{2}}|\rho \underline{u}|_{L^2_{T,R}}.\ee
			By the left hand side of \eqref{eq163}
			\be\label{urhou4}
			|\Dzotr\P\rho \underline{u}|_{L^2_{T,R}}=|\P\Dzotr\rho \underline{u}|_{L^2_{T,R}}\leq |\Dzotr\rho \underline{u}|_{L^2_{T,R}}\leq C e^{-\frac{(b+b_1)TRr^2}{2}}|\rho \underline{u}|_{L^2_{T,R}}.
			\ee
			
			Next, by Lemma \ref{lem813}, \eqref{eq163} and \eqref{urhou4},  we have \be\label{urhou2}C_1e^{-cT}e^{-cTR^{\frac{1}{4}}}{\left|\left(\P\rho \underline{u}-\rho \underline{u}\right)\right|_{L^2_{T,R}}}\leq\left|\Dzotr\left(\P\rho \underline{u}-\rho \underline{u}\right)\right|_{L^2_{T,R}}\leq Ce^{-\frac{(b+b_1)TRr^2}{2}}|\rho \underline{u}|_{L^2_{T,R}}.\ee
			{From \eqref{urhou1} and \eqref{urhou2}, we find \eqref{urhoua}.} Also, note that we could keep the constant $ c $ unchanged under the deformation of $ g^{T\bZ} $ as described in \cref{sec921} and \cref{sec10}. Therefore, the choice of $ R_2 $ in this lemma remains unaffected by this deformation of $ g^{T\bZ} $.


			
			By \eqref{def-eb-pb}, left hand side of \eqref{eq163}, \eqref{eq164} and \eqref{urhou1}, and noticing that $d^{\bZ}|_{\Omega^0}=D_{T,R}^{\bZ}|_{\Omega^0}$, we also have
			\be\label{urhou6}
			|d^{\bZ}\P \rho \underline{u}-d^{\bZ}\eb(u)|_{L^2_{T,R}}\leq Ce^{-\frac{(2b_1+b)TRr^2}{2}}|\rho \underline{u}|_{L^2_{T,R}}\leq Ce^{-\frac{(2b_1+b-b')TRr^2}{2}}|d^{\bZ}\rho \underline{u}|_{L^2_{T,R}}.
			\ee
			
			Let $W$ be the linear space spanned by $(H^{\sm}_{T,R})^0$ and $\rho \underline{u}$ for $u\in (\H^{\bfD}_{T,R})^0$, then we claim that $\|d^{\bZ}|_{W}\|\leq Ce^{-(b+b_1)TRr^2/2}.$
			With this claim, \eqref{eq163} and  \eqref{urhou2}, and again noticing that $d^{\bZ}|_{\Omega^0}=D_{T,R}^{\bZ}|_{\Omega^0}$, if $R$ is large
			\be\label{urhou21}
			\begin{aligned}
				&\ \ \ \ |d^{\bZ}(\P\rho \underline{u}-\rho \underline{u})|_{L^2_{T,R}}\leq Ce^{-\frac{(b_1+b)TRr^2}{2}}|\P\rho \underline{u}-\rho \underline{u}|_{L^2_{T,R}}\\
				&\leq Ce^{-\frac{(b_1+3b)TRr^2}{2}}|\rho \underline{u}|_{L^2_{T,R}}\leq Ce^{-\frac{(b_1+3b-b')TRr^2}{2}}|D_{T,A}^{\bZ}\rho \underline{u}|_{L^2_{T,R}}.
			\end{aligned}
			\ee

			Then \eqref{urhoua1} follows from \eqref{urhou6} and \eqref{urhou21}.
			
			Now it suffices to prove the claim. Let $w=w_1+w_2$, s.t. $w_1\in\Im \P|_{(H^{\sm}_{T,R})^0}$, $w_2=\rho \underline{u}$ for some $u\in (\H^{\bfD}_{T,R})^0$. Let $w':=\frac{w}{|w|_{L^2_{T,R}}}, w_a':=\frac{w_a}{|w|_{L^2_{T,R}}},a=1,2$. It follows from \eqref{eq163} and \eqref{eq164} that
			\begin{align*}
				&\ \ \ \ |d^{\bZ}w'|^2_{L^2_{T,R}}\leq 2|d^{\bZ}w_1'|^2_{L^2_{T,R}}+2|d^{\bZ}w_2'|^2_{L^2_{T,R}}\leq Ce^{-(b+b_1)TRr^2}(|w_1'|^2_{L^2_{T,R}}+|w_2'|^2_{L^2_{T,R}})\\
				&\leq 2Ce^{-(b+b_1)TRr^2},
			\end{align*}
			which conclude the proof of the claim.

            Finally, \eqref{urhoua2} and \eqref{urhoua12} follow easily from \eqref{urhoua} and \eqref{urhoua1}.
		\end{proof}

        Let 
        \be\label{fixr3}
        R_3 = \max \{R_1,2R_2(r)\},
        \ee
        with $R_1$ in Definition \ref{fixr1} and $R_2(r)$ in Lemma \ref{urhou}.
        
		\textbf{From now on, we further assume that $R\geq R_3$.} 
		
		\def\tH{\tilde{\H}}
		\def\tnatr{{\tilde{\nabla}}}
		\def\tnatrs{{\tilde{\nabla}^*}}
		Let $\tH^0_{T,R}$ be the space generated by $\rho \underline{u}$ and $\tH^1_{T,R}$ be the space generated by $d^\bZ\rho \underline{u}$ for $u\in(\Hotro)^0$. Let 
        \begin{equation}
        \label{def-tHTR}
            {\tH}_{T,R}:=\tH^0_{T,R}\oplus\tH^1_{T,R}\oplus \Hotr,
        \end{equation} and $\tilde{\P}$ be the orthogonal projection onto $\tH_{T,R}$, with respect to $h^{\cE}_{T,R}$. Let $\tnatr:=\tilde{\P}\nabla^\cE\tilde{\P}$, and $\tnatrs$ be the adjoint connection with respect to $h^{\cE}_{T,R}$. 
		
		\def\ddtrzo{{\tilde{\mathcal{D}}_{t,T,R}}}
		\def\ddtrti{{\tilde{\cD}_{t,T,R}}}

		{By \eqref{degree01sm}}, when restricted to degree $>1$, $\tilde{\P}$ and $ \P$ agree, and  $(\H^{\sm}_{T,R})^l = \{0\}$ for $2\leq l \leq N - 1$, so  it follows that \be\label{isdiff}(\tilde{\P} d^{\bZ} \tilde{\P})^2 = 0.\ee  
		Let $\ddtrti:=\tnatrs-\tnatr+\frac{\sqrt{t}}{2}\tilde{\P}(d^{\bZ}-\dzotr)\tilde{\P}$,
		then we define
		$$
		{h}_{\sm}^{\wedge,\prime}\left(\cC_t^{\prime}, h^{\cE}_{T,R}\right):=\psi \operatorname{Tr}_s\left(\frac{N^\bZ}{2} \tilde{\P}h^{\prime}\left(\ddtrti\right)\tP\right).
		$$
		
		Let 
		\begin{align*}&\ \ \ \ \T_{\sm,\tau}^{\mL,\prime}(T^H\M,g^{T\bZ},h^{\cE}_{T,R})\\
			&:=-\int_\tau^\infty \left(h_{\sm}^{\wedge,\prime}\left(\cC_t^{\prime}, h^{\cE}_{T,R}\right)-\frac{\chi'(\bZ,\cF)}{2}\right)+\frac{\chi'(\bZ,\cF)-\chi'(\Hotrsm)}{2}h'(\frac{\sqrt{-1}\sqrt{t}}{2})\frac{dt}{t}.\end{align*}

		\begin{prop}\label{prop112}
        If $R\geq R_3$,
			\[\lim_{T\to\infty}\left(\T_{\sm,\tau}^{\mL,\prime}(T^H\M,g^{T\bZ},h^{\cE}_{T,R})-\T_{\sm,\tau}^{\mL}(T^H\M,g^{T\bZ},h^{\cE}_{T,R})\right)=0.\]
		\end{prop}
		\begin{proof}

			

			Note that	 
			\be\label{prop112eq1}
			(\l-\cD^{\sm}_{t,T,R})^{-1}-(\l-\tilde\cD_{t,T,R})^{-1}=	(\l-\cD^{\sm}_{t,T,R})^{-1}(\cD^{\sm}_{t,T,R}-\tilde\cD_{t,T,R})(\l-\tilde\cD_{t,T,R})^{-1}.
			\ee
			
			We will now prove that  
			\be\label{prop112eq2}
			\|\P-\tilde{\P}\|\leq Ce^{-0.4TRr^2}. 
			\ee
 
 If $ w \in (\H_{T,R}^{\sm})^0 $, then by \eqref{degree01sm} and \eqref{def-eb-pb}, there exists $ v \in (\H_{T,R}^\bfD)^0 $ such that $ w = \eb(v) $. By \eqref{urhoua2}, it follows that
\begin{equation}\label{prop112eq31}
|w - \tP(w)|_{L^2_{T,R}} \leq |w - \rho\underline{v}|_{L^2_{T,R}} \leq e^{-0.4TRr^2} |w|_{L^2_{T,R}},
\end{equation}
where the first inequality holds because $ \rho\underline{v} \in \Im(\tP) $.

Similarly to \eqref{eq163}, for any $ u \in (\H_{T,R}^\bfD)^0 $, we have
\begin{equation}\label{prop112eq41}
|\Delta_{T,R}^\bZ \rho\underline{u}|_{L^2_{T,R}} \leq C e^{-(b + b_1)TR^2/2} |\rho\underline{u}|_{L^2_{T,R}}.
\end{equation}

By \Cref{lem813}, the orthogonal complement $ (\Im \P)^\perp $ is spanned by eigenforms corresponding to eigenvalues $ \geq C e^{-cT - cTR^{1/4}} $. Thus, using \eqref{prop112eq41}, for any $ u \in (\H_{T,R}^\bfD)^0 $, we obtain
\begin{equation}\label{prop112eq51}
|(1 - \P)\rho\underline{u}|_{L^2_{T,R}} \leq C e^{-bTR^2} |\rho\underline{u}|_{L^2_{T,R}}.
\end{equation}

Let $ \{\rho\underline{u}_j\},u_j\in (\H_{T,R}^\bfD)^0 $, be an orthonormal basis of $ \tilde{\H}^0_{T,R} $. Then for any $ w \in (\Im \P)^\perp $ of degree $0$, by \eqref{prop112eq41} and \eqref{prop112eq51} (with $ b = 0.8 $), we have
\begin{equation}\label{prop112eq61}
\begin{aligned}
|\tP(w)|_{L^2_{T,R}} &= \left| \sum_j (\rho\underline{u}_j, w)_{L^2_{T,R}} u_j \right|_{L^2_{T,R}} \leq \sum_j \left| (\rho\underline{u}_j, w)_{L^2_{T,R}} \right| \\
&\leq |(1 - \P)\rho\underline{u}_j|_{L^2_{T,R}} \cdot |w|_{L^2_{T,R}} \leq C e^{-0.8TR^2} |w|_{L^2_{T,R}}.
\end{aligned}
\end{equation}

When restricted to 0-forms, \eqref{prop112eq2} directly follows  from \eqref{prop112eq31} and \eqref{prop112eq61}.

Proceeding similarly for 1-forms, and replacing  \eqref{degree01sm} and \eqref{def-eb-pb} by \eqref{dotdt} and \eqref{urhoua2} by \eqref{urhoua12}, we obtain the same estimate, hence \eqref{prop112eq2} also holds in this case. Finally, for forms of higher degree, $\tP$ and $\P$ coincide, so \eqref{prop112eq2} holds trivially.

           We claim that
			\be\label{prop112eq3}
			\|\tP(d^{\bZ}-\dzotr)\tilde{\P}\|\leq Ce^{-0.8TRr^2}.
			\ee
           {Indeed, This follows from \eqref{eq163} and \eqref{def-tHTR} in degree 0. The degree 1 case then follows from the fact that $\tP\dzotr\tilde{\P}$ is the adjoint of $\tP d^{\bZ}\tilde{\P}$, and in higher degree, the left hand side of \eqref{prop112eq3} vanishes by \eqref{def-tHTR}.}
			By \eqref{eq164},
			\be\label{prop112eq4}
			\|\P(d^{\bZ}-\dzotr){\P}\|\leq Ce^{-0.8TRr^2}.
			\ee

			By \eqref{prop112eq2}-\eqref{prop112eq4}, we obtain
			\be\label{prop112eq5}
			\|\cD^{\sm}_{t,T,R}-\tilde\cD_{t,T,R}\|\leq C(e^{-0.4TRr^2}+\sqrt{t}e^{-1.2TRr^2})
			\ee
            
			Let $\gamma$ be the contour given by $\{z\in\C:|\Re(z)|=1\}.$
			
			Note that
			\be \label{prop112eq6}h^{\prime}\left(\cD^{\sm}_{t,T,R}\right)-h^{\prime}\left(\ddtrti\right)=\int_\gamma  h'(\lambda)\left((\l-\cD^{\sm}_{t,T,R})^{-1}-(\l-\tilde\cD_{t,T,R})^{-1}\right)d\l.\ee
			By \eqref{prop112eq1}, \eqref{prop112eq5} and \eqref{prop112eq6},
			\be \label{prop112eq7}\left|\int_\tau^{e^{2.1TRr^2}} \Big(h_{\sm}^{\wedge}\left(\cC_t^{\prime}, h^{\cE}_{T,R}\right)-h_{\sm}^{\wedge,\prime}\left(\cC_t^{\prime}, h^{\cE}_{T,R}\right)\Big)\frac{dt}{t}\right|\leq C|\ln(\tau)|e^{-0.15TRr^2}.\ee
            By Lemma \ref{lem813}, the nonzero eigenvalues of $(\P D_{T,R}^{\bZ} \P)^2$ are bounded below by $e^{-2.08TRr^2}$. Moreover, by \eqref{def-tHTR} and \eqref{isdiff}, and because $\tP d\tP$ commutes with $(\tilde\P D_{T,R}^{\bZ}\tilde\P)^2$, the nonzero eigenvalues of $(\tilde\P D_{T,R}^{\bZ}\tilde\P)^2$ are the same on $\tH^0_{T,R}$ and $\tH^1_{T,R}$, and do not exist on $\Hotr$. Thus, by \eqref{eq163}, we deduces that the non-zero eigenvalues of this operator are bounded below by $e^{-1.96TRr^2}$. Proceeding as in \cite[Theorem 2.13]{bismut1995flat}, we find
			\be \label{prop112eq8}
			\left|h_{\sm}^{\wedge}\left(\cC_t^{\prime}, h^{\cE}_{T,R}\right)-h_{\sm}^{\wedge,\prime}\left(\cC_t^{\prime}, h^{\cE}_{T,R}\right)\right|\leq\frac{Ce^{1.04TRr^2}}{\sqrt{t}}.
			\ee
			So
			\be\label{prop112eq9}
			\left|\int_{e^{2.1TRr^2}}^\infty h_{\sm}^{\wedge}\left(\cC_t^{\prime}, h^{\cE}_{T,R}\right)-h_{\sm}^{\wedge,\prime}\left(\cC_t^{\prime}, h^{\cE}_{T,R}\right)\frac{dt}{t}\right|\leq Ce^{-0.01TRr^2}.
			\ee
			
			The Proposition \ref{prop112} then follows from \eqref{prop112eq7} and \eqref{prop112eq9}.
		\end{proof}
		
		Let $\tilde{\P}^1$ be the orthogonal projection with respect to $\tH^0_{T,R} \oplus \tH^1_{T,R}$, and let $\tilde{\P}^2 := \tilde{\P} - \tilde{\P}^1$. Due to degree considerations, the images of $\tilde{\P}^1$ and $\tilde{\P}^2$ are orthogonal.  By \cite[Proposition 2.13]{bismut1995flat}, since the image of $\tilde{\P}^2$ consists of harmonic forms, the form associated with the image of $\tilde{\P}_2$ contributes only to degree zero terms of ${\T}_{\sm, \tau}^{\mL, \prime}(T^H\M, g^{T\bZ}, h^{\cE}_{T,R})$. So only the form associated with the image of $\tilde{\P}^1$ contributes to $\overline{\T}_{\sm, \tau}^{\mL, \prime}(T^H\M, g^{T\bZ}, h^{\cE}_{T,R})$ by our bar convention. However, the image of $\tilde{\P}^1$ is supported inside $\bfD$, and following the same arguments as in the proof of Theorem \ref{diffsm1},
		\[ 
		\overline{\T}_{\sm, \tau}^{\mL, \prime}(T^H\M, g^{T\bZ}, h^{\cE}_{T,R}) = 0.
		\]
		
		By Proposition \ref{prop112} and the discussion above, we have
		\begin{cor}\label{cor111} There exists $R_4\geq R_3$ such that if $R\geq R_4$,
			\[\lim_{T\to\infty}\overline{\T}_{\sm,\tau}^{\mL}(T^H\M,g^{T\bZ},h^{\cE}_{T,R})=0.\]
			Here the bar convention is used.
		\end{cor}

        Proceeding as in the proof of Theorem \ref{int1}, using Lemma \ref{lem811}, Lemma \ref{lem813}, Lemma \ref{lem916}, and Lemma \ref{lem920} instead, we see that if we fix $R_5$ so that $R_5 \gg R_5^{\frac{1}{4}}$, then for $R\geq R_5$,
		\be\label{torsion-sm-tau-to0}
        \lim_{T\to\infty}\mathcal{T}^{\mL}_{\la,\tau}\left(T^H \M, g^{T \bZ}, \hotrf\right)-\mathcal{T}^{\mL}_{\tau}\left(T^H\M^-, g^{T \bZ}|_{\bZ^-}, h^{\cF}_{T,R}\right)=0.\ee

    \begin{defn}\label{def-of-R0}
        We fix $R_0$\index{R0@$R_0$} such that $R_0\geq \max\{R_4,R_5\}$, with $R_4$ in \Cref{cor111} and $R_5$ above \eqref{torsion-sm-tau-to0}.
    \end{defn}
		
        Theorem \ref{int20} follows from \Cref{cor111}, \eqref{torsion-sm-tau-to0} and the above definition.
		
		\begin{proof}[Proof of Theorem \ref{diffsm} on $\Omega_1''$.]
			We would like to emphasize that here $ \overline{\cdot}$ refers to the bar convention, not to the complex conjugation.
			Proceeding as in the proof of \eqref{eq159}, we have
			\be\label{eq1591}
			\overline{\psi \operatorname{Tr}_s\left(\frac{N^\bZ}{2} \tilde{h}\left({\cB}^2_{t,T,R}\right)\right)-\psi \operatorname{Tr}_s\left(\frac{N^\bZ}{2} \tilde{h}\left(({\cB}^{\bd}_{t,T,R})^{2}\right)\right)}=O(e^{-\frac{c}{t}}).
			\ee
			
			Let $\cB^{\bd}_{t,T,R,2}$ be the restriction of $\cB_{t,T,R}$ on $\bZ^-$ with relative boundary conditions, by Theorem \ref{diffsm1},
			\be
			\overline{\psi \operatorname{Tr}_s\left(\frac{N^\bZ}{2} \tilde{h}\left(({\cB}^{\bd}_{t,T,R})^{2}\right)\right)}=	\overline{\psi \operatorname{Tr}_s\left(\frac{N^\bZ}{2} \tilde{h}\left(({\cB}^{\bd}_{t,T,R,2})^{2}\right)\right)}.
			\ee
			
			Noticing \eqref{eq39} and proceeding as in the proof of \cite[Theorem 9.10 and Theorem 9.12]{bismut2001families},
			\be\label{eq1601}
			\overline{\psi \operatorname{Tr}_s\left(\frac{N^\bZ}{2} \tilde{h}\left(({\cB}^{\bd}_{t,T,R,2})^2\right)\right)}=O(t).
			\ee
			for $s\in\Omega_1''$. By \eqref{eq1591}-\eqref{eq1601}, Theorem \ref{diffsm} holds on $\Omega_1''$.
		\end{proof}


		\section{Proof of Theorem \ref{int2}}\label{sec11}

		\subsection{On $\Omega_0''$}\label{sec121}
		
		This subsection aims to prove the first estimate in Theorem \ref{int2}. 
		
		We use the same notation as in \cref{defct}, \cref{sec10} and \cref{proofint20}.
		
		{Recall that $S'=S-L$}. Because on {$S'$} there are no birth-death points, similarly to \cite[Definition 5.7]{bismut2001families}, we have the following definition.
		\begin{defn}\label{defn121}
            Let 
            \begin{multline}
                \Omega_{\AG}^\bullet(\M, \cF)|_{(\ppi)^{-1}(\Omega_0'')}:= \Big\{\alpha \in\Omega^\bullet(\M, \cF)|_{(\ppi)^{-1}(\Omega_0'')}\::\\ \a \text{ satisfies \eqref{eq39} {fiberwisely, uniformly on compact sets of }} \Omega_0''\Big\}.
            \end{multline}
             We define  a map $I\colon \Omega_{\AG}^\bullet(\M, \cF)|_{(\ppi)^{-1}(S')}\to \Omega^{\bullet}\left(S', V\right)$  by
			$$
			I( \alpha)=\sum_{p \in \Sigma(\bbf)} \int_{{\rW^\ru(p)}} \alpha. 
			$$
			Here over each unstable manifold $\rW^{\ru}(p)$, the bundle  $\cF$ is identified with the trivial bundle of fiber $\cF_p$ thanks to $\nabla^\cF$. Proceeding as in \cite[p. 659]{DY2020cohomology}, one can show that $I$ is well-defined.
		\end{defn}

		Proceeding as in \cite[Theorem 5.8]{bismut2001families}, by Remark \ref{rem56}, if both $\a$ and $d^{\M}\a$ satisfy estimate \eqref{eq39}, we have
		\be\label{dicommut}
		Id^{\M}\a=\A'I\a,
		\ee
       {where $\A'$ is defined in \eqref{def-sc-V}.} 
		

		Let $\Cozt:=\mathcal{C}_{t,T,0}$ and $\Dozt:=\mathcal{D}_{t,T,0}.$ 

		Although at some fibers near $\ppi(\Sigma(\bbf)) \subset S$, the pair $(\mathfrak{f},g^{T\bZ})$ might not be standard near some Morse points (see item (\ref{3a3}) in Lemma \ref{lem441}), the following theorem, similar to \cite[Theorems 10.2 and 10.9 and Proposition 10.24]{bismut2001families}, holds. The proof is essentially the same as in \cite[Theorems 10.2 and 10.9 and Proposition 10.24]{bismut2001families}.

		\begin{thm}
        \label{sp(CtT)-and-ItT}
        There exists $T_0>0$ large enough, such that on $\Omega_0''$, for all $T \geq T_0$
			$$
			\operatorname{Spec}\left((\Cozt)^2\right)=-\operatorname{Spec}\left((\Dozt)^2\right) \subset\left[0, \frac{t}{4}\right] \cup(4 t, \infty).
			$$
			
			Let $\gamma$ be the oriented contour given by $\{a\in \C:|a|=1\}.$
			
			Let \[
            \begin{aligned}
                &\P_{t,T}\index{PtT@$\P_{t,T}$}:=\frac{1}{2\pi\sqrt{-1}}\int_{t\gamma}(\l-(\Cozt)^2)^{-1}d\l,\\
                & W_{t,T}:=\Im(\P_{t,T}).
            \end{aligned}
            \]

			Then $W_{t,T}$ form a finite dimensional vector bundle $W_{t,T} \rightarrow \Omega_0''$. Moreover, by \cite[Theorem 1.1]{DY2020cohomology} and our discussions in \cref{proofthm54}, $W_{t,T}\subset \Omega_{\AG}^\bullet(\M, \cF)|_{(\ppi)^{-1}(\Omega_0'')}$, so
			$$
			I_{t,T}\index{ItT@$I_{t,T}$}:=I|_{W_{t,T}}: W_{t,T} \rightarrow \Omega^{\bullet}\left(\Omega_0'', V\right)
			$$
			is well-defined. Finally, $I_{t,T}$ is a linear isomorphism.
		\end{thm}

                
                In the sequel, we set $I_T=T_{t=1,T}$\index{IT@$I_T$}.

        	\def\homega{{\hat{\Omega}_0''}}		
		\begin{proof}[\textbf{Proof of Theorem \ref{thm54} in $\Omega_0''$}]
			Recall that the first part of Theorem \ref{thm54} have been proved in \cref{proofthm54}, so here we show that $ \J $ is an isomorphism of bundles in $ \Omega_0'' $. To achieve this, it suffices to prove that $ \cL_T $ is an isomorphism when $ T $ is large. When $ T $ is large, $ \cL_T $ is the morphism induced by the $ I_T $, which is an isomorphism by Theorem \ref{sp(CtT)-and-ItT}, thus making $ \cL_T $ an isomorphism.
		\end{proof}

        	Let  $\P_{t,T}^c:=1-\P_{t,T}$, then for the same reason that \cite[Theorem 10.5]{bismut2001families} holds, we have:
            
		\begin{thm}\label{thm103}
			For any $\tau>0$, there exists a constant $\epsilon_0 \in(0,1)$ and $T_0=T_0(\tau)$, such that for $T \geq T_0$, on $\Omega_0''$,
			$$
			\int_\tau^{\infty}\left(\operatorname{Tr}_s\left(N^\bZ h^{\prime}\left(\Dozt\right)\right)-\operatorname{Tr}_s\left(N^\bZ \P_{t,T}h^{\prime}\left({\P}_{t, T} \Dozt {\P}_{t, T}\right)\P_{t,T} \right)\right)\frac{d t}{2 t}=O\left(T^{-\epsilon_0}\right).
			$$
			That is,
			$$
			\int_\tau^{\infty}\operatorname{Tr}_s\left(N^\bZ \P^c_{t,T}h^{\prime}\left({\P}^c_{t, T} \Dozt {\P}^c_{t, T}\right)\P^c_{t,T} \right)\frac{d t}{2 t}=O\left(T^{-\epsilon_0}\right).
			$$
		\end{thm}

		Let $\hV\to {\homega} $ be the pullback of the bundle $V\to {\Omega_0''}$ with respect to canonical projection ${\homega}:=\Omega_0''\times(0,\infty) \to\Omega_0''$.	Let $\bh_T^V$ and $\bh_T^\hV$ be the generalized metrics on $V\to {\Omega_0''} $ and $\hV\to{\homega}$ given by (c.f. \cite[Definitions 10.20 and 10.30]{bismut2001families})
		\begin{equation}\label{def-bhTV}
        \begin{aligned}
            &\bh_T^V(\cdot,\cdot)\index{hVT@$\bh^V_T$ and $\bh_T^\hV$}:=h^{\cE}_{T}(I_T^{-1}\cdot,I_T^{-1}\cdot),\\
            &\bh_T^\hV(\cdot,\cdot)|_{S\times\{t\}}=h^{\cE}_{T}(I_{t,T}^{-1}\cdot,I_{t,T}^{-1}\cdot).
        \end{aligned}
		\end{equation}

		By \cite[Proposition 10.24]{bismut2001families}, 
		\be\label{eq167}(\bh_T^\hV)(\cdot,\cdot)|_{S\times\{t\}}=t^{-N^S-\frac{\dim(\bZ)}{2}} \bh_T^V\left(t^{\frac{N^S+N^V}{2}} \cdot, t^{\frac{N^S+N^V}{2}} \cdot\right),\ee
		where $N^V,N^S\in \Omega^{\bullet}(S,\End(V))$, s.t. for any section $s$ of $V$ of degree $k$ and for any differential form $w$ of degree $l$, $N^Sw\otimes s=lw\otimes s$ and $N^V w\otimes s=kw\otimes s.$ 
		
		It follows from \cite[Theorem 10.31]{bismut2001families} that $\bh_T^\hV$ satisfies Condition \ref{assum21} with respect to $\bh_T^V$.

		Let $$\T^{\mL,\P}_{T,\tau}:=-\int_\tau^{\infty}\operatorname{Tr}_s\left(N^\bZ \P_{t,T}h^{\prime}\left({\P}_{t, T} \Dozt {\P}_{t, T}\right)\P_{t,T} \right)^{>0}\frac{d t}{t}.$$

		Then by \eqref{dicommut} {and Theorem \ref{sp(CtT)-and-ItT}, we have} {on $\Omega_0''$:}
		\be\label{thm6411}
		\overline{\T}_\tau^{\mL}(\A',\bh_T^\hV)=\overline{\T}^{\mL,\P}_{T,\tau}.\ee
        {Using this equation and Theorem \ref{thm103}, we get {on $\Omega_0''$:}
        \begin{equation}
        \label{comp-torsion-an-torsion-bh_T}
            \lim_{T\to\infty} \left( \overline{\T}_\tau^{\mL}\big(T^H\M,g^{T\bZ},h_{T}^\cF\big)-\overline{\T}_\tau^{\mL}(\A',\bh_T^\hV)\right)=0
        \end{equation}
        }
		Recall that $\bbf^V\in\End(V)$ is given by $\bbf^V|_{o^{\ru,*}_p\otimes \cF_p}=\bbf(p). \mathrm{Id}_{ o^{\ru,*}_p\otimes \cF_p}$ for $p\in\Sigma(\bbf)$. Using the same reasoning as in \cite[(10.244),(10.246) and (10.249)]{bismut2001families} as well as Agmon estimates \eqref{eq39}, we find:
		\begin{lem}\label{1010}
			There exists a $h^V$-self-adjoint operator $R_T \in \Omega^{\geq 2}\left({\Omega_0''}  ; \End^{\mathrm{even}} V\right)$\index{RT@$R_T$} commuting with $\nabla^V, \bbf^V$ and $N^V$, which is a polynomial in $\frac{1}{T}$ without constant term, and another $h^V$-self-adjoint $R_T' \in  \Omega^{\bullet}({\Omega_0''}  ;\End (V))$ with \be\label{RT-norm}\left\|R_T^{\prime}\right\|+\left\|\nabla^VR_T^{\prime}\right\|=O\left(e^{-cT}\right)\ee uniformly {on $\Omega_0''$} for some $T$-independent  $c>0$, such that 
			\be\label{eq169}
			\bh_T^V=\left(\left(\mathrm{id}_V+R_T+R_T^{\prime}\right) e^{-T \bbf^V}\left(\frac{T}{\pi}\right)^{\frac{N^V}{2}-\frac{\dim(\bZ)}{4}}\right)^* h^V.
			\ee
			
			Recall that $V:=\oplus_{p\in\Sigma^{(0,k)}(\bbf)} o^{\ru,*}_p\otimes \cF_{p}$, we also have 
			\be 
            \label{RT-preserves-crit-point}
            R_T \mbox{ preserves } o^{\ru,*}_p\otimes \cF_{p}.\ee
			
		\end{lem}
		

		\begin{proof}[\textbf{Proof of Proposition \ref{prop63} on $\Omega_0''$}]
			\def\fR{\mathfrak{R}}            
			Recall that locally, {when crossing} $ L := \ppi(\Sigma^{(1)}(\bbf)) $, as in the proof of Theorem \ref{int3}, $ V $ has the decomposition $ V = \tilde{V} \oplus W ${, with $\tilde{V}$ in \eqref{def-Vtilde-on-Omega1}}.
            More precisely, let $\Omega\subset \Omega_1$ be a simply connected open set. Let $t_1$ be the function in Lemma \ref{lem441}, and let 
            \be\label{def-Omega^pm}
            \Omega^\pm:=\{\s\in\Omega:\pm t_1(\s)>0\}.
            \ee
            Then on $\Omega^+$, $ V = \tilde{V} \oplus W $ and on $\Omega^-$, $ V = \tilde{V}$. By \eqref{RT-preserves-crit-point}, there is $R^1_T \in  \Omega^{\geq 2}\left(\Omega\cap\Omega_0'' ; \End^{\mathrm{even}} \tilde{V}\right)$ and $R^2_T \in  \Omega^{\geq 2}\left(\Omega^+\cap \Omega_0''; \End^{\mathrm{even}} W\right)$ such that 
            \begin{equation}
                R_T|_{\Omega^+{\cap \Omega_0''}}= 
                \begin{pmatrix}
                   R^1_T|_{\Omega^+{\cap \Omega_0''}} &0 \\
                   0 & R^2_T|_{\Omega^+\cap\Omega_0''}
                \end{pmatrix}
                \quad \text{and} \quad R_T|_{\Omega^-{\cap \Omega_0''}}=  R^1_T|_{\Omega^-{\cap \Omega_0''}}.
            \end{equation}

            Now, let $ \eta \in C^\infty(S) $ such that $ 0 \leq \eta \leq 1 $, $ \eta|_{\Omega_0''-\Omega_1} \equiv 1 $, and $ \eta \equiv 0 $ on $\Omega_1''$. We define $\bfR_T \in \Omega^{\geq 2}\left(\Omega_0''\cap \Omega  ; \End^{\mathrm{even}} V\right)$ by 
            \begin{equation}
            \label{def-bfR_T}
                \bfR_T|_{\Omega^+\cap\Omega_0''}= 
                \begin{pmatrix}
                   R^1_T|_{\Omega^+\cap\Omega_0''} &0 \\
                   0 & \eta R^2_T|_{\Omega^+\cap\Omega_0''}
                \end{pmatrix}
                \quad \text{and} \quad \bfR_T|_{\Omega^-\cap\Omega_0''}=  R^1_T|_{\Omega^-\cap\Omega_0''}.
            \end{equation}
            Note that $ \bfR_T|_{\tilde{V}} $ commutes with $ \nabla^V $, $ \bbf^V $, and $ N^V $. Also, $ \bfR_T = O\left(\frac{1}{T}\right) $ {uniformly on $\Omega_0''\cap \Omega $} and it preserves $ o^{\ru,*}_p \otimes \cF_p $. {Finally, $\bfR_T$ coincide with $R_T$ on $\Omega\cap( \Omega_0''-\Omega_1)$.}
            
            Moreover, if $ \tilde{\bfR}_T$ is constructed in the same way from another such connected open set $\tilde{\Omega} \subset \Omega_1$, we see that $\tilde{\bfR}_T=\bfR_T$ on $\Omega\cap\tilde{\Omega}$. So we can define globally $\bfR_T \in \Omega^{\geq 2}\left(\Omega_0'' ; \End^{\mathrm{even}} V\right)$\index{RT@$\bfR_T$!On $\Omega_0''$}, such that 
            \begin{equation}
            \label{def-bfR_T-global}
                \bfR_T|_{\Omega_0''-\Omega_1}=R_T|_{\Omega_0''-\Omega_1}.
            \end{equation}

		{ We define generalized metrics $\fh_T^V$ on $V\to {\Omega_0''} $ and  $\fh^{\hV}_T$ on $\hV\to {\homega} $ as follows:
        \be\label{hvt}
        \begin{aligned}
            &\fh^V_T\index{hVT@$\fh^V_T$ and $\fh^{\hV}_T$!On $\Omega_0''$}:=\left(\left(\mathrm{id}_V+\bfR_T\right) e^{-T \bbf^V}\left(\frac{T}{\pi}\right)^{\frac{N^V}{2}-\frac{\dim(\bZ)}{4}}\right)^* h^V,\\
            &\fh^{\hV}_T(\cdot,\cdot)|_{\Omega_0''\times\{t\}}=t^{-N^S-\frac{\dim(\bZ)}{2}}\fh^V_T(t^{\frac{N^S+N^V}{2}}\cdot,t^{\frac{N^S+N^V}{2}}\cdot).
        \end{aligned} 
        \ee
In the Section \ref{sec122}, we will extend these generalized metrics to  $S'$ and $\hS'$.}

			\begin{rem}
            \label{rem-T(fh)-indep-eta}
            \begin{enumerate}[(1)]
                \item  Note that, by our bar convention and for the same reasons as in the proof of Theorem \ref{int3}, we can see that the associated torsion  $\overline{\T}_\tau^{\mL}(\A',\fh_T^{\hV})$ on $\Omega_0''$ is independent of the choice of $\eta$ {because it coincide on $\Omega_0''\cap \Omega_1$ with the (bar-version of) the torsion defined only using the part $\tilde{V}$ of $V$}. 
                 \item\label{item 2 in rem-T(fh)-indep-eta} Moreover, for the same reason, if  we define a generalized metric $\tilde{\fh}^{\hV}_T$ in the same way as $\fh^{\hV}_T$ but with $\bfR_T$ replaced by $R_T$, then on $\Omega_0''$, $\overline{\T}_\tau^{\mL}(\A',\fh_T^{\hV})=\overline{\T}_\tau^{\mL}(\A',\tilde{\fh}_T^{\hV})$.

                 \item By the last point, at this stage, the introduction of the cutoff function $\eta$ and of $\bfR_T$ may appear redundant. 
However, both $\eta$ and $\bfR_T$ become essential when proving \Cref{prop63} on $\Omega_1''$.

            \end{enumerate}

			\end{rem} 
			\def\tE{\tilde{E}}
			
			For $s\in[0,1]$, set 
            \begin{equation}
            \label{def-EtildesTt-hfraksTt}
                \begin{aligned}
			  &\tE_{s,T,t}:=\left(1+s\bfR_T\right) e^{-T \bbf^V}\left(\frac{T}{\pi}\right)^{s\left(\frac{N^V}{2}-\frac{\dim(\bZ)}{4}\right)}t^{\frac{N^S}{2}}, \\
              &\fh_{s,T,t}^{V}:=t^{-N^S-\frac{\dim(\bZ)}{2}}\tE_{s,T,t}^*h^V.
			\end{aligned}
            \end{equation}

			By a direct computation (see \cite[p. 77]{goette2001morse}), on $\Omega_0''$ \begin{align}\begin{split}\label{eq177} &\ \ \ \ L_{s,T,t}:= \tE_{s, T,t}\left(\fh_{s, T,t}^V\right)^{-1} \frac{\partial \fh_{s, T,t}^V}{\partial s} \tE_{s, T,t}^{-1}\\&=\left(N^V-\frac{\dim(\bZ)}{2}\right)(\log T-\log \pi)+2 t^{-\frac{N^S}{2}}\bfR_T\left(1+s \bfR_T\right)^{-1}t^{\frac{N^S}{2}}.\end{split}\end{align}

			{For the rest of this proof, we will denote $\underline{V}$ to be either $V$ on $\Omega_0''-\Omega_1$, or $\tilde{V}$ on $\Omega_0''\cap\Omega_1$}. Recall that $\A'=\nabla^V+\p$, $\A'_t=\nabla^V+\sqrt{t}\p$, $\nabla^{V,*}$ and $\p^*$ are the adjoint of $\nabla^V$ and $\p$ with respect to $h^V$, $\p_{T}:=e^{-T\bbf^V}\p e^{T\bbf^V}$, and $\p_T^*:=e^{T\bbf^V}\p^*e^{-T\bbf^V}.$	Let $\A''_{s,T,t}$ be the adjoint of $\A'_t$ with respect to the metric $\fh_{s,T,t}^{V}.$ 
			Let $\A_{s,T,t}:=(\A''_{s,T,t}-\A'_t)/2$ and 
			$\cC^V_{s,T,t}:=\tE_{s,T,t}\A_{s,T,t}\tE^{-1}_{s,T,t}$. For $t\leq1,$ as in \eqref{eq52} {and \eqref{added55}} ({see also \cite[p. 78]{goette2001morse}}), we find 
			\begin{align}\begin{split}\label{eq178}
					2\cC^V_{s,T,t}|_{\underline{V}}&=(\nabla^{V,*}-\nabla^V-2T[\nabla^V,\bbf^{V}])|_{\underline{V}}+O(e^{-cT}),
			\end{split}\end{align}
		{where{$\nabla^{V,*}$ is the dual of $\nabla^V$ with respect to $h^V$ and} $O(e^{-cT})$ is uniform {on $\Omega_0''$ (or, more precisely, on $\Omega_0''-\Omega_1$ or on $\Omega_0''\cap\Omega_1$ depending on whether $\underline{V}$ is $V$ or $\tilde{V}$)} and in $(s,t)$.}

			Let $\o_{T}:=\nabla^{V,*}-\nabla^V-2T[\nabla^V,\bbf^{V}]$. Because $\o_{T}|_{\tilde{V}}$ is odd and commutes with $N^V+2 \bfR_T\left(1+\bfR_T\right)^{-1}|_{\tilde{V}}$, while $h^{\prime}$ is an even function, and supertraces of even powers of odd operators vanish, we get that {on $S$}, for $t\in(0,1)$, 
            \begin{equation}
            \label{Trs(LsTth'(CVsTt)}
			\begin{aligned}
            &\Tr_s(L_{s,T,t}|_{\underline{V}}h'(\cC^V_{s,T,t}|_{\underline{V}}))\\
            &=\Tr_s\left(\frac{1}{2}\left(\left(N^V-\frac{\dim(\bZ)}{2}\right)(\log T-\log \pi)|_{\underline{V}}+2t^{-\frac{N^S}{2}} \bfR_T\left(1+\bfR_T\right)^{-1}t^{\frac{N^S}{2}}|_{\tilde{V}} h^{\prime}\left(\o_T\right)\right)\right.\\
            &+O\left(t^{-\dim(S)}e^{-c T}\right)\\ &=\left(\frac{\chi^{\prime}(\tilde{V})}{2}-\frac{\dim(\bZ) \chi(\tilde{V})}{4}\right)(\log T-\log \pi) h^{\prime}(0)+O\left(\frac{1}{T}\right)+O\left(t^{-\dim(S)}e^{-c T}\right).
			\end{aligned}
            \end{equation}
			As a result, on $\Omega_0''$, for $t\in(0,1),$
			\be
            \label{Trs(LsTth'(CsTt))-on-Vtilde}
			\Tr_s(L_{s,T,t}|_{\underline{V}} h'(\cC^V_{s,T,t}|_{\underline{V}}))^{>0} = O(t^{-\dim(S)} T^{-1}).
			\ee
			For the same reason as in the proof of Theorem \ref{int3}, we have {on $\Omega_1\cap\Omega_0''$,}
			\be
            \label{Trs(LsTth'(CsTt))-on-W}
			\overline{\Tr_s(L_{s,T,t}|_W h'(\cC^V_{s,T,t}|_W))} = 0.
			\ee
			For the same reason as in the proof of Theorem \ref{thm57}, along with \eqref{Trs(LsTth'(CsTt))-on-Vtilde} and \eqref{Trs(LsTth'(CsTt))-on-W}, in $ \Omega^{\bullet}(S)/d^S \Omega^{\bullet}(S)$, we have,
			\begin{equation}
            \label{end-of-proof-prop63}
			\begin{aligned}
            &\btiT_\tau^\mL(\A',\fh^\hV_T,h^{\bfH}_{T,R,L^2},h^{\bfH_{\Cb^m}}_{T,R,L^2}) \\
            &= \btiT_\tau^\mL(\A',(e^{-T\bbf^V})^*h^V,h^{\bfH}_{T,R,L^2},h^{\bfH_{\Cb^m}}_{T,R,L^2}) +\int_0^1 \overline{\Tr_s(L_{s,T,\tau}|_{\underline{V}} h'(\cC^V_{\tau,s,T}|_{\underline{V}}))  }ds \\
				&= \btiT_\tau^\mL(\A',(e^{-T\bbf^V})^*h^V,h^{\bfH}_{T,R,L^2},h^{\bfH_{\Cb^m}}_{T,R,L^2}) + O(\tau^{-\dim(S)}T^{-1}).
			\end{aligned}
			\end{equation}
			 \Cref{prop63} is proved on $\Omega_0''.$ 
		\end{proof}

		For $s\in[0,1]$, define {on $\Omega_0''$}
        \be
        \label{def-EsTt-bhsTt}
        \begin{aligned}
            &E_{s,T,t}:=\left(1+R_T+sR_T^{\prime}\right) e^{-T \bbf^V}\left(\frac{T}{\pi}\right)^{\frac{N^V}{2}-\frac{\dim(\bZ)}{4}}t^{\frac{N^S}{2}},\\
            &\bh_{s,T,t}^{V}:=t^{-{N^S}-\frac{\dim(\bZ)}{2}}E_{s,T,t}^*h^V,\\
            & \p_{s,T,t}:=E_{s,T,t}\p E_{s,T,t}^{-1} \quad \text{and} \quad \p_{s,T,t}^*:=E_{s,T,t}^{-1}\p^* E_{s,T,t}.
        \end{aligned}
        \ee

		\def\fA{{\mathfrak{A}}}
		\def\fB{{\mathfrak{B}}}
		
		Let $\fA''_{s,T,t}$ be the adjoint of $\A'_t$ with respect to the metric $\bh_{s,T,t}^{V}.$ 
		Let $\fA_{s,T,t}:=(\fA''_{s,T,t}-\A'_t)/2.$
		
		Let $\fB_{s,T,t}:=E_{s,T,t}\fA_{s,T,t}E_{s,T,t}^{-1}$.
		
		Note that degree $0$ component is irrelevant in bar convention. By \eqref{eq167}, we can see  that
		\[
		\overline{\T}_\tau^{\mL}(\A',\bh_T^\hV) = \int_\tau^\infty \overline{\Tr_s \left( \frac{N^V}{2} h'(\fB_{s=1,T,t}) \right)} \frac{dt}{2t}.
		\]
		Similarly, by \eqref{hvt}, {\eqref{def-EsTt-bhsTt} and item \eqref{item 2 in rem-T(fh)-indep-eta} in  \Cref{rem-T(fh)-indep-eta},}
		\be
		\overline{\T}_\tau^{\mL}(\A',\fh_T^\hV) = \int_\tau^\infty \overline{\Tr_s \left( \frac{N^V}{2} h'(\fB_{s=0,T,t}) \right)} \frac{dt}{2t}.
		\ee

		Let $D^{\bh}_{s,T,t}:=(\p_{s,T,t}+\p_{s,T,t}^*)/2$ and $V^{\bh}_{s,T,t}:=(\p_{s,T,t}^*-\p_{s,T,t})/2$. Let $\Delta^{\bh}_{s,T,t}:=(D^{\bh}_{s,T,t})^2$, then we also have $\Delta^{\bh}_{s,T,t}=-(V^{\bh}_{s,T,t})^2.$
		Moreover, 
        \begin{equation}
        \label{computation-dVsTt/ds}
            \frac{\p}{\p s}V^{\bh}_{s,T,t}=\frac{1}{2}\Big(\big[ \p_{s,T,t},\frac{\p E_{s,T,t}}{\p s} E_{s,T,t}^{-1} \big] + \big[ \p_{s,T,t}^*,E_{s,T,t}^{-1}\frac{\p E_{s,T,t}}{\p s} \big]\Big).
        \end{equation}

		For $\l\in\C$ with $|\Re(\l)|=1$, we have
        \begin{equation}
        \label{bound-(l-VsTt)^(-1)DsTt}
            \begin{aligned}
            & &\|(\l- V_{s,T,t}^{\bh})^{-1} V^{\bh}_{s,T,t}\|^2 &\leq C|\lambda|^2,\\
              &  &\|(\l- V_{s,T,t}^{\bh})^{-1} D^{\bh}_{s,T,t}\|^2&=\|(\l- V_{s,T,t}^{\bh})^{-1} D^{\bh}_{s,T,t}((\l- V_{s,T,t}^{\bh})^{-1} D^{\bh}_{s,T,t})^*\|\\
              & & &= \|(\l- V_{s,T,t}^{\bh})^{-1}\Delta^{\bh}_{s,T,t}( \overline\l+ V_{s,T,t}^{\bh})^{-1}\|\\
             &  & &\leq \|(\l- V_{s,T,t}^{\bh})^{-1} V^{\bh}_{s,T,t}\|^2\leq C|\l|^2.
            \end{aligned}
        \end{equation}
		
		Note that $\frac{\p E_{s,T,t}}{\p s}E_{s,T,t}^{-1}=O(e^{-cT})$ and $E_{s,T,t}^{-1}\frac{\p E_{s,T,t}}{\p s}=O(e^{-cT})$, so by \eqref{computation-dVsTt/ds} and \eqref{bound-(l-VsTt)^(-1)DsTt}, we can see that if $t\geq\tau$,
		\be\label{eq175}
		\frac{\partial}{\partial s} (\lambda - V_{s,T,t}^{\bh})^{-1} = (\lambda - V_{s,T,t}^{\bh})^{-1} \frac{\p V^{\bh}_{s,T,t}}{\p s} (\lambda - V_{s,T,t}^{\bh})^{-1} = |\lambda|^2 O(e^{-cT}).
		\ee
		
		Let $U_{s,T,t}:=\fB_{s,T,t}- V_{s,T,t}^{\bh}$, then by an equation similar to \eqref{eq178}  (see also \cite[p. 78]{goette2001morse}), we have for $t\geq\tau$,
		\be\label{eq182}
		\frac{\p}{\p s}U_{s,T,t}=O(e^{-cT}),
		\ee
      {where $O(e^{-cT})$ is uniform on $(s,t)$.}

		It follows from \eqref{eq175} and \eqref{eq182} that
		\be\label{eq183}
		(\l-\fB_{t,0,T})^{-1}-(\l-\fB_{t,1,T})^{-1}=|\l|^\bfn O(e^{-cT}).
		\ee
		for some positive integer $\bfn.$	
        
		On the other hand, using \eqref{added55} abd proceeding as what we did in \Cref{lem810}, we can see that there exists $(T,s,t)$-independent $c_1,c_2>0$, s.t. all nonzero eigenvalues of $D_{s,T,t}^{\bh}$ are inside $[\sqrt{t}e^{-c_1T},\sqrt{t}e^{-c_2T}]$.
		Proceeding as in \cite[Theorems 2.13]{bismut1995flat}, we then get that for $t>0$, in positive degree, 
		\be\label{eq184}
		((\l-\fB_{t,0,T})^{-1}-(\l-\fB_{t,1,T})^{-1})^{>0}=|\l|^{\bfn}O\left(\frac{e^{c_1T}}{\sqrt{t}}\right).
		\ee

		Let $\gamma$ be the contour given by $\{z\in\C:|\Re(z)|=1\}$, then
		\[h'(\fB_{s,t,T})=\int_\gamma h'(\l)(\l-\fB_{s,t,T}).\]
		
		Thus, on the one hand using \eqref{eq183} we find
		\[
		\int_\tau^{e^{T^2}} \left( \Tr_s \left( \frac{N^V}{2} h'(\fB_{t,1,T}) \right) - \Tr_s \left( \frac{N^V}{2} h'(\fB_{t,0,T}) \right) \right)^{>0} \frac{dt}{2t} \leq C \int_\tau^{e^{T^2}} e^{-cT} \frac{dt}{2t} \leq C (T^2 - \ln \tau) e^{-cT},
		\]
		while on the other hand using \eqref{eq184} we find
		\[
		\int_{e^{T^2}}^\infty \left( \Tr_s \left( \frac{N^V}{2} h'(\fB_{t,1,T}) \right) - \Tr_s \left( \frac{N^V}{2} h'(\fB_{t,0,T}) \right) \right)^{>0} \frac{dt}{2t} \leq C {\int_{e^{T^2}}^\infty} e^{c_1 T} \frac{dt}{2 t^{3/2}} \leq C e^{c_1 T - T^2 / 2}.
		\]
		As a result, \be\label{eq185}\lim_{T\to\infty}\left(\overline{\T}_\tau^{\mL}(\A',\bh_T^\hV)-\overline{\T}_\tau^{\mL}(\A',\fh_T^\hV)\right)=0.\ee

		{Recall that the connection $\nabla^\bfH$  on $\bfH$ was defined below \eqref{defnkt}.}

		Let $\bh^V_{T,L^2}$ and $\fh^V_{T,L^2}$ be the {metrics on $\bfH$  induced  by the Hodge metric on $H^V$ associated with degree $0$ component of $\bh_T^V$ and $\fh_T^V$ respectively,  together with {$\J^{-1}$, with $\J$ in Theorem \ref{thm54}}. By Lemma \ref{1010}, \eqref{def-bfR_T-global} and \eqref{hvt}, we have}
		\be\label{eq189}
		\bh^V_{T,L^2}=\fh^V_{T,L^2}(1+O(e^{-cT})).
		\ee

{As below \eqref{torsion-with-prime-metric-and-nonprime-metric}, we denote by $h^{\bfH}_{T,R,L^2}$  the metric on $\bfH$ induced by $h^{\cE}_{T,R}$ on $\cE =  \Omega^{\bullet}(\bZ; \cF|_{\bZ})$ and the isomorphism $J_T=\K_1\J_T\K_T^{-1}\colon \bfH_{T}\overset{\sim}{\rightarrow}\bfH$, where $\J_T$ and $\K_T$ are defined in \eqref{defjt} and \eqref{defnkt} respectively. Then since $\cL_T$ is induced by $I_T$, it follows from the definition of $\J$ {in Theorem \ref{thm54}} that \be\label{eq1981} h^{\bfH}_{T,0,L^2} = \bh^V_{T,L^2}.\ee}

		\def\abc{0}
		\if\abc0
		\def\bp{{\mathbf{p}}}
		\def\fp{{\mathfrak{p}}}
		\def\til{{\tilde{\l}}}

		
		
		
		
		\fi
		
		Let $\P_{T,R}$ be the orthogonal projection {onto $\H_{T,R}$ in \eqref{def-Hotr-Hotro-Hotrt} with respect to $h^{\cE}_{T,R}$}.
		By \cite[Porposition 2.6]{bismut1995flat}, which is still valid on our non-compact $\M$,   one can see that
		\be\label{eq200}\begin{aligned}\left\|\left(h^{\bfH}_{T,R,L^2}\right)^{-1}\left(\nabla^{\bfH}h^{\bfH}_{T,R,L^2}\right)\right\|&=\left\|\P_{T,R}\left(\nabla^{\cE,*}_{T,R}-\nabla^{\cE}\right)\P_{T,R}\right\|\\
			&=\left\|\P_{T,R}\left(\nabla^{\cE,*}-\nabla^{\cE}-2d^S\bbf_{T,R}\right)\P_{T,R}\right\|=O(TR),\end{aligned}\ee
		where $\nabla^{\cE,*}$ is the adjoint connection of $\nabla^{\cE}$ with respect to $h^{\cE}_{T=0,R=0}.$ {In the same way, \be\label{eq200-bis}\left\|\left(h^{\bfH}_{T,0,L^2}\right)^{-1}\left(\nabla^{\bfH}h^{\bfH}_{T,0,L^2}\right)\right\|= O(T).\ee}


		\def\chH{{\check{\bfH}}}
		Recall that $\chS:=S\times[0,1]$ and $\chH\to \chS$ is the pullback of $\bfH\to S$ with respect to canonical projection $\chS\to S$. 	Let $h_l^{\bfH}=lh^{\bfH}_{T,R,L^2}+(1-l)\fh^V_{T,L^2},l\in[0,1]$, and $h^{\chH}$  be a metric on $\chH\to\chS$, s.t. $h^{\check{\bfH}}|_{S\times\{l\}}=h_l^{\bfH}$. To prove Theorem \ref{int2}, we still need to estimate $\tilde{h}(\nabla^H,h^{\chH}).$
		
		By Lemma \ref{lem915}, \eqref{eq189} and \eqref{eq1981},  		\be\label{eq202}
		\left\|\left(h_l^{\bfH}\right)^{-1}\frac{\p h_l^{\bfH}}{\p l}\right\|=O(e^{-cT}).
		\ee
     {In particular, for $l\in[0,1]$ and $i=0$ or $1$,
        \be
        \label{comp-h_0-h_1}
        h_l^{\bfH} = h_i^{\bfH}\big(1+O(e^{-cT})\big)
        \ee}
		
		By \eqref{eq200}{, \eqref{eq200-bis}, \eqref{RT-norm} and \eqref{comp-h_0-h_1},}
		
		\begin{multline}
        \label{eq199}
	\left\|\left(h_l^{\bfH}\right)^{-1}\left({\nabla^\bfH}h^{\bfH}_l\right)\right\|
			=\left\|l\left(1+O(e^{-cT})\right)(h^{\bfH}_{T,R,L^2})^{-1}(\nabla^{\bfH}h^{\bfH}_{T,R,L^2}) \right. \\
            \left.+(1-l)\left(1+O(e^{-cT})\right)(\fh^V_{T,L^2})^{-1}(\nabla^{\bfH}\fh^V_{T,L^2}) \right\|
			=O(TR).
		\end{multline}
		
		Here $O(e^{-cT})$ may represent different operator on its each occurrence. 
		
		So by \eqref{htilde}, \eqref{eq189}, \eqref{eq202}  and \eqref{eq199}
		\be\label{eq191}
		\tilde{h}(\nabla^\bfH,h^{\chH})=O\left(e^{-cT}(TR)^{\dim(S)}\right).
		\ee

		The first estimate in Theorem \ref{int2} then follows from {Definition \ref{assotor},} \eqref{comp-torsion-an-torsion-bh_T}, \eqref{eq185} and \eqref{eq191}.

		\subsection{On $\Omega_1''$}\label{sec122}
		Since $S_k=\emptyset,k\geq2$, we can take $\Omega_1$ as a small enough tubular neighborhood of $L=\pi(\Sigma^{(1)}(f)),$ so that the function $t_1$ in Lemma \ref{lem441} is small enough. Throughout this section we focus our discussion on $\Omega_1''.$

		We use the same notation as in \cref{diskremov}, \cref{sec10} and \cref{proofint20}.

		Similar to \cite[Definition 5.7]{bismut2001families}, We have the following definition:
		\begin{defn}
        \label{def-I_R^-}
        Let 
            \begin{multline}\Omega_{\AG}^\bullet(\M^-, \cF|_{\M^-})|_{(\ppi)^{-1}(\Omega_1'')}:=\\
            \Big\{\alpha \in\Omega^\bullet(\M^-, \cF|_{\M^-})|_{(\ppi)^{-1}(\Omega_1'')}\:: \a \text{ satisfies the relative boundary conditions} \\\text{and \eqref{eq39} {fiberwisely, uniformly on compact sets of }} \Omega_1''\Big\}.
            \end{multline}

			Let $I^{-}_R:\Omega^{\bullet}_{\AG}(\M^-, \cF|_{\M^-})|_{(\ppi)^{-1}(\Omega_1'')} \rightarrow \Omega^{\bullet}\left(\Omega_1'', V^-\right)$ be the map given by
			$$
			I^-_{R}( \alpha)\index{IRminus@$I^-_{R}$}=\sum_{p \in \Sigma(\bbf_{R}|_{\M^-})} \int_{{\rW^\ru(p)}} \alpha. 
			$$
			Here $\bbf_{R}:=\bbf_{T=1,R}.${The well-definedness of $I^-_{R}$ is discussed below.}
		\end{defn}
		
		\begin{rem}
        \label{non-compact-W.Lu}
        In this section, we want to apply the results of \cite{lu2017thom}, but in a non-compact setting. These non-compact version of \cite[Theorem 1.1-Theorem 1.3 and Remark 5.7]{lu2017thom} can be established using the techniques of  \cite{DY2020cohomology}. Indeed, using Agmon estimates as in \cite[Corollary 7.2 and \S 7.5]{DY2020cohomology}, when working on non-compact manifold we can actually, up to a small error, remove a neighborhood of infinity to get a manifold with boundary, and then compactify it using its double. 
		\end{rem}

        Recall that, as explained in Remark \ref{rem-construction-V^-}, no points of $\partial \bZ^-$ appear in the Thom-Smale-Witten complex of $\bZ^-$.

		Proceeding as in the proof of \cite[Theorem 5.8]{bismut2001families}, using the non-compact relative version of \cite[Theorem 1.3]{lu2017thom} (see also \cite[Remark 5.7]{lu2017thom}), we can see that  if $d^{\M^-}\a$ also satisfies the Agmon estimate \eqref{eq39}, we have
		\be\label{dicommutr}
		I^{-}_Rd^{\M^-}\a=\A^{\prime,-}I^{-}_R\a.
		\ee

		\def\Coztr{{\mathcal{C}_{t,T,R}}}	
		\def\Doztr{{\mathcal{D}_{t,T,R}}}
        \def\Coztrm{{\mathcal{C}^-_{t,T,R}}}	
		\def\Doztrm{{\mathcal{D}^-_{t,T,R}}}
		
		
        Let $\cE^-:=\Omega^{\bullet}(\bZ^-,\cF|_{\bZ^-})$\index{Eminus@$\cE^-$} and let  $h^{\cE,-}_{T,R}$  \index{hETRminus@$h^{\cE,-}_{T,R}$} be the $L^2$ metric  on $\cE^-$ induced by $T^H\M$, $g^{T\bZ}$ and $h_{T,R}^{\cF}$. Let $\cC_{t,T,R}^-$\index{CtTRminus@$\cC_{t,T,R}^-$} and $\cD_{t,T,R}^-$\index{DtTRminus@$\cD_{t,T,R}^-$} be the defined as $\cC_{t,T,R}$ and $\cD_{t,T,R}$ in \eqref{defdt}, but replacing $\bZ$, $\M$, $\cE$ and $h^{\cE}_{T,R}$ respectively by $\bZ^-$, $\M^-$, $\cE^-$ and $h^{\cE,-}_{T,R}$. 
        {We now state the analogue of Theorem \ref{sp(CtT)-and-ItT}.}
		\begin{thm}
        \label{thm-ItTRminus-isom}
			There exists $T_0=T_0(R_0),$ such that for $R=R_0$ and for all $T > T_0$,
			\be\label{thm-ItTRminus-isom-eq-spectrum}
			\operatorname{Spec}\left((\Coztrm)^2\right)=-\operatorname{Spec}\left((\Doztrm)^2\right) \subset\left[0, \frac{t}{4}\right] \cup(4 t, \infty),
			\ee
			
			Let $\gamma$ be the oriented contour given by $\{a\in \C:|a|=1\}.$
			
			Let \[\begin{aligned}
			    &\P^-_{t,T,R}\index{PtTRminus@$\P^-_{t,T,R}$}:=\frac{1}{2\pi\sqrt{-1}}\int_{t\gamma}(\l-(\Coztrm)^2)^{-1}d\l,\\
                &W^-_{t,T,R}:=\Im(\P^-_{t,T,R}).
			\end{aligned}\] 
			Then $W^-_{t,T,R}$ form a finite dimensional vector bundle $W^-_{t,T,R} \rightarrow \Omega_1''$, by \cite[Theorem 1.1]{DY2020cohomology} and our discussions in \cref{proofthm54}, $W_{t,T,R}^-\subset \Omega_{\AG}^\bullet(\M^-, \cF|_{\M^-})|_{(\ppi)^{-1}(\Omega_1'')}$, so
			$$
			I^-_{t,T,R}\index{ItTRminus@$I^-_{t,T,R}$}=I_R^-|_{W^-_{t,T,R}}: W^-_{t,T,R} \rightarrow \Lambda T^*S\hat\otimes V^-
			$$
			is well-defined. Moreover, $I^-_{t,T,R}$ is a linear isomorphism.
		\end{thm}

        		In the sequel, we will denote  $I^-_{T,R}:=I^-_{t=1,T,R}$\index{ITRminus@$I^-_{T,R}$}.
        {\begin{proof}
             This theorem is proved in the same way as Theorem \ref{sp(CtT)-and-ItT}, proceeding as in \cite[Theorems 10.2 and  Proposition 10.24]{bismut2001families}. However here we need to use non-compact version of \cite[Theorem 1.1-Theorem 1.3 and Remark 5.7]{lu2017thom} discussed after Definition \ref{def-I_R^-}. Moreover, here $R$ varies, and we need to prove that the bounds on the spectrum given in \eqref{thm-ItTRminus-isom-eq-spectrum} can be taken uniformly in $R$. This is a consequence of Remark \ref{indepr}, which ensure that the critical points created by the $R$-deformation, as described in Lemma \ref{sixcrits}, behave nicely.
        \end{proof}}

		\begin{proof}[\textbf{Proof of Theorem \ref{thm54} in $\Omega_1''$}]
			In this proof we fix $R=R_0$. {Recall that the first part of Theorem \ref{thm54} have been proved in \cref{proofthm54}, so here we show that $ \J $ is an isomorphism of bundles in $ \Omega_1'' $.} {As $R$ is fixed here, we will not make it appear in the notations in this proof to keep them simpler.}
            
            Recall that $ \bfH_T^- \to \Omega_1'' $ is the $ L^2 $-cohomology  associated with $ (d^{\bZ}|_{\bZ^-}, g^{T\bZ}|_{\bZ^-}, e^{-2T\bbf_R} h^{\cF}|_{\bZ^-}) $  with relative boundary conditions and  $\bfH^-:=\bfH^-_{T=1}$. Recall also that $H^{V^-}\to\Omega_1''$ is the cohomological bundle associated to $(V^-,\p^-)$. As in \cref{combinatorial-complex}, we {will define} a map $ \J^-: \bfH^- \to H^{V^-} $. 
            
            Note that {here, we cannot have a simple formula for} the analogue of $ \J_T $ as in \eqref{defjt}, but {for $ T' \in [\frac{7T}{8}, \frac{9T}{8}] $, we can construct} {an isomorphism} $ J^-_{T,T'}{\colon \bfH_{T}^-\to\bfH_{T'}^-}$ analogous to {$ \J_{T,T'} $ defined above Lemma \ref{lem51} and to} $ J_{T,T'} $ defined above Lemma \ref{lem53}. {Moreover, the second part of Lemma \ref{lem53} is still true in the present situation where $\bZ$ is replaced with $\bZ^{-}$, so we can see that if $T$, $T'$ and $T''$ are close to each other, $J^-_{T,T''}=J^-_{T',T''}\circ J^-_{T,T'}$. In particular, for any $T\leq T'$, by subdivising $[T,T']$ as $\bigcup_{i=1}^N [T_i,T_{i+1}]$ with $T_i$ and $T_{i+1}$ close, we can define $J^-_{T,T'}$ as the composition of all the $J^-_{T_i,T_{i+1}}$, and the result does not depend on the choice of the subdivision.}
            
            We set 
            \be  \label{def-Jminus}\J^-\index{Jminus@$\J^-$} := \cL_T^-\circ(J^{-}_{T,1})^{-1}, \ee
            where the map $ \cL_T^- $ is similar to $ \cL_T $ in \eqref{def-cL_T} {using $I_{T,R}^-$ rather that $L^a_T$, $a=1,2$}. 
            Using the same arguments as {in \cref{proofthm54}}, we can show that $ \cL_T^-\circ(J^{-}_{T,1})^{-1} $ is independent of $ T $.
			Also, by Theorem \ref{thm-ItTRminus-isom}, using the same argument as in the proof of Theorem \ref{thm54} in $\Omega_0''$ in \cref{sec121}, we can show that $ \J^- $ is an isomorphism of bundles. 
            
            We will now use the notation from \cref{sec10} and \cref{proofint20}.
			
			Here we will use the map $\I$ defined in \eqref{eq-def-Ical} for $\bZ^-\subset \bZ$. Let $\mathbf{P}=\P\I$\index{P@$\mathbf{P}$}, with $\P$ the orthogonal projection onto $\Hotrsm$ as below \eqref{eq-def-Ical}. Denote by  $ \mathbf{P}^{>1} $ its the restrictions to components of degree greater than 1.  Then, by \eqref{eq158}, \eqref{Decompo-Hotrsm} and \eqref{degree01sm}, $\mathbf{P}^{>1}|_{(\H_{T,R}^{\bZ^-})^{>1}}:(\H_{T,R}^{\bZ^-})^{>1}\to  {(\Hotrsm)^{>1}}={\H_{T,R}^{>1}}$  is an isomorphism. { Moreover, as $\bbf$ and $\bbf_R$ coincide outside of a compact set, it is equivalent for a smooth form to be $L^2$ with respect to the metric induced by $ g^{T\bZ} $ and $ e^{-2{T}\bbf}h^{\cF}$ and with respect to the one induced by $ g^{T\bZ} $ and $ e^{-2{T}\bbf_R}h^{\cF}$. Thus, using again the same kind of construction as for $\J_{T,T'}$, $J_{T,T'}$ and $J^-_{T,T'}$, we get a natural identification between  $ \bfH_{T} $ and the $ L^2 $-cohomology of $ ( \Omega^{\bullet}(\bZ, \cF|_{\bZ}), d^{\bZ}) $ with respect to the $ L^2 $ metric induced by $ g^{T\bZ} $ and $ e^{-2{T}\bbf_R}h^{\cF} $. Combining the two  isomorphisms above and Hodge theory, we get an isomorphism \be\label{Phi-T} \Phi_T \colon (\bfH_T^-)^{>1}\to (\bfH_{T})^{>1}.\ee}
			
			Let $ w $ be a harmonic form associated with $ (d^{\bZ}|_{\bZ^-}, g^{T\bZ}|_{\bZ^-}, e^{-2T\bbf_R} h^{\cF}|_{\bZ^-})$. By integration by parts, we can see that for any compactly supported smooth form $ \a \in  \Omega^{\bullet}(\bZ, \cF|_{\bZ}) $,
			\[ \int_\bZ \langle w, d_{T,R}^{\bZ,*} \a \rangle e^{-2\bbf_{T,R}} \, \dvol_{\bZ} = 0. \]
			Thus, $ \I(w) $ is $ d^{\bZ} $-closed in $  \Omega^{\bullet}(\bZ, \cF|_{\bZ}) $.

			Following the same arguments as in \cite[p. 673]{DY2020cohomology}, there exists a harmonic form $ w' \in  \Omega^{\bullet}(\bZ, \cF|_{\bZ}) $ and an $ L^2 $-form $ \b \in  \Omega^{\bullet}(\bZ, \cF|_{\bZ}) $ such that
			\be\label{eq1344} \I(w) = w' + d^{\bZ} \b. \ee
			Moreover, $ \tilde{\b} := e^{-\bbf_{T,R}} \b $ satisfies \eqref{eq39}.

            Let $ \J^{-,> 1} $ be the restrictions of $ \J^- $ to components of degree greater than 1. 

            {By \eqref{Def-V_k}, \eqref{d-bd-point}, \eqref{H^V=bfH} and \eqref{def-Vtilde-on-Omega1}, we see that $\bfH^{>1}\simeq H^{>1}(V,\partial) = H^{>1}(\tilde{V}, \partial|_{\tilde{V}})$. On the other hand, by \eqref{def-Vtilde-on-Omega1}, \eqref{decvr} and \eqref{decompo-partial_R}, we see that $H^{>1}(V^-,\partial^-) = H^{>1}(\tilde{V}, \partial|_{\tilde{V}})\oplus \ker(c)\oplus \mathrm{coker}(c)$. Moreover, by \cite{lu2017thom} and Remark \ref{non-compact-W.Lu}, we have $H^{>1}(V^-,\partial^-) \simeq \bfH^{-,>1}$. Thus, using \eqref{eq158}, we find that $\ker(c)\oplus \mathrm{coker}(c)=0$ so that we can identify $H^{>1}(V,\partial)$ and $ H^{>1}(V^-,\partial^-)$.}
            As a result, by Stokes' formula (c.f. \cite[Corollary 7.2]{DY2020cohomology} and Remark \ref{rem56}), $\bfH_T=\bfH_T^{>1}$ and \eqref{eq1344}, we have for $\J$ in the statement of Theorem \ref{thm54}:
			\be \label{J-J-minus}\J = \J^{-,>1} J^{-}_{T,1} {\Phi_T}^{-1} J_{1,T}. \ee
			Thus, $ \J $ is an isomorphism in $ \Omega_1'' $.
		\end{proof} 

        Again using Remark \ref{indepr} as in the proof of Theorem \ref{thm-ItTRminus-isom}, and proceeding as in \cite[Theorem 10.5]{bismut2001families}, we have

		\begin{thm}\label{thm103r}
			There exists a constant $\epsilon_0 \in(0,1)$ and $T_0=T_0(R_0,\tau),$ such that for $R=R_0$ and $T > T_0$,
			$$
			\int_\tau^{\infty}\left(\operatorname{Tr}_s\left(N^\bZ h^{\prime}\left(\Doztrm\right)\right)-\operatorname{Tr}_s\left(N^\bZ \P^-_{t,T,R}h^{\prime}\left({\P}^-_{t, T,R} \Doztrm {\P}^-_{t, T,R}\right)\P^-_{t,T,R} \right)\right)\frac{d t}{2 t}=O\left(T^{-\epsilon_0}\right).
			$$
			That is,
			$$
			\int_\tau^{\infty}\operatorname{Tr}_s\left(N^\bZ \P^{-,c}_{t,T,R}h^{\prime}\left({\P}^{-,c}_{t, T,R} \Doztrm {\P}^{-,c}_{t, T,R}\right)\P^{-,c}_{t,T,R} \right)\frac{d t}{2 t}=O(T^{-\epsilon_0}),
			$$
			where $\P^{-,c}_{t,T,R}:=1-\P^-_{t,T,R}$.
		\end{thm}

        As before, $\hat{\Omega}_1'':=\Omega_1''\times(0,\infty)$.
		Let $\hV^-\to\hat{\Omega}_1''$ be the pullback of the bundle $V^-\to \Omega_1''$ with respect to canonical projection $\hat{\Omega}_1''\to \Omega_1''$.
	Let $\bh_{T,R}^{V^-}$ and $\bh_{T,R}^{\hV^-}$ be the generalized metrics on $V^-\to \Omega_1''$ and  $\hV\to{\homega}$ given by (c.f. \cite[Definitions 10.20 and 10.30]{bismut2001families})	
		\be
         \label{def-bh-TRV-}
         \begin{aligned}
             &\bh_{T,R}^{V^-}(\cdot,\cdot)\index{hVTRminus@$\bh_{T,R}^{V^-}$ and $\bh_{T,R}^{\hV^-}$}:=h^{\cE,-}_{T,R}((I^-_{T,R})^{-1}\cdot,(I^{-}_{T,R})^{-1}\cdot),\\
             &\bh_{T,R}^{\hV^-}(\cdot,\cdot)|_{S\times\{t\}}=h^{\cE,-}_{T,R}((I^{-}_{t,T,R})^{-1}\cdot,(I^{-}_{t,T,R})^{-1}\cdot).
         \end{aligned}
        \ee
		  Then, as in \cite[Proposition 10.24]{bismut2001families}),
		\[\bh_{T,R}^{\hV^-}(\cdot,\cdot)|_{S\times\{t\}}=t^{-N^S-\frac{\dim(\bZ)}{2}} \bh_{T,R}^{V^-}\left(t^{\frac{N^S+N^{V^-}}{2}} \cdot, t^{\frac{N^S+N^{V^-}}{2}} \cdot\right),\]
		where $N^{V^-}\in\End(V^-)$, s.t. for any section $s$ of $V^-$ of degree $k$, $N^{V^-}s=ks$.
		
		It follows from \cite[Theorem 10.31]{bismut2001families} that $\bh_{T,R}^{\hV^-}$ satisfies  Condition \ref{assum21}.

		Let \begin{align*}\T^{\mL,-}_{\tau,T,R}:&=-\int_\tau^{\infty}\operatorname{Tr}_s\left(N^\bZ \P^-_{t,T,R} h^{\prime}\left(\P^-_{t,T,R} \Doztrm \P^-_{t,T,R}\right)\P^-_{t,T,R} \right)\\
			&-\left(\frac{\chi(\bZ^-,\cF)\dim(\bZ)-2\chi'(\bZ^-,\cF)}{4}\right){h'(\frac{\sqrt{-t}}{2})}\frac{d t}{t}.\end{align*}
	Then by Remark \ref{remeq}, \eqref{dicommutr}  and Theorem \ref{thm-ItTRminus-isom}, it's straightforward to check that
		\be\label{thm6411r}
		\overline{\T}^{\mL}_\tau(\A^{\prime,-},\bh_{T}^{\hV^-})=\overline{\T}^{\mL,-}_{\tau,T,R}.\ee
        Using this equation and Theorem \ref{thm103r}, we get on $\Omega_1''$:
        \begin{equation}
        \label{comp-torsion-Mminus-and-torsion-bh_TR}
            \lim_{T\to\infty} \left( \overline{\T}_\tau^{\mL}\big(T^H\M^-,g^{T\bZ}|_{\bZ^-},h_{T}^\cF\big)-\overline{\T}^{\mL}_\tau(\A^{\prime,-},\bh_{T,R}^{\hV^-})\right)=0
        \end{equation}

		Let $\bbf_{T,R}^{V^-}\in\End(V^-)$\index{fTRminus@$\bbf_{T,R}^{V^-}$}, such that for $p\in\Sigma(\bbf_{T,R})$,
		$\bbf_{T,R}^{V^-}|_{o^{\ru,*}_p\otimes \cF_p}=\bbf_{T,R}(p). \mathrm{Id}_{ o^{\ru,*}_p\otimes \cF_p}$. By Remark \ref{indepr} and Agmon estimates derived in \cref{sec9}, for essentially the same reason that \cite[(10.244), (10.246) and (10.249)]{bismut2001families} holds, we have:
		\begin{lem}\label{1010r}
			There exists a $h^{V^-}$-self-adjoint operator $R_{T,R} \in \Omega^{>1}\left(\Omega_1'' ; \End^{\mathrm{even}} V^-\right)$\index{RTR@$R_{T,R}$} commuting with $\nabla^{V^- }$, $\bbf^{V^-}_{1,R}$ and $N^{V^-}$, which is a polynomial in $\frac{1}{T}$ without constant term, and another $h^{V^-}$-self-adjoint $R_{T,R}' \in  \Omega^{\bullet}(\Omega_1'' ;\End (V^-))$ with \be\label{RTR-norm}\left\|R_{T,R}^{\prime}\right\|+\left\|\nabla^{V^-}R_{T,R}^{\prime}\right\|=O\left(e^{-c T}\right)\ee uniformly for some $(T,R)$-independent constant $c$, such that
			$$
			\bh_{T,R}^{V^-}=\left(\left(\mathrm{id}_{V^-}+R_{T,R}+R_{T,R}^{\prime}\right) e^{-T \bbf_{1,R}^{V^-}}\left(\frac{T}{\pi}\right)^{\frac{N^{V^-}}{2}-\frac{\dim(\bZ)}{4}}\right)^* h^{V^-}.
			$$

			As in the proof of Theorem \ref{int3},  on $\Omega_1''$ the bundle $V^-$ admits a decomposition
            \be \label{decompo-V^--on-Omega_1}V^- = \tilde{V} \oplus W',\ee
              where $\tilde{V}$ is defined in \eqref{def-Vtilde-on-Omega1}, and $W'$ corresponds to the three critical points described in Lemma \ref{sixcrits}. Then we also have, 
			\be\label{eq193} 
            \begin{aligned}
            &R_{T,R} \mbox{ preserves } o^{\ru,*}_p\otimes \cF_{p},\\
            &R_{T,R}|_{\tilde{V}} \mbox{ does not depend on  }R,\\
            &R_{T,R}|_{\tilde{V}}=R_{T}|_{\tilde{V}} \text{ on } \Omega_0''\cap\Omega_1'' ,
            \end{aligned}\ee
		{with $R_T$ in Lemma \ref{1010}.}
		\end{lem}

		Let \be\label{hvtr}
        \begin{aligned}
            &\fh^{V^-}_{T,R}\index{hVTRminus@$\fh^{V^-}_{T,R}$ and $\fh^{\hV^-}_{T,R}$}:=\left(\left(\mathrm{id}_{V^-}+R_{T,R}\right) e^{-T \bbf_{1,R}^{V^-}}\left(\frac{T}{\pi}\right)^{\frac{N^{V^-}}{2}-\frac{\dim(\bZ)}{4}}\right)^* h^{V^-},\\
            &\fh^{\hV^-}_{T,R}(\cdot,\cdot)|_{S\times\{t\}}=t^{-N^S-\frac{\dim(\bZ)}{2}}\fh^{V^-}_{T,R}(t^{\frac{N^S+N^{V^-}}{2}}\cdot,t^{\frac{N^S+N^{V^-}}{2}}\cdot).
        \end{aligned}\ee  

	\begin{proof}[\textbf{Proof of Proposition \ref{prop63} on $\Omega_1''$}]
\def\fR{\mathfrak{R}}
By \eqref{eq193}, with respect to the decomposition \eqref{decompo-V^--on-Omega_1}, the operator $R_{T,R}$ takes the form
\[
R_{T,R} =
\begin{pmatrix}
R_{T}^1 & 0 \\
0 & R_{T,R}^2
\end{pmatrix}.
\]

Similarly, the bundle $V$ decomposes locally near $L$ as
\[
V = \tilde{V} \oplus W,
\]
in the sense made precise in the proof of Proposition \ref{prop63} over $\Omega_0''$ in Section \ref{sec121}.

{In \eqref{def-bfR_T}-\eqref{def-bfR_T-global}, we defined an operator $\bfR_T$ (acting on $V\to \Omega_0''$). We now extend it to $S'$ by setting for any connected open set $\Omega\subset \Omega_1''$:}
\be
\label{def-extension-bfR_T}
\bfR_T|_{\Omega^+} =
\begin{pmatrix}
R_{T}^1|_{\Omega^+} & 0 \\
0 & 0
\end{pmatrix},
\quad \text{and} \quad
\bfR_T|_{\Omega^-} = R^1_{T}|_{\Omega^-},
\ee
with $\Omega^\pm$ in \eqref{def-Omega^pm}. By the construction of the function $\eta$ above \eqref{def-bfR_T} and the second identity in \eqref{eq193}, this definition of $\bfR_T$\index{RT@$\bfR_T$!On $\Omega_1''-L$} agrees with the previous one on $\Omega_1'' \cap \Omega_0''$.

The operator $\bfR_T$ commutes with $\nabla^V$, $\bbf^V$, and $N^V$, and satisfies the estimate
\[
\bfR_T = O\left(\frac{1}{T}\right) \quad \text{uniformly on } S'.
\]
Moreover, it preserves the subbundle $o^{\ru,*}_p \otimes \cF_p$.

We now extend the metrics $\fh^V_T$ and $\fh^{\hV}_T$\index{hVT@$\fh^V_T$ and $\fh^{\hV}_T$!On $\Omega_1''-L$} on $S'$ and $\hS'$ in \eqref{hvt} using the extension of $\bfR_T$ constructed above.

By repeating the same argument as in \Cref{int3}, we see that the form 
\[
\btiT_\tau^\mL(\A',\fh^{\hV}_T,h^{\bfH}_{T,R,L^2},h^{\bfH_{\Cb^m}}_{T,R,L^2})
\]
extends smoothly over $S$.

Finally, applying the same reasoning as in the case of $\Omega_0''$, we obtain the proof of \Cref{prop63} on $\Omega_1''$.

\end{proof}

		From here on, and for the remainder of \cref{sec122}, we fix $R=R_0$.
		
		Proceeding as we did  {in \eqref{eq185}}, we get \be\label{eq195}\lim_{T\to\infty}\left(\overline{\T}_\tau^{\mL}(\A^{\prime,-},\bh_{T,R}^{\hV^-})-\overline{\T}_\tau^{\mL}(\A^{\prime,-},\fh_{T,R}^{\hV^-})\right)=0.\ee
		
		By \eqref{eq193}, proceeding as in Theorem \ref{int3}, we have {on $\Omega_1''$}
		\be\label{eq196}
		\overline{\T}_\tau^{\mL}(\A^{\prime,-},\fh_{T,R}^{\hV^-})=	\overline{\T}_\tau^{\mL}(\A^{\prime},\fh_{T}^{\hV})
		\ee

		As $\mathbf{P}^{>1}|_{(\H_{T,R}^{\bZ^-})^{l>1}}:(\H_{T,R}^{\bZ^-})^{l>1}\to  \H_{T,R}$ is an isomorphism, proceeding as in Lemma \ref{urhou}, for any harmonic form $w\in(\H_{T,R}^{\bZ^-})^{l>1}$ with unit $L^2$-norm,
		\be\label{eq197}
		\int_Z|\mathbf{P} u-\I u |^2=O(e^{-cRT}).
		\ee

{Note that $\bfH^{\leq 1}=0$ by \eqref{eq157}.} Let $h_{T,R,L^2,-}^{\bfH}$ be the metric on $\bfH$ induced by $h_{T,R}^{\cE^-}$ and the isomorphism $J_T \circ \Phi_T : \bfH^{-,>1}_T \to \bfH$, {with $\Phi_T$ in \eqref{Phi-T} and $J_T=\K_1\J_T\K_T^{-1}\colon \bfH_{T}\overset{\sim}{\rightarrow}\bfH$, where $\J_T$ and $\K_T$ are defined in \eqref{defjt} and \eqref{defnkt} respectively. By the definition of $h_{T,R,L^2}^{\bfH}$ below \eqref{torsion-with-prime-metric-and-nonprime-metric} and} \eqref{eq197}, we have
\begin{equation}\label{minus-no minus}
h_{T,R,L^2}^{\bfH} = h_{T,R,L^2,-}^{\bfH} (1 + O(e^{-CRT})).
\end{equation}

Similarly to \eqref{eq200}, we can prove that
\begin{equation}\label{eq200-}
\left\| \left(h^{\bfH}_{T,R,L^2,-}\right)^{-1} \left( \nabla^{\bfH} h^{\bfH}_{T,R,L^2,-} \right) \right\| = O(TR).
\end{equation}

Since we can naturally identify $H^{V^-,>1}$ with $H^V$, let $\bh^V_{T,R,L^2,-}$ (resp. $\fh^V_{T,R,L^2,-}$) denote the {Hodge} metrics on $\bfH$ induced by $\bh_{T,R}^{V^-}$ {in \eqref{def-bh-TRV-}} (resp. $\fh_{T,R}^{V^-}$ {in \eqref{hvtr}}), via the identification above and $\J$ {in Theorem \ref{thm54}}.

By \eqref{J-J-minus}, we have
\begin{equation}
h^{\bfH}_{T,R,L^2,-} = \bh^V_{T,R,L^2,-}.
\end{equation}

From \eqref{eq193}, and the construction of $\fh^V$ (in particular, see \eqref{def-extension-bfR_T}), we also have
\begin{equation}
\fh^V_{T,L^2} = \fh^V_{T,R,L^2,-}.
\end{equation}

By Lemma~\ref{1010r} and \eqref{hvtr}, it follows that
\begin{equation}\label{eq198}
\bh^{V^-}_{T,R,L^2,-} = \fh^{V^-}_{T,R,L^2,-} (1 + O(e^{-cT})).
\end{equation}

Using \eqref{RTR-norm} together with \eqref{minus-no minus}--\eqref{eq198}, and proceeding as in the end of \Cref{sec121}, we conclude that the estimate \eqref{eq191} also holds on $\Omega_1''$.

		The second estimate in Theorem \ref{int2} then follows from Definition \ref{assotor}, \eqref{comp-torsion-Mminus-and-torsion-bh_TR}, \eqref{eq195}, \eqref{eq196} and \eqref{eq191}.

		\section{More Than One Birth-death Points}\label{ineresult1}
		
		Here we do not assume $S_2=\emptyset$  anymore, {but we still assume for now that $\bbf_\s$ has at most 2 birth-death points for $\s\in S$}. Following the same arguments as in \cref{ineresult}, we only need to replace the three theorems (Theorem \ref{int1}-Theorem \ref{int2}) in \cref{ineresult} for $\bbf_{T,R}$ with those for $\bbf_{2,T,R,R'}$. Let us give more details.
		
		We will use the notation introduced in \cref{ineresult}. For brevity, we let $\bbf_{T,R,R'} := \bbf_{2,T,R,R'}$.

		Let $h_{T,R,R'}^\cF:=e^{-2\bbf_{T,R,R'}}h^{\cF}$. Let $h^{\cE}_{T,R,R'}$ be the metric induced by $T^H\M$, $g^{T\bZ}$ and $h_{T,R,R'}^{\cF}$. Let $h^{\bfH}_{T,R,R',L^2}$ be the metric on $\bfH$ induced by Hodge theory and $h^{\cE}_{T,R,R'}$.

		
		Let $R_0,R_0'$ be large enough. We will set $h^{\cF}_T:=h^{\cF}_{T,R=0,R'=0}$ and $h^{\cF}_{T,R}:=h^{\cF}_{T,R,R'=0}$.

		Theorem \ref{int1}-Theorem \ref{int2} are replaced by:
		
		\begin{thm}\label{int1r}
			On $\Omega_0''$, for $R=R_0$ and $R'=R_0'$,
			\[\lim_{T\to\infty}\left(\overline{\T}_\tau^{\mL}(T^H\M,g^{T\bZ},h_{T,R,R'}^\cF)-\overline{\T}_\tau^{\mL}(T^H\M,g^{T\bZ},h_{T}^\cF)\right)=0.\]	
		\end{thm}

		\begin{thm}\label{int20r}
			On $\Omega_1''$, for $R=R_0$ and $R'=R_0'$,
			\begin{align*}
				\lim_{T\to\infty}\left(\mathcal{\overline{T}}_{\tau}^\mL\left(T^H \M, g^{T \bZ}, h^{\cF}_{T,R,R'}\right)-\mathcal{\overline{T}}_\tau^{\mL}\left(T^H\M^-, g^{T \bZ}|_{\bZ^-}, h^\cF_{T,R}\right)\right)=0.
			\end{align*}
			
			Here, {$\M^-\to\Omega_1''$ is defined as in \cref{sec53} by removing fiberwisely one ball of radius $\frac{r_1+r_2}{2}$, using Observation \ref{obs7} to `choose' coherently the ball over $\Omega_1''\cap\Omega_2$, and $\T_{\tau}^\mL\left(T^H\M^-, g^{T \bZ}|_{\bZ^-}, h^{\cF}_{T,R}\right)$ is the torsion form for the fibration $\M^- \to \Omega_1''$ with relative boundary conditions.}

			On $\Omega_2''$, for $R=R_0$ and $R'=R_0'$,
			\begin{align*}
				\lim_{T\to\infty}\left(\mathcal{\overline{T}}_{\tau}^\mL\left(T^H \M, g^{T \bZ}, h^{\cF}_{T,R,R'}\right)-\mathcal{\overline{T}}_\tau^{\mL}\left(T^H\M^{--}, g^{T \bZ}|_{\bZ^{--}}, h^\cF_{T,R,R'}\right)\right)=0.
			\end{align*}

			Here, $\M^{--}$ and $\bZ^{--}$ are defined as in \cref{sec53} but removing fiberwisely two balls of radius $\frac{r_1+r_2}{2}$, and $\T_{\tau}^\mL\left(T^H\M^{--}, g^{T \bZ}|_{\bZ^{--}}, h^{\cF}_{T,R,R'}\right)$ is the torsion form for the fibration $\M^{--} \to \Omega_2''$ with relative boundary conditions.  
		\end{thm}
		
		\begin{thm}\label{int2r}
			On $\Omega_0''$, for $R=R_0$ and $R'=R_0'$, 
			\[\lim_{T\to\infty}\left(\overline{\T}_\tau^{\mL}(T^H\M,g^{T\bZ},h_{T}^\cF)-\btiT_\tau^\mL(\A',\fh_T^V,h^{\bfH}_{T,R,R',L^2},h^{\bfH_{\Cb^m}}_{T,R,R',L^2})\right)=0;\]	
			on $\Omega_1''$,   \[\lim_{T\to\infty}\left(\overline{\T}_\tau^{\mL}(T^H\M^-,g^{T\bZ}|_{\bZ^-},h_{T,R}^\cF)-\btiT_\tau^\mL(\A',\fh_T^V,h^{\bfH}_{T,R,R',L^2},h^{\bfH_{\Cb^m}}_{T,R,R',L^2})\right)=0;\]
			on $\Omega_2''$,   \[\lim_{T\to\infty}\left(\overline{\T}_\tau^{\mL}(T^H\M^{--},g^{T \bZ}|_{\bZ^{--}},h_{T,R,R'}^\cF)-\btiT_\tau^\mL(\A',\fh_T^V,h^{\bfH}_{T,R,R',L^2},h^{\bfH_{\Cb^m}}_{T,R,R',L^2})\right)=0.\]	
		\end{thm}
		{Since there is no essential difference or difficulty in extending all estimates in \cref{conag} and \cref{sec9} to the case where the function may have more than one birth-death point}, proofs of Theorems \ref{int1r}-\ref{int2r} are essentially identical to those of Theorems \ref{int1}-\ref{int2}.

		Finally, when $\bbf_\s$ has at most $ \fk $ birth-death points for $ \s \in S $ and $ \fk > 2 $, by Remark \ref{rmkkpts}, we just need to replace the three theorems above stated for the two associated non-Morse functions with those for $ \fk $ associated non-Morse functions. Again, the proofs of these results are essentially identical to those of Theorems \ref{int1}-Theorem \ref{int2}.
		
		\def\appen{1}
		\if\appen0
		\appendix
		\section{On Thom-Smale Transversality Conditions}
		
		In this appendix, we will briefly discuss a strategy for the removal of the fiberwise Thom-Smale transversality condition. Future exploration will be in our subsequence paper.
		
		In \cref{seca1}, we provide a concise exposition of Igusa-Klein's higher algebra construction, as detailed in \cite{igusa2002higher}, which does not require fiberwise Thom-Smale transversality conditions. We also briefly explain Goette's construction in \cref{seca2}, which assumes the existence of a fiberwise Morse function. According to \cite{goette2001morse}, this approach enables the expression of topological higher torsion as differential forms without assuming fiberwise Thom-Smale transversality conditions, inspired by the higher algebra construction.
		\subsection{Igusa-Klein's Higher Algebra Construction}\label{seca1}
		We adopt the notation introduced in \cref{combitor} and \cref{ineresult}. Igusa and Klein's Construction is applicable to fiberwise generalized Morse functions. Here, for simplicity, we assume that $ \bbf $ is fiberwise Morse.  If $ (\bbf, g^{T\bZ}) $ further satisfies the fiberwise Thom-Smale transversality condition, for any small circle $ \gamma: [0,1] \to S $ with $ \gamma(0) = \gamma(1) = \s \in S $, the parallel transport (with respect to $ \nabla^V $) of the chain complex $ (V, \partial) $ along $ \gamma $ yields the trivial endomorphism (the identity) of the cochain complex $ (V, \partial) $ at $ \s $.
		
		Before moving on, we introduce the following definition, which appears in \cite{goette2001morse}.
		Recall that $\bbf^V$ is introduced in \cref{sec54},
		\begin{defn}[$\bbf^V$-upper triangular]
			An endomorphism of $V$ is called $\bbf^V$-upper triangular if it maps each $\lambda$-eigenvector of $\bbf^V$ to the sum of the $\mu$-eigenspaces with $\mu>\lambda$.
		\end{defn}

		Now if we drop the fiberwise Thom-Smale transversality condition, we can still show, as in \cite[$\S$7.3]{DY2020cohomology}, that all stable and unstable manifolds intersect in a compact subset of $\bZ$ fiberwisely. Therefore, for a generic fiberwise Riemannian metric $ g^{T \bZ} $ (such that $ g^{T\bZ} = g^{T\bZ'} $ outside some compact subset), the fiberwise gradient $ \nabla^v \bbf $ will satisfy the Thom-Smale transversality condition over an open dense subset of $ S $. In this case, the parallel transport (with respect to $ \nabla^V $) of the cochain complex $ (V, \partial) $ along $ \gamma $ may results in a nontrivial endomorphism of $ (V, \partial) $ at $ \s $. However, this endomorphism is homotopic to the identity through an $ \bbf^V $-upper triangular homotopy.Such homotopies are themselves related by $ \bbf^V $-upper triangular higher homotopies, and so on.

		{
			Let us give more details. Take a sufficiently small $k$ simplex $\sigma$ in the base $S$. Since $\bbf$ is Morse over every point in $\sigma$, the critical set forms a disjoint union of sheets\be\label{iso-appendix}
			\Sigma(\bbf|_\sigma) \cong \Sigma\left(\bbf_b\right) \times \Delta^k
			\ee
			where $b \in \sigma$ is, say, the barycenter.
			Moreover, by the density of the locus where the Thom-Smale condition is satisfied, we may assume that 
			\begin{cond}
				If $\s$ is the barycenter of a subsimplex of $\sigma$, then $(\bbf|_{\s},g^{T\bZ}|_{\bZ_\s})$ satisfies the Thom-Smale transversality conditions.
			\end{cond}
			
		}
		
		By the assumption above, over each vertex $v_i$ of $\sigma$ we obtain, as in \cref{combitor}, a cochain complex $(V_i,\phi_0(i)):=(V,\p)|_{v_i}$.
		
		Note that $\phi_{0}(i)$ is an endomorphism of $V_i$ satisfying the following
		\begin{enumerate}[(1)]
			\item $\phi_{0}(i)$ is homogeneous of degree $-1$ and $\phi_{0}(i)$ is $\bbf^V$-upper triangular.
			\item $\big(\phi_{0}(i)\big)^2=0$
		\end{enumerate}
		Over each edge $\left(v_i, v_j\right)$ in $\sigma$ with $i<j$ we have an upper triangular cochain isomorphism
		$$
		\Phi_1(i, j): (V_j,\phi_{0}(j)) \xrightarrow{\approx}(V_i,\phi_{0}(i)) .
		$$
		
		Via \eqref{iso-appendix}, this cochain isomorphism can be written as
		$$
		\Phi_1(i, j)=1+\phi_1(i, j).
		$$
		
		Then $\phi_1(i, j)$ will be a $\bbf^V$-upper triangular endomorphism of $V$ of degree 0 satisfying the equation
		$$
		\phi_0(i) \phi_1(i, j)-\phi_1(i, j) \phi_0(j)=\phi_0(j)-\phi_0(i)
		$$
		
		Given three vertices $v_i, v_j, v_k$ with $i<j<k$, we have upper triangular cochain homotopy
		$$
		\phi_2(i, j, k): \Phi_1(i, j) \Phi_1(j, k) \simeq \Phi_1(i, k) .
		$$
		
		This can be written in the form
		$$
		\begin{aligned}
			&\ \ \ \ \phi_0(i) \phi_2(i, j, k)-\phi_1(i, j) \phi_1(j, k)+\phi_2(i, j, k) \phi_0(k) \\
			& =\phi_1(j, k)-\phi_1(i, k)+\phi_1(i, j)
		\end{aligned}
		$$
		Keep on doing this, we obtain an $A_\infty$ higher algebra structure associated to $\{\phi_k\}_{k\geq0}$ above.

		\subsection{Goette's construction of higher Igusa-Klein torsion}\label{seca2}
		
		Assuming the existence of fiberwise Morse functions, the bundle $V \to S$ as described in \cref{combitor} is well-defined as a vector bundle. If we further assume that $(\bbf, g^{T\bZ})$ satisfies the fiberwise Thom-Smale transversality condition, as in \cref{combitor} or \cite[$\S$5]{bismut2001families}, one obtains a well-defined differential $\partial: V \to V$ and a superconnection $\A' := \nabla^V + \partial$. However, if we drop the fiberwise Thom-Smale transversality condition, a globally well-defined differential $\partial$ no longer exists. Instead, one has endomorphisms $\{\phi_k\}_{k \geq 0}$ of $V$ associated with higher homotopies for each simplex $\sigma$. In \cite[$\S$1]{goette2001morse}, given a ``fine'' simplicial structure on $S$, Goette is able to ``glue" all these $\{\phi_k\}_{k \geq 0}$ and construct a smooth superconnection on $V\to S$:
		\[ 
		\A^{\prime} = a_0 + \left(\nabla^V + a_1\right) + a_2 + \ldots \quad\text{with } a_k \in \Omega^k\left(S; \mathrm{Hom}(V^{\bullet}, V^{\bullet+1-k})\right).
		\]
		Here, $a_k$ is closely related to $\phi_k$ in Igusa-Klein's construction. 
		
		
		Then, Goette is able to construct a combinatorial torsion form $\mathcal{T}(\A', V)$ (see \cite[Definition 2.46]{goette2001morse}) similar to those in \cite{bismut1995flat,bismut2001families}. However, we would like to emphasize that since $a_k$ for $k \geq 2$ is nontrivial, \cite[Proposition 2.18]{bismut1995flat} does not hold. As a result, we cannot directly replace the integral $\int_\tau^\infty$ in \Cref{defn26} with $\int_0^\infty$ . Consequently, the torsion form $\mathcal{T}(\A', V)$ is not constructed exactly the same way as in \cite{bismut1995flat,bismut2001families}.
		
		Goette then extends the arguments from \cite{bismut2001families} in \cite[$\S$4]{goette2001morse} to compare combinatorial torsion form $\T(\A',V)$ with Bismut-Lott torsion, and then compares combinatorial torsion form $\T(\A',V)$ with Igusa-Klein torsion in \cite{goette2003morse}, requiring only the existence of a fiberwise Morse function.

		\subsection{Without fiberwise Morse functions and Thom-Smale transversality conditions}
		
		Here we will briefly explain how to extend Goette's construction to generalized Morse functions.
		
		We will adopt the settings described in \cref{combitor}, with the exception of not assuming fiberwise Thom-Smale transversality condition. Consider $ f: M \to \mathbb{R} $ as a fiberwise generalized Morse function, and let $ (\M, \bbf) $ denote a double suspension of $ (M, f) $. Let $ g^{T\bZ} $ be a metric satisfying \Cref{5a} and Observation \ref{5c}. Let $S':=S-L$, then $\bbf_\s$ is Morse for any $\s\in S'$. Then we have vector bundle $V\to S'$ as in \cref{combitor} (the rank of $V$ may be different in different connected components of $S'$).
		
		Take a sufficiently small $k$ simplex $\sigma$ in the base $S$.  If $\sigma\cap L\neq\emptyset$, we assume that $\sigma$ intersects $L$ transversely. Then we delete all components of $\Sigma(\bbf|_\sigma)$ which intersects $\Sigma^{(1)}(\bbf)$ (actually, more components need to be deleted, see \cite[pp.156]{igusa2002higher}). Denote the remaining set of sheets by ${\Sigma_\sigma}$. One can still associate higher algebra structures $\{\phi_k\}_{k\geq0}$ to $\Sigma_\sigma$ as in \Cref{seca1} if $\sigma$ is compatible with $(\bbf,g^{T\bZ})$ in the following sense:
		
		\begin{defn}\label{good-simplex}
			We say a simplex $\sigma$ is compatible with $(\bbf,g^{T\bZ})$, if it satisfies the following conditions.
			\begin{itemize}
				\item If $\sigma \cap L \neq \emptyset$, $\sigma$ intersects $ L $ transversely.
				\item If $\sigma \cap L \neq \emptyset$, we can choose $\sigma$ to be sufficiently small so that birth-death points as well as points collapsing to a birth-death point are independent of any other critical points.
				
			\end{itemize}
		\end{defn}
		
		With these higher algebraic structures, Igusa and Klein define higher Reidemeister-Fraz torsion (Igusa-Klein torsion) using higher K-theory. More precisely,
		\begin{thm}[Igusa \cite{igusa2002higher}]
			The (filtered) higher algebraic structures constructed above for each fiberwise framed function $ \bbf $ lead to a classifying map
			$$
			\xi_\bbf(\M / S ; \cF): S \longrightarrow \mathrm{W h}(R, G)
			$$
			that is unique up to homotopy.
			Here, $ R := M_r(\C) $, $ G := U_r(\C) $ and $ \mathrm{Wh}(R,G) $ denotes the higher Whitehead space associated with the algebra $ R $ and group $ G $ (refer to \cite{igusa2002higher} for further details).
		\end{thm}
		
		Assume that $F$ carries a flat Hermitian metric. Then Igusa constructs cohomology classes
		$$
		\tau=\sum_{k=1}^{\infty} \tau_{2 k} \quad \text { with } \quad \tau_{2 k} \in H^{4 k}\left(\mathrm{Wh}(R, G) ,\mathbb{R}\right).
		$$
		
		\begin{defn}
			Under some further assumption on the cohomological bundle $H\to S$,
			the Igusa-Klein torsion of a smooth fiber bundle $\pi: M \rightarrow S$ (or its double suspension $\ppi:\M\to S$) is defined as
			$$
			\tau(M / S ; F)=\xi(\M / S ; \cF)^* \tau \in H^{4 \bullet}(S ,\mathbb{R}) .
			$$
		\end{defn}

		Given a ``fine" simplicial structure on $ S $ where each simplex is sufficiently small and compatible with $(\bbf, g^{T\bZ})$, we can still ``glue" all these $\{\phi_k\}_{k \geq 0}$ and construct a smooth superconnection on $ V \to S' $ as in \cite{goette2001morse}:
		\[ 
		\A^{\prime} = a_0 + \left(\nabla^V + a_1\right) + a_2 + \ldots \quad\text{with } a_k \in \Omega^k\left(S; \mathrm{Hom}(V^{\bullet}, V^{\bullet+1-k})\right).
		\]
		
		Next, noting from a result similar to \Cref{obs5} that the contribution of deleted components vanishes in some sense, we could construct a combinatorial torsion form $\mathcal{T}(\A', V)$ on $ S' $, similar to those in \cite{goette2001morse}, and this form could then be extended smoothly to $S$ as in \cref{combitor}.
		
		By extending the arguments in this paper similarly to what Goette did in \cite[$\S$4]{goette2001morse}, we expect to derive a comparison between combinatorial torsion forms $\T(\A',V)$ and Bismut-Lott torsion, and then compare combinatorial torsion forms $\T(\A',V)$ with Igusa-Klein torsion as in \cite{goette2003morse}. Eventually, we expect to establish a fully generalized higher Cheeger-M\"uller/Bismut-Zhang theorem.
		
		\fi
		\bibliography{lib}
		
		\bibliographystyle{plain}

        \printindex
        \printindex[p2]
		
	\end{document}